\documentclass[aap]{imsart}

\startlocaldefs
\RequirePackage{amsthm,amsmath,amsfonts,amssymb}
\RequirePackage[numbers,sort&compress]{natbib}
\theoremstyle{plain}

\newtheorem{claim}{Claim}[section]
\newtheorem{theorem}{Theorem}[section]
\newtheorem{lemma}[theorem]{Lemma}
\newtheorem{corollary}[theorem]{Corollary}
\newtheorem{proposition}[theorem]{Proposition}
\theoremstyle{remark}
\newtheorem{condition}{Condition}[section]
\newtheorem{remark}{Remark}

\usepackage{amssymb}
\usepackage{enumerate}
\usepackage{enumitem}
\usepackage[OT1]{fontenc}
\usepackage{bm}
\usepackage{multirow}
\usepackage{diagbox}
\usepackage{amsmath,amssymb,mathrsfs}
\usepackage{xcolor}
\usepackage{bbm, dsfont}
\usepackage{upgreek}
\usepackage{mathtools}
\usepackage{algorithm}
\usepackage{algorithmic}
\usepackage{smile}
\usepackage{multirow,stfloats,booktabs}
\usepackage[english]{babel}
\usepackage{graphicx} 
\usepackage{subcaption}
\usepackage{makecell}
\RequirePackage[colorlinks,citecolor=blue,urlcolor=blue]{hyperref}

\def\tr{\mathop{\text{tr}}\kern.2ex}

\def\supp{\mathop{\text{supp}}}

\long\def\comment#1{}

\def\tr{\mathop{\text{Tr}}}

\def\cS{{\mathcal{S}}}

\newcommand{\bel}{\begin{eqnarray}\label}
	\newcommand{\eel}{\end{eqnarray}}
\newcommand{\bes}{\begin{eqnarray*}}
	\newcommand{\ees}{\end{eqnarray*}}
 \def\Corr{\operatorname{Corr}}

\let\emptyset\varnothing
\let\hat\widehat
\let\tilde\widetilde

\def\Var{\operatorname{Var}}

\def\EE{{\mathbb E}}

\def\supp{\mathop{\text{supp}\kern.2ex}}

\def\argmax{\mathop{\text{\rm arg\,max}}}

\def\tr{{\rm{Tr}}}

\def\supp{\mathop{\text{supp}}}

\def\tr{\mathrm{Tr}}

\usepackage{mathrsfs}

\def\sep{\;\big|\;}

\usepackage{tabularx,array}

\usepackage{hyperref}
\usepackage[    protrusion=true,
expansion=true,
final,
babel
]{microtype}
\def\##1\#{\begin{align}#1\end{align}}
\def\$#1\${\begin{align}#1\end{align}}
\usepackage{enumitem}
\newcommand\numberthis{\addtocounter{equation}{1}\tag{\theequation}}
\theoremstyle{plain}

\usepackage{soul}
\theoremstyle{mytheoremstyle}

\usepackage{lineno}

\def\cN{{\mathcal{N}}}

\def\cL{\mathcal{L}}
\def\EE{\mathbb{E}}
\def\cE{\mathcal{E}}
\def\PP{\mathbb{P}}
\def\Ib{\mathbf{I}}

\def\RR{\mathbb{R}}

\endlocaldefs
\begin{document}
	
\begin{frontmatter}
\title{Anti-Concentration Inequalities for the Difference of Maxima of Gaussian Random Vectors}
%\title{A sample article title with some additional note\thanksref{t1}}
\runtitle{Anti-Concentration for Gaussian maximum difference}
%\thankstext{T1}{A sample additional note to the title.}

\begin{aug}
%%%%%%%%%%%%%%%%%%%%%%%%%%%%%%%%%%%%%%%%%%%%%%%
%% Only one address is permitted per author. %%
%% Only division, organization and e-mail is %%
%% included in the address.                  %%
%% Additional information can be included in %%
\author[A]{\fnms{Alexandre}~\snm{Belloni}\ead[label=e1]{abn5@duke.edu}},
\author[B]{\fnms{Ethan X.}~\snm{Fang}\ead[label=e2]{ethan.fang@duke.edu}}
\and
\author[A,B]{\fnms{Shuting}~\snm{Shen}\ead[label=e3]{ss1446@duke.edu}}
%%%%%%%%%%%%%%%%%%%%%%%%%%%%%%%%%%%%%%%%%%%%%%
%% Addresses                                %%
%%%%%%%%%%%%%%%%%%%%%%%%%%%%%%%%%%%%%%%%%%%%%%
\address[A]{Fuqua School of Business,
Duke University\printead[presep={,\ }]{e1,e3}}

\address[B]{Department of Biostatistics \& Bioinformatics,
Duke University\printead[presep={,\ }]{e2}}

%\address[C]{WW FBA, Amazon{}}
\end{aug}

\begin{abstract}
We derive novel anti-concentration bounds for the difference between the maximal values of two Gaussian random vectors across various settings. Our bounds are dimension-free, scaling with the dimension of the Gaussian vectors only through the smaller expected maximum of the Gaussian subvectors. In addition, our bounds hold under the degenerate covariance structures, which previous results do not cover. In addition, we show that our conditions are sharp under the homogeneous component-wise variance setting,  while we only impose some mild assumptions on the covariance structures under the heterogeneous variance setting. We apply the new anti-concentration bounds to derive the central limit theorem for the maximizers of discrete empirical processes. Finally, we back up our theoretical findings with comprehensive numerical studies.
\end{abstract}

\begin{keyword}[class=MSC]
\kwd[Primary ]{60E15}
\kwd{60G15}
\kwd[; secondary ]{62E20}
\end{keyword}

\begin{keyword}
\kwd{Anti-concentration}
\kwd{L\'evy concentration function}
\kwd{Gaussian maximum difference}
\kwd{Gaussian comparison}
\kwd{central limit theorem}
\kwd{bootstrap approximation}
\end{keyword}

\end{frontmatter}
	\section{Introduction}\label{sec: intro}
The concentration patterns of random variables are pivotal in empirical process theory, and find applications to different communities such as statistics \cite{chen2020robinf}, probability~\cite{Chernozhukov2013ComparisonAA}, and discrete mathematics \cite{kwan2023anticonramsy}. For example, in high-dimensional statistical analysis, the celebrated work \citep{cck2013aos} derives the anti-concentration bounds of the maximum of Gaussian random vectors to establish central limit theorems for 
 % various statistics,
 % including 
 the maximum of a sum of high-dimensional random vectors. Such results provide the foundation for constructing confidence regions through Gaussian and other bootstrap approximations under the high-dimensional setting, which find many important applications \cite{Chernozhukov2013ComparisonAA, CHERNOZHUKOV20163632, cck2017aop, Deng2017BeyondGA, belloni2018high}. This motivated several other works to generalize the results under different settings \citep{Deng2017BeyondGA,zhou2018gaussiancomprandmat,cck2022improvedbootstrap,kato2023wildbtstrp,liu2023lagrangian,zhou2023smoothedquantreg}.

% In high-dimensional analysis, the construction of confidence regions through Gaussian and other bootstrap approximations attracts substantial interest following the seminal work on high-dimensional central limit theorems for statistics including the maximum of high-dimensional random vectors \citep{cck2013aos}. In particular, such results unlock the construction of simultaneous confidence intervals even when there are more parameters than the sample size \cite{Chernozhukov2013ComparisonAA, CHERNOZHUKOV20163632, cck2017aop, Deng2017BeyondGA, belloni2018high}. Several other works further generalize the results under different settings/applications \citep{Deng2017BeyondGA,zhou2018gaussiancomprandmat,cck2022improvedbootstrap,kato2023wildbtstrp,zhou2023smoothedquantreg}. 

%As discussed in \citep{cck2013aos}, to establish such Gaussian or other approximation results, a key step is to derive 
Our focus is to study the anti-concentration bounds that characterize how fast different statistics concentrate,
%To measure how fast such concentration is, 
and we rely on the L\'evy concentration function~\citep{levy1954} to characterize the concentration, which is defined as 
$$ \mathcal{L}(Y, \varepsilon) := \sup_{t\in \mathbb{R}} \mathbb{P}(|Y-t|\leq \varepsilon),$$
where $\varepsilon\in\RR_+$, and $Y$ is the random variable of interest.
For example, suppose that $Y$ is the maximum of a Gaussian random vector. \cite{CHERNOZHUKOV20163632} shows that $ \mathcal{L}(Y, \varepsilon) \leq 2\varepsilon (\sqrt{ 2\log p} + 2) / \underline{\sigma}$, where
% $C>0$ is a constant, and
$\underline{\sigma}^2$ is the minimal component-wise variance of the Gaussian random vector. 
Other works derive sharper bounds for such anti-concentration results in different forms. Specifically, \cite{Chernozhukov2013ComparisonAA,Kuchibhotla2021cltimp,Giessing2023AnticoncentrationOS} derive dimension-free bounds that depend on the expectation or the variance of the maximum component of the Gaussian random vector, allowing for generalization of the results to infinite-dimensional scenarios.

We establish new anti-concentration bounds for the difference of the maxima of two Gaussian random vectors. Let $X = (X_1,...,X_p)^\top \in \RR^{p}$ be a Gaussian random vector with marginal means
 $\{\mu_i\}_{i=1}^p$ and variances $\{\sigma^2_i\}_{i=1}^p$,  $\{\cA,\cB\}$ be a nontrivial partition of $[p] = \{1,\ldots, p\}$ such that both $\cA$ and $\cB$ are nonempty,
 % with $1 \le |\cA| \le p-1$. 
 and $(M_{\cA}, M_{\cB}) =  (\max_{i \in \cA} X_i, \max_{j \in \cB} X_j)$. We aim to bound the L\'{e}vy concentration function of $M_{\cB} - M_{\cA}$, namely: 
\begin{equation}\label{eq:basic:bound}
\mathcal{L}(M_{\cB} - M_{\cA}, \varepsilon) =  \sup_{t \in \RR}\PP(|M_{\cB} - M_{\cA} - t| \le \varepsilon), \quad \text{for any   } \varepsilon > 0.
\end{equation}
This difference in maxima appears naturally when we are interested in the maximizers of some empirical processes \citep{imaizumi2021gaussian}.  We also note that \eqref{eq:basic:bound}  plays a crucial role in establishing the validity of bootstrap approximations for the distribution of the maxima difference of sums of (non-Gaussian) random vectors (see Section~\ref{sec: application}).
% Note that to establish central limit theorems using Gaussian approximation techniques for smooth functions, we may follow the framework in previous works \cite{cck2013aos, cck2022improvedbootstrap} to establish central limit theorems for this difference, where an appropriate anti-concentration bound specifically for the Gaussian maximum difference is unknown to the best of our knowledge. This paper bridges this gap by developing a new anti-concentration inequality tailored to this scenario.
% ,\st{ but to the best of our knowledge, existing work does not establish anti-concentration bounds for this case despite its importance. } 
% {\red Indeed a key step establishing this is the anti-concentration bounds we are studying here.}

To establish the anti-concentration bounds, we  bound the density function of the difference of the maxima of two Gaussian random vectors.
% under weak assumptions.
%Our result generalizes existing works under strictly weaker assumptions. In particular, 
Under the assumption that the minimum eigenvalue of the covariance matrix of $X$ is strictly positive, \cite{imaizumi2021gaussian} develops a conditional anti-concentration inequality, where they bound the density by conditioning on one of the Gaussian random vectors and applying the anti-concentration bound for a single Gaussian maximum in \cite{CHERNOZHUKOV20163632} to the other Gaussian random vector. Thus, the bound scales with the minimal eigenvalue of the full covariance matrix.
Although the minimum eigenvalue is bounded away from zero in many settings, as we will discuss later, this assumption does not hold in some applications of our interest.  Our result does not require such a minimal eigenvalue assumption. Indeed, even if the covariance matrix has zero eigenvalues, 
%where the conditional anti-concentration bounds become infinite by the earlier work, 
the L\'evy concentration function remains well-behaved.\footnote{ Let $\xi_1,\xi_2$ be  i.i.d. $\cN(0,1)$ random variables and define $X_1 = \xi_1$, $X_2 = \xi_2$, $X_3=\frac{1}{\sqrt{2}}(\xi_1-\xi_2)$ and $X_4 = \frac{1}{\sqrt{2}}(\xi_1+\xi_2)$. Then $\sigma_i^2=1, i=1,\ldots,4$ and $\lambda_{\min}(\mathbb{E}[XX^{\top}])=0$, while ${ \sup_{t\in\mathbb{R}}} \ \mathbb{P}\big( | \max\{ X_3,X_4\} - \max\{X_1,X_2\}  - t |\leq \varepsilon\big) \leq  { (8/\sqrt{\pi})} \varepsilon$.} One such setting of interest is when $X_\cA=\Gamma_\cA Z$ and $X_\cB=\Gamma_\cB Z$, where $Z$ is a $d$-dimensional Gaussian random vector and $\Gamma_\cA,\Gamma_\cB$ are two $p\times d$ matrices with $p>d$. In this case, the minimal eigenvalue assumption is violated, rendering prior results inapplicable. In comparison, our results still hold, which sheds new insights into anti-concentration patterns.

\subsection{Major contributions}
Our primary theoretical contribution is to establish new anti-concentration bounds for the difference between the maxima of two Gaussian vectors. { 
A significant challenge in this context is the non-convexity of the function that maps a vector to the difference of the maximum values in two parts of the vector,}
%its maximum difference between partitions, 
rendering existing techniques for establishing anti-concentration bounds inapplicable. To address this, we develop new proof strategies. Central to our approach is a new smoothing technique, where we transform the density into disjoint integrals over centered Gaussian variables, and then utilize the derivative of a conditional probability for Gaussian random vectors as an intermediate bound to achieve the final anti-concentration bound.

Our new anti-concentration bounds hold under generic conditions.  The bound is dimension-free and depends on the minimum of the expectation of the two Gaussian maxima.
Indeed, if one of the maxima concentrates fast enough due to the much higher subset dimension, this maximum behaves essentially like a constant, and the anti-concentration is governed by the variation in the other maximum. 
% Furthermore, we explore applications of our results to various settings. 
Under the condition that all components have equal variances, i.e., \(\sigma_i=\sigma\) for all \(i \in [p]\), our conditions are broadly applicable. In particular, we have that
%when  there is no positive perfect correlation between the two subsets $\cA$ and $\cB$, the following bound holds that
\begin{equation}\label{eq: bound normalize}
    \cL(M_{\cB} - M_{\cA}, \varepsilon ) \le C \cdot \frac{\varepsilon /{ \sigma} }{1 -\bar\rho}\cdot\min\big\{\EE \big(\max_{i \in \cA } |X_i - \mu_i|  { /\sigma }\big) , \EE \big(\max_{j \in \cB} |X_j - \mu_j|/ \sigma \big)\big\} ,
\end{equation}
where \(C > 0\) is a constant independent of \(p\), and \(\bar\rho = \max_{i \in \cA, j \in \cB} \Corr(X_i,X_j)\) is the highest correlation across the subsets.
% {\red between-partition }. 
In the case of unequal marginal variances, our bound scales with the reciprocal of \(\min_{i\in\mathcal{A}, j\in \mathcal{B}}\sigma_i - \sigma_{ij}/\sigma_i\) or \(\min_{j\in\mathcal{B}, i\in \mathcal{A}}\sigma_j - \sigma_{ij}/\sigma_j\) (see Theorem~\ref{thm: anti con} for details). { This term is similar to the minimal marginal standard deviation in the anti-concentration bounds for the single Gaussian maximum \cite{Chernozhukov2013ComparisonAA}. }
% our bounds also scale based on the reciprocal of the minimum value of \(1- \Corr(X_i,X_j)\), for \(i \in \cA, j \in \cB\). 
For the case where \(\Corr(X_i, X_j) = 1\) for some \(i \in \cA, j \in \cB\), a bound in the form of \eqref{eq: bound normalize} cannot hold. Instead, we show that the L\'evy concentration function satisfies \(\mathcal{L}(M_{\cB} - M_{\cA}, \varepsilon) \leq C_X \varepsilon + \zeta_X\),{ where \(C_X\) is the concentration rate that depends on the covariance structure of \(X\), and \(\zeta_X\) is the residual probability of \(X_i\) and \(X_j\) being the maxima, which is negligible in many settings (see Corollary~\ref{col: anti con equal var cor 1} and Proposition~\ref{prop: lwr bd case}). } % for which we provide examples.

% {\red For application of our results, we adapt the Gaussian approximation techniques}
% For illustrative application of our results, 

Combining our new anti-concentration bounds with the Gaussian approximation techniques in \cite{cck2022improvedbootstrap}, we derive a new central limit theorem (CLT) and bootstrap approximation for the distribution of the maximizers of empirical processes on a discrete domain.
% Under the same scaling condition as in \cite{cck2022improvedbootstrap}, an application of our anti-concentration bounds establishes comparable Gaussian and bootstrap approximation rates for the empirical maximizers as in earlier works \citep{cck2013aos,Deng2017BeyondGA,cck2017aop,cck2022improvedbootstrap}.
  Specifically, consider a normalized empirical process \(Y = \{Y_j\}_{j=1}^p\) with each component \(Y_j = n^{-1/2}\sum_{i=1}^n (\xi_{ij} +a_j)\) standardized to unit variance, where \(\xi_{ij}\) are independent centered random variables, and \(a_j\)'s are deterministic constants. For a subset $\cA \subsetneq [p]$, the probability that \(\argmax_{j \in [p]} Y_j\) is a subset of \(\cA\) is  characterized by the probability of a  difference of maxima being positive:
$$
\PP(\argmax_{j \in [p]} Y_j \subseteq \cA) = \PP(\max_{i \in \cA} Y_i - \max_{j \in [p]\setminus\cA} Y_j > 0).
$$
% any partition \([p] = \cA \cup \cB\) devoid of positive unit correlations between \(\cA\) and \(\cB\). 
By using a maximizer of a Gaussian process \(X = (X_j)_{j=1}^p\) mirroring \(Y\)'s mean and covariance structure,
 %approximates the distribution of the maximizer of \(Y\), and 
 we have 
\begin{equation*}
    \begin{aligned}
       & \big|\PP(\textstyle\argmax_{j \in [p]} Y_j \subseteq \cA) - \PP(\textstyle\argmax_{j \in [p]} X_j \subseteq \cA)  \big| \\
       & \quad = \Big|\PP(\max_{i \in \cA} Y_i - \max_{j \in [p]\setminus\cA} Y_j > 0) - \PP(\max_{i \in \cA} X_i - \max_{j \in [p]\setminus\cA} X_j > 0)\Big|\\
&\quad \le \frac{C}{1 -\bar\rho} \cdot \min\Big\{\EE \big(\max_{i \in \cA } |X_i - \mu_i|  \big), \EE \big(\max_{j \in [p]\setminus \cA} |X_j - \mu_j| \big)  \Big\}\left(\frac{\log^3 (pn) }{n}\right)^{1/4},
    \end{aligned}
\end{equation*}
where \(\bar\rho = \max_{i \in \cA, j \in [p] \backslash \cA} \Corr(Y_i,Y_j)\) is the highest correlation between the components in different subsets.
% Comment/check if bounds are dimension-free
% Comment/check if bounds are dimension-free
% Comment/check if bounds are dimension-free

%For comprehensive details under broader covariance structures, refer to Section~\ref{sec: application}.

% Comment/check if bounds are dimension-free

\subsection{Literature review}
Studies of anti-concentration inequalities for maximal statistics predominantly focus on the distribution of continuous random variables, particularly those related to functions of Gaussian random variables \citep{carbery2001distributional, nazarov2003maximal, klivans2008gaussurf, Ylvisarer1965expectzero, kengo2019chianticon, Ylvisaker1968tangencies, Chernozhukov2013ComparisonAA}. \cite{Chernozhukov2013ComparisonAA} sets a precedent by bounding the L\'evy concentration function of the maximum of centered Gaussian processes. \cite{Chernozhukov2013ComparisonAA} utilizes { the monotonicity of a component }of the density function to establish a dimension-free bound, which scales inversely with the minimal component-wise standard deviation. An alternative approach for non-centered Gaussian processes appears in \cite{CHERNOZHUKOV20163632}, where Nazarov's inequality is used to bound the concentration at the intersection boundaries of half-spaces. 
\cite{Deng2017BeyondGA} advances these bounds by replacing the inverse of the minimal standard deviation with the inverse of a function of ordered standard deviations, which can significantly exceed the minimal standard deviation under specific conditions. \cite{Giessing2023AnticoncentrationOS} introduces a novel proof strategy that relies on the \(s\)-concavity of the Gaussian measure, providing upper and lower bounds for the anti-concentration of Gaussian order statistics based on the variance. These developments in characterizing anti-concentration patterns have been extensively applied to establish central limit theorems for various statistics, contributing to recent advances in uncertainty quantification of stochastic processes \citep{CHERNOZHUKOV20163632, kato2020bootstrap, imaizumi2021gaussian}, the improvement of the Berry–Esseen bounds in high dimensions \citep{lopes2020decayclt, Kuchibhotla2021cltimp, lopes2022mbsqrtn}, the bootstrap approximation for regression estimates and random matrices \citep{pan2021quantreg, chen2020robinf,yao2023bsteigen, lopes2024improved}, and the construction of confidence bands in non-parametric settings \citep{confidencebands2014aos, kato2018cbdecon, kato2019cbnonpar}.

Anti-concentration inequalities are also of interest in combinatorial applications such as random graphs. For example, \cite{FOX_KWAN_SAUERMANN_2021} investigates the anti-concentration of polynomials of Bernoulli random variables within a combinatorial context to study subgraph counts in random graphs. \cite{kwan2023anticonramsy} explores anti-concentration bounds for edge statistics in \(C\)-Ramsey graphs, focusing on the number of edges within random vertex subsets of such graphs. Additionally, \cite{rudelson2008littlewood} offers precise estimates for the Littlewood-Offord problem \citep{littlewood1938} by characterizing the anti-concentration of weighted sums of independent random variables with bounded variance and limited third moments. \citep{tao2006littlewood, Razborov2013realadv, meka2016anticonpolyn} further extend the Littlewood-Offord problem to higher-degree polynomials.

Recall that, except for \cite{imaizumi2021gaussian}, all existing works do not apply to the difference of maximum values of two Gaussian random vectors. Compared to \cite{imaizumi2021gaussian}, our bounds rely on (pairwise) correlations rather than the minimal eigenvalue of the covariance matrix.
% and the minimal eigenvalue assumption may not hold.

\vspace{3pt}

\noindent{\bf Paper organization.} 
The rest of the paper is organized as follows. Section~\ref{sec: main thm results} presents the main theoretical results on the anti-concentration inequalities under homogeneous and heterogeneous variance conditions, respectively. Section~\ref{sec: application} applies the anti-concentration results to establish the validity of bootstrap approximations for estimating the distribution of maximizers of empirical processes. Section~\ref{sec: simu} conducts numerical simulations to evaluate the anti-concentration bounds under various settings. The proofs for the theoretical results are provided in Section~\ref{sec: proofs}, followed by concluding remarks in Section~\ref{sec: conclusion}.

\vspace{3pt}

\noindent{\bf Notations.}
For real numbers $x, y \in \RR$, denote $x \vee y := \max(x,y)$ and $x \wedge y := \min(x,y)$, and define $(x)_+ = x \vee 0$ . For an integer $p$, denote by $[p]$ the index set $\{1,\ldots, p\}$. For a vector $v \in \RR^p$, denote by $\|v\|_2 = (\sum_{j=1}^p v_j^2)^{1/2}$ the Euclidean norm and by $\|v\|_{\infty} = \max_{j \in [p]} |v_j|$ the infinity norm. For a matrix $M \in \RR^{p \times q}$, given subsets of indices \(\cA \subseteq [p]\) and \(\cB \subseteq [q]\), denote $M_{\cA \cB} = (M_{ij})_{i \in \cA, j \in \cB} $, and when $\cA$ and $\cB$ coincide, denote $M_{\cA} = M_{\cA \cA}$. We define \(\II\{\cdot\}\) as the indicator function, which takes the value of 1 if the condition inside \(\{\cdot\}\) is met and 0 otherwise. For ease of notation, we denote \(\II_{\cS} (i) = \II\{i \in \cS\}\) for a given index \(i\) and index set~\(\cS\). For two positive sequences \(x_n\) and \(y_n\), denote \(x_n \lesssim y_n\) if there exists a constant \(C > 0\) independent of \(n\) such that \(x_n \le C y_n\).
% $v_{\cS} = (v_i)_{i \in \cS}$ subvector corresponding to $\cS$

\section{Anti-concentration inequalities}\label{sec: main thm results}

In this section, we establish new anti-concentration bounds. We begin with the case where the Gaussian vectors possess homogeneous component-wise variances. Then, we generalize the analysis to encompass the case with heterogeneous variances.

\subsection{Anti-concentration under homogeneous {component-wise} variances}\label{sec: eq var}
Our first theorem provides a dimension-free anti-concentration bound under the homogeneous variance setting. 
%where we derive a dimension-free bound for the L\'evy concentration function.
 We provide the proof of Theorem~\ref{col: anti con equal var no cor 1} in Section~\ref{sec: proof col anti con eq var no cor 1} (we note that the proof relies on Theorem~\ref{thm: anti con}, whose proof is self-contained).

%We present the main anti-concentration results for the equal variance case in the following theorem, which provides a dimension-free bound for the L\'evy concentration function when no pair of Gaussian components from different subsets is perfectly positively correlated.
\begin{theorem}\label{col: anti con equal var no cor 1} Let $(X_1, X_2, \ldots, X_p)^{\top}$ be a  Gaussian random vector with mean $\mu = (\mu_1,...,\mu_p)^\top$ and equal variances, $\Var(X_i) = \sigma^2$ for all $i\in [p]$.  Let $\{\cA,\cB\}$ be a nontrivial partition of $[p]$. Define $M_{\cA} = \max_{i \in \cA} X_i$ and $M_{\cB} = \max_{j \in \cB} X_j$.
% , $\bar\lambda_1 = |\max_{i \in \cA} \mu_i - \min_{i \in \cA} \mu_i |/2$ and $\bar\lambda_2 = |\max_{i \in \cB} \mu_i - \min_{i \in \cB} \mu_i|/2$.
 % $\bar\lambda = \min_{\mu \in \RR} \max_{i \in \cA} |\mu_i - \mu| \vee \min_{\mu \in \RR} \max_{i \in \cB} |\mu_i - \mu| $. 
 % If $\Corr(X_i,X_j) < 1$ for any $i \in \cA$ and $j \in \cB$, 
 Then for any $\varepsilon >0$, we have that
 % there exists a fixed constant $C > 0$ such that 
    \begin{equation}\label{eq: anti con max diff equal var col}
        % \sup_{t \in \RR}\PP\Big(\big| M_{\cB} - M_{\cA} - t\big| \le \varepsilon \Big) 
        \cL(M_{\cB} - M_{\cA}, \varepsilon) \le  \min\Big\{\EE \big(\max_{i \in \cA } |X_i - \mu_i| /\sigma \big) , \EE \big(\max_{j \in \cB} |X_j - \mu_j|  /\sigma\big)\!\!  \Big\} \cdot \frac{7\varepsilon}{(1 -\bar\rho)\sigma},
    \end{equation} 
    where $\bar{\rho} = \max_{i \in \cA, j \in \cB} \Corr(X_i, X_j)$.
    % Besides, at $t = 0$, we have 
    % \begin{equation}\label{eq: anti con max diff equal var col t 0}
    %    \PP\Big(\big| M_{\cB} - M_{\cA}\big| \le \varepsilon \Big) \le C  \left\{\EE \left(\max_{i \in \cA} |X_i - \mu_i| +\bar\lambda_1 \right) \wedge \EE \left(\max_{i \in \cB} |X_i - \mu_i|  + \bar\lambda_2 \right)    \right\}\frac{\varepsilon}{(1 -  \bar\rho)\sigma^2}.
    % \end{equation}
\end{theorem}

\begin{remark}
     The bound in \eqref{eq: anti con max diff equal var col} indicates that the L\'evy concentration scales with the smaller expected maximum of the two normalized Gaussian random vectors, and is invariant to shifting of the means. 
\end{remark}
 In the presence of perfect positive correlations between the subsets, the following corollary characterizes the bounds of the L\'evy concentration function with a stationary residual probability.
\begin{corollary}\label{col: anti con equal var cor 1}
%Let $(X_1, X_2, \ldots, X_p)^{\top}$ be a Gaussian random vector with mean $(\mu_1,...,\mu_p)^\top$  and a homogeneous marginal variance $\sigma^2 > 0$. For the nontrivial partition $\{\cA, \cB\}$ of $[p]$, let $M_{\cA} = \max_{i \in \cA} X_i$ and $M_{\cB} = \max_{j \in \cB} X_j$.
% , $\bar\lambda_1 = |\max_{i \in \cA}\mu_i - \min_{i \in \cA}\mu_i|/2$ and $\bar\lambda_2 = |\max_{i \in \cB}\mu_i - \min_{i \in \cB}\mu_i|/2$.
% $\bar\lambda_1 = \min_{\mu \in \RR} \max_{i \in \cA} |\mu_i - \mu|$ and $\bar\lambda_2 = \min_{\mu \in \RR} \max_{i \in \cB} |\mu_i - \mu| $. 
Under the setting of Theorem~\ref{col: anti con equal var no cor 1}, for any $\delta \in (0, 1)$, we define the set
\begin{equation} \label{eqn:Nd}
\cN_{\delta} = \{i \in \cA: \exists j \in \cB \text{ such that } \Corr(X_i, X_j) \ge  1 - \delta \},
\end{equation}
% $$
% \cN_{\cB} = \{i \in \cB: \exists j \in \cA \text{ such that } \Corr(X_i, X_j) = 1\},
% $$
and for $\delta \in (0,1)$ such that $\cN_{\delta} \ne \cA$, we define 
% the associated metric
    \begin{equation} \label{eqn:omegad}
     \omega_{\delta}  = \left\{ \begin{array}{ll}
              \exp \Big\{-\frac{1}{8\sigma^2 }\Big(\EE\big[ \displaystyle\max_{i \in \cA \backslash \cN_{\delta}} X_i \big]- \EE\big[ \displaystyle\max_{j \in \cN_{\delta}} X_{j}\big] \Big)_+^2  \Big\}, \quad & \,\text{if \ $\cN_{\delta} \ne \emptyset$,} \\
              0, \quad& \, \text{if \ $\cN_{\delta} = \emptyset$.}
            \end{array} \right.
    \end{equation}

    % $$
    % \omega_2 := \EE \max_{i \in \cB \backslash \cN_{\cB}} (X_i - \mu_i) - \EE \max_{i' \in \cN_{\cB}} (X_{i'} - \mu_{i'}) - 2 \bar\lambda_2.
    % $$
    % Denote $\omega = \omega_1 \vee \omega_2$. If $\omega > 0$, 
    Suppose that $\{\delta \in (0,1): \cN_{\delta} \ne \cA\} \ne \emptyset$. 
    Then for any $\varepsilon > 0$, we have that
    % there exists a fixed constant $C > 0$ such that 
    \begin{equation}\label{eq: anti con max diff equal var col cor 1}
    \begin{aligned}
        % \sup_{t \in \RR}\PP\Big(\!\big| M_{\cB} - M_{\cA} - t\big| \le \!\varepsilon \!\Big)  
        & \cL(M_{\cB} - M_{\cA}, \varepsilon) \\
        & \ \le  \inf_{\delta: \cN_{\delta} \ne \cA}\! \Big \{ \!\min\!\! \Big\{ \EE (\max_{i \in \cA\backslash\cN_{\delta}} \!\!{|X_i - \mu_i| }/{\sigma} ) , \EE (\max_{j \in \cB} \!{|X_j - \mu_j| }/{\sigma} )\! \Big\} \! \cdot \!\frac{7\varepsilon}{\delta\sigma}  + 2 \omega_{\delta} \!\Big \},
        % \exp\!\Big(\!-\frac{\omega_{\delta}^2}{8 \sigma^2}\Big)\!\!\bigg\},
    \end{aligned}
    \end{equation} 
    where the result also holds when $\cA$ and $\cB$ are swapped.
    % \begin{equation}\label{eq: anti con max diff equal var col cor 1}
    % \begin{aligned}
    %     \sup_{t \in \RR}\PP\Big(\!\big| M_{\cB} - M_{\cA} - t\big| \le \!\varepsilon \!\Big)  \le & 7 \left\{\!\EE \!\left(\!\max_{i \in \cA} |X_i - \mu_i|  \right) \!\wedge\! \EE \!\left(\!\max_{i \in \cB} |X_i - \mu_i|  \right)  \right\}\!\frac{\varepsilon}{(1 - \bar\rho)\sigma^2} \\
    %     &  + 4\exp\!\big(-\frac{\omega^2}{8 \sigma^2}\big),
    % \end{aligned}
    % \end{equation} 
    % where 
    % % $\bar\lambda = \bar\lambda_1 \vee \bar\lambda_2$ and 
    % $\bar{\rho} = \max_{i \in \cA\backslash\cN_{\cA}, j \in \cB} \Corr(X_i, X_j) \vee \max_{i \in \cA, j \in \cB\backslash\cN_{\cB}} \Corr(X_i, X_j)$.
\end{corollary}
\begin{proof}
 We provide the proof in Section~\ref{sec: proof col anti con eq var cor 1}.
 \end{proof}
\begin{remark}
   The condition $\{\delta \in (0,1): \cN_{\delta} \ne \cA\} \ne \emptyset$ rules out the case where all the random variables in partition $\cA$ are perfectly positively correlated with elements in $\cB$.  Essentially, the construction of the set \(\cN_{\delta}\) with the threshold \(\delta > 0\) removes the components with a correlation exceeding \(1-\delta\) with elements in the other subset, so that for the remaining components, the gap \(1 - \max_{i \in \cA \backslash \cN_{\delta}, j \in \cB} \Corr(X_i, X_j)\) is at least $\delta$, and applying Theorem~\ref{col: anti con equal var no cor 1} yields a non-trivial bound. This gives the first term on the right-hand side (RHS) of \eqref{eq: anti con max diff equal var col cor 1}, which provides a concentration bound. Meanwhile, the second term on the RHS of \eqref{eq: anti con max diff equal var col cor 1} 
   is a residual term, which 
   is the probability that the maximizer within set \(\cA\) falls within the set \(\cN_{\delta}\). 
   It can be seen that the first concentration bound term decreases as \(\delta\) increases, while the second residual term increases with \(\delta\), as a larger \(\delta\) enlarges the set \(\cN_{\delta}\).
\end{remark}
 The following proposition provides an illustrative example of the residual probability when the distribution of the Gaussian random vector is exchangeable. We provide the proof in Section~\ref{sec: proof prop lwr bd case}.
\begin{proposition}[Lower Bound Case]\label{prop: lwr bd case} 
For some integers $1 \le k <m$ and $p = 2 m - k$, let $(X_1,...,X_{p})^\top$ be a Gaussian random vector with an equal mean $\mu_0$ and equal variance $\sigma^2$, and suppose that $\bar\rho:= \max_{i,j \in [p], i \ne j}\Corr(X_i, X_j) < 1$. Furthermore, assume that the distribution of the random vector is invariant with respect to any permutation of its components, i.e.,
% for all $(x_1, \ldots, x_{p}) \in \RR^{p}$ and $\pi \in \Pi^{[p]}$, we have
\begin{equation}\label{eq: exchange dist}
    \phi(x_1,  \ldots, x_{p}) = \phi\big(x_{\pi(1)}, \ldots, x_{\pi({p})}\big), \quad \text{for all $(x_1, \ldots, x_{p}) \in \RR^{p}$ and $\pi \in \Pi^{[p]}$,}
\end{equation}
where $\phi(x_1, x_2, \ldots, x_{p})$ denotes the joint distribution of $(X_1,...,X_{p})^\top$ at $(x_1, x_2, \ldots, x_{p})$,  and $\Pi^{[p]}$ denotes the set of all bijections from $[p]$ to itself. Define the sets $\cA = [m]$, $\cB = [p] \backslash [m - k]$ 
% and $\cN = \cA \cap \cB$, 
and let $M_{\cA} = \max_{i \in \cA} X_i$ and $M_{\cB} = \max_{j \in \cB} X_j$. Then for any $\varepsilon > 0$, we have that
% there exists a constant $C > 0$ such that
\begin{equation}\label{eq: lwr bd case}
\begin{aligned}
   \frac{k}{p} & \le
   % \sup_{t \in \RR}\!\PP\big(|M_{\cB} - M_{\cA} - t| \!\le \!\varepsilon \big) 
   \cL(M_{\cB} - M_{\cA}, \varepsilon) \\
   &\le  \min \Big\{ \EE \big(\max_{i \in \cA\backslash \cN} \!{|X_i - \mu_0|}/{\sigma}   \big)  , \EE \big( \max_{j \in \cB\backslash\cN} \!{|X_j - \mu_0|}/{\sigma}   \big)\Big\}\cdot  \!\frac{7\varepsilon}{(1-\bar\rho)\sigma} + \frac{4k}{p + k }.
   \end{aligned}
\end{equation}
Moreover, it holds that
\begin{equation}\label{eq: lwr bd case t > eps}
\begin{aligned}
   & \sup_{t \in \RR: |t| > \varepsilon} \PP\big(|M_{\cB} - M_{\cA} - t|\le \varepsilon \big) \\
    & \qquad  \le\min\Big\{ \EE \big(\max_{i \in \cA} {|X_i - \mu_0|}/{\sigma} \big)  , \EE \big( \max_{j \in \cB} {|X_j - \mu_0|}/{\sigma}  \big) \Big\} \cdot \frac{14\varepsilon}{(1-\bar\rho)\sigma} .
\end{aligned}
\end{equation}
\end{proposition}
%The proof is deferred to Section~\ref{sec: proof prop lwr bd case}.
\begin{remark}
     By \eqref{eq: lwr bd case}, we have that the L\'evy concentration function is lower bounded by the ratio of the components with across-partition correlation equal to 1, and when the number of the perfectly positively correlated components $k$ is sufficiently small relative to the number of components~$p$, the residual probability is negligible compared with the concentration rate. Furthermore, the bound in \eqref{eq: lwr bd case t > eps} shows that the difference of the maxima still achieves the typical anti-concentration rate for $t$ bounded away from 0.
\end{remark}

\subsection{Anti-concentration under heterogeneous variance conditions}\label{sec: uneq var}
Next, we provide anti-concentration results for Gaussian random vectors with heterogeneous variances. 
Under some structural assumptions on the covariance matrix, we establish the following dimension-free bound for the L\'evy concentration function.

\begin{theorem} \label{thm: anti con}
Let  $(X_1,...,X_p)^\top$ be a  Gaussian random vector with $\EE(X_i) = \mu_i$  and  
$\Var(X_i) = \sigma^2_i>0$ for $i \in [p]$. For a nontrivial partition $\{\cA,\cB\}$ of $[p]$, define $M_\cA = \max_{i\in\cA} X_i$ and $M_\cB = \max_{j\in\cB} X_j$.
 % and $[p] = \cA \cup \cB$ is an arbitrary partition of $[p]$.
 % Define { $\bar\lambda_{\cA} = |\max_{i \in \cA} \mu_i - \min_{i \in \cA} \mu_i |/2$, $\bar\lambda_{\cB} = |\max_{i \in \cB} \mu_i - \min_{i \in \cB} \mu_i|/2
 % % \min_{\mu \in \RR} \max_{i \in \cA} |\mu_i - \mu| \vee  \min_{\mu \in \RR} \max_{i \in \cB} |\mu_i - \mu|
 % $},
 % % $\underline\sigma = \min_{i \in [p]} \sigma_i$
 % and $\bar\sigma = \max_{i \in [p]} \sigma_i$. 
 Assume that  $|{\rm Corr}(X_i, X_{j}) | < 1$ for any $i \in \cA$ and $j \in \cB$, %$(\max_{S, S' \in \cA} \sigma_S^{-2} \sigma_{S S'} \vee \max_{S, S' \in \cB} \sigma_S^{-2} \sigma_{S S'} ) \le 1$
 % and define the quantities $C_{\cA, \cB}$ and $C_{\cB, \cA}$ as 
 % $$
 % C_{\cA, \cB} = \left\{\begin{array}{ll}
 %    \! \! \max_{j \in \cB}\frac{1}{\sigma_j - \max_{i \in \cA} { \sigma_{i j} }/{\sigma_j}},  & \, \text{if $\max_{j,j' \in \cB} \frac{\sigma_{j j'}}{\sigma_j^{2}} \le 1$ and $ \min_{j \in \cB}(\sigma_j - \max_{i \in \cA}\frac{\sigma_{i j}} { \sigma_j} )> 0$} \\
 %     \!\! + \infty, & \, \text{otherwise}
 % \end{array}\right.
 % $$
 % and  
 %  $$
 % C_{\cB, \cA} = \left\{\begin{array}{ll}
 %    \! \! \max_{i \in \cA}\frac{1}{\sigma_i - \max_{j \in \cB} { \sigma_{i j} }/{\sigma_i}},  & \, \text{if $\max_{i,i' \in \cA} \frac{\sigma_{i i'}}{\sigma_i^{2}} \le 1$ and $ \min_{i \in \cA}(\sigma_i - \max_{j \in \cB}\frac{\sigma_{i j}} { \sigma_i} )> 0$} \\
 %     \!\! + \infty, & \, \text{otherwise}
 % \end{array}\right.
 % $$
 and one of the following two conditions is satisfied:
 \begin{enumerate}[label=(\Alph*)]
     \item { $\max_{j,j' \in \cB} \sigma_j^{-2} \sigma_{j j'} \le 1$, and $C_{\cA, \cB} = \min_{j \in \cB, i \in \cA}(\sigma_j - \sigma_j^{-1} \sigma_{i j} )> 0$,}
     \item  $\max_{i,i' \in \cA} \sigma_i^{-2} \sigma_{i i'} \le 1$, and $C_{\cA, \cB} = \min_{i \in \cA, j \in \cB}(\sigma_i - \sigma_i^{-1} \sigma_{i j} )> 0$,
 \end{enumerate}
 \noindent where $\sigma_{ij} =  \EE [(X_j - \mu_j) (X_i - \mu_i) ]$. Then for any $\varepsilon >0$, we have that
 % there exists a fixed constant $C > 0$ such that 
    \begin{equation}\label{eq: anti con max diff}
        % \sup_{t \in \RR}\PP\Big(\big| M_{\cB} - M_{\cA} - t\big| \le \varepsilon \Big) 
        \cL(M_{\cB} - M_{\cA}, \varepsilon) \le 2 \cdot
        % \left\{
     \EE \big({\max_{{\ell} \in \cS }{|X_{\ell} - \mu_{\ell} |} /{\sigma_{\ell}}} \big)
     % \bigwedge  C_{\cA, \cB}  \EE \left[ {\max_{{ j} \in \cB }\frac{|X_j - \mu_j |} {\sigma_j}} \right]
        % + \bar\lambda_{\cS} 
        % + \bar\sigma \sqrt{0 \vee \log(\bar\sigma/\varepsilon)} 
        % \right\}
        \frac{\varepsilon}{    C_{\cA, \cB}},
    \end{equation}
    where the set $\cS = \cB$ when condition (A) is satisfied,  and $\cS = \cA$ when condition (B) is satisfied.
    % and $\bar\sigma_{\cS}^2 := \max_{i \in \cS} \sigma_i^2$.
\end{theorem}
\begin{proof}
Before proving Theorem~\ref{thm: anti con}, we need the following lemma that characterizes the joint probability density function for $(M_{\cA}, M_{\cB})$, and we provide the proof in  Appendix~\ref{sec: proof lm joint max dens}.

\begin{lemma}\label{lm: joint max dens}
    Let $(X_1,...,X_p)^\top$ be a  Gaussian random vector with $\EE(X_i) = \mu_i$  and  
$\Var(X_i) = \sigma^2_i$ for $i \in [p]$.  For a nontrivial partition $\{\cA,\cB\}$ of $[p]$, define $M_\cA = \max_{i\in\cA} X_i$ and $M_\cB = \max_{j\in\cB} X_j$.
    % $[p] = \cA \cup \cB$ is a partition of $[p]$.
    
    \noindent (1) Suppose that $\sigma^2_i > 0$ for all $i \in [p]$, ${\rm Var} (X_i - X_j) > 0$ for any $i \ne j$, and $|{\rm Corr}(X_i, X_j) | < 1$ for any $i \in \cA$ and $j \in \cB$, then the joint distribution of $ (M_{\cA}, M_{\cB})$  is absolutely continuous with respect to the Lebesgue measure of $\RR^2$, and a version of its density is 
\begin{equation}\label{eq: max joint den}
    f(x, y) = \sum_{i \in \cA} \sum_{j \in \cB} \PP \Big(\max_{i' \in \cA \backslash \{i\}} X_{i'} \le x, \max_{j' \in \cB \backslash \{j\}} X_{j'} \le y \Big| X_i = x, X_j  = y \Big) \phi_{i,j}(x, y),
\end{equation}
where $\phi_{i,j}(\cdot,\cdot)$ is the joint density function of $(X_i, X_j)$, $i \in \cA, j \in \cB$. 

\noindent (2) If for all $i,j \in \cA$ and  $i \ne j$, either $\Var(X_i - X_j) > 0$ or $\EE (X_i - X_j)  \ne 0$, and $\sigma_i^2 > 0$ for all $i \in \cA$,
% $\II_{\cA}(S)\cdot \II_{\cA}(S') + \II_{\cB}(S)\cdot \II_{\cB}(S') +  \II_{\cC}(S)\cdot \II_{\cC}(S') = 1$, 
then given any $y \in \RR$, $F_y (x):= \PP(M_{\cA} \le x, M_{\cB} \le y)$ is absolutely continuous with respect to the Lebesgue measure on $\RR$, and  a version of its partial derivative is
\begin{equation}\label{eq: max joint part x}
\begin{aligned}
   \frac{\partial F_y (x)}{\partial x} =  \sum_{i \in \cA} \PP \Big(\max_{i' \in \cA \backslash \{i\}} X_{i'} \le x, \max_{j \in \cB } X_j \le y \Big| X_i  = x \Big) \phi_{i}(x),
\end{aligned}
\end{equation}
where $\phi_{i}(\cdot)$ is the marginal density function for $X_i$.
\end{lemma}
Note that when $\cB = \emptyset$, the partial derivative in \eqref{eq: max joint part x} reduces to the density function of $\PP(\max_{i \in [p]} X_i \le x)$. 
    Without loss of generality, we assume that condition (A) holds. 
% First note that for any given $\lambda_1, \lambda_2 \in \RR$ and $t \in \RR$, we have $\max_{i \in \cA} X_i= \max_{i \in \cA} (X_i- \lambda_1) + \lambda_1$, $\max_{j \in \cB} X_{j} = \max_{j \in \cB} (X_{j} - \lambda_2) + \lambda_2$, and in turn 
% \begin{align*}
%     & \PP\big(\big|\max_{j \in \cB} X_{j} - \max_{i \in \cA} X_i- t\big| \le \varepsilon \big) = \PP\big(\big| \max_{j \in \cB} (X_{j} - \lambda_2) - \max_{i \in \cA} (X_i- \lambda_1) - t + \lambda_2 - \lambda_1\big| \le \varepsilon \big) \\
%     & \le \sup_{t' \in \RR}\PP\big(\big| \max_{j \in \cB} (X_{j} - \lambda_2) - \max_{i \in \cA} (X_i- \lambda_1) - t'\big| \le \varepsilon \big).
% \end{align*}
% % where $t' = t$ when $\lambda_1 = \lambda_2$.  
% Hence if we take
% $$
% \lambda_1 = \arg\min_{\mu \in \RR} \max_{i \in \cA} |\mu_i - \mu| = (\max_{i \in \cA} \mu_i + \min_{i \in \cA} \mu_i)/2,
% $$
% $$
% \text{and }\quad  \lambda_2 = \arg\min_{\mu \in \RR} \max_{i \in \cB} |\mu_i - \mu| = (\max_{i \in \cB} \mu_i + \min_{i \in \cB} \mu_i)/2,
% $$
% % and for bounding \eqref{eq: anti con max diff}, we will take $\lambda_1 = \mu_{\cA}^*$ and $\lambda_2 = \mu_{\cB}^*$, and for bounding \eqref{eq: anti con max diff t 0}, we will take $\lambda_1 = \lambda_2 = \mu_{\cB}^*$. Then 
% without loss of generality, we can assume that $\{X_i\}_{i \in [p]}$ have means $\{\mu_i - \lambda_1\}_{i \in \cA}$ and $\{\mu_i - \lambda_2\}_{i \in \cB}$, and for the convenience of presentation, we slightly abuse the notation and denote $\mu_i - \lambda_k$ by $\mu_i$, where $k = 1,2$, $i = 1, \ldots, p$. 
Also, without loss of generality, we assume that ${\rm Var}(X_i- X_{j}) > 0$ for all $i \ne j$. Indeed, the inequality holds for $i \in \cA$ and $j \in \cB$ by the assumption that $|{\rm Corr}(X_i, X_{j}) | < 1$ for any $i \in \cA$ and $j \in \cB$. In addition, for any $i,j \in \cA$ or $i,j \in \cB$ such that $i \ne j$, if ${\rm Var}(X_i- X_{j}) = 0$, there exits a constant $c \in \RR$ such that $X_i= X_{j} + c$ almost surely. Then, we can remove either $X_i$ or $ X_{j}$ based on the sign of $c$. We can repeat removing the components until ${\rm Var}(X_i- X_{j}) > 0$ holds for all $i \ne j$.

Then by Lemma~\ref{lm: joint max dens}, for any $t \in \RR$, we have 
\begin{align}
    & \PP\Big(\big|\max_{j \in \cB} X_{j} -\max_{i \in \cA} X_i- t\big| \le \varepsilon \Big) = \int_{|y - x - t| \le \varepsilon } f(\max_{i \in \cA} X_i= x  , \max_{j \in \cB} X_{j} = y ) d x d y \notag\\
    &  =  \int_{|y - x  - t| \le \varepsilon } \sum_{i \in \cA, j \in \cB} \!\!\bigg\{\PP \Big(\max_{i' \in \cA \backslash \{i\}} \!X_{i'} \le x, \max_{j' \in \cB \backslash \{j\}}\! X_{j'} \le y \Big| X_i  = x, X_j = y \Big)\notag \\
    &\quad \quad \quad \quad  \quad \quad \quad \quad \quad \times \phi_{i,j}(x, y) \bigg\} dx dy \notag\\
    % & \overset{u := y - x}{ =\joinrel= } \!\! \!\! 
    &=\!\!\!\sum_{i \in \cA,j \in \cB} \int_{t-\varepsilon}^{t+ \varepsilon}\!\!\!\int_{ \RR}\!\bigg\{ \!\PP \Big(\max_{i'\in \cA \backslash \{i\}} X_{i'} \le x, \max_{j' \in \cB \backslash \{j\}} X_{j'} \le x + u \Big| X_i  = x, X_j - X_i  =  u\! \Big) \notag\\
    & \quad \quad \quad \quad \quad\quad \quad  \quad   \times \phi_{i,j}(x, x + u) |J| \bigg\} dx du \notag\\
    % & = \sum_{S, S'} \int_{|u| \le \varepsilon} \phi_{S,S'}(u)\int_{x \in \RR} \PP \Big(\max_{\tilde{S} \in \cA,\tilde{S}' \in \cB } X_{i'} \vee (X_{j'} - u) \le x \Big| X_{i'}  = x, X_{j'}  - X_{i'}  =  u \Big) \phi_{S, S'}(x| u)  dx du, 
        & = \sum_{i \in \cA, j \in \cB}  \int_{t-\varepsilon}^{t+ \varepsilon}\!\!\!\int_{ \RR}\bigg\{ \PP \Big(\max_{i' \in \cA, j' \in \cB } \!X_{i'}\! \vee (X_{j'} - u) \le x \Big| X_i  = x, X_j \!- \!X_i  =  u \Big) \label{eq: anti con ini int}\\
        & \quad \quad \quad \quad  \quad \quad \quad \quad  \quad \quad \quad  \times \phi_{i,j}(x, x + u) \bigg\} dx du,\notag
\end{align}
% as $  \phi_{S,S'}(u) = O(1)$, 
% it suffices to bound the integral
% $$
% \int_{- \infty}^{+ \infty} \sum_{S, S'}  \PP \Big(\max_{\tilde{S} \in \cA,\tilde{S}' \in \cB } X_{i'} \vee (X_{j'} - u) \le x \Big| X_{i'}  = x, X_{j'}  - X_{i'}  =  u \Big) \phi_{S, S'}(x| u)  dx .
% $$
where $J$ is the Jacobian matrix {mapping $(x,u):= (x,y-x) \mapsto (x,y)$} {so that $|J|=1$}. Note that we can write  
\begin{align*}
    & \int_{\RR} \sum_{i \in \cA} \sum_{j \in \cB}  \PP \Big(\max_{i' \in \cA, j' \in \cB } X_{i'} \vee (X_{j'} - u) \le x \Big| X_i  = x, X_j - X_i  =  u \Big) \phi_{i,j}(x, x+ u)  dx \notag\\
    & = \!\sum_{j \in \cB}  
    % \bigg \{ 
    \int_{\RR}\sum_{i \in \cA}  \PP \Big(\max_{i' \in \cA,j' \in \cB } X_{i'} \!\!\vee\!\! (X_{j'} - u) \!\le x \Big| X_i  \!=\! x, X_j \!=\! x + u \Big) \phi_{i| j}(x| x+ u) \phi_{ j}( x+ u)   dx ,
    % & \quad +  \int_{- \infty}^{+ \infty} \sum_{S \in \cB}  \PP \Big(\max_{\tilde{S} \in \cA,\tilde{S}' \in \cB } X_{i'} \vee (X_{j'} - u) \le x \Big| X_i  = x + u, X_j = x + u \Big) \phi_{S| S'}(x + u| x+ u) \phi_{ S'}( x+ u)   dx \bigg\},
\end{align*}
where $\phi_{i|j} (\cdot | \cdot)$ denotes the density of $X_i$ conditional on $X_{j}$.
% and the density of $\max_{j \in \cB} X_{j} -\max_{i \in \cA} X_i$ is bounded everywhere:
% \begin{align*}
%    & \left|\int_{- \infty}^{+ \infty} \sum_{i \in \cA} \sum_{j \in \cB}  \PP \Big(\max_{i' \in \cA, j' \in \cB } X_{i'} \vee (X_{j'} - u) \le x \Big| X_i  = x, X_j - X_i  =  u \Big) \phi_{i,j}(x, x+ u)  dx \right| \\
%     & \le \sum_{i \in \cA} \sum_{j \in \cB}  \int_{- \infty}^{+ \infty}  \phi_{i,j}(x, x+ u)  dx = \sum_{i \in \cA} \sum_{j \in \cB}  \phi_{X_j - X_i} (u) < +\infty, \quad \text{for all $u \in \RR$,}
% \end{align*}
% where $\phi_{X_j - X_i} (\cdot)$ denotes the marginal density for $X_{j} - X_i$.
Then for any $j \in \cB$ and $u \in \RR$, we let
% and by assumption $\{\xi_{i'}\}_{\tilde{S} \in [p] \backslash \{S'\}}$ conditional on $X_{j}$ are still multivariate Gaussian with non-degenerate marginal distributions.

% {\red 
% Define $V_{i'} = X_{i'} - \mu_{i'} - \Sigma_{S'}^{-1} \Sigma_{S' \tilde{S}} (X_{j}  - \mu_{j}) $. Note that if the following equation holds, 
% $$
% V_{i'} + \mu_{i'} + \Sigma_{S'}^{-1} \Sigma_{S' \tilde{S}} (x + u  - \mu_{j}) -  u \cdot \II\{\tilde{S} \in \cB \} = V_{j'} + \mu_{j'} + \Sigma_{S'}^{-1} \Sigma_{S' \tilde{S}'} (x + u  - \mu_{j}) -  u \cdot \II\{\tilde{S}' \in \cB \} ,
% $$
% then $V_{i'}  = V_{j'} $, and hence $ \Sigma_{S' \tilde{S}}  \ne \Sigma_{S' \tilde{S}'} $, because else ${\rm Corr} (\xi_{i'}, X_{j'} ) = 1$ and this is contradictory to our assumption. Hence the equation may only hold at 
% $$
% x = (\Sigma_{S' \tilde{S}} - \Sigma_{S' \tilde{S}'})^{-1} \left( \Sigma_{S'}(\mu_{j'}  - \mu_{i'} ) + u \cdot (\II\{\tilde{S} \in \cB \} - \II\{\tilde{S}' \in \cB \}) \right) - u + \mu_{j},
% $$
% which indicates that $X_{i'} -  u \cdot \II\{\tilde{S} \in \cB \} \ne X_{j'} -  u \cdot \II\{\tilde{S}' \in \cB \} $ conditional on $X_{j} = x + u$ for any $\tilde{S} \ne \tilde{S}'$ and a.e. $x \in \RR$.
% }
% Hence by Lemma~\ref{lm: joint max dens}, we know that the density of {\red $\max_{\tilde{S} \in \cA,\tilde{S}' \in \cB \backslash \{S'\}} X_{i'} \vee (X_{j'} - u)$ conditional on $X_{j} = x + u$ at $x$ is } 
\begin{align}
       g_{u, j} (x)  &=  \sum_{i \in \cA} \PP \Big(\max_{i' \in \cA, j' \in \cB } X_{i'} \!\vee\! (X_{j'} - u) \!\le\! x \Big| X_i  = x , X_j = x + u \Big) \phi_{i| j}(x | x+ u);\label{eq: g_u}\\
      H_{u,j}(x) &= \PP\left(\max_{i' \in \cA, j' \in \cB \backslash \{j\}} X_{i'} \vee (X_{j'} - u) \le x | X_{j} = x+u\right). \label{eq: H_u}
\end{align}
% \begin{align}
%      % f_{u, S'} (x)  &=  \sum_{i \in \cA} \bigg\{ \PP \Big(\max_{\tilde{S} \in \cA,\tilde{S}' \in \cB } X_{i'} \vee (X_{j'} - u) \le x \Big| X_i  = x + u \cdot \II\{S \in \cB\}, X_j = x + u \Big) \\
%      % & \quad \times \phi_{S| S'}(x + u \cdot \II\{S \in \cB\}| x+ u)\bigg\}, \quad \text{and} \\
%   \\ 
% \end{align}
% {\red \bf explain why when $h(x,y)$ we can easily get the derivative buy when $h(x)$ harder (x y moving together).}
The following lemma establishes the relationship between  $g_{u,j}(\cdot)$ and $H_{u,j} (\cdot)$
% two claims 
 for a.e. $u \in \RR$ and every $j \in \cB$ (we provide its proof in Appendix~\ref{sec: proof claim 1}).
\begin{lemma}\label{claim: 251}
Let $g_{u,j}(\cdot)$ and $H_{u,j} (\cdot)$ be the functions defined in \eqref{eq: g_u} and \eqref{eq: H_u} respectively.   For every $j \in \cB$ and a.e. $u \in \RR$, we have 
    \begin{equation}\label{eq: claim g = 0}
         g_{u, j}(x) = 0, \quad \text{for all } x < x_{j}(u),
    \end{equation}
    and 
    \begin{equation}\label{eq: claim g < H}
        g_{u, j}(x) \le \Big(1- \max_{i \in \cA} \sigma_{j}^{-2} \sigma_{i j}\Big)^{-1} \frac{d H_{u,j}(x)} { d x}, \quad  \text{for all } x > x_{j}(u),
    \end{equation}
where $x_{j}(u) \in \RR$ is a predetermined value {and might take the value of $-\infty$ depending on the covariance structure (in which case \eqref{eq: claim g = 0} is trivial and \eqref{eq: claim g < H} holds for all $x \in \RR$)}.\footnote{The claim is nontrivial because $H_{u,j} (x)$ is not a (conditional) cumulative function due to the tangling roles of $x$ in both the quantile determination and the conditioning on $X_j$. Specifically, if we consider the function $G_{u,j} (x,y) := \PP\left(\max_{i' \in \cA, j' \in \cB \backslash \{j\}} X_{i'} \vee (X_{j'} - u) \le x | X_{j} = y+u\right)$, then $d H_{u,j}(x) / d x= \partial_x G_{u,j}(x,x) + \partial_y G_{u,j}(x,x) $. While the first component $\partial_x G_{u,j}(x,x)$ gives rise to an upper bound of $g_{u, j} (x)$ by applying Lemma~\ref{lm: joint max dens}, the second component $\partial_y G_{u,j}(x,x)$ highly depends on the covariance structure of $X$.}
\end{lemma}
% \begin{itemize}
% %     \item \textbf{Claim 1: } for every $j \in \cB$ and a.e. $u \in \RR$, {$H_{u,j}(x)$ is non-decreasing on $x \in \RR$};
%     % \item    
%    \item \noindent \textbf{Claim 2.5.1: } 
% \end{itemize}

\noindent Applying Lemma~\ref{claim: 251}, since $C_{\cA,\cB} > 0$ by condition~(A),
 we have that the following holds for a.e. $u \in \RR$,
{\small
\begin{align*}
    & \int_{- \infty}^{+ \infty} \sum_{i \in \cA} \sum_{j \in \cB}  \PP \Big(\max_{i' \in \cA, j' \in \cB } X_{i'} \vee (X_{j'} - u) \le x \Big| X_i  = x, X_j - X_i  =  u \Big) \phi_{i,j}(x, x+ u)  dx \notag\\
    & =  \sum_{j \in \cB} 
    \int_{- \infty}^{+ \infty}  \phi_{j}( x+ u)  g_{u, j}(x)  dx  \overset{\rm (a)}{=} \sum_{j \in \cB} 
    \int_{x_{j}(u)}^{+ \infty}  \phi_{j}( x+ u)  g_{u, j}(x)  dx \notag \\
    & \overset{\rm (b)}{\le}  \sum_{j \in \cB} 
    \int_{ x_{j}(u)}^{+\infty} \phi_{ j}( x+ u) \Big(1- \max_{i \in \cA} \sigma_{j}^{-2} \sigma_{ij}\Big)^{-1} d H_{u,j}(x) \notag\\
    & = \sum_{j \in \cB} \!\Big(1- \max_{i \in \cA} \frac{\sigma_{ij}}{\sigma_{j}^{2} }\Big)^{-1} \!\!\!\int_{ x_{j}(u)}^{+\infty} \! \phi_{j} (x+ u) d \PP\Big(\max_{i' \in \cA, j' \in \cB \backslash \{j\}} X_{i'} \vee (X_{j'} - u) \le x | X_{j} = x+u\Big)\notag\\
    & \overset{\rm (c)}{=} \sum_{j \in \cB} \!\Big(1- \max_{i \in \cA} \frac{\sigma_{ij}}{\sigma_{j}^{2} }\Big)^{-1}  \Big\{ \Big[\phi_{j} (x+ u) \cdot \PP\Big(\max_{i' \in \cA, j' \in \cB } X_{i'} \vee (X_{j'} - u) \le x | X_{j} = x+u\Big) \Big]\Big |_{x_{j}(u)}^{+\infty} \\
    % & \quad + \phi_{j} (x+ u) \cdot \PP(\max_{i' \in \cA, j' \in \cB } X_{i'} \vee (X_{j'} - u) \le x | X_{j} = x+u) \Big |_{-\infty}^{x_{j}(u)^-} \notag\\
    & \quad - \int_{ x_{j}(u)}^{+\infty}\PP\Big(\max_{i' \in \cA, j' \in \cB } X_{i'} \vee (X_{j'} - u) \le x | X_{j} = x+u\Big) d \phi_{j} (x+ u) \Big\}, \notag
    \end{align*}
    }where equality (a) follows from \eqref{eq: claim g = 0}, inequality (b) follows from \eqref{eq: claim g < H}, and equality (c) follows by integration by parts.
  % Rearranging terms for display \eqref{eq: dens cont}, 
Denote by ${\rm RHS}_{(c)}$ the term on the right hand side of equality (c), then we further have
    \begin{align*}
    {\rm RHS}_{(c)}
     % \eqref{eq: dens cont}  
     % & = \sum_{j \in \cB} \!\Big(1- \max_{i \in \cA} \sigma_{j}^{-2} \sigma_{ij}\Big)^{-1}  \!\Big\{ \phi_{j} (x+ u) \cdot \PP(\max_{i' \in \cA, j' \in \cB } X_{i'} \vee (X_{j'} - u) \le x | X_{j} = x+u) \Big |_{x_{j}(u)^+}^{x_{j}(u)^-} \notag\\
    % & \quad - \int_{\RR}\PP(\max_{i' \in \cA, j' \in \cB } X_{i'} \vee (X_{j'} - u) \le x | X_{j} = x+u) d \phi_{j} (x+ u) \Big\} \notag\\
    & \overset{\rm (d)}{\le} \! \sum_{j \in \cB}  \int_{x_{j}(u)}^{+ \infty} \!\PP(\max_{i' \in \cA, j' \in \cB } X_{i'} \!\vee \! (X_{j'} - u) \!\le x | X_{j} = x+u) \frac{x+u - \mu_{j}}{\sigma_{j}^2 - \displaystyle\max_{i \in \cA}  \sigma_{ij}} \phi_{j} (x+ u) d x \notag\\
    &  \overset{\rm (e)}{\le}   \sum_{j \in \cB}   \int_{x_{j}(u)}^{+ \infty} \PP(\max_{j' \in \cB } X_{j'}  \le x + u| X_{j} = x+u ) \left| \frac{x+u- \mu_{j}}{\sigma_{j}^2 - \max_{i \in \cA}  \sigma_{ij}} \right| \phi_{j} (x+ u) d x  \notag\\
    &  \le   \sum_{j \in \cB}   \int_{-\infty}^{+ \infty} \PP(\max_{j' \in \cB } X_{j'}  \le x + u| X_{j} = x+u ) \left| \frac{x+u- \mu_{j}}{\sigma_{j}^2 - \max_{i \in \cA}  \sigma_{ij}} \right| \phi_{j} (x+ u) d x  \notag\\
    & \overset{\rm (f)}{\leq} C_{\cA, \cB}^{-1} \sum_{j \in \cB}   \int_{- \infty}^{+ \infty} \PP(\max_{j' \in \cB } X_{j'}  \le \tilde{x}| X_{j} = \tilde{x} ) |\tilde{x} - \mu_j| /\sigma_j \phi_{j} (\tilde{x}) d \tilde{x}, 
    \end{align*}
where inequality (d) follows from the fact that
    $$
   \Big[\phi_{j} (x+ u) \cdot \PP(\max_{i' \in \cA, j' \in \cB } X_{i'} \vee (X_{j'} - u) \le x | X_{j} = x+u) \Big] \Big |_{x_{j}(u)}^{+\infty} \le 0, \quad \text{for all } j \in \cB,
    $$
    % that $H_{u,j}(x)$ has a non-negative jump at $x = x_j(u)$ by Lemma~\ref{claim: 251}. 
      inequality (e) holds as $\max_{j' \in \cB } (X_{j'} - u) \le \max_{i' \in \cA, j' \in \cB } X_{i'} \vee (X_{j'} - u)$, and inequality (f) follows by taking $\tilde{x} = x+u$ and the definition of $C_{\cA,\cB} = \min_{j \in \cB, i \in \cA} (\sigma_j -  \sigma_j^{-1}\sigma_{ij})$. Furthermore, let ${\rm RHS}_{(f)}$ be the term on the right hand side of inequality (f), and we have
    \begin{align*}
    {\rm RHS}_{(f)} & =  C_{\cA, \cB}^{-1} \sum_{j \in \cB}   \int_{- \infty}^{+ \infty} \EE\Big(\II\{\max_{j' \in \cB \backslash\{j\} } X_{j'}  \le X_{j} \} |X_j - \mu_j|/\sigma_j \, \Big| X_{j} = \tilde{x} \Big)  \phi_{j} (\tilde{x}) d \tilde{x}  \\
    & =  C_{\cA, \cB}^{-1} \,\,  \EE\left[ \sum_{j \in \cB} \II\{\max_{j' \in \cB \backslash\{j\} } X_{j'}  \le X_{j} \} |X_j - \mu_j| /\sigma_j \right]  \\
    & \le C_{\cA, \cB}^{-1}  \Big\{ \EE\Big[  \max_{j \in \cB}\Big|\frac{X_j - \mu_j}{\sigma_j}\Big| \sum_{j \in \cB} \II\{\max_{j' \in \cB \backslash\{j\} } X_{j'}  < X_{j} \}\Big] \\
    &\quad \quad \quad \quad + \EE \Big[  \max_{j \in \cB}\Big|\frac{X_j - \mu_j}{\sigma_j}\Big| \sum_{j \in \cB} \II\{\max_{j' \in \cB \backslash\{j\} } X_{j'}  = X_{j} \}\Big] \Big\} \\
    & \overset{\rm (g)}{\le} C_{\cA, \cB}^{-1} 
    % \bigg\{ 
    \EE \Big( \max_{j \in \cB}\Big|\frac{X_j - \mu_j}{\sigma_j}\Big|\Big),
    % +\! \sum_{j \in \cB} \EE \Big( \max_{j \in \cB}\Big|\frac{X_j - \mu_j}{\sigma_j}\Big|^2 \Big)^{1/2} \PP \Big(\max_{j' \in \cB \backslash\{j\} } X_{j'}  = X_{j} \Big)^{1/2} \bigg\}\\
    % & \overset{\rm (h)}{=} C_{\cA, \cB}^{-1}  \EE \Big( \max_{j \in \cB}\Big|\frac{X_j - \mu_j}{\sigma_j}\Big| \Big),
    % C_{\cA, \cB}^{-1}  \left(\sum_{j \in \cB}   \int_{- \infty}^{+ \infty} \PP(\max_{j' \in \cB } X_{j'}  \le x | X_{j} = x ) | x | \phi_{j} (x) d x + \bar\lambda_{\cB} \right) = C_{\cA, \cB}^{-1} \left\{\EE |\max_{i \in \cB } X_{i} | + \bar\lambda_{\cB} \right\}\\
    % & \overset{\rm (f)}{\le} C_{\cA, \cB}^{-1} \left\{\EE |\max_{i \in \cB }( X_{i} - \mu_i )| + 2\bar\lambda_{\cB} \right\},
\end{align*}
% \textcolor{blue}{AB:Tentative
% \begin{align*}
%     & \overset{\rm (c)}{\le}  \sum_{j \in \cB}  \int_{- \infty}^{+ \infty} \PP(\max_{i' \in \cA, j' \in \cB } X_{i'} \vee (X_{j'} - u) \le x | X_{j} = x+u) \frac{x+u - \mu_{j}}{\sigma_{j}^2 - \max_{i \in \cA}  \sigma_{ij}} \phi_{j} (x+ u) d x\\
%     & \leq \sum_{j \in \cB} \frac{1}{\sigma_{j}^2 - \max_{i \in \cA}  \sigma_{ij}} \int_{- \infty}^{+ \infty} \PP(\max_{ j' \in \cB } X_{j'} \le x+u | X_{j} = x+u) |x+u-\mu_j| \phi_{j} (x+ u) d x\\
%     & = \sum_{j \in \cB} \frac{1}{\sigma_{j}^2 - \max_{i \in \cA}  \sigma_{ij}} \int_{- \infty}^{+ \infty} \PP(\max_{ j' \in \cB } X_{j'} \le \tilde x | X_{j} = \tilde x ) |\tilde x-\mu_j| \phi_{j} (\tilde x) d \tilde x\\
%     & \leq  C_{\cA, \cB}^{-1}\sum_{j \in \cB}  \int_{- \infty}^{+ \infty} \PP(\max_{ j' \in \cB } X_{j'} \le \tilde x | X_{j} = \tilde x ) |\tilde x-\mu_j| \phi_{j} (\tilde x) d\tilde{x} \\
%     & =  C_{\cA, \cB}^{-1}  \mathbb{E}\left[ \sum_{j \in \mathcal{B}} 1\{ X_j\geq \max_{j'\in\mathcal{B}}X_{j'}\}|X_j- \mu_j|  \right]\\
%         &  \overset{\rm (e)}{\le} C_{\cA, \cB}^{-1} \left\{\EE \max_{j' \in \cB } |X_{j'} - \mu_{j'}|  \right\},
% \end{align*}}
{where inequality (g) holds by the Cauchy–Schwarz inequality and the fact that  
$$
\PP \Big(\max_{j' \in \cB \backslash\{j\} } X_{j'}  = X_{j} \Big) = 0, \quad \text{for any $j \in \cB$.}
$$ 
% $$\sum_{j \in \cB} \II\{\max_{j' \in \cB \backslash\{j\} } X_{j'}  < X_{j} \} \le 1,$$
% and the equality (h) is a result of the following facts
% \begin{align*}
% & \EE \Big( \max_{j \in \cB} \big[ |X_j - \mu_j|^2/\sigma_j^2 \big] \Big) \overset{\rm (h1)}{<} + \infty;\\
%      & \PP \Big(\max_{j' \in \cB \backslash\{j\} } X_{j'}  = X_{j} \Big) \le \sum_{j' \in \cB \backslash \{j\} } \PP(X_j - X_{j'} = 0) \overset{\rm (h2)}{=} 0,
% \end{align*}
% where (h1) is a result of Theorem~7.1 and equation~(7.4) in \cite{ledoux2001concentration}, and (h2) is by the fact that $\Var(X_j - X_{j'}) > 0$ for any $j \ne j'$.
% by noting that by triangle inequality we have 
% $$
% \max_{j \in \cB}|x + u - \mu_j | \le |x + u| + \max_{j \in \cB}|\mu_j| = |x + u| + \bar\lambda_{\cB},
% $$
% and the inequality (f) follows from that
% \begin{align*}
%     \max_{i \in \cB} X_i & \le \max_{i \in \cB} (X_i - \mu_i) + \max_{i \in \cB} \mu_i \le |\max_{i \in \cB} (X_i - \mu_i)| + \bar\lambda_{\cB};\\
%     \max_{i \in \cB} X_i & \ge  X_{i^*} =  X_{i^*} - \mu_{i^*} + \mu_{i^*} = \max_{i \in \cB} (X_i - \mu_i)+ \mu_{i^*} \ge -|\max_{i \in \cB} (X_i - \mu_i)| - \bar\lambda_{\cB},
% \end{align*}
% where $i^* \in \argmax_{i \in \cB} (X_i - \mu_i)$.
 Hence applying the above bound to the integral in \eqref{eq: anti con ini int}, for any $t \in \RR$, we have that }
$$
     \PP\Big(\big| \max_{j \in \cB} X_{j} - \max_{i \in \cA} X_i- t\big| \le \varepsilon \Big) \le 2 \EE \big( \max_{j \in \cB}|X_j - \mu_j| /\sigma_j \big)
     % \left\{\EE \big|{\max_{i \in \cB} (X_i - \mu_i ) } \big|
     % + {|t| + \varepsilon} 
     % \right\}
     \frac{\varepsilon}{C_{\cA, \cB}},
     % & \lesssim {\red (1 - C_{\cA, \cB})^{-1}  \left(\frac{\bar\sigma}{\underline\sigma}\sqrt{K \log n} + \bar\lambda /\underline\sigma \right)\varepsilon .}
$$
and \eqref{eq: anti con max diff} follows by taking the supremum over $t$.
\end{proof}
\begin{remark} 
The condition (A) or (B) of Theorem~\ref{thm: anti con} is satisfied when one subset has a homogeneous component-wise variance that is greater than the component-wise variances of the other subset. For example, condition (A) is satisfied when we have $\sigma_j^2 = 1$ for all $j \in \cB$, and  $\sigma_i^2 \le 1$ for any $i \in \cA$. In \cite{Chernozhukov2013ComparisonAA}, for the Gaussian random vector with heterogeneous marginal variances, an upper bound of the concentration function is obtained by scaling the random variables to have homogeneous component-wise variances, whereas in our case, the generalization of results from the equal variance case to the unequal variance case via scaling is nontrivial. For example, consider \(p = 1\), and let \(M_{\cA} = M_{\cB} = X_1 \sim \mathcal{N}(\mu, \sigma^2)\). Then the concentration behavior of \(M_{\cB} - M_{\cA} = X_1 - X_1 = 0\) and the scaled difference \(M_{\cB} - \frac{1}{2} M_{\cA} = X_1 - \frac{1}{2} X_1 = \frac{1}{2} X_1 \sim \mathcal{N}(\mu, \sigma^2/4)\) will be fundamentally different. Specifically, while \(\mathcal{L}(M_{\cB} - M_{\cA}, \varepsilon) = 1\) for any \(\varepsilon > 0\), the scaled difference has a well-bounded L\'evy concentration function \(\mathcal{L}(M_{\cB} - \frac{1}{2} M_{\cA}, \varepsilon) \le \sqrt{8/\pi} \sigma^{-1} \varepsilon\). This indicates that for the case of the maximum difference, scaling can fundamentally change the L\'evy concentration function and is not applicable for generalizing the results from the equal variance case to the unequal variance case.
\end{remark}

The bound in \eqref{eq: anti con max diff} gives a dimension-free bound, similar to the case of the single Gaussian maximum \citep{Chernozhukov2013ComparisonAA}. Even though our approach is developed for the difference of maxima, it yields the following comparable bounds for the case of the single maximum. We provide the proof in Appendix~\ref{sec: proof col single max anti con}.
\begin{corollary}\label{col: single max anti con}
    Let $(X_1, \ldots, X_p)^{\top}$ be a Gaussian random vector with $\EE(X_j) = \mu_j$ and $\Var(X_j) = \sigma_j^2 > 0$ for $j \in [p]$. Then for any $\varepsilon > 0$, we have
    \begin{equation}\label{eq: single max anti con}
        \cL(\textstyle\max_{j \in [p]} X_j, \varepsilon) \le 2 \EE(\textstyle \max_{j \in [p]} |X_j - \mu_j| /\sigma_j) \cdot \varepsilon/\underline\sigma,
    \end{equation}
    where $\underline\sigma = \min_{j \in [p]} \sigma_j$.
\end{corollary}
In particular, when $\Var(X_j) = \sigma^2 > 0$ for $j \in [p]$, \eqref{eq: single max anti con} yields the upper bound $2 \mathbb{E}[\max_{j\in [p]}|X_j-\mu_j|/\sigma] \varepsilon/\sigma$ for $ \cL(\textstyle\max_{j \in [p]} X_j, \varepsilon) $, while \cite{Chernozhukov2013ComparisonAA} obtained the upper bound $4 \mathbb{E}[\max_{j\in [p]} (X_j-\mu_j)/\sigma + 1] \varepsilon/\sigma$ % and \cite{Chernozhukov2013ComparisonAA} obtained the upper bound $\sqrt{12}\varepsilon/\sqrt{\Var(M_{\cB}) + \varepsilon^2/12}$ 
using a different proof. 
% \begin{remark}
%     Consider the case when $\sigma_S = 1$ and $\mu_i = 0$ for all $i \in [p]$. When there exist $S \in \cA$ and $S' \in \cB$ such that $${\rm Corr}(X_i, X_{j}) = 1$$ and consider $|\cA| = |\cB| = 2$. 
% \end{remark}
%{ The detailed proof is given in Section~\ref{sec: proof thm anti con}. }

% When the covariance matrix of the multivariate Gaussian random vector is non-degenerate, 
By Corollary~\ref{col: single max anti con}, we can derive an alternative upper bound for the L\'evy concentration function of the difference of two maxima of Gaussian random vectors in the next proposition.
\begin{proposition}\label{col: min eigen anti con}
     Let $(X_1,...,X_p)^\top { = (X_{\cA}^{\top} , X_{\cB}^{\top})}^{\top} \sim \cN(\mu, \Sigma)$ be a  Gaussian random vector.
     % and assume that there exists $\underline\lambda > 0$ such that $\lambda_{\min} (\Sigma) \ge \underline\lambda$. Let $[p] = \cA \cup \cB$ be a partition of $[p]$, and denote by $X_{\cA} = (X_i)_{i \in \cA}$ and $X_{\cB} = (X_j)_{j \in \cB}$ the corresponding subvectors of $X$ . Then let $\mu_{\cA} = \EE (X_{\cA})$, $\mu_{\cB} = \EE (X_{\cB})$, $\Sigma_{\cS} = \EE[(X_{\cS} - \mu_{\cS})(X_{\cS} - \mu_{\cS})^{\top}]$ for $\cS = \cA, \cB$ and $\Sigma_{\cA\cB} = \Sigma_{\cB\cA}^{\top} = \EE[(X_{\cA} - \mu_{\cA})(X_{\cB} - \mu_{\cB})^{\top}]$.
     Also, let $(\tilde{X}_1,...,\tilde{X}_p)^\top \sim \cN(0, \tilde\Sigma)$ be a Gaussian random vector that satisfies $(\tilde\Sigma_{ij})_{i,j \in \cA} = \Sigma_{\cA} - \Sigma_{\cA \cB} \Sigma_{\cB}^{-1} \Sigma_{\cB \cA}$ and $(\tilde\Sigma_{ij})_{i,j \in \cB} = \Sigma_{\cB} - \Sigma_{\cB \cA} \Sigma_{\cA}^{-1} \Sigma_{\cA \cB}$. 
     % and $\tilde\Sigma_{ij} = \tilde\Sigma_{ji} = 0$ for any $i\in \cA$ and $j \in \cB$.
     Suppose $\min_{j \in [p]} \tilde\sigma_j >0$, where $\tilde\sigma_j^2 = \Var(\tilde{X}_j)$ for $j \in [p]$.  Then for any $\varepsilon > 0$, we have 
     \begin{equation}\label{eq: min eigen anti con}
         \cL(M_{\cB} - M_{\cA}, \varepsilon) \le 2 (\min_{j \in [p]} \tilde\sigma_j)^{-1} \min\Big\{ \EE \big(\max_{i \in \cA} |\tilde{X}_i|/\tilde\sigma_i\big),  \EE \big(\max_{j \in \cB} |\tilde{X}_j|/\tilde\sigma_j\big)\Big\} \cdot \varepsilon,
     \end{equation}
     where $M_{\cA} = \max_{i \in \cA} X_i$ and $M_{\cB} = \max_{j \in \cB} X_j$.
     % and $\tilde\sigma_j = \tilde\Sigma_{jj}$ for $j \in [p]$.
\end{proposition}
\begin{proof}
For any $t \in \RR$ and $\varepsilon > 0$, we have that
\begin{align*}
    &\PP(|\max_{j \in \cB} X_j - \max_{i \in \cA} X_i - t| \le \varepsilon ) = \EE \big( \PP(|\max_{j \in \cB} X_j - \max_{i \in \cA} X_i - t| \le \varepsilon\,\, | X_{\cA}) \Big)\\
    & \quad \overset{\rm (a)}{\le} \EE \big( \sup_{t' \in \RR }\PP(|\max_{j \in \cB} X_j - t'| \le \varepsilon \,\,| X_{\cA}) \Big) \\
    &\quad \overset{\rm (b)}{\le} 2 \EE\big( \EE\big( \max_{j \in \cB} |\tilde{X}_j|/\tilde\sigma_j \,\, \big| X_{\cA}\big) \varepsilon / \min_{j \in \cB} \tilde\sigma_j\big) \\
    & \quad \overset{\rm (c)}{=} 2  \EE\big( \max_{j \in \cB} |\tilde{X}_j|/\tilde\sigma_j \big) \varepsilon / \min_{j \in \cB} \tilde\sigma_j \le 2 (\min_{j \in [p]} \tilde\sigma_j)^{-1} \EE\big( \max_{j \in \cB} |\tilde{X}_j|/\tilde\sigma_j \big) \varepsilon,
\end{align*}
where $(\tilde{X}_j)_{j \in \cB} = X_{\cB} - \mu_{\cB} - \Sigma_{\cB \cA} \Sigma_{\cA}^{-1} (X_{\cA} - \mu_{\cA})$ is the residual vector that is independent of $X_{\cA}$ and $\mu_{\cA} = \EE (X_{\cA})$.  Inequality~(a) follows from the fact that $\max_{i \in \cA} X_i$ is deterministic conditional on $X_{\cA}$, inequality~(b) follows from Corollary~\ref{col: single max anti con}  utilizing the conditional distribution of $X_{\cB}$ on $X_{\cA}$, and equality~(c) follows from the independence between $\tilde{X}_{\cB}$ and $X_{\cA}$.
% $\tilde\sigma_j = ( \Sigma_{\cB} - \Sigma_{\cB \cA} \Sigma_{\cA}^{-1} \Sigma_{\cA \cB})_{jj}$ for $j \in \cB$,
% and the last inequality is by Proposition~1 in \cite{imaizumi2021gaussian}.
Similarly, by conditioning on $X_{\cB}$, we have 
$$
\PP(|\max_{j \in \cB} X_j - \max_{i \in \cA} X_i - t| \le \varepsilon )  \le 2 (\min_{j \in [p]} \tilde\sigma_j)^{-1} \EE\big( \max_{j \in \cA} |\tilde{X}_j|/\tilde\sigma_j \big) \varepsilon,
$$
where  $(\tilde{X}_j)_{j \in \cA} \sim \cN(0,\Sigma_{\cA} - \Sigma_{\cA \cB} \Sigma_{\cB}^{-1} \Sigma_{\cB \cA} )$.
% and $\tilde\sigma_j = (\Sigma_{\cA} - \Sigma_{\cA \cB} \Sigma_{\cB}^{-1} \Sigma_{\cB \cA})_{jj}$ for $j \in \cA$. 
Hence combining the above results, we have that \eqref{eq: min eigen anti con} holds. 
\end{proof}
% Note that by Proposition~1 in \cite{imaizumi2021gaussian}, if  the covariance matrix $\Sigma$ is non-singular, we have the lower bound $\min_{j \in [p]} \tilde\sigma_j^2 \ge  \lambda_{\min} (\Sigma)$. Thus, we have that 
Proposition~\ref{col: min eigen anti con} is a refinement of Lemma~5 in \cite{imaizumi2021gaussian}, which shows
$$
  \max_{\substack{\cA \subsetneq [p]\\\cB = [p] \backslash \cA}}\cL(M_{\cB} - M_{\cA}, \varepsilon) \le \frac{2\varepsilon}{\sqrt{\underline\lambda}} (\sqrt{2\log p} + 2), \quad  \underline\lambda := \min_{\substack{\cA \subsetneq [p]\\\cB = [p] \backslash \cA}} \min_{j \in \cA}(\Sigma_{\cA} - \Sigma_{\cA \cB} \Sigma_{\cB}^{-1} \Sigma_{\cB \cA})_{jj},
$$
and they lower bound $\underline\lambda$ by $\lambda_{\min}(\Sigma)$ for positive definite covariance matrices.
While Lemma~5 in \cite{imaizumi2021gaussian} aims to obtain uniform bound over all $\cA \subsetneq [p]$, \eqref{eq: min eigen anti con} is more adaptive to specific subsets and allows for possibly sharper rates when $|\cA| \ll |\cB|$, e.g., when $|\cA| = O(1)$, the bound in \eqref{eq: min eigen anti con} scales with a constant instead of $\sqrt{\log p}$ as in \cite{imaizumi2021gaussian}.

\begin{remark}
We compare the bound in \eqref{eq: min eigen anti con}  with the bound in \eqref{eq: anti con max diff}. The most favorable scenario for \eqref{eq: min eigen anti con} occurs when \( p = 2\), $X = (X_1, X_2)^{\top}$, \( \sigma_1 = \sigma_2 = 1 \) and \( \Sigma_{12} = \rho < 1 \). Then the density of $M_{\cB} - M_{\cA}$ is upper bounded by $(1/\sqrt{4\pi})(1-\rho)^{-1/2}$. In comparison, the upper bound in \eqref{eq: anti con max diff} gives \( \sqrt{8/\pi} (1-\rho)^{-1}\varepsilon \), while the upper bound in \eqref{eq: min eigen anti con} yields \( \sqrt{8/\pi} \sqrt{(1-\rho)(1+\rho)}\varepsilon \). Hence, it can be seen that relative to the true density, our bound in \eqref{eq: anti con max diff} is conservative by a factor of $O(\sqrt{1-\rho} )$, whereas the bound in \eqref{eq: min eigen anti con} carries an excess factor of $\sqrt{1+\rho}$. Thus, in non-degenerate scenarios, our bound in \eqref{eq: anti con max diff} is tighter when \( \rho < 0 \), and the bound in \eqref{eq: min eigen anti con} is tighter when \( \rho > 0 \). It remains an open problem whether it is possible to close the gap of the extra factor \(\sqrt{1-\bar\rho}\) when $p > 2$. We further investigate the tightness of the bound with respect to \(1-\bar\rho\) for $p > 2$ through numerical studies in Appendix~\ref{sec: app add plot}. However, it is clear that $\lambda_{\min} (\Sigma)$ does not prevent \eqref{eq: anti con max diff} from delivering quite useful bounds.
 % On the other hand, under the uniform variance assumption, our results provide an anti-concentration bound that remains valid even in scenarios where the covariance matrix is degenerate or has a small minimum eigenvalue. 
 % Hence, it can be seen that Corollary~\ref{col: min eigen anti con} might potentially provide better bounds than Theorem~\ref{thm: anti con} when the covariance matrix is well-structured. Yet, in many cases, we will be confronted with an ill-conditioned covariance matrix with minimal eigenvalue converging to 0, where Theorem~\ref{thm: anti con} provides a better characterization of the anti-concentration behaviors.
\end{remark}
\section{Application: Central limit theorems for the distribution of the maximizers of empirical processes}\label{sec: application}
% Following the anti-concentration bounds in the previous section,
Here, we establish the Gaussian approximation for the distribution of the maximizers of empirical processes indexed by a discrete set. Earlier works mainly focus on establishing Gaussian approximations for the supremum of empirical processes \cite{cck2013aos, CHERNOZHUKOV20163632, Deng2017BeyondGA}. However, in some recent applications, such as combinatorial inferences \cite{liu2023lagrangian, neykov2019combinatorial, le2019combinatorial, imaizumi2021gaussian}, the main interest lies in the maximizers of some empirical processes. Therefore, by casting the maximizing event as the difference of two maxima,
% for illustrative application of the anti-concentration results in Section~\ref{sec: main thm results},  
we integrate our new anti-concentration bounds with existing Gaussian approximation techniques to establish novel central limit theorems for this application. In particular, our central limit theorems cover the setting where the covariance matrix is singular, such as when  $X_\cA=\Gamma_\cA Z$ and $X_\cB=\Gamma_\cB Z$, where $Z$ is a $d$-dimensional Gaussian random vector and $\Gamma_\cA,\Gamma_\cB \in \RR^{p\times d}$ with $p>d$. 

 Specifically, let $\xi_1, \ldots, \xi_n \in \RR^p$ be independent random vectors with the same mean $\mu = \EE (\xi_i)$  and possibly different covariance matrices $\Sigma_i = \EE \big[(\xi_i - \mu)(\xi_i - \mu)^{\top}\big]$ for $i\in[n]$. 
% such that $\max_{i \in [N], j \in [p]} \EE |\xi_{ij}|^3 < \infty$.
Meanwhile,  let $\eta_1, \ldots, \eta_n \in \RR^p$ be independent  Gaussian random vectors with $\eta_i \sim \cN(\mu, \Sigma_i)$ for $i = 1, \ldots, n$. We define 
\begin{equation*}
    X = n^{-1/2}\sum_{i=1}^n (\eta_i - \mu + a) , \quad Y = n^{-1/2}\sum_{i=1}^n (\xi_i - \mu + a) ,
\end{equation*}
% $ X = n^{-1/2}\sum_{i=1}^n (\eta_i - \mu + a) $ and $Y = n^{-1/2}\sum_{i=1}^n (\xi_i - \mu + a) $, 
where $a \in \RR^p$ is a known deterministic vector, and denote by $\Sigma^* = \frac{1}{n}\sum_{i=1}^n \Sigma_i
% \EE\big((X - \sqrt{n} a)(X - \sqrt{n} a)^{\top}\big) = \EE\big((Y - \sqrt{n} a)(Y - \sqrt{n} a)^{\top}\big)
$ with $\sigma_j^2 = \Sigma^*_{jj}$ for $j \in [p]$. Here we add the shifting vector \(a\) to allow for more general non-centered distributions, such as when location information is incorporated to perform multiple testing of moment inequalities as in \cite{bai2019practical}.
% $ \bar\mu = \EE(X) = n^{-1/2}\sum_{i=1}^n \mu_i$, 
% $\sigma_j^2 = \Var(X_j)$ ($j \in [p]$) and $\sigma_{ij} = \EE[(X_i - \sqrt{n} a_i)(X_j - \sqrt{n} a_j)],\, i,j \in [p]$.
Let $\{\cA,\cB\}$ be a nontrivial partition of $[p]$. Next, we let 
    $$
    M_{\cA}^X = \max_{j \in \cA} X_j, \quad M_{\cB}^X = \max_{j \in \cB} X_j, \quad M_{\cA}^Y = \max_{j \in \cA} Y_j, \quad M_{\cB}^Y = \max_{j \in \cB} Y_j.
    $$
    Our objective is to estimate the probability that the maximum component of $Y$  lies within the set $\cA$ with the probability that the maximum component of 
 $X$ lies within the same set $\cA$. To do that, we need to control:
    \begin{equation}\label{eq: equiv argmax}
        \begin{aligned}
            & \left| \PP(\textstyle\argmax_{j \in [p]} Y_j \subseteq \cA) - \PP( \textstyle\argmax_{j \in [p]} X_j \subseteq \cA)\right| \\
            &\ = \left| \PP(M_{\cA}^Y - M_{\cB}^Y > 0) - \PP( M_{\cA}^X - M_{\cB}^X  > 0)\right|,
        \end{aligned}
    \end{equation}
which connects with the probability of the difference of the maxima being positive.
    Further, as the mean and the covariance matrix of $X$ are usually unknown in practice, we construct the multiplier bootstrap random vector to approximate $X$:
    \begin{equation}\label{eq: multiplier X}
         \hat{X} = \frac{1}{\sqrt{n}} \sum_{i=1}^n [e_i (\xi_i - \bar\xi) + a], 
    \end{equation}
    where $\bar\xi = \frac{1}{n}\sum_{i=1}^n  \xi_i$ and $e_i \overset{\rm i.i.d.}{\sim} \cN(0,1)$ are independent of $\xi_{1:n} := \{\xi_i\}_{i=1}^n$.
    Then we will control
    % also aim to characterize the difference in the following probabilities, which is the bootstrap equivalent of \eqref{eq: equiv argmax},
    \begin{equation}\label{eq: equiv argmax mb}
        \begin{aligned}
          & \left| \PP(\textstyle\argmax_{j \in [p]} Y_j \subseteq \cA) - \PP( \textstyle\argmax_{j \in [p]} \hat{X}_j \subseteq \cA | \xi_{1:n})\right|\\
          &\  = \left| \PP(M_{\cA}^Y - M_{\cB}^Y > 0) - \PP( M_{\cA}^{\hat{X}} - M_{\cB}^{\hat{X}}  > 0)\right|,
    \end{aligned}
    \end{equation}
 where $M_{\cA}^{\hat{X}} = \max_{j \in \cA} \hat{X}_j$ and $M_{\cB}^{\hat{X}} = \max_{j \in \cB} \hat{X}_j$. Without loss of generality, by the constructions of $X$, $Y$ and $\hat{X}$, we assume that $\mu = \mathbf{0}$. 

The application of the anti-concentration theorems in Section~\ref{sec: main thm results} is essential for bounding \eqref{eq: equiv argmax} and \eqref{eq: equiv argmax mb}. To facilitate our discussion, we introduce some conditions on the covariance structure.
    % To construct an upper bound on the quantity in \eqref{eq: equiv argmax}, we adopt the notations in \cite{cck2017aop} and define $ M_{n}(\delta)  =  M_{n, X}(\delta) +  M_{n, Y}(\delta)$, where 
    % \begin{align*}
    %     M_{n, X}(\delta) & = \frac{1}{n} \sum_{i=1}^n \EE \left[\max_{j \in [p]} |\eta_i - \mu_i|^3 \II\{\max_{j \in [p]} |\eta_i - \mu_i| > \delta \sqrt{n}/(4 \log p)\}\right],\\
    %     M_{n, Y}(\delta) & = \frac{1}{n} \sum_{i=1}^n \EE \left[\max_{j \in [p]} |\xi_i - \mu_i|^3 \II\{\max_{j \in [p]} |\xi_i - \mu_i| > \delta \sqrt{n}/(4 \log p)\}\right].
    % \end{align*}
    % Then the following proposition characterizes the asymptotic distribution of the maximizer of $Y$.
    \begin{condition}[Covariance condition]\label{cond: anti-con cov}
        Let $Z \in \RR^p$ be a random vector with covariance matrix $\Sigma$ and let $\sigma_j^2 = \Sigma_{jj}$ for $j = 1, \ldots, p$. We say $Z$ satisfies the covariance condition with respect to the nontrivial partition $\{\cA, \cB\}$ of $[p]$ if 
        % $\max_{i \in \cA, j \in \cB} \Corr(Z_i, Z_j) < 1$ when $\sigma_j^2 = \sigma^2$ for all $j \in [p]$. When there exists $i,j \in [p]$ such that $\sigma_i^2 \ne \sigma_j^2$, we assume $\max_{i \in \cA, j \in \cB} |\Corr(Z_i, Z_j)| < 1$, and
        one of the following conditions holds,
        \begin{enumerate}[label=(\Alph*)]
     \item $\max_{j,j' \in \cB} \sigma_{j}^{-2} \Sigma_{j j'} \le 1$, and $C_{\cA, \cB} := \min_{j \in \cB, i \in \cA}(\sigma_{j} - \sigma_{j}^{-1} \Sigma_{i j} )> 0$,
     \item  $\max_{i,i' \in \cA} \sigma_{i}^{-2} \Sigma_{i i'} \le 1$, and $C_{\cA, \cB} := \min_{i \in \cA, j \in \cB}(\sigma_{i} -  \sigma_{i}^{-1} \Sigma_{i j} )> 0$,
 \end{enumerate}
 and that $\max_{i \in \cA, j \in \cB} |\Corr(Z_i, Z_j)| < 1$ when there exists $i,j \in [p]$ such that $\sigma_i^2 \ne \sigma_j^2$.
    \end{condition}
   % \begin{remark}
   %     Note that when $\sigma^2_j = \sigma^2$ for all $j \in [p]$, the conditions (A) and (B) are equivalent and can be reduced to $\max_{i \in \cA, j \in \cB} \Corr(Z_i, Z_j) < 1$. The subsequent results also hold for both $\cS_{\rm p} = \cA$ and  $\cS_{\rm p} = \cB$ under the uniform variance case. 
   % \end{remark}
 %    \begin{remark}
 %       Together with \eqref{eq: equiv argmax}, we have that  \eqref{eq: gauss comp rate} implies
 %  \begin{equation*}
 %  % \label{eq: gauss comp rate argmax}
 % \begin{aligned}
 %     & \sup_{ v \in \RR^p} \left|\PP\Big(\argmax_{j \in \cA}(V_j + v_j) \subseteq \cA \Big) - \PP\Big(\argmax_{j \in \cA}(Z_j + v_j) \subseteq \cA \Big)  \right|\\
 %     & \quad \le K { \cdot} \frac{\EE[\max_{k \in \cS_{\rm p}}|Z_k|/\sigma_k]}{C_{\cA,\cB}}\sqrt{\Delta \log p},
 % \end{aligned}
 % \end{equation*}
 % which is of potential interest for perturbation analysis for the maximizers.
 %    \end{remark}
   Next, we bound \eqref{eq: equiv argmax} and \eqref{eq: equiv argmax mb} in the following theorem, and we provide the proof in Section~\ref{sec: proof sec application}.
\begin{theorem}\label{prop: clt gaussian approx rate}
   Let  $\xi_1, \ldots, \xi_n \in \RR^p$ be independent random vectors with $\EE (\xi_i) = \mu$ and $\Sigma_i = \EE \big[(\xi_i - \mu)(\xi_i - \mu)^{\top}\big]$ for $i\in[n]$, and 
% such that $\max_{i \in [N], j \in [p]} \EE |\xi_{ij}|^3 < \infty$.
 let $\eta_1, \ldots, \eta_n \in \RR^p$ be independent  Gaussian random vectors with $\eta_i \sim \cN(\mu, \Sigma_i)$ for $i \in [n]$. Define $X = n^{-1/2}\sum_{i=1}^n (\eta_i - \mu + a)$ and $Y = n^{-1/2}\sum_{i=1}^n (\xi_i - \mu + a) $.
Assume $\Sigma^* = \frac{1}{n}\sum_{i=1}^n \Sigma_i$ satisfies the covariance Condition~\ref{cond: anti-con cov} for a nontrivial partition $\{\cA, \cB\}$ of $[p]$, and suppose that there exist a $B_n \ge 1$ and a constant~$b_0$ such that we have
   \begin{equation}\label{eq: exp rate sample}
        \EE\big(\exp(|\xi_{ij}|/B_n)\big) \le 2, \quad \text{for all } i \in [n], j \in [p],
   \end{equation}
   and 
   \begin{equation}\label{eq: mom rate sample}
       \frac{1}{n} \sum_{i=1}^n \EE(\xi_{ij}^4) \le B_n^2 b_0^2, \quad \text{for all } j \in [p].
   \end{equation}
   Furthermore, assume that  $ B_n^2 \log^5(pn) \le c n$ for a small enough constant $c > 0$.
   Then, there exists a constant  $C >0$  depending only on $b_0$ such that
   \begin{equation}\label{eq: CLT rate}
       \begin{aligned}
       & \sup_{s \in \RR} \left| \PP(M_{\cA}^Y - M_{\cB}^Y \le s) - \PP( M_{\cA}^X - M_{\cB}^X  \le s) \right|\\
       &\quad \le C \cdot \frac{\EE \big(\max_{\ell \in \cS_{\rm p}} |X_{\ell} - \EE( X_{\ell} )|/\sigma_{\ell}\big)}{C_{\cA, \cB}} \left(\frac{B_n^2 \log^3(pn)}{n}\right)^{1/4},
       \end{aligned}
   \end{equation}
  where  $\cS_{\rm p} = \cB$ when  (A) of Condition~\ref{cond: anti-con cov} is satisfied, and $\cS_{\rm p} = \cA$ when (B) of Condition~\ref{cond: anti-con cov} is satisfied. Meanwhile,  for the multiplier bootstrap random vector $\hat{X}$ as defined in \eqref{eq: multiplier X}, with probability at least $1-1/(2n^4) - 1/n - 3 ({B_n^2 \log^3(pn) /n})^{1/2}$, we have that  \begin{equation}\label{eq: CLT rate 1}
       \begin{aligned}
       & \sup_{s \in \RR} \left| \PP(M_{\cA}^Y - M_{\cB}^Y \le s) - \PP( M_{\cA}^{\hat{X}} - M_{\cB}^{\hat{X}}  \le s \,|\, \xi_{1:n})\right| \\
       & \quad \le C \cdot \frac{\EE [\max_{\ell \in \cS_{\rm p}} |X_{\ell} - \EE(X_{\ell})|/\sigma_{\ell}]}{C_{\cA, \cB}} \left(\frac{B_n^2 \log^3(pn)}{n}\right)^{1/4}.
       \end{aligned}
   \end{equation}
   \end{theorem} 

The proof of Theorem~\ref{prop: clt gaussian approx rate} follows the iterative randomized Lindeberg technique from \cite{cck2022improvedbootstrap} with an application of the anti-concentration bounds developed in the previous section.  Moreover, it smooths the function as follows:
% $\{X_i\}_{i\in\NN}, \{Y_i\}_{i\in\NN}$, 
% we aim to bound the distance, 
$$
\sup_{s \in \RR, v \in \RR^p} \left|\PP\Big(\cM_{\cA}^v(Y_n) - \cM_{\cB}^v(Y_n) \le s\Big) - \PP\Big(\cM_{\cA}^v(X_n) - \cM_{\cB}^v(X_n) \le s\Big)  \right|,
$$
where  $X_n$ and $Y_n$ are two $p$-dimensional random vectors with distributions depend on~$n$, and $\cM_{\cS}^v(x) = \max_{j \in \cS} (x_j + v_j)$ for $x \in \RR^p$ and $\cS \in\{ \cA, \cB\}$.
To bound this, we first construct a smooth approximation function $\varphi_s^v: \RR^p \rightarrow \RR$ for the indicator function of the difference of the two maxima such that for all $ (x,v,s) \in \RR^p \times \RR^p \times \RR$, we have
$$
\II\{\cM_{\cA}^v(x) - \cM_{\cB}^v(x) \le s-\delta\} \le \varphi_s^v(x) \le  \II\{\cM_{\cA}^v(x) - \cM_{\cB}^v(x) \le s + 2\delta\} ,
$$
where \(\delta > 0\) is a tuning parameter controlling the accuracy of approximation. Then, by the triangle inequality, we bound the distance by
\begin{align*}
    &\sup_{s \in \RR, v \in \RR^p} \left|\PP\Big(\cM_{\cA}^v(Y_n) - \cM_{\cB}^v(Y_n) \le s\Big) - \PP\Big(\cM_{\cA}^v(X_n) - \cM_{\cB}^v(X_n) \le s\Big)  \right|\\
    & \quad \le %\underbrace{
    \sup_{s \in \RR, v \in \RR^p} \left|\PP\Big(\cM_{\cA}^v(X_n) - \cM_{\cB}^v(X_n) \le s\Big) - \PP\Big(\cM_{\cA}^v(X_n) - \cM_{\cB}^v(X_n) \le s + 3\delta\Big)  \right|
    %}_{\rm I} 
    \\
    & \quad \quad + \sup_{v \in \RR^p, s \in \RR}\left|\EE [\varphi_s^v(Y_n)] - \EE [\varphi_s^v(X_n)]\right|. %}_{\rm II},
\end{align*}
where we bound the first term by applying the anti-concentration bounds developed in Theorem~\ref{col: anti con equal var no cor 1} and Theorem~\ref{thm: anti con}, and we bound the second term on the RHS by the iterative randomized Lindeberg method in \cite{cck2022improvedbootstrap}.

\begin{remark}
    The smoothing and the anti-concentration results developed in the previous section can be applied to achieve more general bootstrap approximation results following the arguments in~\cite{cck2022improvedbootstrap}. We refer interested readers to \cite{cck2022improvedbootstrap} for details.
\end{remark}

\section{Numerical experiments}\label{sec: simu}
We conduct thorough numerical experiments under several different settings to empirically benchmark the anti-concentration bounds developed earlier. We also compare our bounds with \cite{imaizumi2021gaussian} empirically, where the authors derive an anti-concentration bound by conditioning on one of the two Gaussian random vectors, with the L\'evy concentration function scaling with the minimal eigenvalue of the covariance matrix.

Define $\cA = \{1,\ldots, p/2\}$ and $\cB = \{p/2 + 1, \ldots, p\}$. Denote by $X_{\cA}$ and $X_{\cB}$ the  Gaussian random vectors corresponding to the sets $\cA$ and $\cB$, respectively. We generate $X_{\cA}$ and $X_{\cB}$ by 
\begin{equation}\label{eq: simu Xa Xb gen}
    X_{\cA} = \Gamma_{\cA} Z + \mu_{\cA}, \text{ and } X_{\cB} = \Gamma_{\cB} Z + \mu_{\cB}, \quad X_{\cA}, X_{\cB} \in \RR^{p/2},
\end{equation}
where $Z \sim \cN(\mathbf{0}, \Ib_d)$ is a $d$-dimensional standard Gaussian random vector, and we pre-specify $\mu_{\cA}, \mu_{\cB} \in \RR^{p/2}$ and \(\Gamma_{\cA}, \Gamma_{\cB}\in\RR^{p/2 \times d}\) for different settings. Letting $M_{\cA} = \max_{i \in \cA} X_i$ and $M_{\cB} = \max_{j \in \cB}X_j$, we  evaluate the L\'evy concentration function $\cL(M_{\cB} - M_{\cA}, \varepsilon)$ under different covariance structures. 

We repeat the generation scheme 5,000 times for each setting and report the empirical performance. Specifically, to evaluate $\cL(M_{\cB} - M_{\cA}, \varepsilon)$, we present the empirical distribution of $\PP(|M_{\cB} - M_{\cA} - t| \le \varepsilon)$ over a range of $t$, where we compute the empirical probabilities over 1,000 equidistant points of $t$ from the minimum to the maximum realizations of $M_{\cB} - M_{\cA}$ and present the maximum empirical probability. Next, to evaluate the dependence of $\cL(M_{\cB} - M_{\cA}, \varepsilon)$ on $p$, we plot  the normalized interval length $\sqrt{\log p} \varepsilon / \hat{\sigma}$ versus
$\cL(M_{\cB} - M_{\cA}, \varepsilon)$, where {$\hat{\sigma} = p^{-1} \sum_{i=1}^p \sigma_i$ }is the average standard deviation for $\{X_i\}_{i=1}^p$. In Sections~\ref{sec: simu eq} and~\ref{sec: simu uneq}, we consider the cases where $d<p$, and in Section~\ref{sec: simu comp}, we compare our bounds with that in \cite{imaizumi2021gaussian} under the full rank setting.

\subsection{Homogeneous variance case}\label{sec: simu eq}

We begin by discussing the homogeneous variance case. The entries of the matrices \(\Gamma_{\cA}\) and \(\Gamma_{\cB}\) are drawn independently from a standard Gaussian distribution. To ensure homogeneous variance across ultivariate Gaussian random vectors \(X_{\mathcal{A}}\) and \(X_{\mathcal{B}}\), we normalize the row norms of \(\Gamma_{\cA}\) and \(\Gamma_{\cB}\) to unity. We consider dimensions \(p \in \{2000, 3000, 4000, 5000\}\), and change the number of overlapping rows between \(\Gamma_{\cA}\) and \(\Gamma_{\cB}\) to assess the impact of violating the condition \(\max_{i \in \mathcal{A}, j \in \mathcal{B}} \operatorname{Corr}(X_i, X_j) < 1\) on the L\'evy concentration function.

 % We first provide the benchmark anti-concentration results of the single Gaussian maximum $\max_{j \in [p]} X_j = \max\{M_{\cA}, M_{\cB}\}$ with $\max_{i,j \in [p], i \ne j} \Corr(X_i, X_j) < 1$. Figure~\ref{fig: reference single max}(a) in Appendix~\ref{sec: app add plot} shows that $\cL(\max_{j \in [p]} X_j,\varepsilon)$ grows linearly with $\sqrt{\log p} \varepsilon/\hat\sigma$, and Figure~\ref{fig: reference single max}(b) in Appendix~\ref{sec: app add plot} provides the empirical density of the normalized $\max_{j \in [p]} X_j$. {\color{red}EF: If we talk about it first, shall we move the figures here? Or, only mention this at the end of this subsection?}

We first consider $M_{\cB} - M_{\cA}$ when $\max_{i \in \cA,j \in \cB} \Corr(X_i, X_j) < 1$. We present the empirical results in Figure~\ref{fig: equal var cor < 1}. In particular, for the centered case that $(X_i)_{i=1}^p$ with $\mu_{\cA} = \mu_{\cB} = \mathbf{0}$, 
 Figure~\ref{fig: equal var cor < 1}(a) shows that $\cL(M_{\cB} - M_{\cA}, \varepsilon)$ scales linearly with $\sqrt{\log p} \varepsilon / \hat{\sigma}$, and 
Figure~\ref{fig: equal var cor < 1}(b) provides the empirical densities of normalized $M_\cA - M_\cB$. We then consider the case of non-centered $(X_i)_{i=1}^p$, where we draw each $\mu_i$ independently from an exponential distribution with rate 1. We provide the results in Figures~\ref{fig: equal var cor < 1}(c) and~(d), and we observe that the results are similar to the centered case. This is consistent with our theoretical findings.

%Figure~\ref{fig: equal var cor < 1} presents the results when $\max_{i \in \cA,j \in \cB} \Corr(X_i, X_j) < 1$ by setting the number of overlapping rows $K$ between $\Gamma_{\cA}$ and $\Gamma_{\cB}$ to be 0. Figure~\ref{fig: equal var cor < 1}~(a) and (b) provide the results for centered $\{X_i\}_{i=1}^p$ ($\mu_{\cA} = \mu_{\cB} = \mathbf{0}$). We can see from  Figure~\ref{fig: equal var cor < 1}~(a) that the empirical measure of $\cL(M_{\cB} - M_{\cA}, \varepsilon)$ scales well with $\sqrt{\log p} \varepsilon / \hat{\sigma}$.
% where {\red $\hat{\sigma} = p^{-1} \sum_{i=1}^p \sigma_i$ }is the mean standard deviation for $\{X_i\}_{i=1}^p$. Figure~\ref{fig: equal var cor < 1}~(c) and (d) provide the results for non-centered $\{X_i\}_{i=1}^p$, where we draw $\mu_i$'s independently from an exponential distribution with a rate parameter of 1. No significant impact of the non-zero means is observed in the anti-concentration result, which is consistent with our theoretical anti-concentration rate that is independent of the Gaussian means. 

Then, we consider the extreme case where \(\max_{i \in \cA, j \in \cB} \Corr(X_i, X_j) = 1\) with centered \((X_i)_{i=1}^p\), and we provide the empirical results in Figure~\ref{fig: equal var cor = 1}. Figures~\ref{fig: equal var cor = 1}(a) and (b) consider the case where the number of overlapping rows between \(\Gamma_{\cA}\) and \(\Gamma_{\cB}\) is \(K = p/8\), and Figures~\ref{fig: equal var cor = 1}(c) and (d)
consider the case of varying $K \in \{100, 200, 300, 400\}$ with $p  =  4000$. 
By the density plots in Figures~\ref{fig: equal var cor = 1}(b) and (d), we observe that the distributions of $M_\cB - M_\cA$ concentrate towards zero. Also, by  
Figure~\ref{fig: equal var cor = 1}(c), we find that the intercepts of the lines scale approximately in proportion with~$K$.

To verify whether \(\cL(M_{\cB} - M_{\cA}, \varepsilon)\) scales with the lesser of \(\EE(\max_{i \in \cA} |X_i - \mu_i|)\) and \(\EE(\max_{j \in \cB} |X_j - \mu_j|)\), as indicated by Theorem~\ref{col: anti con equal var no cor 1} when \(\cA\) and \(\cB\) are of different sizes, we conduct an additional experiment. In this experiment, we define \(\cA = \{1, \ldots, k_0\}\) and \(\cB = \{k_0+1, \ldots, p\}\) with \(\mu_i = 0\) for all \(i \in [p]\), and examine how \(\cL(M_{\cB} - M_{\cA}, 0.05)\) scales with \(p\) when \(k_0\) is fixed at \{100, 200, 300, 400\}. We present the results in Figure~\ref{fig:levy scale with p}. We can see that the L\'evy concentration function plateaus after \(p \ge 2 k_0\), i.e., \(|\cB| \ge |\cA|\). This observation is consistent with Theorem~\ref{col: anti con equal var no cor 1}.

%Figure~\ref{fig: equal var cor = 1} illustrates outcomes for centered \(\{X_i\}_{i=1}^p\) when \(\max_{i \in \cA, j \in \cB} \Corr(X_i, X_j) = 1\). Figure~\ref{fig: equal var cor = 1}~(a) and (b) examine cases with the number of overlapping rows between \(\Gamma_{\cA}\) and \(\Gamma_{\cB}\) set at \(K = p/8\). Figure~\ref{fig: equal var cor = 1}~(c) and (d) focus on the impact of varying \(K\) within \{100, 200, 300, 400\}, with \(p\) fixed at 4000. The plot reveals a notable trend of concentration around zero in the distribution of \(M_{\cB} - M_{\cA}\), with the intercept shift of \(\cL(M_{\cB} - M_{\cA}, \varepsilon)\) exhibiting a proportional relationship with \(K\). Contrastingly, the concentration diminishes for values diverging from zero, which is consistent with theoretical expectations. In Figure~\ref{fig: equal var cor = 1}~(e) and (f), we further study the anti-concentration results when the variances of the positively correlated random variables are unequal. Specifically, we set $K = p/8$ and scale the norm of the rows in $\Gamma_{\cB}$ overlapping with $\Gamma_{\cA}$ to be equal to 1/2. No significant concentration is observed across $p \in \{2000, 3000, 4000, 5000\}$, which suggests a potential relaxation of the anti-concentration conditions when there are positive perfect correlations present between subsets $\cA$ and $\cB$. 

\begin{figure}[htbp]
    \centering
    \begin{tabular}{cc}
         (a) $\cL(M_{\cB} - M_{\cA}, \varepsilon)$ when $\mu_{\cA} = \mu_{\cB} = \mathbf{0}$& (b) Density curves when $\mu_{\cA} = \mu_{\cB} = \mathbf{0}$ \\
         \!\!\!\!\includegraphics[width = 202 pt]{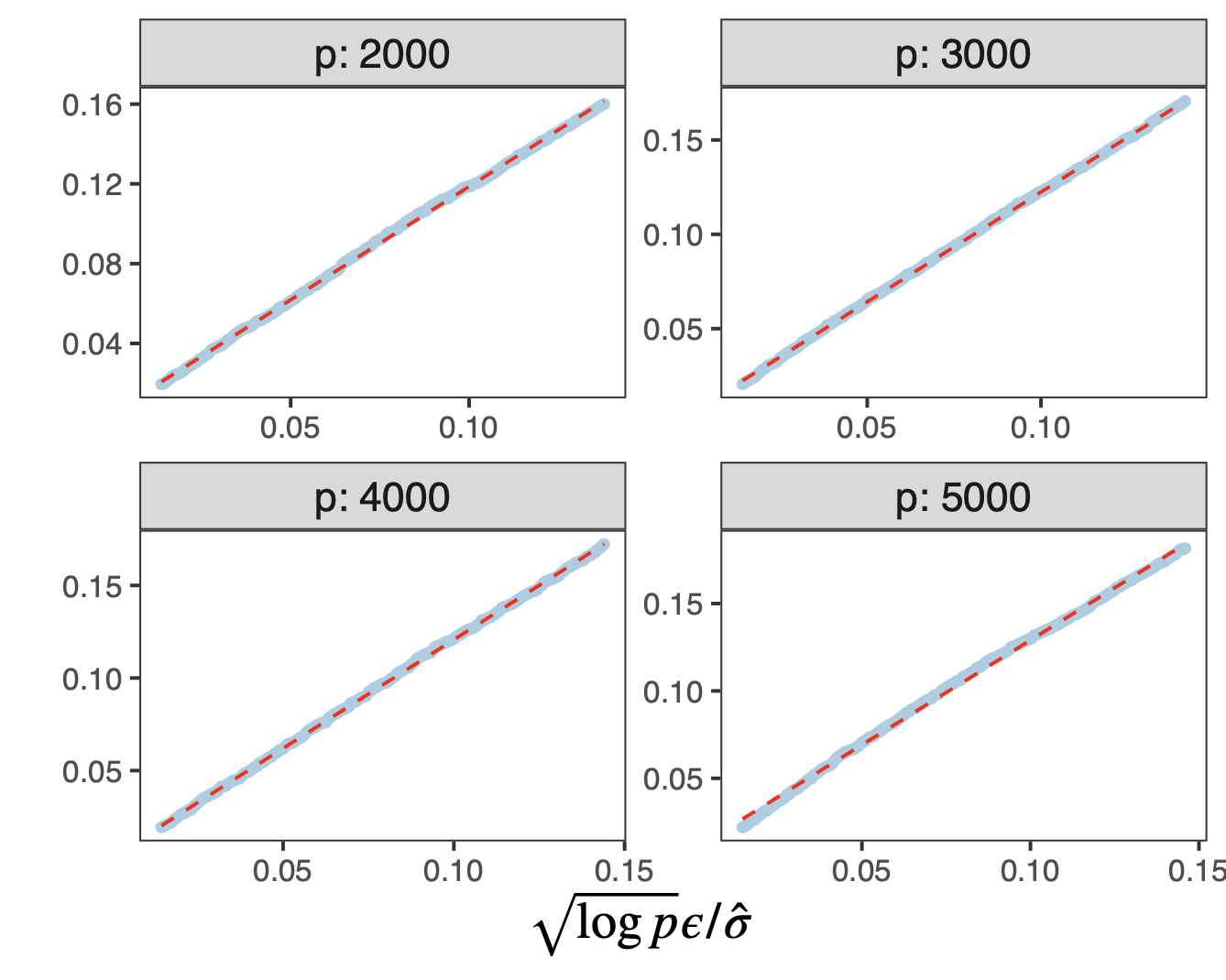}&  \!\!\!\!\!\includegraphics[width = 202 pt]{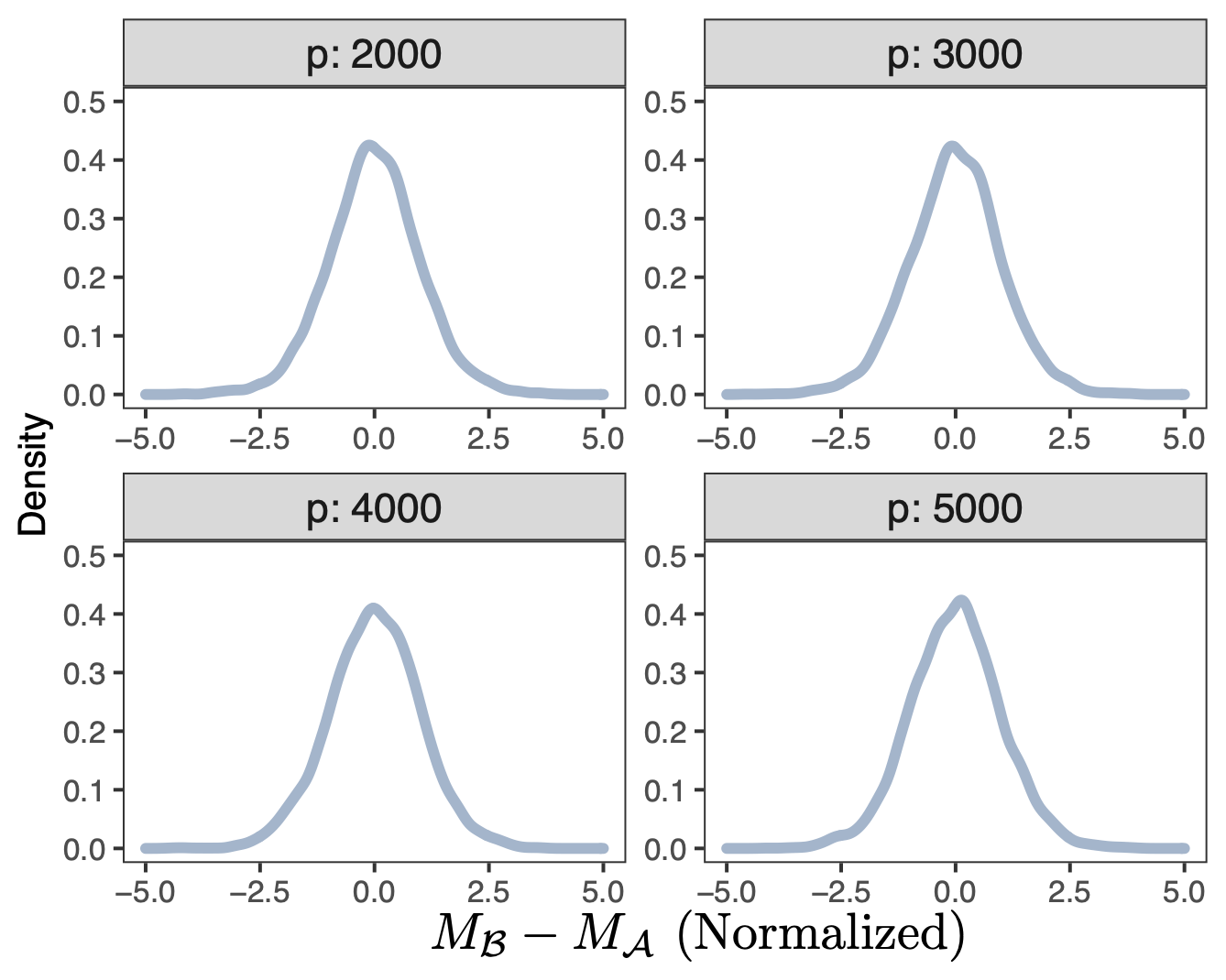}\\
         (c) $\cL(M_{\cB} - M_{\cA}, \varepsilon)$ when $\mu_{\cA}, \mu_{\cB} \ne \mathbf{0}$& (d) Density curves when $\mu_{\cA} , \mu_{\cB} \ne \mathbf{0}$ \\
         \!\!\!\!\includegraphics[width = 202 pt]{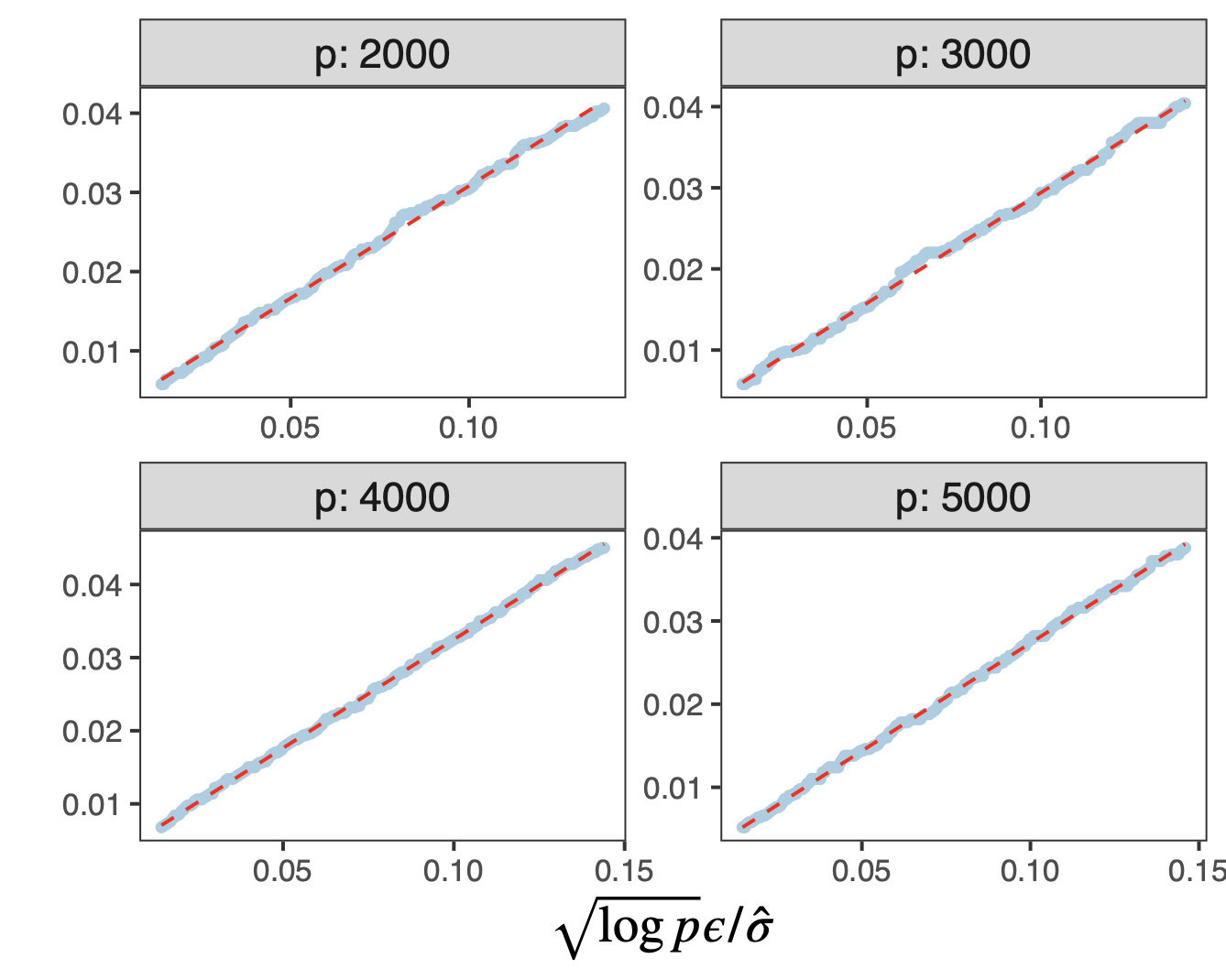}&  \!\!\!\!\!\includegraphics[width = 202 pt]{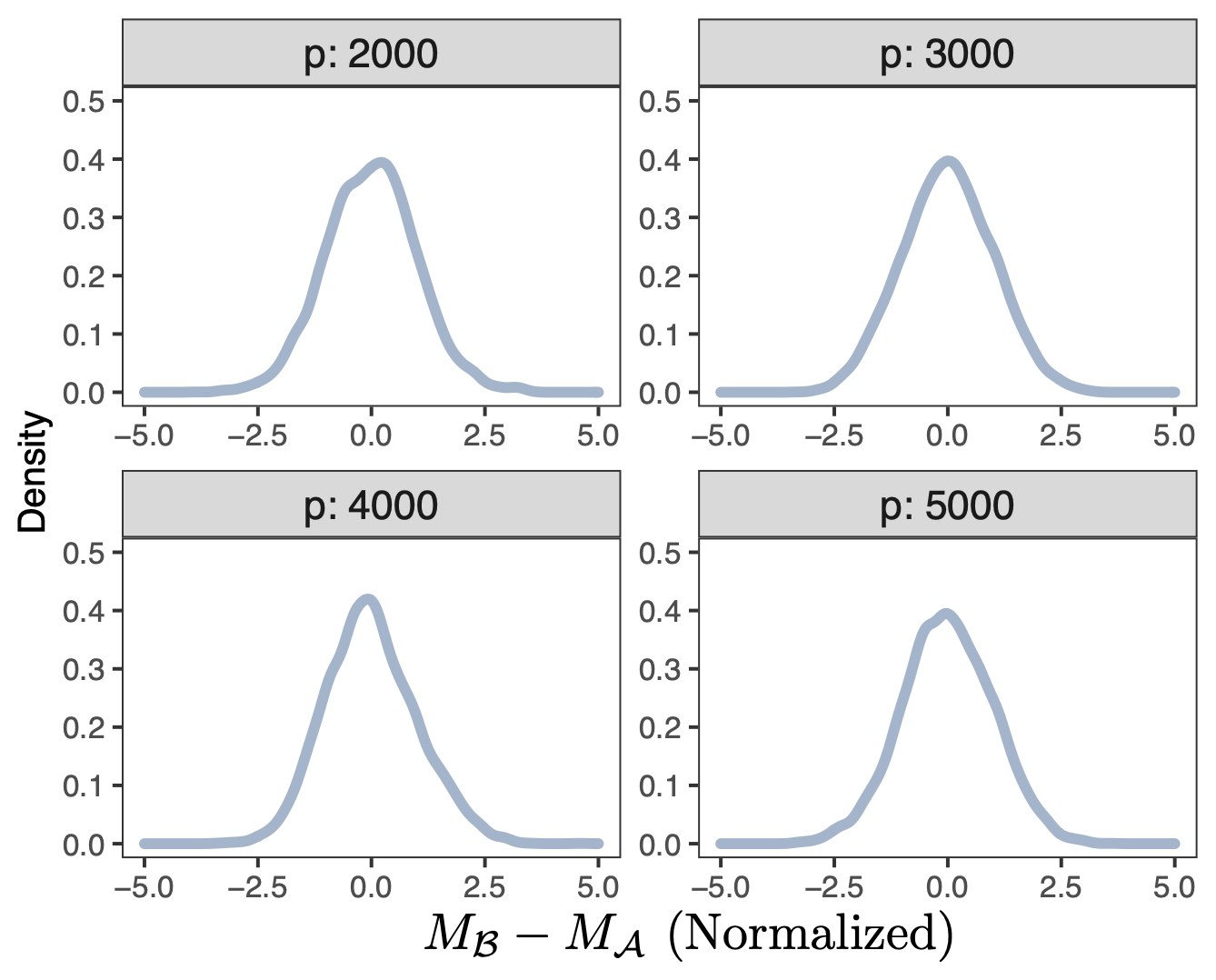}
    \end{tabular}
    \caption{Anti-concentration results of $M_{\cB} - M_{\cA}$ for the homogeneous variance case with $\max_{i \in \cA,j \in \cB} \Corr(X_i, X_j) < 1$. Panels (a) and (c) show that the empirical evaluation of the L\'evy concentration function scales linearly with $\varepsilon$ under both the centered and uncentered cases. Panels (b) and (d) present the density curves, where $M_{\cB} - M_{\cA}$ is normalized by the sample mean and standard deviation for comparability. Empirically, we have $\max_{i \in \cA,j \in \cB} \Corr(X_i, X_j) = \{0.916, 0.913, 0.923, 0.924\}$ for $p = \{2000, 3000, 4000, 5000\}$ respectively.}\label{fig: equal var cor < 1}
\end{figure}
\begin{figure}[htbp]
    \centering
    \begin{tabular}{cc}
         (a) $\cL(M_{\cB} - M_{\cA}, \varepsilon)$ with varied $p$& (b) Density curves with varied $p$\\
         \!\!\!\!\includegraphics[width = 202 pt]{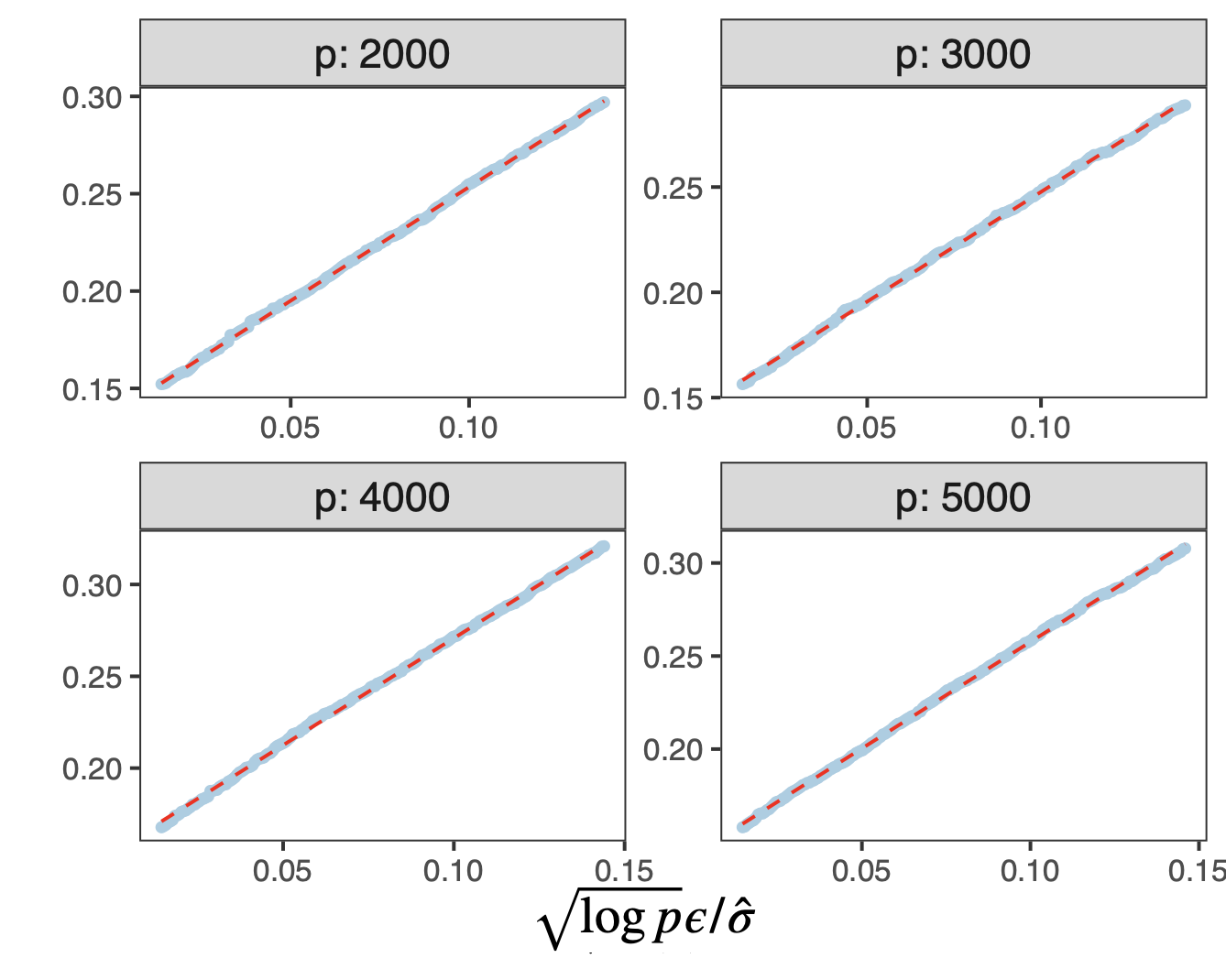}&  \!\!\!\!\!\includegraphics[width = 202 pt]{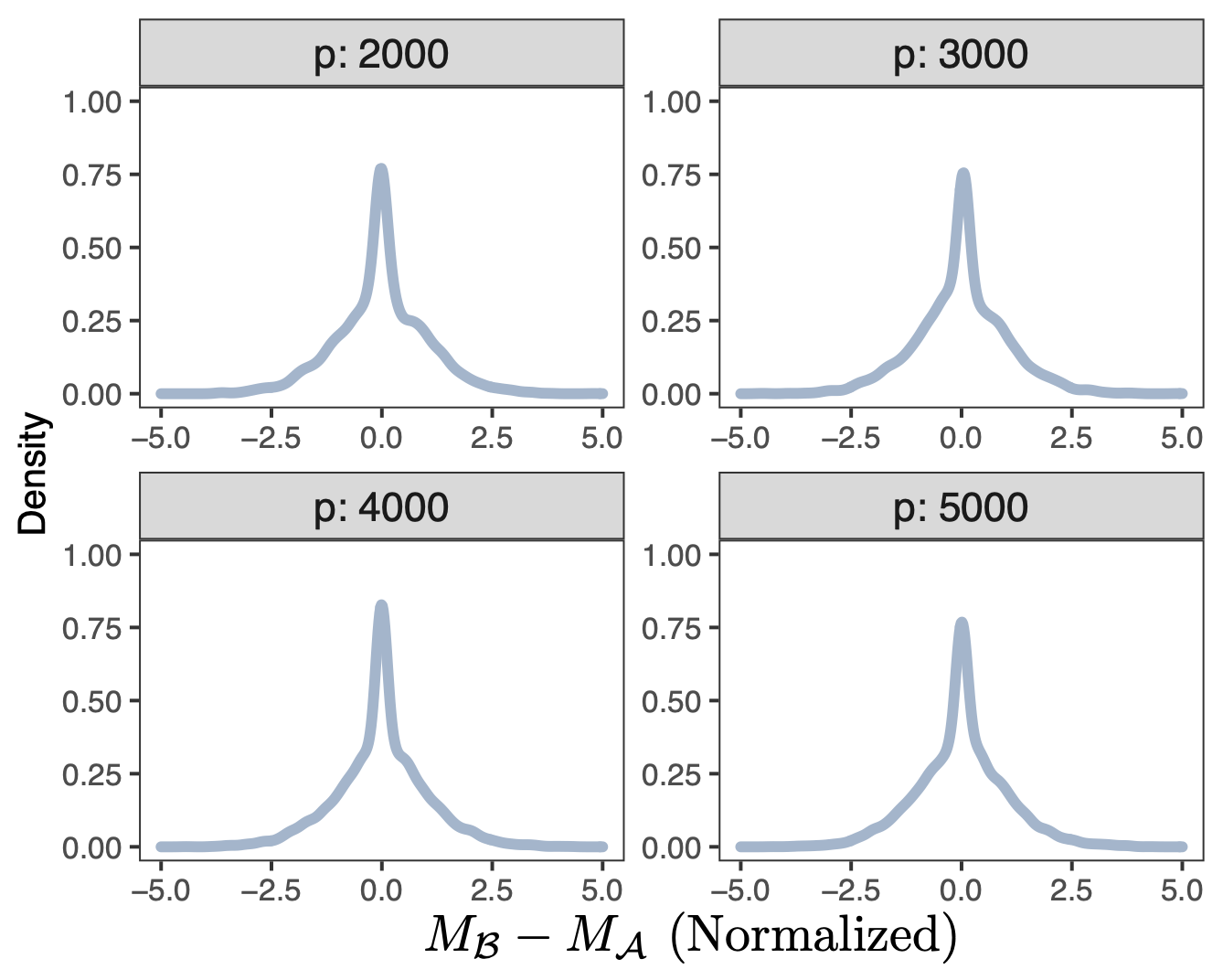}\\
         (c) $\cL(M_{\cB} - M_{\cA}, \varepsilon)$ with varied $K$ & (d) Density curves with varied $K$\\
         \!\!\!\!\includegraphics[width = 202 pt]{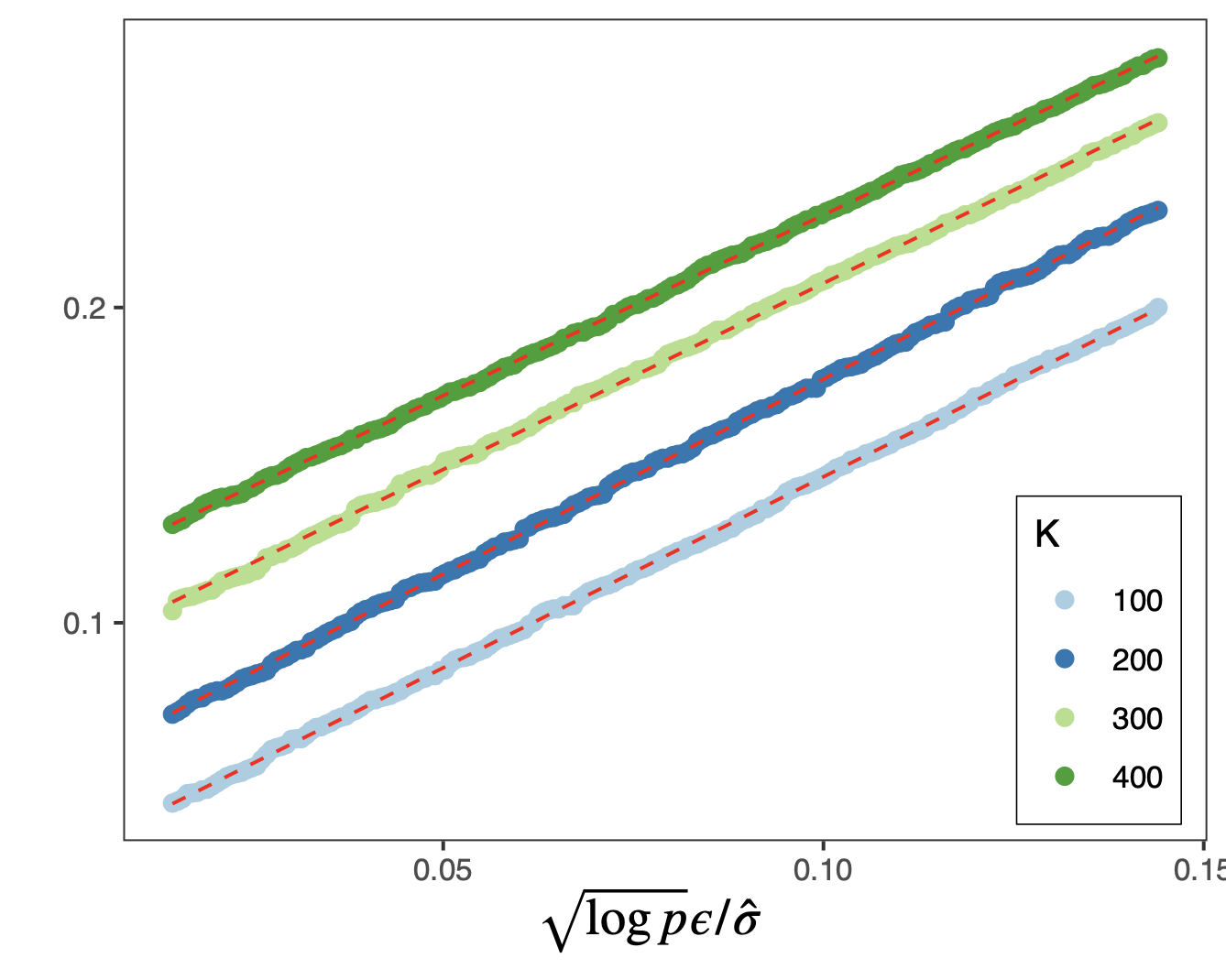}&  \!\!\!\!\!\includegraphics[width = 202 pt]{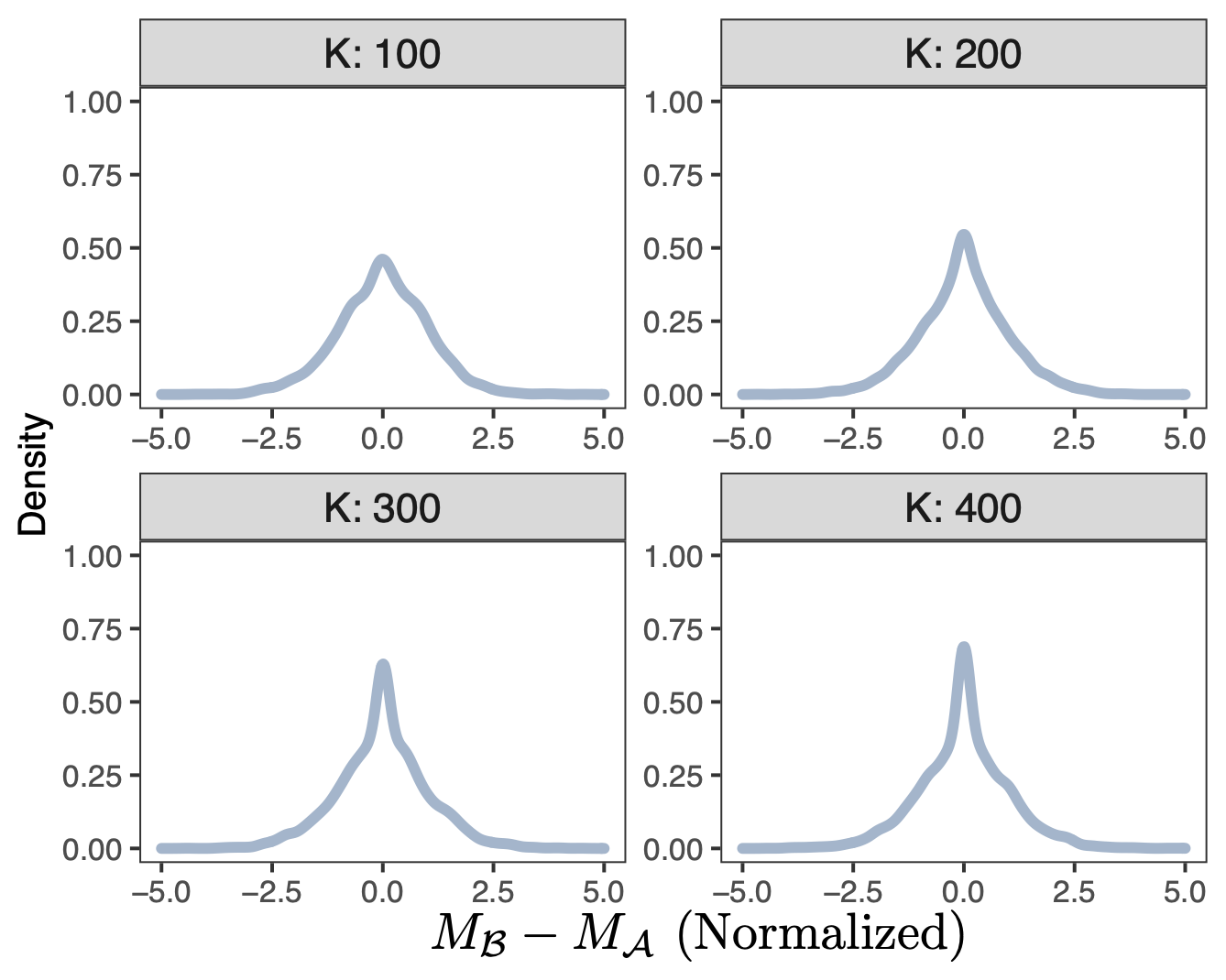}
    \end{tabular}
    \caption{Anti-concentration results of $M_{\cB} - M_{\cA}$ for the homogeneous variance case with $\max_{i \in \cA,j \in \cB} \Corr(X_i, X_j) = 1$. Panels (a) and (b) evaluate the anti-concentration behaviors at different $p$ when the number of overlapping components between $X_{\cA}$ and $X_{\cB}$ is $K = p /8$; Panels (c) and (d) evaluate the anti-concentration behaviors for varying $K$ at $p = 4000$.
    }\label{fig: equal var cor = 1}
\end{figure}
\begin{figure}
    \centering
    \includegraphics[width = 290 pt]{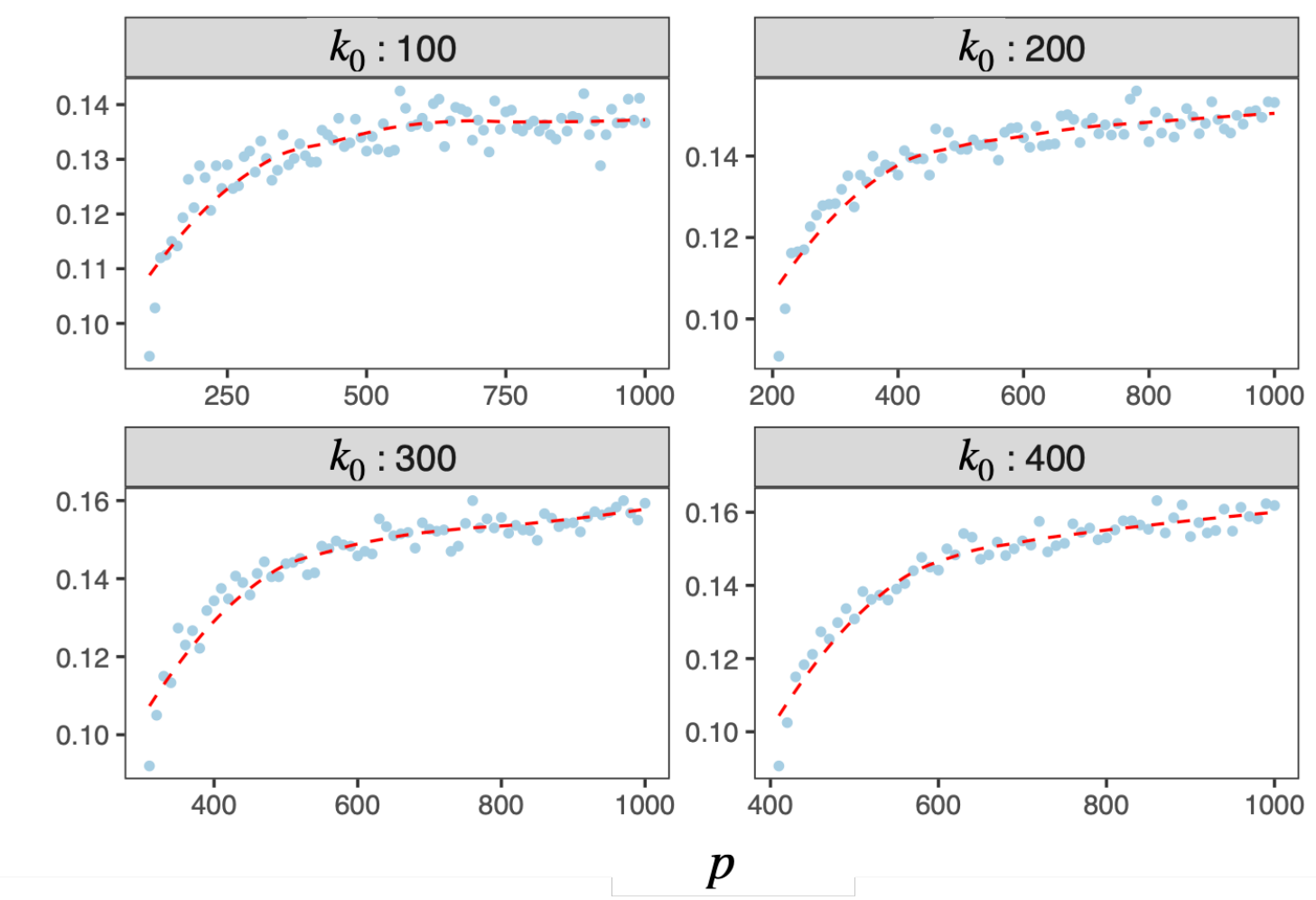}
    \caption{Scaling of \(\cL(M_{\cB} - M_{\cA}, 0.05)\) with \(p\) when \(\cA = [k_0]\) and \(\cB = [p] \backslash [k_0]\) at \(k_0 \in \{100, 200, 300, 400\}\) for centered \((X_i)_{i=1}^p\) with unit component-wise variances. The empirical evaluation of \(\cL(M_{\cB} - M_{\cA}, 0.05)\) plateaus after \(p \ge 2k_0\), i.e., \(|\cB| \ge |\cA|\). This suggests that the Lévy concentration function scales with the lesser of \(\EE(\max_{i \in \cA} |X_i - \mu_i|)\) and \(\EE(\max_{j \in \cB} |X_j - \mu_j|)\) when \(\cA\) and \(\cB\) are of different sizes, which is consistent with Theorem~\ref{col: anti con equal var no cor 1}.}
    \label{fig:levy scale with p}
\end{figure}
\subsection{Heterogeneous variance case}\label{sec: simu uneq}
We evaluate the anti-concentration behaviors under the heterogeneous variance setting with the covariance conditions on $C_{\cA, \cB}$ in Theorem~\ref{thm: anti con} satisfied and violated, respectively, i.e., $C_{\cA, \cB} > 0$ and $C_{\cA,\cB} \le 0$. We take $\mu_{\cA} = \mu_{\cB} = \mathbf{0}$ in this subsection. 

We first consider the case where the covariance condition is satisfied that $\max_{i,j \in \cB} \sigma_i^{-2} \sigma_{i j} \le 1$, and $C_{\cA, \cB} = \min_{j \in \cB, i \in \cA} (\sigma_j - \sigma_j^{-1} \sigma_{i j}) > 0$. In particular, we draw the entries of $\Gamma_{\cA}$ and $\Gamma_{\cB}$ independently from the standard Gaussian distribution. We then normalize the row norms of $\Gamma_{\cB}$ to 1 and scale $\Gamma_{\cA}$ by the inverse of its maximal row norm. In this case, $\sigma_j^2 = 1$ for $j \in \cB$ and $\sigma_i^2 \le 1$ for $i \in \cA$. We present the results in Figure~\ref{fig: unequal var Cs not violated}. In the design illustrated by Figure~\ref{fig: unequal var Cs not violated}, we have that $\cL(M_{\cB} - M_{\cA}, \varepsilon)$ scales linearly with $\sqrt{\log p} \varepsilon/\hat\sigma$ when the covariance condition is satisfied.

Next, we consider the case in which the covariance condition is violated. That is, when $\max_{i \in \cA, j \in \cB} \Corr(X_i, X_j) < 1$, and $\max\{\min_{j \in \cB, i \in \cA} (\sigma_j - \sigma_j^{-1} \sigma_{i j} ), \min_{i \in \cA, j \in \cB} (\sigma_i - \sigma_i^{-1} \sigma_{i j} ) \}\le 0$. Specifically, we let  $\cV_{\cA} = \{i \in \cA: \sigma_i - \max_{j \in \cB}\sigma_i^{-1} \sigma_{i j}  \le 0\}$ and $\cV_{\cB} = \{j \in \cB: \sigma_j - \max_{i \in \cA}\sigma_j^{-1} \sigma_{i j}  \le 0\}$ be the sets of indices violating the condition of $C_{\cA, \cB}>0$ in sets $\cA$ and $\cB$, respectively, and let $\nu_{\cA} = |\cV_{\cA}|/|\cA|$ and $\nu_{\cB} = |\cV_{\cB}|/|\cB|$ be the corresponding violation ratios.
We then let $m_{\cA} = |\cV_{\cA}|^{-1} \sum_{i \in \cV_{\cA}} (\sigma_i - \max_{j \in \cB}\sigma_i^{-1} \sigma_{i j})$ and $m_{\cB} = |\cV_{\cB}|^{-1} \sum_{j \in \cV_{\cB}} (\sigma_j - \max_{i \in \cA}\sigma_j^{-1} \sigma_{i j})$. We consider the full rank and equal correlation setting by letting $\Corr(X_i, X_j) = 0.9$ for all $i \ne j$. Then, we consider different $\{\sigma_j\}_{j \in [p]}$ by scaling the row norms of $\Gamma_{\cA}$ and $\Gamma_{\cB}$ for different violation ratios 
$\nu_{\cA}$ and $\nu_{\cB}$.

We present the empirical results in Figure~\ref{fig: unequal var Cs violated} for different settings. In Figures~\ref{fig: unequal var Cs violated}(a) and (b), we set $\sigma_j^2 = 1$ for all $j \in [p]$ as the benchmark. In Figures~\ref{fig: unequal var Cs violated}(c) and (d), we set 
$$
(\sigma^2_i)_{i \in \cA} = (\sigma^2_j)_{j \in \cB} = (0.9 \cdot \mathbf{1}_{p/4}^{\top}, \,  \mathbf{1}_{p/8}^{\top}, \, 10 \cdot \mathbf{1}_{p/8}^{\top})^{\top},
$$
and we have $\nu_{\cA} = \nu_{\cB} = 0.75$ and $m_{\cA} = m_{\cB} = -8.067$. In Figures~\ref{fig: unequal var Cs violated}(e) and (f), we set 
$$
(\sigma^2_i)_{i \in \cA} = (\sigma^2_j)_{j \in \cB} =  (0.9 \cdot \mathbf{1}_{3p/8}^{\top}, \,  \mathbf{1}_{p/16}^{\top}, \,  15 \cdot \mathbf{1}_{p/16}^{\top})^{\top},
$$
and we have $\nu_{\cA} = \nu_{\cB} = 0.875$ and $m_{\cA} = m_{\cB} = - 12.586$. By Figures~\ref{fig: unequal var Cs violated}(b), (d), and (f), as the violation ratios $\nu_{\cA}$ and $\nu_{\cB}$ increase, we observe a higher concentration. However, this may be due to the reduction in the smallest marginal variance $\min_{j \in [p]} \sigma_j^2$ after the normalization of $M_{\cA} - M_{\cB}$. Meanwhile, Figures~\ref{fig: unequal var Cs violated}(a), (c), and (e) show that the L\'evy concentration function still scales linearly with $\sqrt{\log p} \varepsilon/\hat\sigma$.

Finally, we consider the case where \(\max_{i \in \cA, j \in \cB} \Corr(X_i, X_j) = 1\) with centered \((X_i)_{i=1}^p\), and the perfectly correlated random variables from different subsets have unequal variances. In this setup, we set \(K = p/8\) and scale the norm of the overlapping rows in \(\Gamma_{\cA}\) and \(\Gamma_{\cB}\) to 1 and 1/2, respectively. We provide empirical results in Figures~\ref{fig: unequal var cor = 1}(a) and (b). In this case, we do not observe the concentrations in the empirical density plot as in the homogeneous variance cases. 
\begin{figure}[htbp]
    \centering
    \begin{tabular}{cc}
         (a) $\cL(M_{\cB} - M_{\cA}, \varepsilon)$ & (b) Density curves \\
         \!\!\!\!\includegraphics[width = 202 pt]{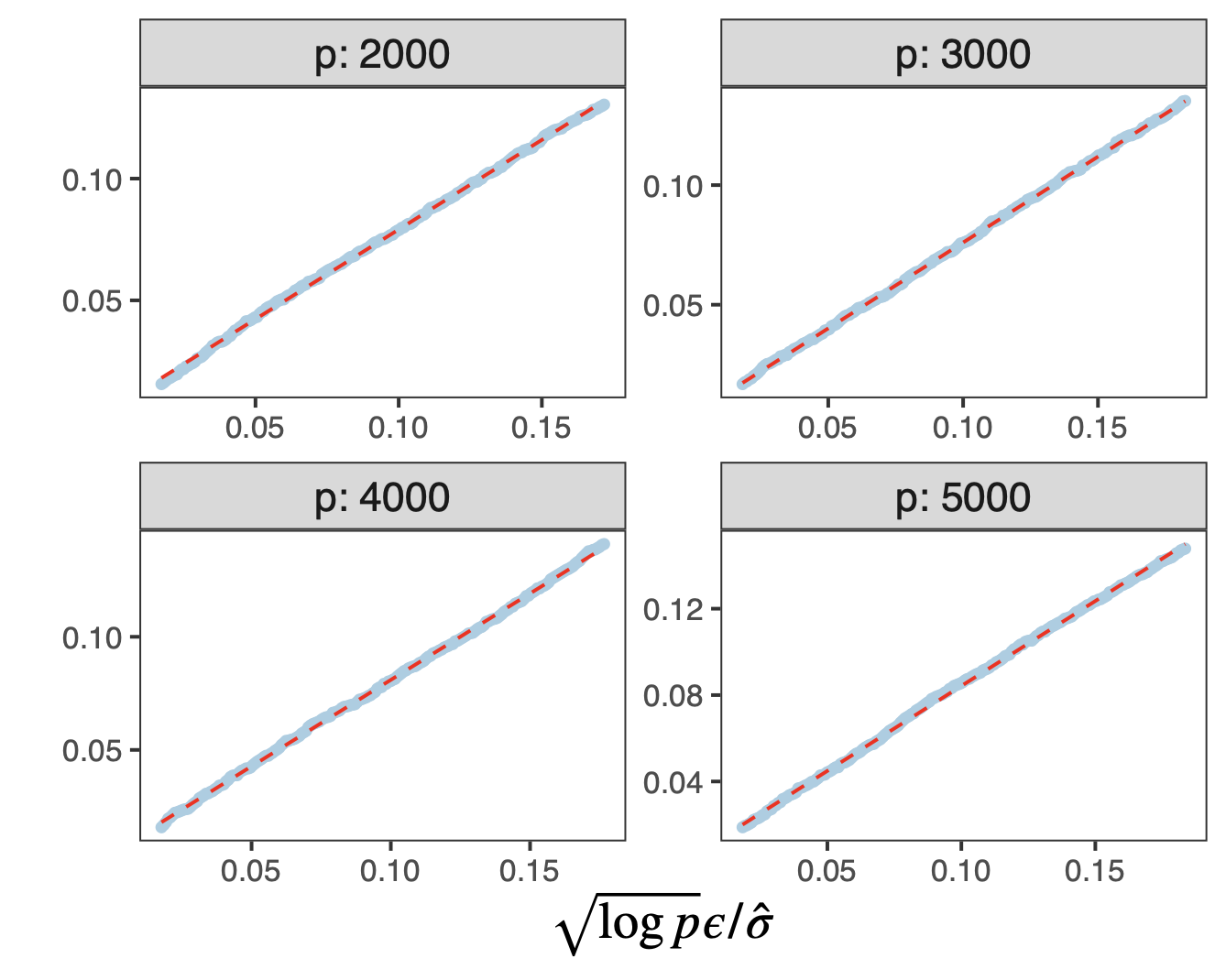}&  \!\!\!\!\!\includegraphics[width = 202 pt]{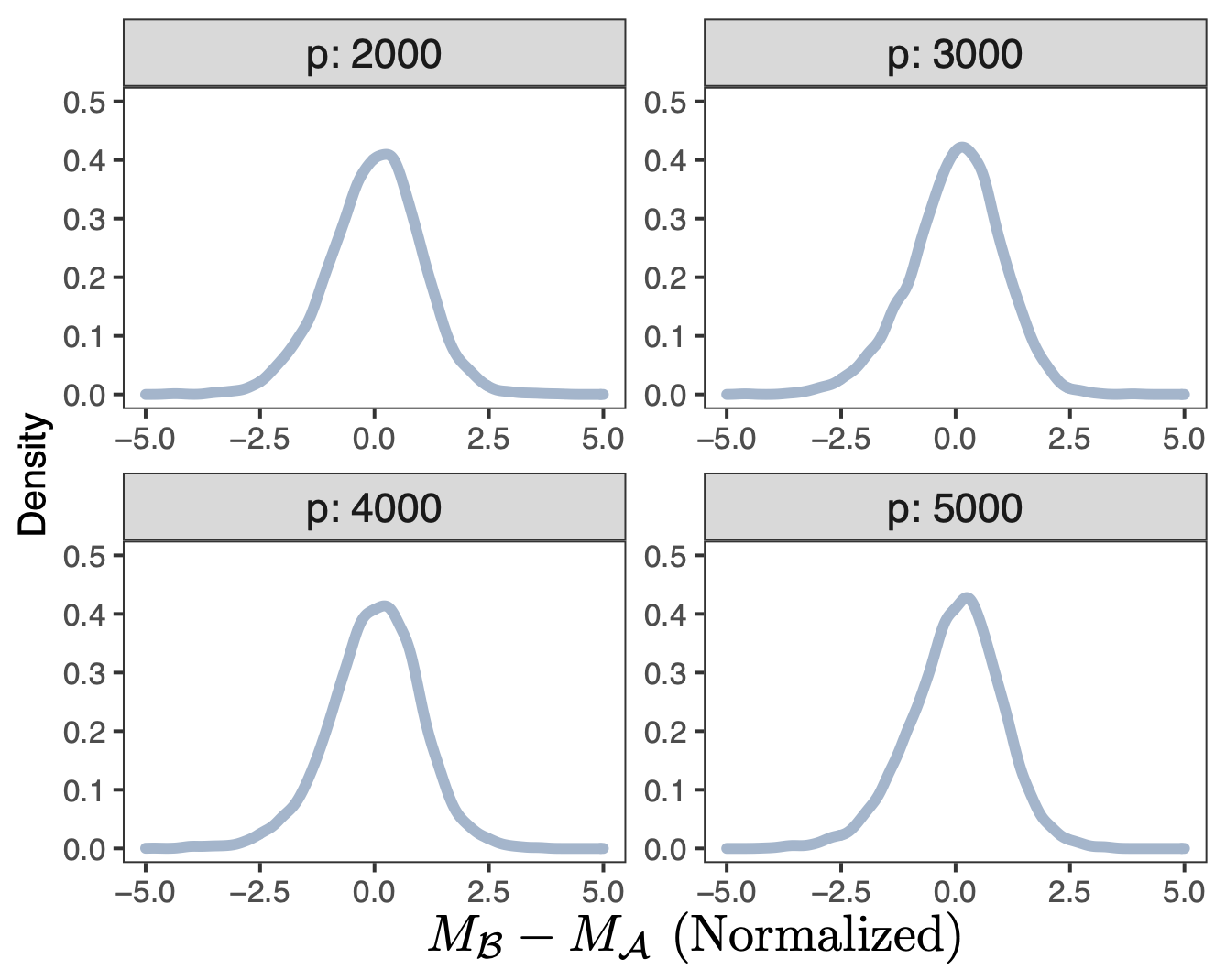}
    \end{tabular}
    \caption{Anti-concentration results of $M_{\cB} - M_{\cA}$ for the heterogeneous variance case when the covariance condition is satisfied, i.e., $\max_{i,j \in \cB} \sigma_i^{-2} \sigma_{i j} \le 1$, and $C_{\cA, \cB} = \min_{j \in \cB, i \in \cA} (\sigma_j - \sigma_j^{-1} \sigma_{i j} ) > 0$. Panel (a) shows that the empirical evaluation of the L\'evy concentration function still scales linearly with $\varepsilon$ under the heterogeneous variance case when the condition for $C_{\cA, \cB}$ is satisfied; Panel (b) presents the density curves for the normalized difference of Gaussian maxima, and there is no significant concentration.  For $p \in \{2000, 3000, 4000, 5000\}$, we have $\max_{i \in \cA,j \in \cB} \Corr(X_i, X_j) \in \{0.876, 0.926, 0.932, 0.923\}$  respectively. }\label{fig: unequal var Cs not violated}
\end{figure}

% \begin{figure}[htbp]
%     \centering
%     \begin{tabular}{cc}
%          (a) $\cL(M_{\cB} - M_{\cA}, \varepsilon)$ & (b) Density curves \\
%          \!\!\!\!\includegraphics[width = 202 pt]{plot_unequal_Cs_vio.pdf}&\!\!\!\!\!  \includegraphics[width = 202 pt]{hist_unequal_Cs_vio.pdf}
%     \end{tabular}
%     \caption{Anti-concentration results of $M_{\cB} - M_{\cA}$ for the unequal variance case with $\min_{i \in \cB} (\sigma_i - \max_{j \in \cA}\sigma_i^{-1} \sigma_{i j} ) \vee \min_{i \in \cA} (\sigma_i - \max_{j \in \cB}\sigma_i^{-1} \sigma_{i j} ) \le 0$ and $\max_{i \in \cA, j \in \cB} \Corr(X_i, X_j) < 1$. The ratio }\label{fig: unequal var Cs violated}
% \end{figure}
\begin{figure}[htbp]
    \centering
    \begin{tabular}{cc}
         (a) $\cL(M_{\cB} - M_{\cA}, \varepsilon)$ ($\nu_{\cA} = \nu_{\cB} = 0$) & (b) Density curves ($\nu_{\cA} = \nu_{\cB} = 0$)  \\
         \!\!\!\!\includegraphics[width = 180 pt]{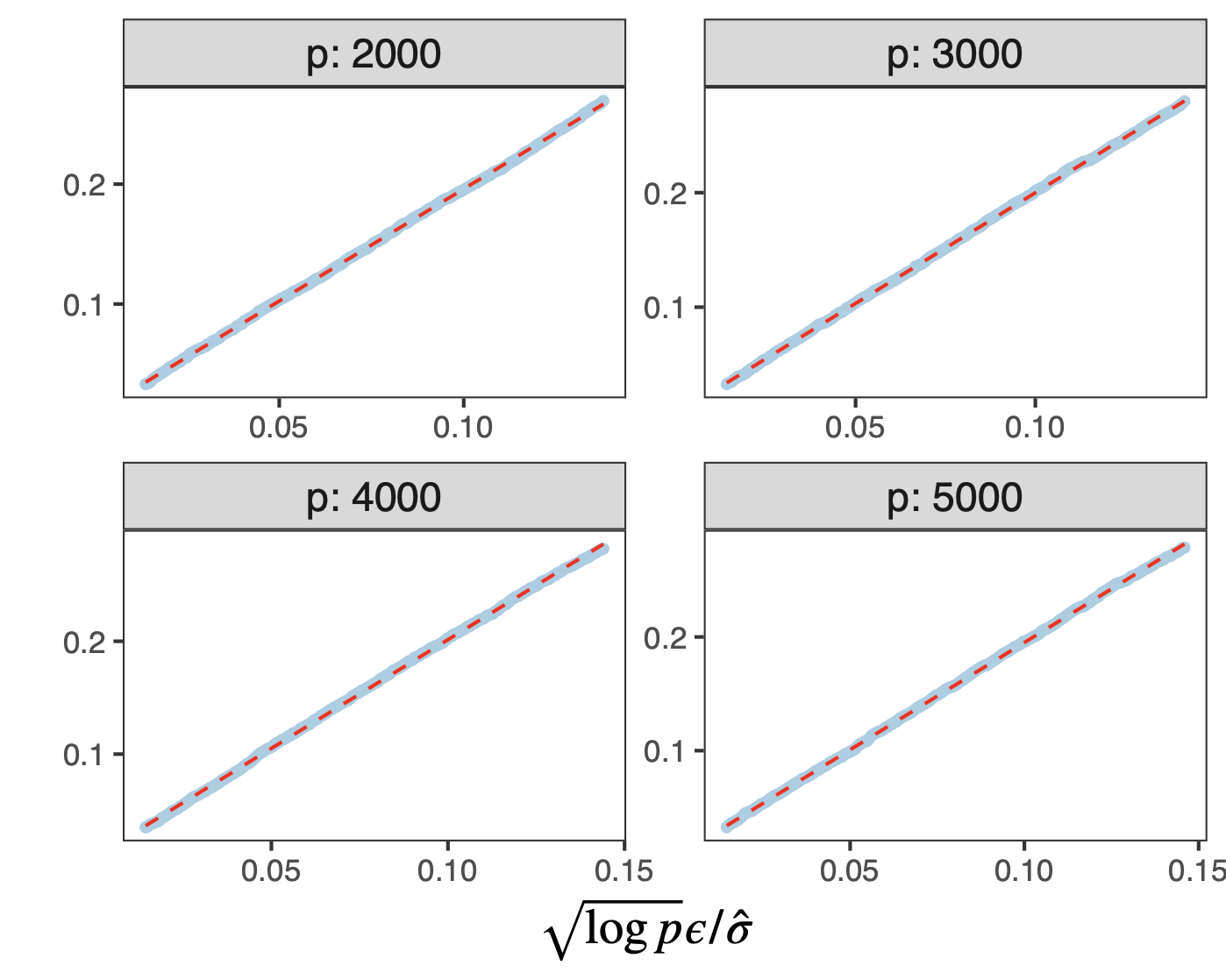}&\!\!\!\!\!  \includegraphics[width = 180 pt]{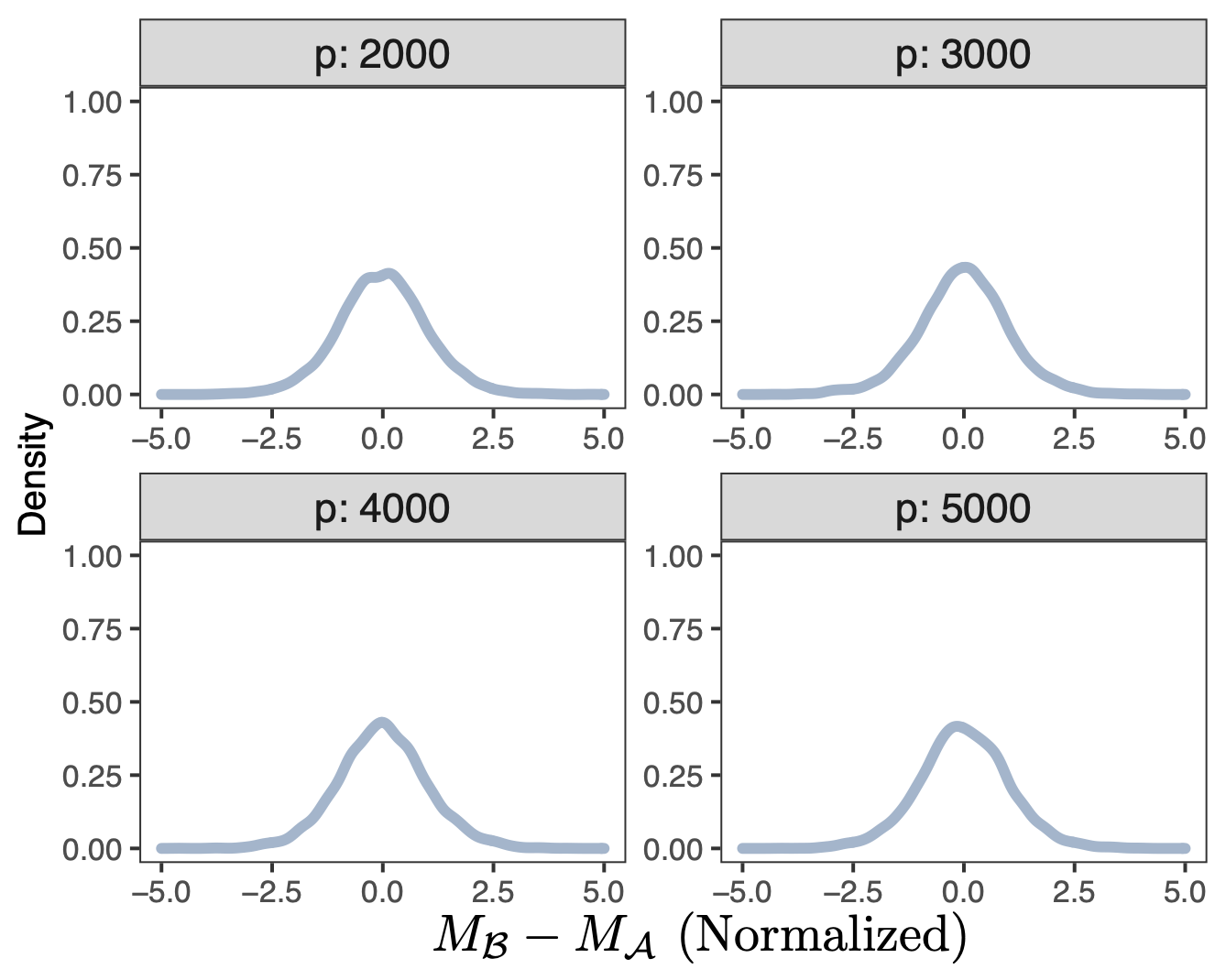}\\
         % (c) $\cL(M_{\cB} - M_{\cA}, \varepsilon)$ at $n=p/16$ & (d) Density curves at $n=p/16$ \\
         % \!\!\!\!\includegraphics[width = 202 pt]{0.2_2_2_2_2_2_2_2_plot.pdf}&\!\!\!\!\!  \includegraphics[width = 202 pt]{0.2_2_2_2_2_2_2_2_hist.pdf}\\
         (c) $\cL(M_{\cB} - M_{\cA}, \varepsilon)$ at ($\nu_{\cA} = \nu_{\cB} = 0.75$)  & (d) Density curves ($\nu_{\cA} = \nu_{\cB} = 0.75$)  \\
         \!\!\!\!\includegraphics[width = 180 pt]{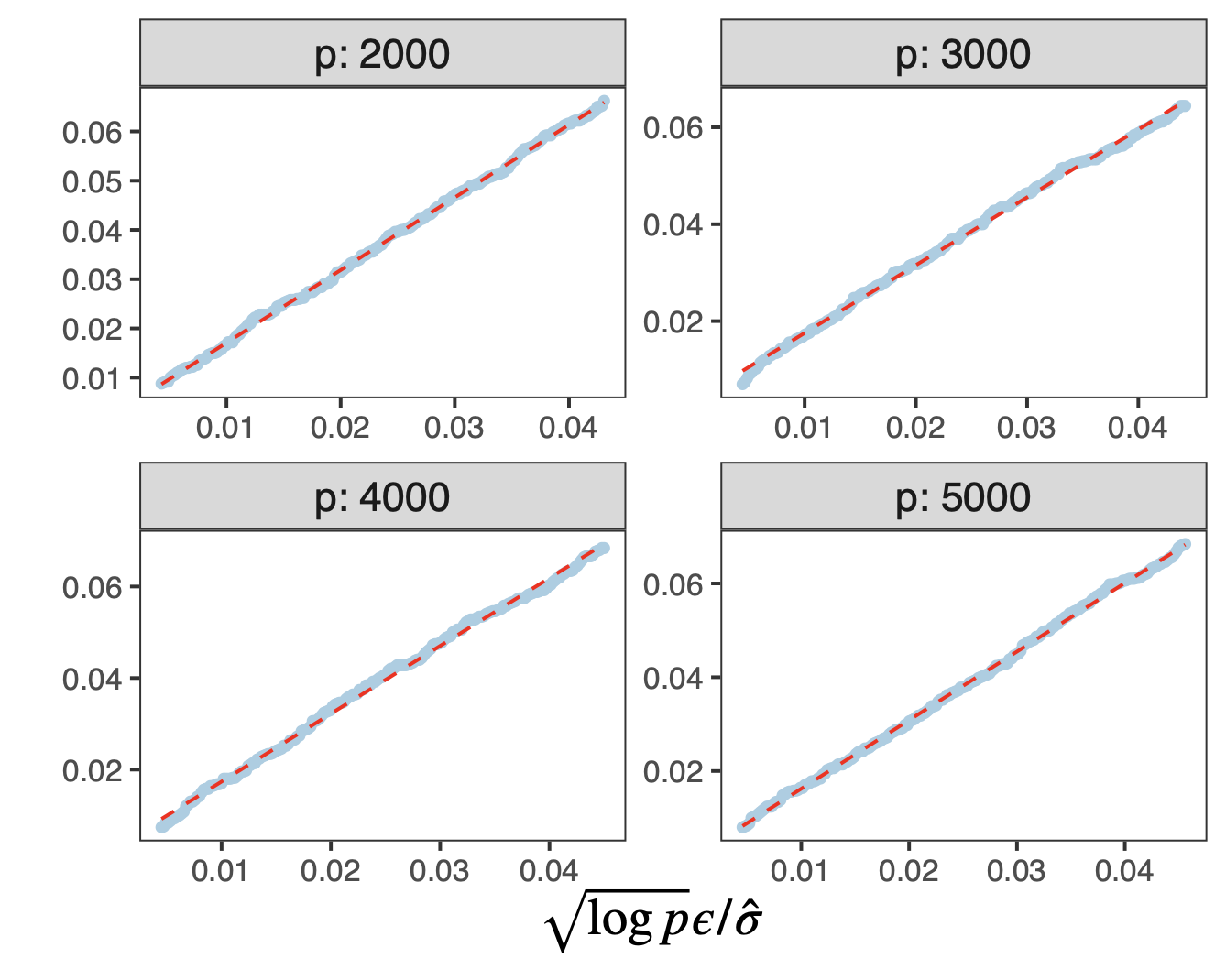}&\!\!\!\!\!  \includegraphics[width = 180 pt]{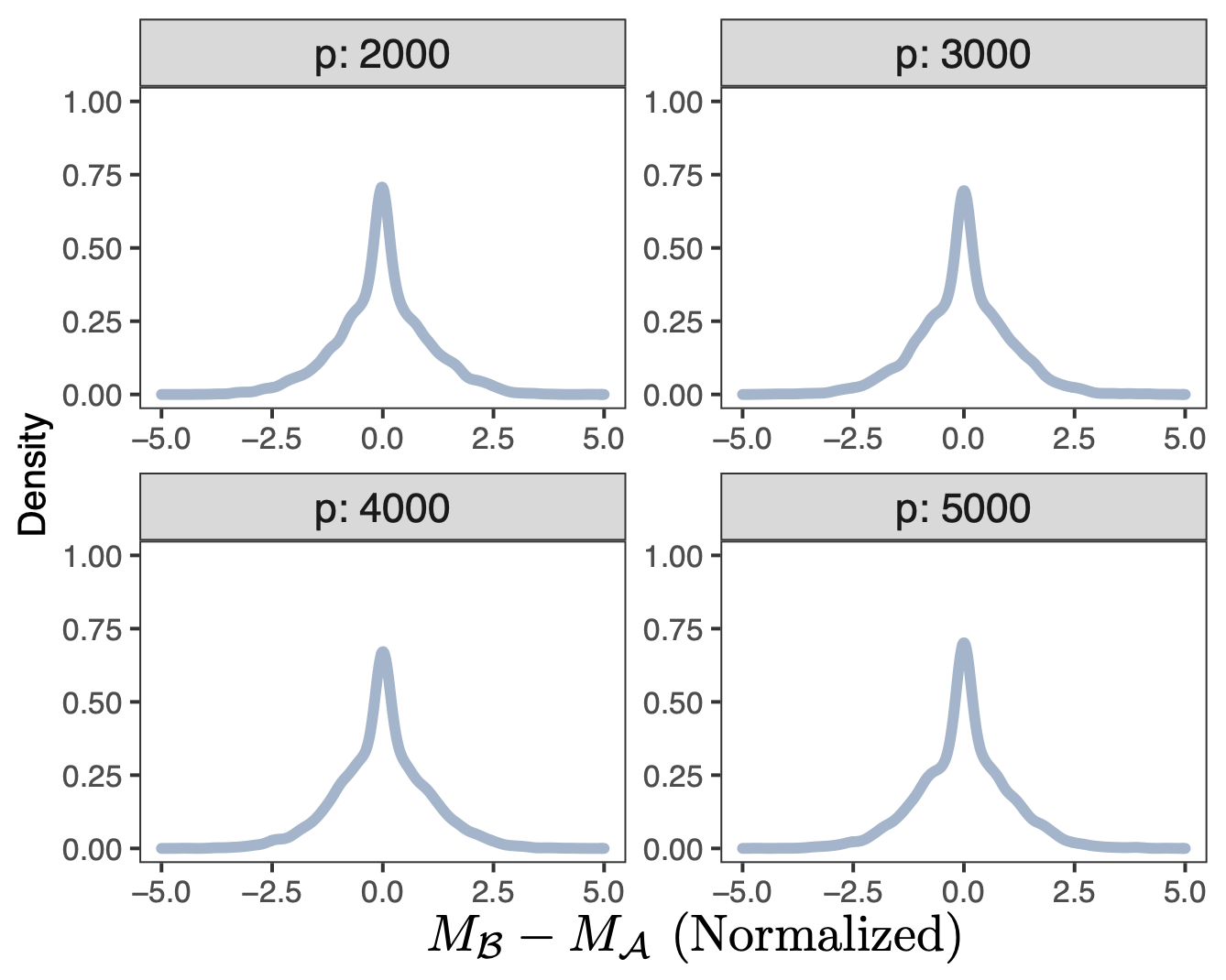}\\
         (e) $\cL(M_{\cB} - M_{\cA}, \varepsilon)$ ($\nu_{\cA} = \nu_{\cB} = 0.875$)  & (f) Density curves ($\nu_{\cA} = \nu_{\cB} = 0.875$) \\
         \!\!\!\!\includegraphics[width = 180 pt]{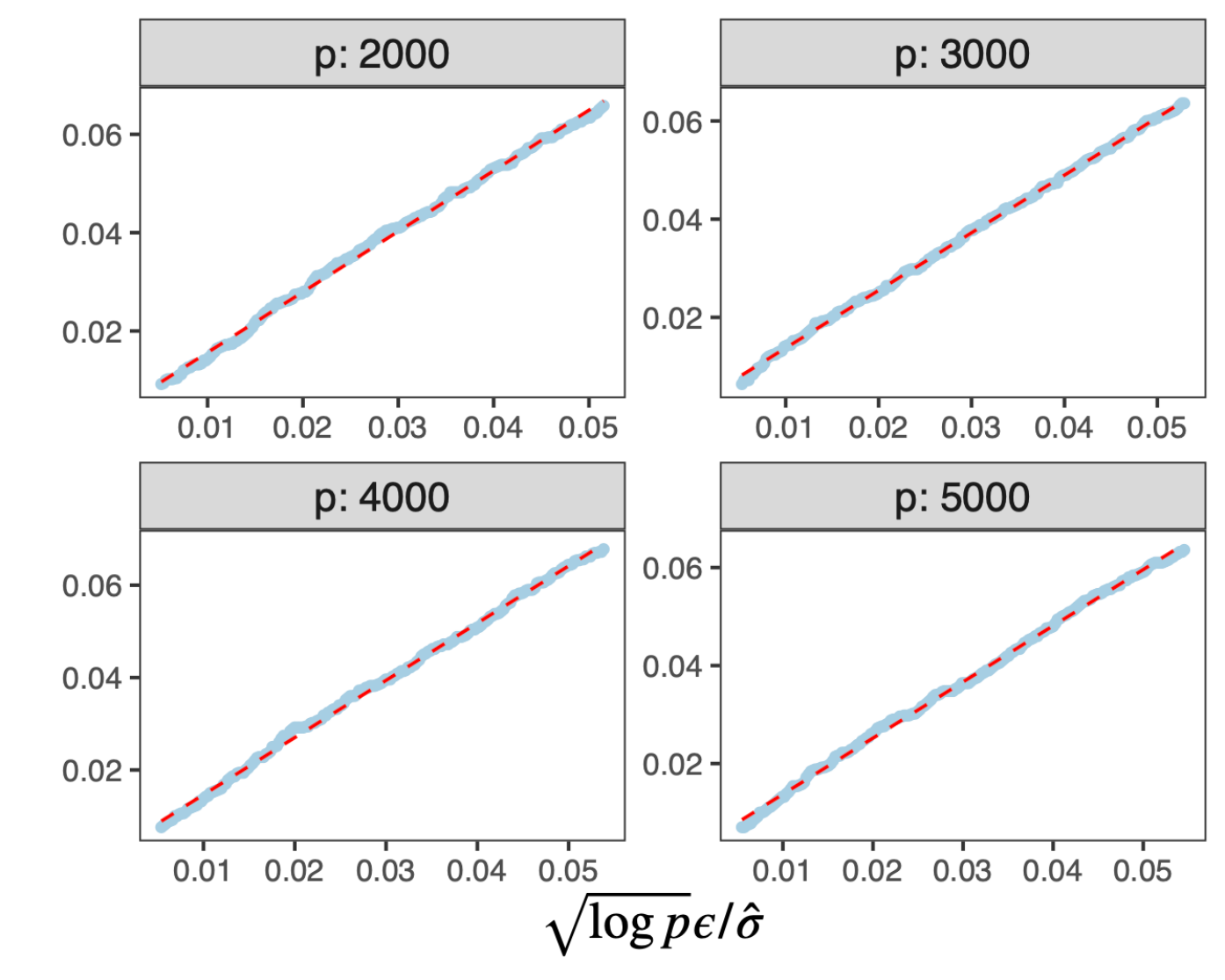}&\!\!\!\!\!  \includegraphics[width = 180 pt]{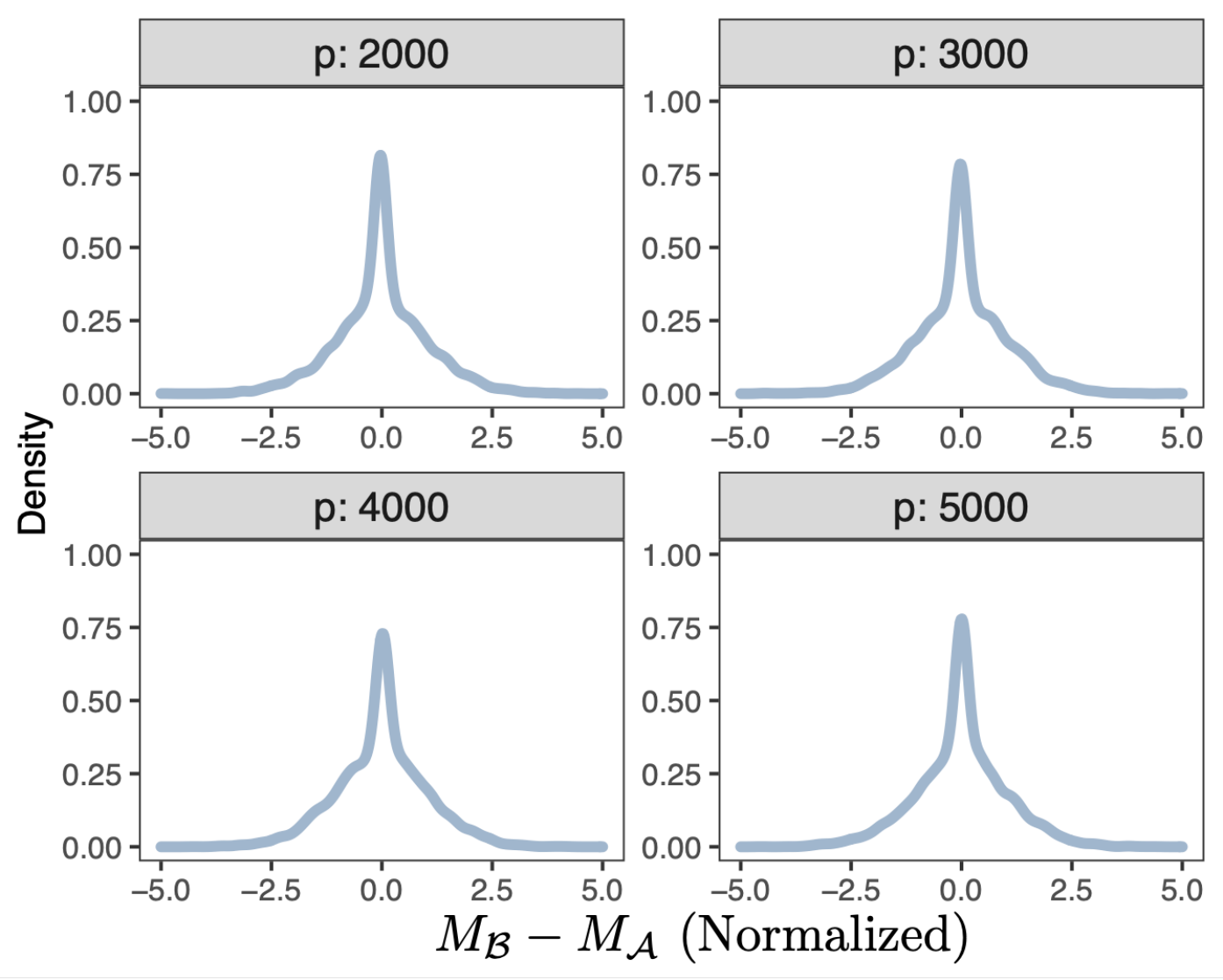}
    \end{tabular}
    \caption{Anti-concentration results for \(M_{\cB} - M_{\cA}\) in the heterogeneous variance case when the covariance condition is violated, i.e., \(\max\{\min_{j \in \cB, i \in \cA} (\sigma_j - \sigma_j^{-1} \sigma_{i j} ) , \min_{i \in \cA, j \in \cB} (\sigma_i - \sigma_i^{-1} \sigma_{i j} ) \}\le 0\) and \(\max_{i \in \cA, j \in \cB} \Corr(X_i, X_j) = 0.9\). Panels (a) and (b) serve as benchmarks, with \(\nu_{\cA} = \nu_{\cB} = 0\) for all \(p\) configurations, showing no notable concentration phenomena. In panels (c) and (d), the violation ratios are \(\nu_{\cA} = \nu_{\cB} = 0.75\), and the violation means are \(m_{\cA} = m_{\cB} = -8.067\) for all \(p\) setups. The density curves in panel (d) show a more marked concentration around 0 compared to panel (b); panels (e) and (f) present the results at \(\nu_{\cA} = \nu_{\cB} = 0.875\) and \(m_{\cA} = m_{\cB} = -12.586\) for all configurations of \(p\), exhibiting a higher concentration around 0 than observed in panel (d).}\label{fig: unequal var Cs violated}
\end{figure}
\begin{figure}[htbp]
    \centering
    \begin{tabular}{cc}
         (a) $\cL(M_{\cB} - M_{\cA}, \varepsilon)$ & (b) Density curves \\
         \!\!\!\!\includegraphics[width = 202 pt]{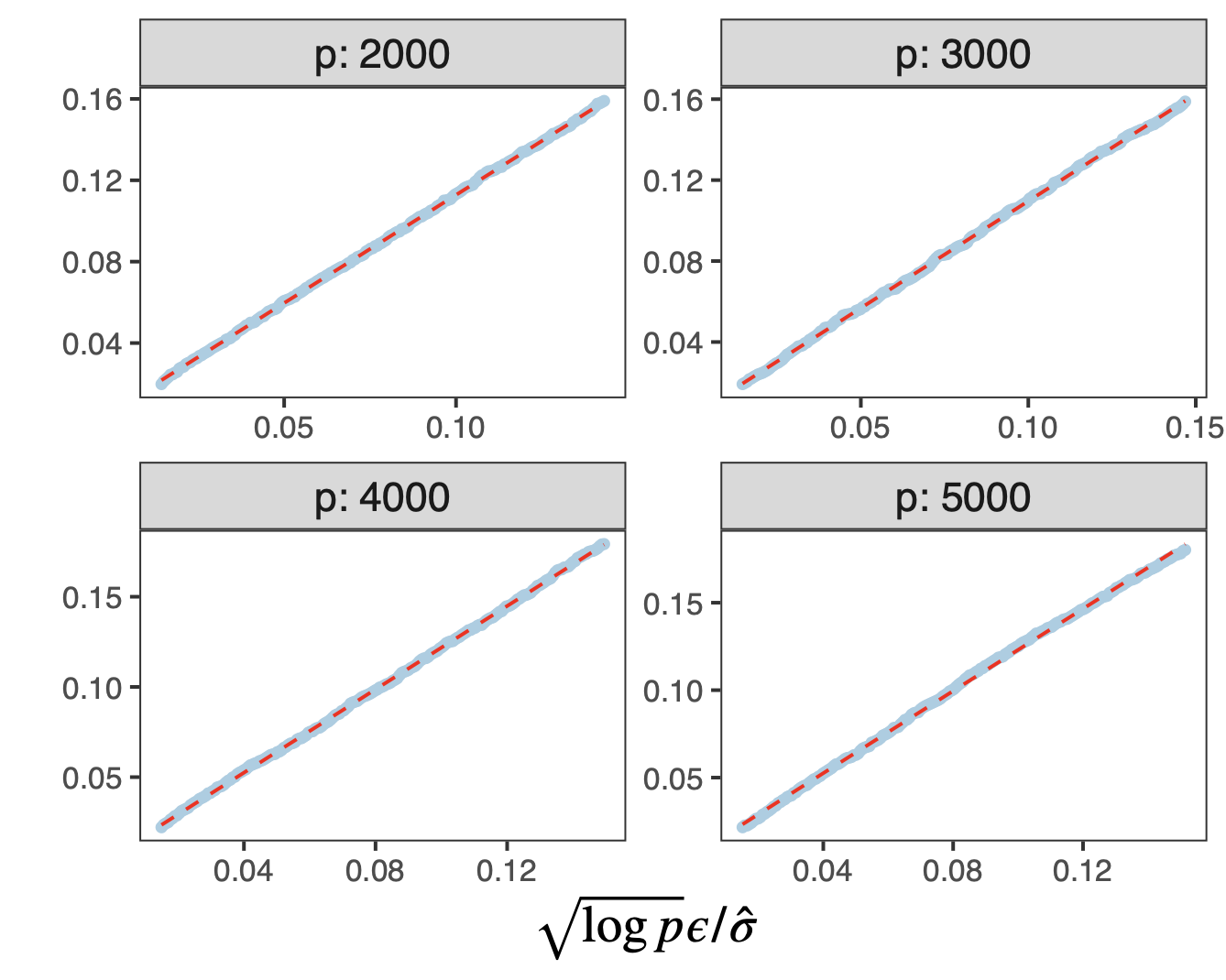}&  \!\!\!\!\!\includegraphics[width = 202 pt]{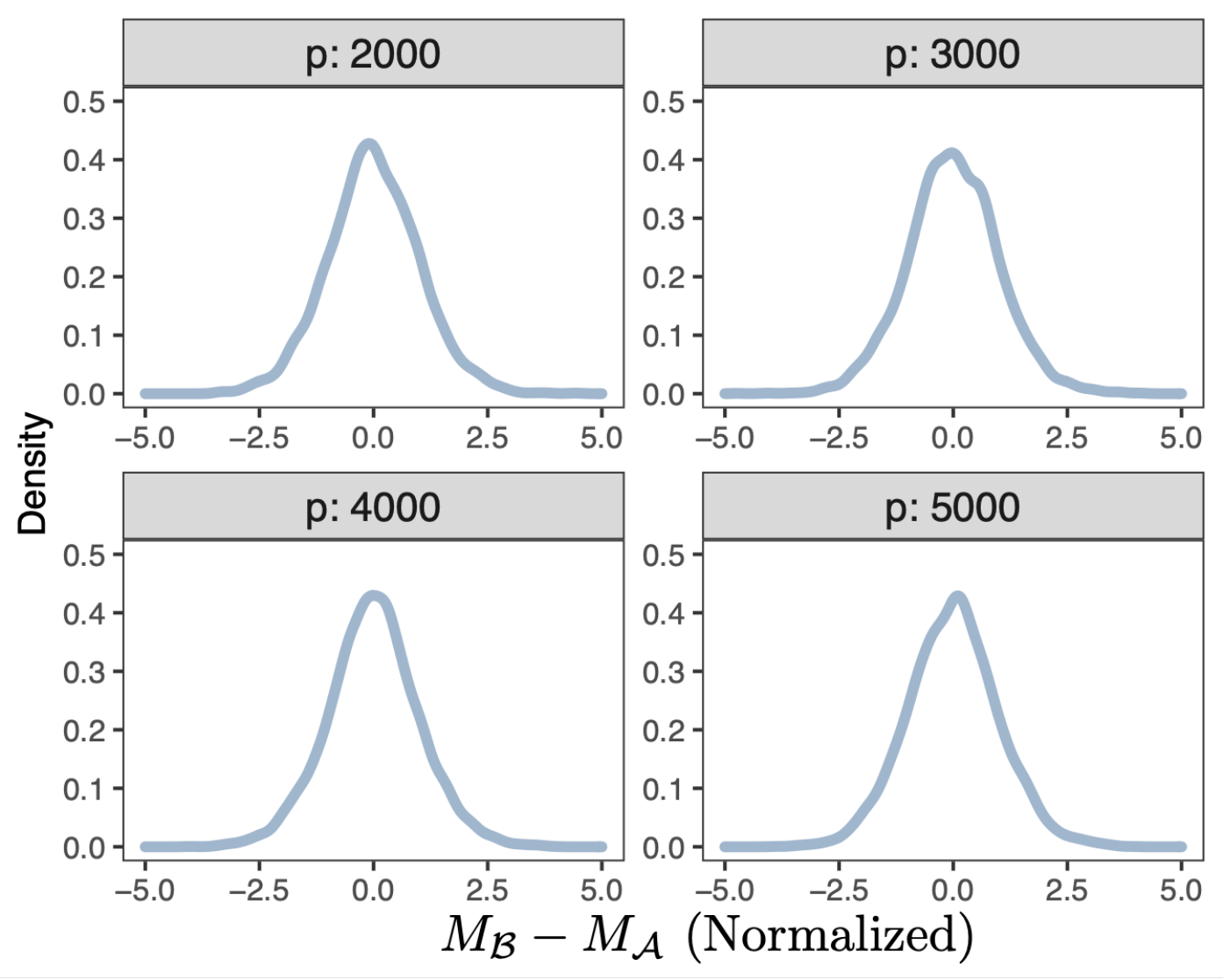}
    \end{tabular}
    \caption{Anti-concentration results for \(M_{\cB} - M_{\cA}\) in the unequal variance case where \(\max_{i \in \cA, j \in \cB} \Corr(X_i, X_j) = 1\) and the variances of the overlapping components are different. We set \(K = p / 8\), scaling the overlapping random variables in \(\cA\) to have a unit variance and those in \(\cB\) to have a constant variance of \(\sigma^2 = 1/2\). No significant concentrations in the density function are observed.
    }\label{fig: unequal var cor = 1}
\end{figure}
\subsection{Empirical comparison of the theoretical bounds with literature}\label{sec: simu comp}

We compare the ratio of $\cL(M_{\cB} - M_{\cA}, \varepsilon)/\varepsilon$ with the theoretical bounds in Theorem~\ref{thm: anti con}, Proposition~\ref{col: min eigen anti con}  and Lemma~5 of \cite{imaizumi2021gaussian}. As discussed earlier in Section~\ref{sec: intro}, there are examples where the previous bounds in literature are not informative. Here we consider designs that highlight the difference even in well-behaved cases  where 
% In order to compute the theoretical bound in \cite{imaizumi2021gaussian} using the minimal eigenvalue of the covariance matrix, 
the covariance matrix of $(X_i)_{i=1}^p$ is non-degenerate. We let $\Sigma^0 = \Gamma \Gamma^{\top} + \Ib_p $, where $\Gamma \in \RR^{p \times \frac{p}{10}}$ is a matrix with entries independently sampled from standard Gaussian distribution. We take $\Sigma = \diag(\Sigma^0 )^{-1/2} \Sigma^0 \diag(\Sigma^0 )^{-1/2} $ as the covariance matrix for $(X_{\cA}^\top, X_{\cB}^\top)^\top$, and let $\mu_{\cA} = \mu_{\cB} = \mathbf{0}$. We consider $p \in \{2000, 3000, 4000, 5000\}$. For each setting, we compute the theoretical bound for the ratio by Theorem~\ref{thm: anti con} as $2 \cdot \max\big\{\EE [\max_{i \in \cA}|X_i |], \EE [\max_{j \in \cB}|X_j |]\big\}/ (1 - \bar\rho)$, where 
% $\bar\EE_n$ stands for the empirical mean, and 
$\bar\rho = \max_{i \in \cA, j \in \cB} \Corr(X_i, X_j)$. The theoretical bound by Proposition~\ref{col: min eigen anti con} is $2 \cdot \max\big\{\EE [\max_{i \in \cA}|\tilde{X}_i /\tilde\sigma_i|], \EE [\max_{j \in \cB}|\tilde{X}_j /\tilde\sigma_j|]\big\}/ \min_{j \in [p] }\tilde\sigma_j$, where $\tilde{X} \sim \cN(0,\tilde\Sigma)$ with $\tilde\Sigma = \diag (\Sigma_{\cA} - \Sigma_{\cA \cB} \Sigma_{\cB}^{-1} \Sigma_{\cB \cA}, \Sigma_{\cB} - \Sigma_{\cB \cA} \Sigma_{\cA}^{-1} \Sigma_{\cA \cB})$ and $\tilde\sigma^2_j = \tilde\Sigma_{jj}$ for $j \in [p]$. We compare the bounds by Theorem~\ref{thm: anti con} and Proposition~\ref{col: min eigen anti con} with the bound of Lemma~5 in \cite{imaizumi2021gaussian}, $2 (\sqrt{2\log p} + 2)/\sigma_{\min}$, where $\sigma_{\min}^2 = \lambda_{\min} (\Sigma_{X})$. Table~\ref{tab:compare} shows the results at \(\varepsilon = 0.05\). Theorem~\ref{thm: anti con} provides a tighter bound by bounding the concentration function based only on pair-wise correlations instead of relying on the minimal eigenvalue of the covariance matrix.
% under the setting of a full-rank covariance matrix with a small minimal eigenvalue.
\begin{table}[H]
    \centering
     \caption{Comparison of theoretical bounds with \cite{imaizumi2021gaussian} against empirical ratio of $\cL(M_{\cB} - M_{\cA}, \varepsilon)/\varepsilon$ at $\varepsilon = 0.05$. }
    \label{tab:compare}
    \begin{tabular}{c|c|c|c|c}
    \hline
    \hline
        $p$ & $\cL(M_{\cB} - M_{\cA}, \varepsilon)/\varepsilon$ & Theorem~\ref{thm: anti con} & Proposition~\ref{col: min eigen anti con}  & $2 (\sqrt{2\log p} + 2)/\sigma_{\min}$   \\
         \hline
         \hline
2000 & 1.872 & 10.295 & 100.890 & 196.078 \\
3000 & 1.816 & 9.754 & 123.743 & 238.469 \\
4000 & 1.960 & 9.678 & 142.134 & 270.224 \\
5000 & 1.916 & 9.477 &162.975 & 305.929 \\
%          2000 & 1.880 & 10.290 & 196.078\\
% 3000 & 1.960 & 9.658 & 233.370\\
% 4000 & 1.960 & 9.642 & 272.987\\
% 5000 & 2.044 & 9.470 & 304.007\\
\hline
\hline
    \end{tabular}
\end{table}
\iffalse{
\begin{figure}[htbp]
    \centering
    \includegraphics[width = 240 pt]{bounds_compare.pdf}
    \caption{Comparison of theoretical bounds with \cite{imaizumi2021gaussian}, where the blue dashed line represents the theoretical bound from Theorem~\ref{thm: anti con} in our paper, and the red dashed line represents the theoretical bound from \cite{imaizumi2021gaussian}.}
    \label{fig: bounds comp}
\end{figure}}\fi
\section{Proofs of main theoretical results}\label{sec: proofs}
We provide the technical proofs for the main results in Sections~\ref{sec: main thm results} and~\ref{sec: application}. 

\subsection{Proof of Theorem~\ref{col: anti con equal var no cor 1}}\label{sec: proof col anti con eq var no cor 1}
% We first provide two lemmas that are essential for proving Theorem~\ref{col: anti con equal var no cor 1}, where we provide the proofs in  Appendix~\ref{sec: proof tech lms}.
% Given the two lemmas above, we present the proof of  Theorem~\ref{col: anti con equal var no cor 1}.
% \begin{proof}[\bf Proof of Theorem~\ref{col: anti con equal var no cor 1}]
% We will show \eqref{eq: anti con max diff equal var col} only, and \eqref{eq: anti con max diff equal var col t 0} can be shown with trivial modifications. 
% With similar arguments as in the proof of Theorem~\ref{thm: anti con}, without loss of generality, we can assume $\{X_i\}_{i \in \cA}$ have means $\{\mu_i - \mu_1^*\}_{i \in \cA}$ and $\{X_i\}_{i \in \cB}$ have means $\{\mu_i - \mu_2^*\}_{i \in \cB}$, where $\mu_1^* = \arg\min_{\mu \in \RR} \max_{i \in \cA} |\mu_i - \mu|$ and $\mu_2^* = \arg\min_{\mu \in \RR} \max_{i \in \cB} |\mu_i - \mu|$. We slightly abuse the notation and denote $\mu_i - \mu_j^*$ by $\mu_i$ for notational convenience ($i=1,\ldots,p$, $j = 1,2$). Also 
Note that when \(\bar{\rho} = \max_{i \in \cA, j \in \cB} \Corr(X_i, X_j) = 1\), the right-hand side of \eqref{eq: anti con max diff equal var col} goes to infinity. Therefore, we only prove the results when \(\bar{\rho} = \max_{i \in \cA, j \in \cB} \Corr(X_i, X_j) < 1\). Without loss of generality, suppose that there is no perfect positive correlation within $\cA$ or $\cB$ that $\Corr(X_i, X_j) < 1$ for any $i,j \in \cA$ or $i, j \in \cB$ and $i \ne j$. If there exists a pair $(X_i, X_j)$ in the same set that is perfectly positively correlated, because $\sigma_i = \sigma_j = \sigma$, there exists a constant $c$ such that $X_i \equiv X_j + c$, and we can remove $X_i$ or $X_j$ according to the sign of $c$ without changing the value of $M_{\cA}$ or $M_{\cB}$. Denote the sets of indices with perfect negative correlation with another variable in the other set as
    \begin{equation}\label{eq: def Na Nb}
\begin{aligned}
    \cN_{\cA} &= \{i \in \cA: \exists j \in \cB \text{ such that } \Corr(X_i , X_j) = -1\}, \\
    \cN_{\cB} &= \{i \in \cB: \exists j \in \cA \text{ such that } \Corr(X_i , X_j) = -1\}.
\end{aligned}
    \end{equation}
 If $\cN_\cA$ and $\cN_\cB$ are empty, the result directly follows from Theorem~\ref{thm: anti con}. Hence we suppose that $\cN_\cA$ and $\cN_\cB$ are nonempty. Lemma~\ref{lm: NA NB equiv} establishes that there is a one-to-one mapping between $\cN_{\cA}$ and $\cN_{\cB}$, and the Gaussian random variables within each set are independent.
% The following lemma establishes the relationship between $\cN_{\cA}$ and $\cN_{\cB}$, as well as the correlation structure within them. {\color{red}EF: Put this lemma into appendix. You only need to refer the result when you need it.}
% The proof of Lemma~\ref{lm: NA NB equiv} is in Appendix~\ref{sec: proof tech lms}. 
% {(\color{red}EF: For Appendix, I suggest we reorganize a bit. In particular, each proof should be a subsection.)}{\purple SS: No problem. I will reorganize.}
For ease of presentation, we assume that $\cA \ne \cN_{\cA}$ and $\cB \ne \cN_{\cB}$ (the results when either $\cA = \cN_{\cA}$ or $\cB = \cN_{\cB}$ follow by similar arguments). We define the events
    $$
    \cE_1 = \{\max_{i \in \cN_{\cA}} X_i \ge \max_{i' \in \cA \backslash \cN_{\cA}} X_{i'}\}, \quad
    \cE_2 = \{\max_{j \in \cN_{\cB}} X_j \ge \max_{j' \in \cB \backslash \cN_{\cB}} X_{j'}\}.
    $$
    Then for any $t \in \RR$, we have
    \begin{align*}
         \PP\Big(\!\big| M_{\cB} \!- \!M_{\cA} \!- \!t\big|\! \le \!\varepsilon \!\Big) & \le  \underbrace{\PP\Big( \!\!\{ \big| M_{\cB} - M_{\cA} - t\big| \le \varepsilon \} \!\cap\! \cE_1^c \Big)}_{\rm I_1} +  \underbrace{\PP\Big(\!\! \{ \big| M_{\cB} - M_{\cA} - t\big| \le \varepsilon \}\! \cap\! \cE_2^c \Big)}_{\rm I_2} \\
        & \quad + \underbrace{\PP\Big( \{ \big| M_{\cB} - M_{\cA} - t\big| \le \varepsilon \} \cap \{\cE_1 \cap \cE_2 \}\Big)}_{\rm I_3}.
    \end{align*}
    We then bound the terms  ${\rm I_1}$, ${\rm I_2}$ and ${\rm I_3}$. For ${\rm I_1}$, we have 
    \begin{align*}
        {\rm I_1} &= \PP \Big(\big\{\big|\max_{i \in \cA \backslash \cN_{\cA}} X_i - \max_{j \in \cB} X_j - t\big| \le \varepsilon \big\} \cap \cE_1^c \Big) \le \PP \Big(\big|\max_{i \in \cA \backslash \cN_{\cA}} X_i - \max_{j \in \cB} X_j - t\big| \le \varepsilon \Big).
    \end{align*}
    By the definition of $\cN_{\cA}$ in \eqref{eq: def Na Nb}, for any $i \in \cA \backslash \cN_{\cA}$ and $j \in \cB$, it holds that $|\Corr(X_i, X_j)| < 1$.  Hence conditions (A) and (B) of Theorem~\ref{thm: anti con} are both satisfied since $\sigma_j - \sigma_j^{-1} \sigma_{ij} = \sigma - \sigma  \cdot \Corr(X_i, X_j) > 0$ for any $i,j \in [p]$ and $i \ne j$. Then applying Theorem~\ref{thm: anti con}, we have that 
    % there exists a fixed constant $C > 0$ such that 
    \begin{align*}
       {\rm I}_1 &\le 2  \Big\{\EE \big(\max_{i \in \cA \backslash \cN_{\cA}} |X_i - \mu_i|  \big) \wedge \EE \big(\max_{i \in \cB} |X_i - \mu_i|  \big)  \Big\}\frac{\varepsilon}{(1 -\bar\rho_1)\sigma^2}\\
       & \le 2  \Big\{\EE \big(\max_{i \in \cA } |X_i - \mu_i|  \big) \wedge \EE \big(\max_{i \in \cB} |X_i - \mu_i|  \big)  \Big\}\frac{\varepsilon}{(1 -\bar\rho)\sigma^2},
    \end{align*}
where 
% $$
% \bar\lambda_1 = \min_{\mu \in \RR} \max_{i \in \cA \backslash\cN_{\cA}} |\mu_i - \mu| \vee \min_{\mu \in \RR} \max_{i \in \cB} |\mu_i - \mu| \le \bar\lambda,
% $$ 
% and 
$\bar\rho = \max_{i \in \cA, j \in \cB} \Corr(X_i, X_j)$ and $\bar{\rho}_1 = \max_{i \in \cA \backslash\cN_{\cA}, j \in \cB} \Corr(X_i, X_j) \le \bar\rho$. Following similar arguments, we have 
$$
{\rm I}_2 \le 2 \Big\{\EE \big(\max_{i \in \cA } |X_i - \mu_i|  \big) \wedge \EE \big(\max_{i \in \cB} |X_i - \mu_i|  \big)  \Big\}\frac{\varepsilon}{(1 -\bar\rho)\sigma^2}.
% \left\{\EE \left(\max_{i \in [p]} |X_i - \mu_i|  \right) + \bar\lambda + \sigma \sqrt{0 \vee \log(\sigma/\varepsilon)}  \right\}\frac{\varepsilon}{(1 -\bar\rho)\sigma^2}.
$$
For the term ${\rm I_3}$, applying Lemma~\ref{lm: anti con equal var no cor 1}, we have that 
\begin{align*}
    {\rm I}_3  &= \PP \Big(\!\!\big\{\big|\max_{i \in \cN_{\cA}} X_i - \max_{j \in \cN_{\cB}} X_j - t\big| \le \varepsilon \big\} \!\cap\! \{\cE_1 \cap \cE_2 \}\!\!\Big) \overset{\rm (a)}{\le} \PP \Big(\big|\max_{i \in \cN_{\cA}} X_i - \max_{j \in \cN_{\cB}} X_j - t\big| \le \varepsilon \Big)\\
    & \overset{\rm (b)}{=} \PP \Big(\big|\max_{i \in \cN_{\cA}} X_i - \max_{j \in \cN_{\cA}} (-X_i + c_i)- t\big| \le \varepsilon \big) \overset{\rm (c)}{\le} 3  \EE \Big(\max_{i \in \cN_{\cA}} |X_i - \mu_i|  \Big)\frac{\varepsilon}{(1 +  \underline\rho_3)\sigma^2}\\
    & \overset{\rm (d)}{\le} 3 \Big\{\EE \big(\max_{i \in \cA } |X_i - \mu_i|  \big) \wedge \EE \big(\max_{i \in \cB} |X_i - \mu_i|  \big)  \Big\}\frac{\varepsilon}{(1 -\bar\rho)\sigma^2},
\end{align*}
where inequality~(a) follows by dropping the intersection, equality~(b) holds by  Lemma~\ref{lm: NA NB equiv}, inequality~(c) follows from Lemma~\ref{lm: anti con equal var no cor 1}, and inequality~(d) follows from $\max_{i \in \cN_{\cA}} |X_i - \mu_i| = \max_{i \in \cN_{\cB}} |X_i - \mu_i|$ due to the equivalence in \eqref{eq: equal NA NB} of Lemma~\ref{lm: NA NB equiv}, and that
% $\bar\lambda_3 = \min_{\mu \in \RR} \max_{i \in \cN_{\cA} } |\mu_i - \mu| \vee \min_{\mu \in \RR} \max_{i \in \cN_{\cB}} |\mu_i - \mu| \le \bar\lambda$ and 
$$
\underline{\rho}_3\! := \min_{i , j \in \cN_{\cA}} \Corr(X_i, X_j) = \min_{i \in \cN_{\cA}, j \in \cN_{\cB}} ( - \Corr(X_i, X_j) \!) = - \!\!\!\max_{i \in \cN_{\cA}, j \in \cN_{\cB}} \!\Corr(X_i, X_j)  \ge - \bar\rho.
$$
Combining the three bounds above, we have that \eqref{eq: anti con max diff equal var col} holds. 
% \end{proof}
\subsection{Proof of Corollary~\ref{col: anti con equal var cor 1}}\label{sec: proof col anti con eq var cor 1}
% Similar as in the proof of Theorem~\ref{col: anti con equal var no cor 1}, 
Without loss of generality, we assume that $\Corr(X_i,X_j) < 1$ for any $i,j \in \cA$ or $i, j \in \cB$ and $i \ne j$.
% Also note that for any given $\lambda_1, \lambda_2 \in \RR$ and $t \in \RR$, we have $\max_{i \in \cA} X_i= \max_{i \in \cA} (X_i- \lambda_1) + \lambda_1$, $\max_{j \in \cB} X_{j} = \max_{j \in \cB} (X_{j} - \lambda_2) + \lambda_2$, and in turn 
% \begin{align*}
%     & \PP\big(\big|\max_{j \in \cB} X_{j} - \max_{i \in \cA} X_i- t\big| \le \varepsilon \big) = \PP\big(\big| \max_{j \in \cB} (X_{j} - \lambda_2) - \max_{i \in \cA} (X_i- \lambda_1) - t + \lambda_2 - \lambda_1\big| \le \varepsilon \big) \\
%     & \le \sup_{t' \in \RR}\PP\big(\big| \max_{j \in \cB} (X_{j} - \lambda_2) - \max_{i \in \cA} (X_i- \lambda_1) - t'\big| \le \varepsilon \big),
% \end{align*}
% and hence without loss of generality we can assume that $\{X_i\}_{i \in \cA}$ have means $\{\mu_i - \mu_1^*\}_{i \in \cA}$ and $\{X_i\}_{i \in \cB}$ have means $\{\mu_i - \mu_2^*\}_{i \in \cB}$, where $\mu_1^* = \arg\min_{\mu \in \RR} \max_{i \in \cA} |\mu_i - \mu|$ and $\mu_2^* = \arg\min_{\mu \in \RR} \max_{i \in \cB} |\mu_i - \mu|$. We slightly abuse the notation and denote $\mu_i - \mu_j^*$ by $\mu_i$ for notational convenience ($i=1,\ldots,p$, $j = 1,2$). 
    % Besides, without loss of generality, we assume that $\omega_1 \ge \omega_2$ and $\omega = \omega_1$.
To prove the claim \eqref{eq: anti con max diff equal var col cor 1}, it suffices to show that, for any $\delta \in (0,1)$ such that the corresponding $\cN_\delta$ defined in \eqref{eqn:Nd} satisfies $\cN_{\delta} \ne \cA$, it holds that
\begin{equation}\label{eq: anti con max diff equal var col cor 1 0}
    \begin{aligned}
        % \sup_{t \in \RR}\PP\Big(\!\big| M_{\cB} - M_{\cA} - t\big| \le \!\varepsilon \!\Big)  
        \cL(M_{\cB} - M_{\cA}, \varepsilon) \le  \left\{ \EE \!\left(\!\max_{i \in \cA\backslash\cN_{\delta}} \!\!\frac{|X_i - \mu_i| }{\sigma} \right) \!\wedge\! \EE \!\left(\!\max_{i \in \cB} \!\frac{|X_i - \mu_i| }{\sigma} \right)  \right\} \cdot \!\frac{7\varepsilon}{\delta\sigma}  + 2 \omega_{\delta}.
        % \exp\!\Big(\!-\frac{\omega_{\delta}^2}{8 \sigma^2}\Big)\!\!\bigg\},
    \end{aligned}
    \end{equation} 
    
First, we consider the case $\cN_\delta = \emptyset$. Then, for any $\delta \in (0,1)$, we have
$$
\bar\rho = \textstyle\max_{i \in \cA, j \in \cB} \Corr(X_i, X_j) < 1 - \delta,
$$
and the assumptions of Theorem~\ref{col: anti con equal var no cor 1} are satisfied. Applying Theorem~\ref{col: anti con equal var no cor 1}, together with the fact $\bar\rho < 1 - \delta$, we have that \eqref{eq: anti con max diff equal var col cor 1 0} holds as desired.

Next, consider the case $\cN_{\delta} \ne \emptyset$ and $D_{\delta} = \EE (\max_{i \in \cA \backslash \cN_{\delta}} X_i )- \EE (\max_{j \in \cN_{\delta}} X_{j} ) \le 0$. We have that $\omega_\delta = 1$ by its definition in \eqref{eqn:omegad}, and \eqref{eq: anti con max diff equal var col cor 1 0} holds since its right-hand side is larger than 1.

    Next,  consider the case  of $\delta \in (0,1)$ such that $D_{\delta}  = \EE (\max_{i \in \cA \backslash \cN_{\delta}} X_i )- \EE (\max_{j \in \cN_{\delta}} X_{j} )  > 0$. We first define the event 
    $$
    \cE = \Big\{\max_{i \in \cN_{\delta}} X_i \ge \max_{i' \in \cA \backslash \cN_{\delta}} X_{i'}\Big\}.
    $$
   For any subset \(\cS \subseteq [p]\), let \(X_{\cS} = (X_i)_{i \in \cS}\) be the subvector corresponding to \(\cS\). Note that we can write \(X_{\cS} = \Sigma_{\cS}^{1/2} Z_{\cS} + \mu_{\cS}\), where \(\mu_{\cS} = \EE X_{\cS}\), $\Sigma_{\cS} = \EE[(X_{\cS} - \mu_{\cS})(X_{\cS} - \mu_{\cS})^{\top}]$, \(\Sigma_{\cS}^{1/2}\) is a symmetric matrix that satisfies \(\Sigma_{\cS} = \big(\Sigma_{\cS}^{1/2}\big)^2 \), $Z_{\cS} = \big(\Sigma_{\cS}^{1/2}\big)^{-1}(X_{\cS} - \mu_{\cS})$, and the components of \(Z_{\cS}\in\RR^{|\cS|}\) are independent standard Gaussian random variables. Further,  for any $x, y \in\RR^{|\cS|}$, by the Cauchy-Schwarz inequality we have that
   $$
  \left|\max_{i \in \cS} ( \Sigma_{\cS}^{1/2} x + \mu_{\cS})_i - \max_{i \in \cS} ( \Sigma_{\cS}^{1/2} y + \mu_{\cS})_i\right| \le \max_{i \in \cS}\|(\Sigma_{\cS}^{1/2} )_i\|_2 \|x - y\|_2 = \sigma \|x-y\|_2,
   $$
   where $(\Sigma_{\cS}^{1/2} )_i$ denotes the $i$-th row of $\Sigma_{\cS}^{1/2}$. Hence we know that both $\max_{i \in \cS} X_i$ and $-\max_{i \in \cS} X_i$ are $\sigma$-Lipschitz with respect to $Z_{\cS}$.  By the log-Sobolev inequality (see Lemma~\ref{lm: log-sob ineq Gaussian} in Appendix~\ref{sec: log sob ineq}), for any $r > 0$, we have 
   \begin{equation}\label{eq: log sob ineq}
       \begin{aligned}
            &\PP\big(\max_{i \in \cS} X_i - \EE \max_{i \in \cS} X_i \ge r \big) \le \exp\Big(-\frac{r^2}{2\sigma^2}\Big),\\
            \text{and} \quad &\PP(\max_{i \in \cS} X_i - \EE \max_{i \in \cS} X_i \le -r ) \le \exp\Big(-\frac{r^2}{2\sigma^2}\Big).
       \end{aligned}
   \end{equation}
    Hence, for $D_{\delta}  > 0$, we have that 
    \begin{align*}
        \PP(\cE^c) & \ge \PP\left\{\left\{ \max_{i \in \cN_{\delta}} X_i \le \EE \max_{i \in \cN_{\delta}} X_i + \frac{D_{\delta}}{2}\right\} \cap \left\{\max_{i \in \cA \backslash \cN_{\delta}} X_i \ge \EE \max_{i \in \cA \backslash \cN_{\delta}} X_i - \frac{D_{\delta}}{2}\right\}\right\}\\
        & \ge 1 - \PP \left\{ \max_{i \in \cN_{\delta}} X_i \ge \EE \max_{i \in \cN_{\delta}} X_i + \frac{D_{\delta}}{2}\right\}  - \PP \left\{\max_{i \in \cA \backslash \cN_{\delta}} X_i \le \EE \max_{i \in \cA \backslash \cN_{\delta}} X_i - \frac{D_{\delta}}{2}\right\} \\
        % & \ge 1 - \PP \left\{ \max_{i \in \cN_{\cA}} X_i \ge \EE \max_{i \in \cN_{\cA}} (X_i -\mu_i) + \max_{i \in \cN_{\cA}} (\mu_i) + \frac{\omega}{2}\right\} \\
        % & \quad - \PP\left\{\max_{i \in \cA \backslash \cN_{\cA}} X_i \le \EE \max_{i \in \cA \backslash \cN_{\cA}} (X_i - \mu_i) \!-\! \max_{i \in \cA \backslash\cN_{\cA}} (-\mu_i) - \frac{\omega}{2}\right\}\\
        % & \overset{\rm(a)}{\ge}\! 1 \!- \!\PP\! \left\{ \!\max_{i \in \cN_{\cA}} (X_i - \mu_i)
        % \!\ge \!\EE \!\max_{i \in \cN_{\cA}} (X_i -\mu_i) \!+\! \frac{\omega}{2}\right\} \!- \!\PP\!\left\{\max_{i \in \cA \backslash \cN_{\cA}} (X_i - \mu_i )\!\le\! \EE \!\max_{i \in \cA \backslash \cN_{\cA}}\! (X_i - \mu_i)  \!-\! \frac{\omega}{2}\right\}\\
        & \ge 1 - 2 \exp \left(-\frac{D_{\delta}^2}{8\sigma^2}\right) = 1 - 2 \omega_{\delta},
    \end{align*}
    where the last inequality holds by \eqref{eq: log sob ineq}.
    % where inequality (a) follows from the fact that 
    % $$
    % \max_{i \in \cN_{\cA}} X_i \le \max_{i \in \cN_{\cA}} (X_i - \mu_i) +  \max_{i \in \cN_{\cA}} \mu_i, \quad \max_{i \in \cA \backslash \cN_{\cA}} (X_i - \mu_i) \le \max_{i \in \cA \backslash \cN_{\cA}}  X_i + \max_{i \in \cA \backslash \cN_{\cA}} (-\mu_i),
    % $$
    % and inequality (b) follows from Theorem~7.1 in \cite{ledoux2001concentration}.
    Then, for any $t \in \RR$ and $\varepsilon > 0$, we have that 
    \begin{align*}
        \PP(|M_{\cB} - M_{\cA} - t| \le \varepsilon ) &= \PP\big(\{|M_{\cB} - M_{\cA} - t| \le \varepsilon \} \cap \cE \big)+\PP\big(\{|M_{\cB} - M_{\cA} - t| \le \varepsilon \} \cap \cE^c \big) \\
       & \quad\le \PP(\cE) + \PP\big( \big|\max_{j \in \cB} X_j - \max_{i' \in \cA \backslash \cN_{\delta}} X_{i'} - t\big| \le \varepsilon\big) \\
       & \quad \le 2\omega_{\delta} +  7\left\{\!\EE \!\left(\!\max_{i \in \cA} |X_i - \mu_i|  \right) \wedge \EE \!\left(\!\max_{j \in \cB} |X_j - \mu_j|  \right) \right\}\!\frac{\varepsilon}{\delta\sigma^2},
    \end{align*}
    % where $\bar\rho_1 = \max_{i \in \cA \backslash \cN_{\cA}, j \in \cB} \Corr(X_i, X_j) \le \bar\rho$, and 
    where the first inequality holds by definitions and $\Corr(X_i, X_j) < 1 - \delta $ for any $i \in \cA\backslash\cN_{\delta}$ and $j \in \cB$, 
    the second inequality holds by Theorem~\ref{col: anti con equal var no cor 1}. Then \eqref{eq: anti con max diff equal var col cor 1 0}  follows,  and \eqref{eq: anti con max diff equal var col cor 1} holds by taking the infimum on $\delta$ that satisfies $\cN_{\delta} \ne \cA$, which concludes our proof. 
\subsection{Proof of Proposition~\ref{prop: lwr bd case}}\label{sec: proof prop lwr bd case}
Without loss of generality, we assume that $\mu_0 = 0$. 
We begin with showing the first inequality of  \eqref{eq: lwr bd case}. Recall that by definition, $\cL(M_{\cB}-M_{\cA},\varepsilon) = \sup_{t\in\RR}\PP(|M_{\cB} - M_{\cA} - t|\le \varepsilon)$. Define $\cN = \cA \cap \cB = [m] \backslash [m-k]$, and the events 
$$
\cE_1 = \left\{\max_{i \in \cN} X_i \ge \max_{i \in \cA \backslash \cN} X_i \right\}, \quad \cE_2 = \left\{\max_{j \in \cN} X_j \ge \max_{j \in \cB \backslash \cN} X_j \right\}, \quad \text{and } \quad \cE = \cE_1 \cap \cE_2 ,
$$
where we note that $ \cE = \left\{\max_{i \in \cN} X_i \ge \max_{i \in [p] \backslash \cN} X_i \right\}$. Taking $t = 0$,  we have
\begin{equation}\label{eq: LHS eq lwr bd case}
   \begin{aligned}
    & \PP\big(|M_{\cB} - M_{\cA}| \le \varepsilon \big) \ge \PP\big(\big\{ |M_{\cB} - M_{\cA}| \le \varepsilon \big\} \cap \cE \big) \\
    & = \PP\big(\{|\textstyle\max_{i \in \cN} X_i - \textstyle\max_{i \in \cN} X_i|\le \varepsilon \} \cap \cE \big) \\
    &=  \PP\big(\{0 \le \varepsilon \} \cap \cE \big) = \PP(\cE), 
\end{aligned} 
\end{equation}
 where the first equality holds because under the event $\cE$, we have that $M_{\cB} = M_{\cA} = \max_{i \in \cN} X_i$.
 Note that by the exchangeability of $(X_1,\ldots,X_{p})^{\top}$ in \eqref{eq: exchange dist} and the assumption that $\bar\rho:= \max_{i,j \in [p], i \ne j}\Corr(X_i, X_j) < 1$,  for any subset $\cS \subseteq [p]$ with $|\cS| \ge 2$, we have that the maximizer of $\{X_i\}_{i \in \cS}$ is unique with probability 1, and for any $j \in \cS$, we have
 \begin{equation}\label{eq: sym S}
     \PP(\textstyle\arg\max_{i \in \cS} X_i = j) = 1/|\cS|.
 \end{equation}
Therefore, taking $\cS = [p]$, we have
\begin{equation}\label{eq: sym E}
    \PP(\cE) = \PP\Big(\argmax_{i \in [p]} X_i \in \cN\Big) = \sum_{j \in \cN} \PP\Big(\argmax_{i \in [p]} X_i = j \Big) = \frac{k}{p}, 
\end{equation}
which combined with \eqref{eq: LHS eq lwr bd case} gives the first inequality of \eqref{eq: lwr bd case}.

 Next, we prove the second inequality of \eqref{eq: lwr bd case}. Following similar arguments to \eqref{eq: sym E}, by taking $\cS = \cA$ and $\cS = \cB$ in \eqref{eq: sym S}, respectively, we have $\PP(\cE_1) = \PP(\cE_2) = k / m = 2k / (p + k)$.  
% The following lemma helps to characterize the distribution of the maximizer when the distribution of $\{X_i\}_{i=1}^{p}$ is exchangeable. 
% \begin{lemma}\label{lm: exchange S argmax prob}
%     Let $(X_1,\ldots,X_{p})^{\top}$ be a Gaussian random vector with a homogeneous component-wise mean $\mu_0$ and equal marginal variance $\sigma^2$, and assume that the distribution of the random vector is exchangeable and satisfies \eqref{eq: exchange dist}.  Let $\cS \subseteq [p]$ be an arbitrary subset of $[p]$ with $|\cS| \ge 2$, then the maximizer of $\{X_i\}_{i \in \cS}$ is unique with probability 1, and for any $j \in \cS$, we have $\PP(\textstyle\arg\max_{i \in \cS} X_i = j) = 1/|\cS|$.
% \end{lemma}
% \begin{proof}
%     See Appendix~\ref{sec: proof lm exchange S argmax prob}.
% \end{proof}
% Applying Lemma~\ref{lm: exchange S argmax prob},
% Define the events:
% $$
% \cE_1 = \left\{\max_{i \in \cN} X_i \ge \max_{i \in \cA \backslash \cN} X_i \right\}, \quad \cE_2 = \left\{\max_{i \in \cN} X_i \ge \max_{i \in \cB \backslash \cN} X_i \right\}, 
% $$
% and it can be seen that $\cE_1 \cap \cE_2 = \cE$. 
Then, for any $t \in \RR$ and $\varepsilon > 0$, we have that
\begin{align*}
    & \PP\big(|M_{\cB} - M_{\cA} - t| \le \varepsilon \big)\\
    &\quad \le  \PP\big(\!\big\{ |M_{\cB} \!-\! M_{\cA} \!-\! t| \le \varepsilon \big\} \!\cap \!(\cE_1^c \cap \cE_2^c) \big) + \PP(\cE_1 ) + \PP(\cE_2) \\
    &\quad \le  \PP\left(\Big|\max_{i \in \cB \backslash \cN} X_i - \max_{i \in \cA \backslash \cN} X_i- t\Big| \le \varepsilon\right)  + \PP(\cE_1) + \PP(\cE_2) \\
    &\quad \le 7  \left\{\EE \!\Big[\max_{i \in \cA\backslash\cN} \!|X_i - \mu_0|  \!\Big]  \wedge \EE \!\Big[ \max_{i \in \cB\backslash\cN} \!|X_i - \mu_0|  \!\Big]  \right\} \frac{\varepsilon}{(1-\bar\rho)\sigma^2} + \frac{2k}{m},
\end{align*}
where the second inequality holds due to the fact that under the event $\cE_1^c \cap \cE_2^c$, we have $M_{\cB} = \max_{i \in \cB \backslash \cN} X_i $ and $M_{\cA} = \max_{i \in \cA \backslash \cN} X_i $, and the last inequality holds by Theorem~\ref{col: anti con equal var no cor 1} and plugging in $\PP(\cE_1) = \PP(\cE_2) = k / m$.

% where $C > 0$ is a fixed constant that does not change with $t$. 
Finally, we prove \eqref{eq: lwr bd case t > eps}.  For any $t \in \RR$ such that $|t| > \varepsilon$, we have
\begin{align*}
    & \PP\big(|M_{\cB} - M_{\cA} - t| \le \varepsilon \big)\\
    & \le  \PP\big(\!\big\{ |M_{\cB} \!-\! M_{\cA} \!-\! t| \!\le \!\varepsilon \big\} \!\cap \!\cE\big) \!+ \!\PP\big(\!\big\{ |M_{\cB} \!-\! M_{\cA} \!-\! t| \!\le \!\varepsilon \big\} \!\cap \!\cE_1^c\big) \!\\
    & \quad +\! \PP\big(\!\big\{ |M_{\cB} \!-\! M_{\cA} \!-\! t| \!\le \!\varepsilon \big\} \!\cap \!\cE_2^c\big) \\
    & \le  \PP\left(\Big|\max_{i \in \cB  } X_i - \max_{i \in \cA \backslash \cN} X_i- t\Big| \le \varepsilon\right) +  \PP\left(\Big|\max_{i \in \cB \backslash \cN} X_i - \max_{i \in \cA  } X_i- t\Big| \le \varepsilon\right) +\PP(|t| \le \varepsilon)\\
    & \le 14  \left\{\EE \!\Big[\max_{i \in \cA} \!|X_i - \mu|  \!\Big]  \wedge \EE \!\Big[ \max_{i \in \cB} \!|X_i - \mu|  \!\Big]  \right\} \frac{\varepsilon}{(1-\bar\rho)\sigma^2},
\end{align*}
where the second inequality holds by the fact that under the event $\cE$, we have
$$
\big\{ |M_{\cB} - M_{\cA} - t| \le \varepsilon \big\} = \big\{ |\max_{i \in \cN} X_i - \max_{i \in \cN} X_i - t| \le \varepsilon \big\}  = \big\{ |t| \le \varepsilon \big\} ,
$$
and the last one holds by Theorem~\ref{col: anti con equal var no cor 1} and the assumption $|t|>\varepsilon$. Thus,   \eqref{eq: lwr bd case t > eps} follows.

\subsection{Proofs of results in Section~\ref{sec: application}}\label{sec: proof sec application}
We provide the proof of 
% Lemma~\ref{lm: gauss comp rate} and
Theorem~\ref{prop: clt gaussian approx rate} in Section~\ref{sec: application}. Before we delve into the proof, we provide some auxiliary lemmas to be used in the proof.

We first provide a comparison inequality using the Stein kernel or the difference between the maxima of two random vectors, which is a modified version of Theorem~3.2 in \cite{cck2022improvedbootstrap}. 
    \begin{lemma}\label{lm: gauss comp rate}
       Let \(\cC_b^2 (\cdot) : \RR^p\rightarrow \RR\) be the class of twice continuously differentiable functions, with both the functions and their first and second-order partial derivatives bounded. Let $V \in \RR^p$ be a centered random vector,  and suppose that there exists a mensurable function (Stein kernel) $\tau: \RR^p \rightarrow \RR^{p \times p}$ such that 
        \begin{equation}\label{eq: stein kernel cond}
            \sum_{j=1}^p \EE(\partial_j \varphi(V) V_j) = \sum_{j,k=1}^p \EE(\partial_j \partial_k \varphi(V) \tau_{jk}(V)), \quad \forall \varphi \in \cC_b^2 (\cdot). 
        \end{equation} 
        Let $Z \in \RR^p$ be a centered Gaussian random vector with covariance matrix $\Sigma$, and let $\{\cA,\cB\}$ be a nontrivial partition of $[p]$. Suppose that $Z$ satisfies Condition~\ref{cond: anti-con cov} for the partition $\{\cA,\cB\}$. Then, we have that
 \begin{equation}\label{eq: gauss comp rate}
 \begin{aligned}
     & \sup_{s \in \RR, v \in \RR^p} \left|\PP\Big(\cM_{\cA}^v(V) - \cM_{\cB}^v(V) \le s\Big) - \PP\Big(\cM_{\cA}^v(Z) - \cM_{\cB}^v(Z) \le s\Big)  \right|\\
     & \quad \le K \cdot \frac{\EE(\max_{\ell \in \cS_{\rm p}}|Z_{\ell}|/\sigma_{\ell})}{C_{\cA,\cB}}\sqrt{\Delta \log p},
 \end{aligned}
 \end{equation}
 where  $\cM_{\cS}^v(x) = \max_{j \in \cS} (x_j + v_j)$ for $x \in \RR^p$ and $\cS \in\{ \cA, \cB\}$, $K > 0$ is a universal constant, and $\Delta = \EE(\max_{j,k \in [p]} |\tau_{jk}(V) - \Sigma_{jk}|)$. 
    \end{lemma}
\begin{proof}
    See Appendix~\ref{sec: proof lm gauss comp rate}.
\end{proof}
Next, we provide a tailored version of Theorem~3.1 in \cite{cck2022improvedbootstrap}. We modify the statistics and anti-concentration conditions in \cite{cck2022improvedbootstrap} to our setting.
Let $V_1, \ldots, V_n, Z_1, \ldots, Z_n \in \RR^p$ be independent zero-mean random vectors, and $B_n \ge 1$ is the rate defined in \eqref{eq: exp rate sample} and \eqref{eq: mom rate sample}. We impose the following conditions on the random vectors.
\begin{condition}\label{cond: fourth moment}
    There exists a constant  $C > 0$ such that for all $j \in [p]$, we have  
    $$
    \frac{1}{n} \sum_{i=1}^n \EE [V_{ij}^4 + Z_{ij}^4] \le C B_n^2.
    $$
\end{condition}
\begin{condition}\label{cond: tail prob}
    There exists a constant $C \ge 1$ such that for all $i \in [n]$, we have  
    $$
    \PP\big(\|V_i\|_{\infty} \vee \|Z_i\|_{\infty} > C B_n \log (pn)\big) \le 1/n^4. 
    $$
\end{condition}
\begin{condition}\label{cond: 8th inf mm bd}
    There exists a constant $C > 0$ such that for all $i \in [n]$, we have 
    $$
   \EE\big[\|V_i\|_{\infty}^8 + \|Z_i\|_{\infty}^8  \big] \le C B_n^8 \log^8 (pn).
    $$
\end{condition}
\begin{condition}\label{cond: anti-con rate}
    Define $S_n^Z = \frac{1}{\sqrt{n}} \sum_{i=1}^n Z_i$ and let $\{\cA,\cB\}$ be the non-trivial partition of $[p]$ specified in Theorem~\ref{prop: clt gaussian approx rate}. There exists a constant $C > 0$, and deterministic sequences $\varepsilon_n \ge 0$ and $ r_n \ge  1$ such that for
    % all $v \in \RR^p$, $s \in \RR$ and 
    any $t \in (0,\infty)$, we have  
    $$
    \sup_{v \in \RR^p, s \in \RR}\!\Big\{ \!\PP\Big(\!\cM_{\cA}^v(S_n^Z) - \cM_{\cB}^v(S_n^Z) \le s + t \Big) - \PP\Big(\cM_{\cA}^v(S_n^Z) - \cM_{\cB}^v(S_n^Z) \le s \Big)\!\!\Big\} \le C (r_n t + \varepsilon_n),
    $$
  where $\cM_{\cS}^v(x) := \max_{j \in \cS} (x_j + v_j)$ for $x \in \RR^p$, and $v \in \RR^p$ stands for an arbitrary shift in the mean of the random vector.
\end{condition}
Conditions~\ref{cond: fourth moment}-\ref{cond: 8th inf mm bd} correspond to the Conditions~V, P and B given in Section~3.1 of \cite{cck2022improvedbootstrap}, and Condition~\ref{cond: anti-con rate} is the anti-concentration condition for $S_n^Z$ corresponding to Condition~A in \cite{cck2022improvedbootstrap}, which is tailored to fit our setting. 
% Same as in \cite{cck2022improvedbootstrap}, the constants $C_v, C_e, C_b$ and $C_a$ are fixed and do not change with the indices. 
Then we present a modification of Theorem~3.1 in \cite{cck2022improvedbootstrap} (see Appendix~\ref{sec: proof prop modified thm 3.1} for the proof).

\begin{lemma}\label{prop: modified thm 3.1}
Let $V_1, \ldots, V_n, Z_1, \ldots, Z_n \in \RR^p$ be independent centered random vectors, and suppose that Conditions~\ref{cond: fourth moment}-\ref{cond: anti-con rate} hold. Assume that
    \begin{equation}\label{eq: m2 cond}
        \max_{j,k \in [p]} \left| \frac{1}{\sqrt{n}} \sum_{i=1}^n \big(\EE (V_{ij} V_{ik}) -\EE (Z_{ij} Z_{ik}) \big)\right| \le C_0 B_n \sqrt{\log (pn)},
    \end{equation}
    and 
    \begin{equation}\label{eq: m3 cond}
        \max_{j,k,\ell \in [p]} \left| \frac{1}{\sqrt{n}} \sum_{i=1}^n \big(\EE (V_{ij} V_{ik} V_{i\ell}) -\EE (Z_{ij} Z_{ik} Z_{i\ell}) \big)\right|\le C_0 B_n^2 \sqrt{\log^3 (pn)},
    \end{equation}
    for some constant $C_{0} > 0$. Then letting  $S_n^V = \frac{1}{\sqrt{n}}\sum_{i=1}^n V_i$ and $S_n^Z = \frac{1}{\sqrt{n}} \sum_{i=1}^n Z_i$, under the condition that $B_n^2 \log^5(pn) \le c_0 n$ for some small enough constant $c_0 > 0$, we have 
    \begin{equation}\label{eq: CLT initial results}
        \begin{aligned}
            & \sup_{v \in \RR^p, s \in \RR} \left|\PP\Big(\cM_{\cA}^v(S_n^V) - \cM_{\cB}^v(S_n^V) \le s  \Big) - \PP\Big(\cM_{\cA}^v(S_n^Z) - \cM_{\cB}^v(S_n^Z) \le s \Big)\right|\\
            & \quad \le C_1 \left(r_n \left(\frac{B_n^2 \log^3(pn)}{n}\right)^{1/4}+ \varepsilon_n\right),
        \end{aligned}
    \end{equation}
    where $C_1 > 0$ is a constant only depending on $c_0, C_0$ and the constants in Conditions~\ref{cond: fourth moment}-\ref{cond: anti-con rate}.
    % {\red \bf the constants in conditions 5.1-5.4} $C_v$, $C_e$, $C_b$, $C_a$ and $C_m$,    
\end{lemma}
We then prove Theorem~\ref{prop: clt gaussian approx rate}.
\begin{proof}[\bf Proof of Theorem~\ref{prop: clt gaussian approx rate}]
    Recall that due to the constructions of $X$, $Y$ and $\hat{X}$, without loss of generality, we can assume $\EE(\xi_i) = \mu =  \mathbf{0}$ for $i \in [n]$. Similarly to the proof of Theorem~2.1 in \cite{cck2022improvedbootstrap}, we denote by $\cA_n$ the event that (33) and (35)-(37) in \cite{cck2022improvedbootstrap} hold, i.e.,  $\cA_n = \cA_{n,1} \cap \cA_{n,2} \cap \cA_{n,3} \cap \cA_{n,4}$, where
    \begin{align*}
        \cA_{n,1} &= \big\{\textstyle\max_{i \in [n]}\|\xi_i\|_{\infty} \le 5 B_n \log (pn) \big\} ,\\
        \cA_{n,2} &= \big\{\textstyle\max_{j \in [p]} \frac{1}{n} \textstyle\sum_{i=1}^n (\xi_{ij} - \bar\xi_j)^4 \le 2 B_n^2 b_0^2 \big\},\\
        \cA_{n,3} & = \big\{  \textstyle\max_{j,k \in [p]} \big|n^{-1/2}\textstyle\sum_{i=1}^n [ (\xi_{ij} - \bar\xi_j) (\xi_{ik} - \bar\xi_k) - (\Sigma_i)_{jk} ]\big| \le C B_n \sqrt{\log (pn)} \big\},\\
        \cA_{n,4} & = \Big\{  \!\max_{j,k,\ell \in [p]} \Big|\frac{1}{\sqrt{n}}\sum_{i=1}^n [ (\xi_{ij} - \bar\xi_j) (\xi_{ik} - \bar\xi_k) (\xi_{i\ell} - \bar\xi_{\ell}) - \EE[\xi_{ij}\xi_{ik}\xi_{i\ell}] ]\Big| \!\le\! C B_n^2 \sqrt{\log^3 (pn)} \Big\},
    \end{align*}
% where we replace the event in (35) of \cite{cck2022improvedbootstrap} by {\color{red}EF: What is $b_0$? No $j$ here, do you mean $i$? $\xi$ is not defined.}{\purple SS: the notations were defined in the proposition. I'll restate below.}
%     $$
%     \frac{1}{n} \sum_{i=1}^n (\xi_{ij} - \bar\xi_j)^4 \le 2 B_n^2 b_0^2, \quad \text{for all } j \in [p],
%     $$
    where $C > 0$ is a universal constant, $B_n \ge 1$ and $b_0 > 0$ are defined in \eqref{eq: exp rate sample} and \eqref{eq: mom rate sample}, and $\bar\xi = \frac{1}{n}\sum_{i=1}^n \xi_i$. Also, we let
    \begin{equation}\label{eq: nu rates clt}
         \nu_{n,1} = \frac{\EE [\max_{\ell \in \cS_{\rm p}} |X_{\ell} - \EE(X_{\ell})|/\sigma_{\ell}]}{C_{\cA, \cB}} \left(\frac{B_n^2 \log^3(pn)}{n}\right)^{1/4} \!\!\!
 , \quad \nu_{n,2} = \sqrt{\frac{B_n^2 \log^3 (pn)}{n}},
    \end{equation}
 and, without loss of generality, we assume $1/(2 n^4) + 1/n + 3 \nu_{n,2}  < 1$ by taking the constant $C >0$ in \eqref{eq: CLT rate} large enough. Then by Lemma~4.1 and Lemma~4.2 in \cite{cck2022improvedbootstrap}, we have that 
 \begin{equation} \label{eqn:AnP}
 \PP (\cA_n) \ge 1 - 1/(2 n^4) - 1/n - 3 \nu_{n,2} > 0.
 \end{equation}
 Then we will show that \eqref{eq: CLT rate} and \eqref{eq: CLT rate 1} hold conditioning on the event $\cA_n$ via triangle inequality. Specifically,  as in \cite{cck2022improvedbootstrap}, we define the multiplier bootstrap random vector generated from independent standardized ${\rm Beta}(1/2, 3/2)$ random variables,
 \begin{equation}\label{eq: X beta mult}
     X^* = \frac{1}{\sqrt{n}} \sum_{i =1}^n [\kappa_i (\xi_i -\bar\xi) + a], 
 \end{equation}
 where $\kappa_i = \Var(\kappa_i^0)^{-1/2}(\kappa_i^0 - \EE(\kappa_i^0) )$ and $\kappa_i^0 \overset{\rm i.i.d.}{\sim} {\rm Beta} (1/2, 3/2)$ are independent of $\xi_{1:n} = \{\xi_i\}_{i=1}^n$. %Then $\EE (\kappa_i^3) = 1$.
We propose the following claims, based on which we establish  \eqref{eq: CLT rate} and  \eqref{eq: CLT rate 1} via the triangle inequality:
\begin{claim}\label{claim: CLT rate X_star X}
Under the event $\cA_n$, the following holds,
    \begin{equation}\label{eq: CLT rate X* X}
         \sup_{v \in \RR^p, s \in \RR}\left| \PP(\cM_{\cA}^v({X^*}) - \cM_{\cB}^v({X^*}) \!\le s \,|\, \xi_{1:n}) - \PP( \cM_{\cA}^v({X})\! - \!\cM_{\cB}^v({X} ) \le s)\right| \lesssim \nu_{n,1}.
\end{equation}
\end{claim}
\begin{proof}
    See Section~\ref{sec: proof claim CLT rate X* X}.
\end{proof}
\begin{claim}\label{claim: CLT rate X_star Y}
Under the event $\cA_n$, we have
\begin{equation}\label{eq: CLT rate X* Y}
    \sup_{v \in \RR^p, s \in \RR}\left| \PP(\cM_{\cA}^v({X^*}) - \cM_{\cB}^v({X^*}) \le s \,|\, \xi_{1:n}) - \PP( \cM_{\cA}^v({Y}) - \cM_{\cB}^v({Y} ) \le s )\right| \lesssim  \nu_{n,1}.
\end{equation} 
\end{claim}
\begin{proof}
    See Section~\ref{sec: proof claim clt rate X_star Y}.
\end{proof}
\begin{claim}\label{claim: CLT rate X X_hat}
Under the event $\cA_n$, we have
\begin{equation}\label{eq: CLT rate X X_hat}
     \sup_{v \in \RR^p, s \in \RR}\left| \PP( \cM_{\cA}^v({X}) - \cM_{\cB}^v({X} ) \le s) - \PP(\cM_{\cA}^v({\hat{X}}) - \cM_{\cB}^v(\hat{X}) \le s \,|\, \xi_{1:n}) \right| \lesssim \nu_{n,1}.
\end{equation}
\end{claim}
\begin{proof}
    See Section~\ref{sec: proof claim clt rate X X_hat}.
\end{proof}
By the triangle inequality, combining \eqref{eq: CLT rate X* X} and \eqref{eq: CLT rate X* Y} shows that \eqref{eq: CLT rate} holds under the event \(\cA_n\). Furthermore, since both sides of   \eqref{eq: CLT rate} are deterministic and $\PP(\cA_n) > 0$, we have that \eqref{eq: CLT rate} holds unconditionally. Combining \eqref{eq: CLT rate X* X}, \eqref{eq: CLT rate X* X} and \eqref{eq: CLT rate X X_hat}, we have that \eqref{eq: CLT rate 1} holds under the event $\cA_n$  by triangle inequality. Besides, by \eqref{eqn:AnP}, we have that  $\cA_n$ holds with probability at least $1-1/(2n^4) - 1/n - 3 ({B_n^2 \log^3(pn) /n})^{1/2}$, and the results of Theorem~\ref{prop: clt gaussian approx rate} hold subsequently. 
\end{proof}
\subsubsection{Proof of Claim~\ref{claim: CLT rate X_star X}}\label{sec: proof claim CLT rate X* X}
 % \begin{lemma}\label{lm: anticon cond for beta mult}
 %   Under the same conditions as Theorem~\ref{prop: clt gaussian approx rate}, let $X^*$ be the Beta multiplier bootstrap statistic defined in \eqref{eq: X beta mult}, and then Condition~\ref{cond: anti-con rate} is satisfied by $X^*$  with $r_n = {\EE [\max_{\ell \in \cS_{\rm p}} |X_{\ell} - \EE(X_{\ell})|/\sigma_{\ell}]}/{C_{\cA, \cB}}$ and $\varepsilon_n = \nu_{n,1}$, where $\nu_{n,1}$ is the rate defined in \eqref{eq: nu rates clt}.
 %   % for the constant  $C = 2 K_0$. 
 % \end{lemma}
We will apply Lemma~\ref{lm: gauss comp rate} to show \eqref{eq: CLT rate X* X}. Specifically, let $\tau(\cdot): \RR^p \rightarrow \RR^{p \times p}$ be the Stein kernel of $X^*$ defined in \eqref{eq: stein kernel cond} conditioned on $\xi_{1:n}$. We will check the rate of $\Delta = \EE(\max_{j,k \in [p]} |\tau_{jk}(V) - \Sigma_{jk}|)$. 
 % that satisfies the condition in \eqref{eq: stein kernel cond}. 
By Lemma~4.6 of \cite{koike2021notesdimdep},
 % the last equation in the proof of Corollary~3.2 in \cite{cck2022improvedbootstrap}, 
under event $\cA_n$, we have that
 \begin{equation}\label{eq: stein beta ineq}
      \EE\Big(\max_{j,k \in [p]} \big|\tau_{jk} (X^*) - \frac{1}{n} \sum_{i=1}^n (\xi_{ij} - \bar\xi_j) (\xi_{ik} - \bar\xi_k)\big|  \, \,\, \Big| \xi_{1:n}\Big) \le C_1 B_n \sqrt{\frac{ \log p}{n}}  
 \end{equation}
 where $C_1 > 0$ is a constant.  Further, under the event $\cA_n$, we have
{\small \begin{align*}
&  \EE\Big(\max_{j,k \in [p]} \left|\tau_{jk} (X^*) - \Sigma_{jk}^*\right| \,\, \Big| \xi_{1:n}\Big) \\
& \  \le  \EE\Big(\max_{j,k \in [p]} \big|\tau_{jk} (X^*) \!-\! \frac{1}{n} \sum_{i=1}^n (\!\xi_{ij} - \bar\xi_j\!)(\!\xi_{ik} - \bar\xi_k)\big|  \Big| \xi_{1:n} \Big) \!+\! \max_{j,k \in [p]} \Big|\frac{1}{n} \sum_{i=1}^n (\!\xi_{ij} - \bar\xi_j) (\!\xi_{ik}\! -\! \bar\xi_k) - \Sigma_{jk}^*\Big| \\
& \ \le C_1 B_n \sqrt{\frac{ \log p}{n}}  + C_2 B_n \sqrt{\frac{\log (pn)}{n}}\\
  & \ \le C_3 B_n \sqrt{\frac{\log (pn)}{n}},
\end{align*}}
where $\Sigma^* = \frac{1}{n} \sum_{i=1}^n \Sigma_i$ with $\Sigma_i = \EE[\xi_i \xi_i^{\top}]$. The first inequality holds by the triangle inequality, the second inequality holds by \eqref{eq: stein beta ineq} together with that event $\cA_{n,3}$ holds,
% the condition that event (36) of \cite{cck2022improvedbootstrap} holds 
% under the event $\cA_n$,
and the last inequality holds for a large enough $C_3$.
 Also, by assumption, the random vector $X$ satisfies the covariance Condition~\ref{cond: anti-con cov}. Applying Lemma~\ref{lm: gauss comp rate}, under the event $\cA_n$, we have that $\Delta = C_3 B_n \sqrt{{\log (pn)}/{n}}$ so that
\begin{equation}\label{eq: bound beta anti-con}
    \begin{aligned}
    &\sup_{v \in \RR^p, s \in \RR}\left| \PP(\cM_{\cA}^v({X^*}) - \cM_{\cB}^v({X^*}) \le s \,|\, \xi_{1:n}) - \PP( \cM_{\cA}^v({X}) - \cM_{\cB}^v({X} ) \le s)\right| \\
    & \qquad \lesssim  \frac{\EE [\max_{\ell \in \cS_{\rm p}} |X_{\ell} - \EE(X_{\ell})|/\sigma_{\ell}]}{C_{\cA, \cB}} \left(\frac{B_n^2 \log^3(pn)}{n}\right)^{1/4} \lesssim \nu_{n,1},
\end{aligned}
\end{equation}
where $C_{\cA,\cB}$ is defined in Condition~\ref{cond: anti-con cov}. Hence \eqref{eq: CLT rate X* X} holds under the event $\cA_n$.
\subsubsection{Proof of Claim~\ref{claim: CLT rate X_star Y}}\label{sec: proof claim clt rate X_star Y}
Define $Y' = n^{-1/2} \sum_{i=1}^n (\xi_i' - \mu + a)$, where $\xi_{1:n}'$ are independent copies of $\xi_{1:n}$. Then $Y'$ and $Y$ are independently and identically distributed. We show \eqref{eq: CLT rate X* Y} by applying  Lemma~\ref{prop: modified thm 3.1} to $X^*$ and $Y'$. We need to verify that the conditions of Lemma~\ref{prop: modified thm 3.1}  are satisfied conditional on $\xi_{1:n}$ under the event $\cA_n$. By Lemma~4.1 and Lemma~4.2 in \cite{cck2022improvedbootstrap}, we have that Conditions~\ref{cond: fourth moment}-\ref{cond: 8th inf mm bd} and the moment conditions \eqref{eq: m2 cond} and \eqref{eq: m3 cond} are satisfied conditional on the event $\cA_n$. Therefore, we are only left with verifying Condition~\ref{cond: anti-con rate} for $X^*$.
Applying Claim~\ref{claim: CLT rate X_star X}, under the event $\cA_n$, for all $v \in \RR^p$, $s \in \RR$, and $t \in (0,\infty)$, we have 
\begin{align*}
    &\PP(\cM_{\cA}^v({X^*}) - \cM_{\cB}^v({X^*}) \le s + t\,|\, \xi_{1:n}) \le \PP( \cM_{\cA}^v({X}) - \cM_{\cB}^v({X} ) \le s + t) + K_0 \nu_{n,1} \\
    &\quad \le  \PP( \cM_{\cA}^v({X}) - \cM_{\cB}^v({X} ) \le s ) + K_0t\cdot\frac{\EE [\max_{\ell \in \cS_{\rm p}} |X_{\ell} - \EE(X_{\ell})|/\sigma_{\ell}]}{C_{\cA, \cB}}   + K_0 \nu_{n,1} \\
    & \quad \le  \PP( \cM_{\cA}^v({X^*}) - \cM_{\cB}^v({X^*} ) \le s \,|\, \xi_{1:n}) + K_0t\cdot\frac{\EE [\max_{\ell \in \cS_{\rm p}} |X_{\ell} - \EE(X_{\ell})|/\sigma_{\ell}]}{C_{\cA, \cB}}  + 2 K_0 \nu_{n,1} ,
\end{align*}
where $K_0 > 0$ is a large enough constant; the first and the third inequalities hold by \eqref{eq: bound beta anti-con}, and the second inequality holds by applying 
% Theorem~\ref{col: anti con equal var no cor 1} and 
Theorem~\ref{thm: anti con} to $X$ since $X$ is a Gaussian random vector. Then we have that Condition~\ref{cond: anti-con rate} is satisfied by $X^*$  with $r_n = {\EE [\max_{\ell \in \cS_{\rm p}} |X_{\ell} - \EE(X_{\ell})|/\sigma_{\ell}]}/{C_{\cA, \cB}}$ and $\varepsilon_n = \nu_{n,1}$ for the constant  $C = 2 K_0$. 

% Next, by Lemma~4.1 and Lemma~4.2 in \cite{cck2022improvedbootstrap}, we have that Conditions~\ref{cond: fourth moment}-\ref{cond: 8th inf mm bd} and the moment conditions \eqref{eq: m2 cond} and \eqref{eq: m3 cond} are satisfied conditioning on the event $\cA_n$. 
Therefore, applying Lemma~\ref{prop: modified thm 3.1}, under the event $\cA_n$ we have that
\begin{align*}
&\sup_{v \in \RR^p, s \in \RR}\left| \PP(\cM_{\cA}^v({X^*}) - \cM_{\cB}^v({X^*}) \le s \,|\, \xi_{1:n}) - \PP( \cM_{\cA}^v({Y}) - \cM_{\cB}^v({Y} ) \le s )\right| \\
    &\quad = \sup_{v \in \RR^p, s \in \RR}\left| \PP(\cM_{\cA}^v({X^*}) - \cM_{\cB}^v({X^*}) \le s \,|\, \xi_{1:n}) - \PP( \cM_{\cA}^v({Y'}) - \cM_{\cB}^v({Y'} ) \le s \,|\, \xi_{1:n} )\right| \\
    & \quad \lesssim  \nu_{n,1},
\end{align*}
where $Y' = n^{-1/2} \sum_{i=1}^n (\xi_i' - \mu + a)$ with $\xi_{1:n}'$ being independent copies of $\xi_{1:n}$. This completes the proof of Claim~\ref{claim: CLT rate X_star Y}.
% Then, \eqref{eq: CLT rate} holds by triangle inequality conditioning on the event $\cA_n$. Furthermore, since both sides of   \eqref{eq: CLT rate} are deterministic, we have that \eqref{eq: CLT rate} holds unconditionally.
\subsubsection{Proof of Claim~\ref{claim: CLT rate X X_hat}}\label{sec: proof claim clt rate X X_hat}
Note that the Stein kernel for a Gaussian random vector is its covariance matrix. Then, by similar arguments of showing \eqref{eq: bound beta anti-con} and Lemma~\ref{lm: gauss comp rate}, conditional on the event $\cA_n$,  we have that 
\begin{align*}
     &\sup_{v \in \RR^p, s \in \RR}\left| \PP( \cM_{\cA}^v({X}) - \cM_{\cB}^v({X} ) \le s) - \PP(\cM_{\cA}^v({\hat{X}}) - \cM_{\cB}^v(\hat{X}) \le s \,|\, \xi_{1:n}) \right| \\
     & \ \lesssim  \frac{\EE [\max_{\ell \in \cS_{\rm p}} |X_{\ell} - \EE(X_{\ell})|/\sigma_{\ell}]}{C_{\cA, \cB}} \Big(\log p\cdot \max_{j,k \in [p]} \Big|\frac{1}{n} \sum_{i=1}^n (\xi_{ij} - \bar\xi_j) (\xi_{ik} - \bar\xi_k) - \Sigma_{jk}^*\Big|\Big)^{1/2}\\
     & \ \lesssim \nu_{n,1},
\end{align*}
and Claim~\ref{claim: CLT rate X X_hat} follows.
% Then by triangle inequality, we have that
%  \begin{align*}
%     &\sup_{v \in \RR^p, s \in \RR}\left| \PP( \cM_{\cA}^v({Y}) - \cM_{\cB}^v({Y} ) \le s) - \PP(\cM_{\cA}^v({\hat{X}}) - \cM_{\cB}^v(\hat{X}) \le s \,|\, \xi_{1:n}) \right|   \\
%     &\qquad \le \sup_{v \in \RR^p, s \in \RR}\left| \PP( \cM_{\cA}^v({Y}) - \cM_{\cB}^v({Y} ) \le s) - \PP( \cM_{\cA}^v({X}) - \cM_{\cB}^v({X} ) \le s) \right| \\
%     & \quad \quad \quad + \sup_{v \in \RR^p, s \in \RR}\left| \PP( \cM_{\cA}^v({X}) - \cM_{\cB}^v({X} ) \le s) - \PP(\cM_{\cA}^v({\hat{X}}) - \cM_{\cB}^v(\hat{X}) \le s \,|\, \xi_{1:n}) \right| \\
%     &\qquad \lesssim \nu_{n,1},
% \end{align*}
% which shows that \eqref{eq: CLT rate 1} holds conditioning on the event $\cA_n$. By \eqref{eqn:AnP}, we have that  $\cA_n$ holds with probability at least $1-1/(2n^4) - 1/n - 3 ({B_n^2 \log^3(pn) /n})^{1/2}$, which completes the proof.

% if your bibliography is in bibtex format, uncomment commands:
\section{Conclusion}\label{sec: conclusion}
To conclude, we derive new dimension-free anti-concentration bounds for the difference between the maxima of two Gaussian random vectors. In contrast to existing literature, our bounds depend only on pairwise correlations and not on the minimum eigenvalue of the covariance matrix.
% By developing new proof techniques, we show that the anti-concentration properties depend on the correlations between two Gaussian random vectors. 
In particular, under the homogeneous variance setting, our bound holds when there is no perfect positive correlation between the two random vectors. Meanwhile, for the heterogeneous variance setting, our bound holds under a mild covariance condition. 
 We also apply our new anti-concentration bounds to establish the central limit theorem for the maximizer of empirical processes indexed by a discrete set, and we propose a valid Gaussian multiplier approximation for the distribution of the maximizer.
We conduct extensive simulation studies to validate our findings empirically. Also, in our simulation studies, we find that the bounds may hold under more relaxed assumptions. We plan to further relax the assumptions in our future work.

%With novel proofs we demonstrate that the anti-concentration properties are influenced by the pairwise correlation structure of the Gaussian vectors. This improves the initial findings by \cite{imaizumi2021gaussian}, which were based on conditioning one maximum on the entire Gaussian vector from the other partition, thereby tying the bound to the reciprocal of the minimum eigenvalue of the covariance matrix. When the Gaussian random vector is standardized to have a homogeneous component-wise variance, our theoretical framework necessitates only the absence of perfectly positively correlated pairs between partitions, a condition validated as optimal through a lower bound case. For cases with heterogeneous variances, we impose the technical condition that either \(\min_{j \in \cB} (\sigma_j - \max_{i \in \cA} \sigma_j^{-1} \sigma_{ij}) > 0\) or \(\min_{i \in \cA} (\sigma_i - \max_{j \in \cB} \sigma_i^{-1} \sigma_{ij}) > 0\) must hold. Simulation studies indicate potential for refining this condition, aiming for anti-concentration bounds akin to those in the uniform variance scenario, even when the current pairwise covariance constraints are violated. Improvements over the current technical conditions will be the focus of our future works.
%%%%%%%%%%%%%%%%%%%%%%%%%%%%%%%%%%%%%%%%%%%%%%
%% Example with single Appendix:            %%
%%%%%%%%%%%%%%%%%%%%%%%%%%%%%%%%%%%%%%%%%%%%%%

\begin{appendix}
\section{Proofs of Technical Lemmas in the Main Text}\label{appx: add proof} %% if no title is needed, leave empty \section*{}.
We provide in Appendix~\ref{appx: add proof} the proofs for technical lemmas in the main text. For notational simplicity, for two random variables $X_1, X_2 \in \RR$, we denote by \(X_1 \equiv X_2\) when \(X_1=X_2\) almost surely.
% for two random variables $X_1, X_2 \in \RR$.

\subsection{Proof of Lemma~\ref{lm: joint max dens}}\label{sec: proof lm joint max dens}
\begin{proof}[\unskip\nopunct]
   We start by proving that case (1). First, the absolute continuity for the joint distribution of $ (M_{\cA}, M_{\cB}) =  (\max_{i \in \cA} X_i, \max_{j \in \cB} X_j)$ follows from the assumption that $|{\rm Corr}(X_i, X_j) | < 1$ for any $i \in \cA$ and $j \in \cB$ together with the fact that for any Borel measurable subset $\cI \in \cB(\RR^2)$, we have 
   $$
   \PP((M_{\cA}, M_{\cB}) \in \cI) \le \sum_{i \in \cA, j \in \cB} \PP((X_i, X_j) \in \cI).
   $$
   
     Next, we prove \eqref{eq: max joint den} by showing that 
     \begin{equation}\label{eq: max joint dens lim}
         \lim_{\varepsilon \downarrow 0}\varepsilon^{-2} \PP(x < M_{\cA} \le x + \varepsilon, y < M_{\cB} \le y+\varepsilon) = f(x, y), \quad \text{for a.e. $(x,y) \in \RR^2$.}
     \end{equation}
     First, note that for any $(x,y) \in \RR^2$, we have: 
    $$
    \{x < M_{\cA} \le x + \varepsilon, y < M_{\cB} \le y+\varepsilon \} = \cup_{q=1}^{|\cA|}\cup_{k=1}^{|\cB|} \cE_{\varepsilon, x, y}^{q,k},
    $$
    where the event $\cE_{\varepsilon, x, y}^{q,k}$ is defined as 
    $$
    \cE_{\varepsilon, x, y}^{q,k} = \left\{\exists I_1 \subseteq \cA, |I_1| =q  \text{ and } I_2 \subseteq \cB, |I_2| = k: \begin{array}{ll}
       x < X_i \le x + \varepsilon,  & \,\, \forall i \in I_1  \\
        X_{i'} \le x , &\,\,   \forall i' \in \cA \backslash I_1 \\
          y < X_j \le y + \varepsilon,  & \,\, \forall j \in I_2  \\
        X_{j'} \le y , &\,\,   \forall j' \in \cB \backslash I_2
    \end{array}  \right\}.
    $$
   By construction, the sets $\cE_{\varepsilon, x, y}^{q,k}$'s are disjoint. 
   Then, we propose the following two claims.
   % that are essential for proving  \eqref{eq: max joint dens lim}.
   \numberwithin{claim}{section}
   \begin{claim}\label{claim: 241}
       For a.e. $(x,y) \in \RR^2$, we have that
        \begin{align*}
       \lim_{\varepsilon \downarrow 0} \varepsilon^{-2}   \PP(\cE_{\varepsilon, x, y}^{1,1} ) = \sum_{i \in \cA,j \in \cB}\PP \Big(\max_{i' \in \cA \backslash \{i\}} X_{i'} \le x, \max_{j' \in \cB \backslash \{j\}} X_{j'} \le y \Big| X_i  = x, X_j = y \Big) \phi_{i,j}(x, y).
   \end{align*}
   \end{claim}
   \begin{proof}
   See Appendix~\ref{sec: proof claim 241}.
   \end{proof}
   \begin{claim}\label{claim: 242}
       For $\cE_{\varepsilon, x, y}^{q,k} $ with  $q \ge 2$ or $k \ge 2$, $\PP(\cE_{\varepsilon, x, y}^{q,k}) = o(\varepsilon^2)$ when $\varepsilon \downarrow 0$ for a.e. $(x,y) \in \RR^2$.
   \end{claim}
   \begin{proof}
   See Appendix~\ref{sec: proof claim 242}.
   \end{proof}
Combining Claim~\ref{claim: 241} and Claim~\ref{claim: 242}, we have that, for a.e. $(x,y) \in \RR^2$, 
   \begin{align*}
           \lim_{\varepsilon \downarrow 0}\varepsilon^{-2} \PP(x < M_{\cA} \le x + \varepsilon, y < M_{\cB} \le y+\varepsilon)  &=  \lim_{\varepsilon \downarrow 0} \varepsilon^{-2}   \sum_{q \in [|\cA|], k \in [|\cB|]}\PP(\cE_{\varepsilon, x, y}^{q,k} ) \\
           & \quad =  \lim_{\varepsilon \downarrow 0} \varepsilon^{-2}   \PP(\cE_{\varepsilon, x, y}^{1,1} )  = f(x, y),
   \end{align*}
   which completes the proof of \eqref{eq: max joint den} for case (1). 

   Next, we show that case (2) holds. The absolute continuity follows from the condition that $\sigma_i > 0$ for any $i \in \cA$, along with the bound that for any Borel measurable subset $\cI \in \cB(\RR)$, 
   $$
   \PP(M_{\cA} \in \cI, M_{\cB} \le y) \le \sum_{i \in \cA} \PP(X_i\in \cI, M_{\cB} \le y) \le \sum_{i \in \cA} \PP(X_i\in \cI) \le |\cA|.
   $$
   
   Then, similar to the proof of case (1), to show that \eqref{eq: max joint part x} holds, we only need to show that 
   \begin{equation}\label{eq: joint max dens par lim}
       \lim_{\varepsilon \downarrow 0}\varepsilon^{-1} \PP(x < M_{\cA} \le x + \varepsilon, M_{\cB} \le y) = f_y(x), \quad \text{for a.e. $x \in \RR$ given any $y \in \RR$,}
   \end{equation} 
   where $ f_y(x) :=\sum_{i \in \cA} \PP \Big(\max_{i' \in \cA \backslash \{i\}} X_{i'} \le x, \max_{j \in \cB } X_j \le y \Big| X_i  = x \Big) \phi_{i}(x)$.
   First, we have the following disjoint decomposition
 that $$
    \{x < M_{\cA} \le x + \varepsilon, M_{\cB} \le y\} = \cup_{k=1}^{|\cA|} \cE_{\varepsilon, x, y }^{k},
    $$
    where each event $\cE_{\varepsilon, x, y}^{k}$ is defined as
    $$
    \cE_{\varepsilon, x, y}^{k} = \left\{\exists I \subseteq \cA, |I| =k: \begin{array}{ll}
       x < X_i\le x + \varepsilon,  & \,\, \forall i \in I \\
        X_{i'} \le x,  &\,\,   \forall i' \in \cA \backslash I \\
           X_{j} \le y,  & \,\, \forall j \in \cB 
            % \xi_{S''} \le z,  & \,\, \forall S'' \in \cC
    \end{array}  \right\}.
    $$
    Then, similar to the proof of case (1), the proof of \eqref{eq: joint max dens par lim} depends on the following two claims.
    \begin{claim}\label{claim: 243}
        Given any $y \in \RR$, we have
        $$
         \lim_{\varepsilon \downarrow 0} \varepsilon^{-1}   \PP(\cE_{\varepsilon, x, y}^{1} ) = \sum_{i \in \cA} \PP \Big(\max_{i' \in \cA \backslash \{i\}} X_{i'} \le x, \max_{j \in \cB } X_j \le y \Big| X_i  = x \Big) \phi_{i}(x) = f_y(x),
        $$
        for a.e. $x \in \RR$.
    \end{claim}
    \begin{proof}
    See Appendix~\ref{sec: proof claim 243}.
    \end{proof}
    \begin{claim}\label{claim: 244}
        Given any $y \in \RR$, for $k \ge 2$, $\PP(\cE_{\varepsilon, x, y}^{k}) = o(\varepsilon)$ as $\varepsilon \downarrow 0$ for a.e. $x \in \RR$.
    \end{claim}
    \begin{proof}
    See Appendix~\ref{sec: proof claim 244}.
    \end{proof}
 Combining Claim~\ref{claim: 243} and Claim~\ref{claim: 244}, for any given $y \in \RR$, we have that, for a.e. $x \in \RR$, 
    \begin{align*}
        \lim_{\varepsilon \downarrow 0}\varepsilon^{-1} \PP(x < M_{\cA} \le x + \varepsilon, M_{\cB} \le y) = \lim_{\varepsilon \downarrow 0}\varepsilon^{-1} \sum_{k = 1}^{|\cA|} \PP(\cE_{\varepsilon, x, y}^{k} ) = \lim_{\varepsilon \downarrow 0}\varepsilon^{-1} \PP(\cE_{\varepsilon, x, y}^{1} ) = f_y(x).
    \end{align*}
    Consequently, \eqref{eq: joint max dens par lim} holds, which completes the proof of \eqref{eq: max joint part x}.
   \end{proof}
\subsection{Proof of Lemma~\ref{claim: 251}}\label{sec: proof claim 1}
	\begin{proof}[\unskip\nopunct]
		We show that \eqref{eq: claim g = 0} and \eqref{eq: claim g < H} in Lemma~\ref{claim: 251}  hold
		% that Lemma~\ref{claim: 251} and Claim~2 hold 
		for a.e. $u \in \RR$ and every $j \in \cB$. For a given $j \in \cB$, let
		$$
		\cS_j^+ := \{i \in [p] \backslash \{j\}: 1 - \sigma_j^{-2} \sigma_{i j} > 0\},
		\quad 
		\cS_j^0 := \{i \in [p] \backslash \{j\} : 1 - \sigma_j^{-2} \sigma_{i j} =  0\},
		$$
		and 
		\begin{equation}\label{eq: C_j def}
			\cC_j^1 := \cS_j^+ \cap \{i \in [p] \backslash \{j\}:| {\rm Corr} (X_i, X_{j})| = 1\},
		\end{equation}
		where by assumption we have that $\cA \subseteq \cS_j^+ \backslash \cC_j^1$.
		Then we have
		\begin{align*}
			H_{u,j}&(x) = \PP\Big(\max_{i \in \cA, j' \in \cB \backslash \{j\}} X_{i} \vee (X_{j'} - u) \le x | X_{j} = x+u \Big) \\
			& = \PP \left( X_{i} \le x + u \cdot \II\{i \in \cB \}, \,\, i\in [p] \backslash \{j\} | X_{j} = x+u \right)\\
			% & = \PP \left(V_{i} \le x + u \cdot \II\{\tilde{S} \in \cB \} - \mu_{i} - \Sigma_{S'}^{-1} \Sigma_{S' \tilde{S}} (y + u - \mu_{j}) , \,\, \tilde{S} \in [p] \backslash \{S'\}| X_{j} = y+u \right)\\
			& = \PP\left(
			V_{i} \le x + u \cdot \II\{i \in \cB \} - \mu_{i} - \sigma_{j}^{-2} \sigma_{ij} (x + u - \mu_{j})  ,\, i \in [p] \backslash\{j\} 
			|\, X_{j} = x+u \right)\\
			& \overset{\rm (a)}{=} \PP\left\{\begin{array}{ll}
				(1 - \sigma_{j}^{-2} \sigma_{ij})^{-1} \left( V_{i}  - u \cdot \II\{i \in \cB \} +\mu_{i}  + \sigma_{j}^{-2} \sigma_{ij} (u - \mu_{j}) \right) \le x , &  \, i\in \cS_j^+ \\
				V_{i} \le u \cdot \II\{i \in \cB \} - \mu_{i} - \sigma_{j}^{-2} \sigma_{ij} ( u - \mu_{j})  , &  \, i \in \cS_j^0 
			\end{array} \right\}\\
			& \overset{\rm (b)}{=}  \PP\left\{\begin{array}{ll}
				(1 - \sigma_{j}^{-2} \sigma_{ij})^{-1} \!\left( V_{i}  - u \cdot \II\{i \in \cB \} +\mu_{i}  + \sigma_{j}^{-2} \sigma_{ij} (u - \mu_{j}) \right) \le x ,  &  \, i \in \cS_j^+ \backslash\cC_j^1\\
				(1 - \sigma_{j}^{-2} \sigma_{ij})^{-1} \left(  \mu_{i} - u \cdot \II\{i \in \cB \}  + \sigma_{j}^{-2} \sigma_{ij} (u - \mu_{j}) \right) \le x  , &  \, i \in \cC_j^1\\
				V_{i} \le u \cdot \II\{i \in \cB \} - \mu_{i} - \sigma_{j}^{-2} \sigma_{ij} ( u - \mu_{j}),   &  \, i \in \cS_j^0 
			\end{array} \right\}\\
			& \overset{\rm (c)}{=} \PP\left\{\begin{array}{ll}
				(1 - \sigma_{j}^{-2} \sigma_{ij})^{-1} \left( V_{i} +\mu_{i}  + \sigma_{j}^{-2} \sigma_{ij} (u - \mu_{j}) \right) \le x,   &  \, i \in \cA\\
				(1 - \sigma_{j}^{-2} \sigma_{ij})^{-1} \left( V_{i}   +\mu_{i}  - \sigma_{j}^{-2} \sigma_{ij} \mu_{j} \right) - u \le x,   &  \, i \in \cS_j^+ \backslash\cC_j^1\backslash \cA \\
				(1 - \sigma_{j}^{-2} \sigma_{ij})^{-1} \left(  \mu_{i} - u \cdot \II\{i \in \cB \}  + \sigma_{j}^{-2} \sigma_{ij} (u - \mu_{j}) \right) \le x,   &  \, i \in \cC_j^1\\
				V_{i} \le u \cdot \II\{i \in \cB \} - \mu_{i} - \sigma_{j}^{-2} \sigma_{ij} ( u - \mu_{j}){ ,}   &  \, i \in \cS_j^0 
			\end{array} \right\},
		\end{align*}
		where $V_{i} = X_{i} - \mu_{i} - \sigma_{j}^{-2} \sigma_{ij} (X_{j}  - \mu_{j}) $ is the orthogonal residual of $X_i$ independent of $X_{j}$. Equality (a) follows from the fact that \([p] \backslash \{j\} = \cS_j^+ \cup \cS_j^0\) when condition (A) holds, equality (b) separates variables perfectly correlated with \(X_j\) that are constants conditional on \(X_j = x+u\), and equality (c) follows from the assumption that \(\max_{i \in \cA, j \in \cB} |\Corr(X_i, X_j)| < 1\) such that \(\cA \subseteq \cS_j^+ \backslash \cC_j^1\).
		% Hence it can be seen that Lemma~\ref{claim: 251} follows.
		When $\cC_j^1 \ne \emptyset$, we let 
		$$
		x_j(u) = \max_{i \in \cC_j^1} (1 - \sigma_{j}^{-2} \sigma_{ij})^{-1} \left(  \mu_{i} - u \cdot \II\{i\in \cB \}  + \sigma_{j}^{-2} \sigma_{ij} (u - \mu_{j}) \right),
		$$ 
		and when $\cC_j^1 = \emptyset$, we let $x_j(u) = -\infty$. 
		
		In what follows, we discuss the case when $x_j(u) > -\infty$, and the results for the case $x_j(u) = -\infty$ follow by similar arguments. First, we observe that
		\begin{enumerate}
			\item when $x \ge x_j(u)$: 
			\begin{align*}
				%&H_{u,j}(x) =\\
				&H_{u,j}(x) = \PP\left\{\begin{array}{ll}
					(1 - \sigma_{j}^{-2} \sigma_{ij})^{-1} \left( V_{i} +\mu_{i}  + \sigma_{j}^{-2} \sigma_{ij} (u - \mu_{j}) \right) \le x ,  &  \, i \in \cA\\
					(1 - \sigma_{j}^{-2} \sigma_{ij})^{-1} \left( V_{i}   +\mu_{i}  - \sigma_{j}^{-2} \sigma_{ij} \mu_{j} \right) - u \le x ,  &  \, i \in \cS_j^+ \backslash\cC_j^1\backslash \cA \\
					% (1 - \sigma_{S'}^{-2} \sigma_{S' \tilde{S}})^{-1} \left(  \mu_{i} - u \cdot \II\{\tilde{S} \in \cB \}  + \sigma_{S'}^{-2} \sigma_{S' \tilde{S}} (u - \mu_{j}) \right) \le x   &  ,\, \tilde{S} \in \cC_j\\
					V_{i} \le u \cdot \II\{i \in \cB \} - \mu_{i} - \sigma_{j}^{-2} \sigma_{ij} ( u - \mu_{j}),   &  \, i \in \cS_j^0 
				\end{array} \right\},
			\end{align*}
			\item when $x < x_j(u)$: $H_{u,j}(x) = 0$ as there exists an $i \in \cC_j^1$ that violates the constraint:
			$$
				(1 - \sigma_{j}^{-2} \sigma_{ij})^{-1} \left(  \mu_{i} - u \cdot \II\{i \in \cB \}  + \sigma_{j}^{-2} \sigma_{ij} (u - \mu_{j}) \right) \le x.
			$$
		\end{enumerate}
		Moreover, recall the definition of $g_{u, j} (x) $ is 
		$$
			g_{u, j} (x)  =  \sum_{i \in \cA} \PP \Big(\max_{i' \in \cA, j' \in \cB } X_{i'} \!\vee\! (X_{j'} - u) \!\le\! x \Big| X_i  = x , X_j = x + u \Big) \phi_{i| j}(x | x+ u),
		$$
	where for any $i \in \cA$ we have 
		\begin{align*}
			&\PP \Big(\max_{i' \in \cA, j' \in \cB } X_{i'} \vee (X_{j'} - u) \le x \Big| X_i  = x , X_j = x + u \Big)  \\
   & = \PP \Big(X_{i'} \le x + u \cdot \II\{i' \in \cB \},    \, i' \in [p] \backslash \cup \{i,j\} \Big| X_i  = x , X_j = x + u \Big)  \\
			& = \PP\left\{\begin{array}{ll}
			X_{i'} \le x + u \cdot \II\{i' \in \cB \},  &  \, i' \in [p] \backslash (\cC_j^1 \cup \{i,j\})\\
				X_{i'} \le x + u \cdot \II\{i' \in \cB \},  &  \, i' \in \cC_j^1 
			\end{array} \Bigg| X_i  = x , X_j = x + u  \right\}\\
			& =\PP\left( \!\!\left\{ X_{i'} \le x + u \cdot \II\{i' \in \cB \}, i' \in [p] \backslash (\cC_j^1 \cup \{i,j\})\right\} \cap \left\{ x \ge x_j(u)\right\} \, \Big|  X_i  = x , X_j = x + u  \!\right).
		\end{align*}
		The second equality follows from the fact that \( \cC_j^1 \subseteq [p] \backslash \{i,j\} \) by the assumption \( \cA \subseteq \cS_j^+ \backslash \cC_j^1 \) such that \( i \notin \cC_j^1 \) for any \( i \in \cA \), and the third equality follows from the equivalence for all \( i' \in \cC_j^1 \), conditional on \( X_j = x + u \):
$$
\big\{X_{i'} \le x + u \cdot \II\{i' \in \cB \} \big\} = \Big \{ (1 - \sigma_{j}^{-2} \sigma_{i'j})^{-1} \Big(  \mu_{i'} - u \cdot \II\{i' \in \cB \} + \sigma_{j}^{-2} \sigma_{i'j} (u - \mu_{j}) \Big) \le x \Big \}.
$$
Therefore, $  g_{u, j} (x) = 0 = d H_{u,j}(x) /dx$ when $x < x_j(u)$, and \eqref{eq: claim g = 0} follows.
		
		Next, we move on to show that \eqref{eq: claim g < H}  holds when $x > x_j(u)$. We first rewrite $ H_{u,j}(x)$ as
		\begin{align*}
			H_{u,j}(x) &= \PP\left\{\begin{array}{ll}
				(1 - \sigma_{j}^{-2} \sigma_{ij})^{-1} \left( V_{i} +\mu_{i}  + \sigma_{j}^{-2} \sigma_{ij} (u - \mu_{j}) \right) \le x ,  &  \, i \in \cA\\
				(1 - \sigma_{j}^{-2} \sigma_{ij})^{-1} \left( V_{i}   +\mu_{i}  - \sigma_{j}^{-2} \sigma_{ij} \mu_{j} \right) - u \le x ,  &  \, i \in  \cS_j^+ \backslash\cC_j^1\backslash \cA \\
				% (1 - \sigma_{S'}^{-2} \sigma_{S' \tilde{S}})^{-1} \left(  \mu_{i'} - u \cdot \II\{\tilde{S} \in \cB \}  + \sigma_{S'}^{-2} \sigma_{S' \tilde{S}} (u - \mu_{j}) \right) \le x   &  ,\, \tilde{S} \in \cC_j\\
				V_{i} \le u \cdot \II\{i \in \cB \} - \mu_{i} - \sigma_{j}^{-2} \sigma_{ij} ( u - \mu_{j}) ,  &  \, i \in \cS_j^0 
			\end{array} \right\}\\
			& = \PP\left(\begin{array}{ll} 
				\tilde{V}_{i} \le x  , &  \, i \in \cS_j^+ \backslash\cC_j^1 \\
				\tilde{U}_{i} \le 0 ,&  \, i \in \cS_j^0 
			\end{array} \right) = \PP\left(\max_{i \in \cS_j^+ \backslash\cC_j^1} \tilde{V}_{i}  \le x, \quad \max_{i \in \cS_j^0} \tilde{U}_{i} \le 0\right),
		\end{align*}
		where  
		$$
		\tilde{V}_{i}  = (1 - \sigma_{j}^{-2} \sigma_{ij})^{-1} \left( V_{i}  - u \cdot \II\{i \in \cB \} +\mu_{i}  + \sigma_{j}^{-2} \sigma_{ij} (u - \mu_{j}) \right) , \quad i \in \cS_j^+ \backslash\cC_j^1,
		$$
		$$
		\text{and} \quad \tilde{U}_{i} = V_{i} - u \cdot \II\{i \in \cB \} + \mu_{i} + \sigma_{j}^{-2} \sigma_{i j} ( u - \mu_{j}), \quad i \in \cS_j^0,
		$$
		and the second equality follows from that since $\cA \subseteq \cS_j^+ \backslash\cC_j^1$, we have $\cA \cup  \cS_j^+ \backslash\cC_j^1 \backslash \cA  = \cS_j^+ \backslash\cC_j^1$.
		
		To apply \eqref{eq: max joint part x} of Lemma~\ref{lm: joint max dens} to compute the derivative of $H_{u,j}(x)$ on $x$, we verify that the conditions in case (2) of Lemma~\ref{lm: joint max dens} hold for a.e. $u \in \RR$. First, by the construction of the set $\cC_j^1$ in \eqref{eq: C_j def}, we have that $\Var ( \tilde{V}_{i} ) > 0$ for any $i \in \cS_j^+ \backslash\cC_j^1$. Next, we remove the duplicate terms in $\{\tilde{V}_{i} \}_{i \in \cS_j^+ \backslash\cC_j^1}$, so that for the remaining $i,j \in  \cS_j^+ \backslash\cC_j^1$ and $i \ne j$, it holds that either $\Var(\tilde{V}_i - \tilde{V}_j) > 0$ or $\EE (\tilde{V}_i - \tilde{V}_j)  \ne 0$. To do so, we show that
		 %for a.e. $u \in \RR$, 
		 the duplicate terms only occur within the set $\cS_j^+\backslash \cC_j^1 \backslash \cA$:
		\begin{claim}\label{claim: dup term thm 2.5}
			For a.e. $u \in \RR$, we have the following hold
			\begin{equation}\label{eq: dup term A}
				\tilde{V}_i \not\equiv \tilde{V}_{i'}, \quad \text{for any $i,i' \in \cA$ and $i \ne i'$};
			\end{equation}
			\begin{equation}\label{eq: dup term cross}
				\tilde{V}_i \not\equiv \tilde{V}_{i'}, \quad \text{for any $i \in \cA$ and $i' \in \cS_j^+\backslash \cC_j^1 \backslash \cA$}.
			\end{equation}
		\end{claim}
		\begin{proof}
			See Appendix~\ref{sec: proof claim dup term thm 2.5}.
		\end{proof}
		By Claim~\ref{claim: dup term thm 2.5}, we have that for a.e. $u \in \RR$, the identity $\tilde{V}_i \equiv \tilde{V}_{i'}$ for $i \ne i'$ can only occur to $i,i' \in \cS_j^+ \backslash \cC_j^1 \backslash \cA$. 
		% Define 
		% $$
		%  \tilde{V}_{i}  = (1 - \sigma_{j}^{-2} \sigma_{ij})^{-1} \left( V_{i}  - u \cdot \II\{i \in \cB \} +\mu_{i}  + \sigma_{j}^{-2} \sigma_{ij} (u - \mu_{j}) \right) . 
		% $$
	We denote by \(\tilde\cB_j \subseteq \cS_j^+ \backslash \cC_j^1 \backslash \cA\) the subset with all duplicates of \(\tilde{V}_{i}\)'s removed, i.e., \(\tilde\cB_j\) is any subset of \(\cS_j^+ \backslash \cC_j^1 \backslash \cA\) such that for any \(i, i' \in \tilde\cB_j\) and \(i \ne i'\), we have \(\tilde{V}_{i} \not\equiv \tilde{V}_{i'}\), and if \(\cS_j^+ \backslash \cC_j^1 \backslash \cA \backslash \tilde\cB_j \ne \emptyset\), then for any \(k \in \cS_j^+ \backslash \cC_j^1 \backslash \cA \backslash \tilde\cB_j\), there exists a \(k' \in \tilde\cB_j\) such that \(\tilde{V}_{k} \equiv \tilde{V}_{k'}\). 
		
		Then we have that 
		\begin{align*}
			H_{u,j}(x) &= \PP\left(\begin{array}{ll} 
				\tilde{V}_{i} \le x  , &  \, i \in \cS_j^+ \backslash\cC_j \\
				\tilde{U}_{i} \le 0 ,&  \, i \in \cS_j^0 
			\end{array} \right) 
			%             \PP\left\{\begin{array}{ll}
				% (1 - \sigma_{j}^{-2} \sigma_{ij})^{-1} \left( V_{i} +\mu_{i}  + \sigma_{j}^{-2} \sigma_{ij} (u - \mu_{j}) \right) \le x ,  &  \, i \in \cA\\
				% (1 - \sigma_{j}^{-2} \sigma_{ij})^{-1} \left( V_{i}   +\mu_{i}  - \sigma_{j}^{-2} \sigma_{ij} \mu_{j} \right) - u \le x ,  &  \, i \in \tilde\cB\\
				% % (1 - \sigma_{S'}^{-2} \sigma_{S' \tilde{S}})^{-1} \left(  \mu_{i'} - u \cdot \II\{\tilde{S} \in \cB \}  + \sigma_{S'}^{-2} \sigma_{S' \tilde{S}} (u - \mu_{j}) \right) \le x   &  ,\, \tilde{S} \in \cC_j\\
				% V_{i} \le u \cdot \II\{i \in \cB \} - \mu_{i} - \sigma_{j}^{-2} \sigma_{ij} ( u - \mu_{j}) ,  &  \, i \in \cS_j^0 
				% \end{array} \right\}
			= \PP\left(\begin{array}{ll} 
				\tilde{V}_{i} \le x  , &  \, i \in \cA \cup \tilde\cB_j\\
				\tilde{U}_{i} \le 0 ,&  \, i \in \cS_j^0 
			\end{array} \right) \\
			& = \PP\left(\max_{i \in \cA \cup \tilde\cB_j} \tilde{V}_{i}  \le x, \quad \max_{i \in \cS_j^0} \tilde{U}_{i} \le 0\right),
		\end{align*}
		% where $\tilde{U}_{i} = V_{i} - u \cdot \II\{i \in \cB \} + \mu_{i} + \sigma_{j}^{-2} \sigma_{i j} ( u - \mu_{j}) $. 
		% Then 
		% % applying \eqref{eq: max joint part x} of Lemma~\ref{lm: joint max dens} by taking $y = 0$, 
		% we have that 
		and 
		\begin{align*}
			&  d H_{u,j}(x) / d x  \overset{\rm (d)}{=}  \sum_{i \in \cA \cup \tilde\cB_j} \Bigg \{ \PP\left(\begin{array}{ll} 
				\tilde{V}_{i'} \le x ,  &  \, i' \in \cA \cup \tilde\cB_j\\
				\tilde{U}_{i'} \le 0, &  \, i' \in \cS_j^0 
			\end{array} \Bigg| \tilde{V}_i = x \right) f_{\tilde{V}_i} (x)\Bigg\} \\
			&\quad \overset{\rm (e)}{=}  \sum_{i \in \cA \cup \tilde\cB_j} \Bigg \{ \PP\left(\begin{array}{ll} 
				\tilde{V}_{i'} \le x ,  &  \, i' \in \cS_j^+ \backslash\cC_j\\
				\tilde{U}_{i'} \le 0 , &  \, i' \in \cS_j^0 
			\end{array} \Bigg| \tilde{V}_i = x \right) f_{\tilde{V}_i} (x)\Bigg\} \\
			&\quad   =  \sum_{i \in \cA \cup \tilde\cB_j} \Bigg \{ \PP\!\left(\begin{array}{ll} 
				\tilde{V}_{i'} \le x  , & \, i' \in \cS_j^+ \backslash\cC_j\\
				\tilde{U}_{i'}  \le 0,  &  \, i' \in \cS_j^0 \\
				x_j(u) \le x & 
			\end{array} \Bigg| \tilde{V}_i = x \right) f_{\tilde{V}_i} (x)\Bigg\}  \\
			& \quad \overset{\rm (f)}{=}  \sum_{i \in \cA \cup \tilde\cB_j} \Bigg \{ \PP \Big(\max_{i' \in \cA, j' \in \cB } X_{i'} \vee (X_{j'} - u) \le x \Big| X_i  = x +  u \cdot \II\{i \in \cB \} , X_j = x + u \Big) \\
			& \qquad \times (1 - \sigma_{j}^{-2} \sigma_{ij}) \phi_{i|j} (x +  u \cdot \II\{i \in \cB \} | x+u) \Bigg\}\\
			& \quad \overset{\rm (g)}{\ge} \sum_{i \in \cA} \!\Bigg \{\! \PP \Big(\max_{i' \in \cA, j' \in \cB } X_{i'}\! \vee \!(X_{j'} - u) \le x \Big| X_i  = x , X_j = x + u \Big)  (1 - \sigma_{j}^{-2} \sigma_{ij}) \phi_{i|j} (x | x+u)\!\! \Bigg\} \\
			& \quad \ge  \Big(1- \max_{i \in \cA} \sigma_{j}^{-2} \sigma_{ij}\Big)  g_{u, j} (x) ,
		\end{align*}
		where $f_{\tilde{V}_i} (\cdot)$ denotes the density function of ${\tilde{V}_i} $; equality (d) follows from  \eqref{eq: max joint part x} of Lemma~\ref{lm: joint max dens} by taking $y = 0$, equality (e) follows by the fact that adding back the constraints of the duplicate terms does not change the probability, equality (f) follows from the fact that for any $i \in \cA \cup \tilde\cB_j$, $f_{\tilde{V}_i} (x) = (1 - \sigma_{j}^{-2} \sigma_{ij}) \phi_{i|j} (x +  u \cdot \II\{i \in \cB \} | x+u) $ and the fact that $\tilde{U}_{i}$'s and $\tilde{V}_{i}$'s are independent of $X_j$  by construction, so that we have the equivalence:
		\begin{align*}
			&\PP\!\left(\begin{array}{ll} 
				\tilde{V}_{i'} \le x  , & \, i' \in \cS_j^+ \backslash\cC_j\\
				\tilde{U}_{i'}  \le 0,  &  \, i' \in \cS_j^0 \\
				x_j(u) \le x & 
			\end{array} \Bigg| \tilde{V}_i = x \right) = \PP\!\left(\begin{array}{ll} 
			\tilde{V}_{i'} \le x  , & \, i' \in \cS_j^+ \backslash\cC_j\\
			\tilde{U}_{i'}  \le 0,  &  \, i' \in \cS_j^0 \\
			x_j(u) \le x & 
		\end{array} \Bigg| \tilde{V}_i = x, \,  X_j = x + u \right)\\
	& \quad =  \PP \Big(\max_{i' \in \cA, j' \in \cB } X_{i'} \vee (X_{j'} - u) \le x \Big| X_i  = x +  u \cdot \II\{i \in \cB \} , X_j = x + u \Big),
		\end{align*}
and inequality (g) follows by noting that $1 - \sigma_{j}^{-2} \sigma_{ij} \ge 0$ for all $i \in \cA \cup \tilde\cB_j$ due to condition (A). 
		Hence \eqref{eq: claim g < H} in Lemma~\ref{claim: 251} also follows, which completes the proof.
		% Claim~2 follows by taking $x_{j}(u) = x_{j}^0(u)$ when $\cC_j \ne \emptyset$, and we will take {\red $x_{j}(u) = -\infty$ } when $\cC_j = \emptyset$ for notational convenience. 
	\end{proof}

	\subsection{Proof of Corollary~\ref{col: single max anti con}}\label{sec: proof col single max anti con}
	\begin{proof}[\unskip\nopunct]
		First observe that for any given $t \in \RR$, we have 
		\begin{equation}\label{eq: single max anti con ineq}
			\PP(|\max_{j \in [p]} X_j - t| \le \varepsilon ) \le   \PP(|\max_{j \in [p]} (X_j - t)/\sigma_j| \le \varepsilon/\underline\sigma ),
		\end{equation}
		where $\underline\sigma = \min_{j \in [p]} \sigma_j$ and the inequality follows from the fact that $|\max_{j \in [p]} \underline\sigma(X_j - t)/\sigma_j| \le |\max_{j \in [p]} X_j - t| $. Then we will employ Theorem~\ref{thm: anti con} to establish the single Gaussian maximum inequality by augmenting the scaled Gaussian random vector $(\tilde{X}_j^t)_{j=1}^p : = \big((X_j - t)/ \sigma_j\big)_{j=1}^p$ with i.i.d. standard normal random variables.  Specifically, define the multivariate Gaussian random vector $(\tilde{X}_1^t, \ldots, \tilde{X}_p^t, \tilde{X}_{p+1}, \ldots, \tilde{X}_{p+n})$, where $\tilde{X}_j \overset{\rm i.i.d.}{\sim} \cN(0,1)$ for $j \in [p+n] \backslash [p]$ and $\tilde{X}_i^t \perp\!\!\!\!\perp \tilde{X}_j$ for any $i \in [p]$ and $j \in [p+n] \backslash [p]$. 
		% \eqref{eq: anti con max diff} also recovers the anti-concentration results of a single Gaussian maximum with $\sigma_{ij} \le \sigma_i^2$ for any $i,j$. More specifically, 
		Consider the partition $\cB = [p]$ and $\cA = [p+n]\backslash[p]$, and define $M_{\cA} = \max_{i \in \cA} \tilde{X}_i$ and $M_{\cB} = \max_{j \in \cB} \tilde{X}_j^t$. By Theorem~\ref{thm: anti con}, for all $\varepsilon > 0$, $\delta > 0$ and $t' \in \RR$, we have that 
		\begin{align}
			&\PP(|M_{\cB} - t'| \le \varepsilon )\notag \\
			&\ \le \PP\bigg(\!\Big\{\!|M_{\cB} - M_{\cA} - t' + \EE M_{\cA} |\! \le \varepsilon + \frac{C}{\sqrt{\log n}} \Big\} \!\cap\! \Big\{|M_{\cA} - \EE M_{\cA} |\!\le\! \frac{C}{\sqrt{\log n}}\Big\}\!\bigg) \notag\\
			&\quad +\! \PP \Big(\!|M_{\cA} - \EE M_{\cA} | > \!\frac{C}{\sqrt{\log n}}\Big)\notag\\
			&\ \le \cL\Big(M_{\cB} - M_{\cA}, \varepsilon + \frac{C}{\sqrt{\log n}}\Big) +  C^{-2} \log n \Var(M_{\cA})\notag\\
			& \ \le (2 + \delta) \EE \Big[ {\max_{j \in [p] } \big|\tilde{X}_j^t - \EE(\tilde{X}_j^t) \big| } \Big] {\varepsilon} = (2 + \delta) \EE \Big[ {\max_{j \in [p] } |{X}_j - \mu_j | /\sigma_j} \Big] {\varepsilon},\label{eq: single max anti con uni}
		\end{align}
		where the first inequality follows from triangle inequality, the second inequality follows from the definition of the L\'evy concentration function and Markov's inequality, and
		% $c$ in the last inequality is an arbitrarily small constant, and 
		the last inequality holds by choosing $C > 0$ and $n$ properly large together with the fact that $\log n \Var(M_{\cA}) = O(1)$ as $M_{\cA}$ is the maximum of i.i.d. $\cN(0,1)$ \citep{chatterjee2014superconcentration}. Then taking the infimum over $\delta > 0$, \eqref{eq: single max anti con uni} yields 
		\begin{equation}\label{eq: single max eq var anti con}
			\sup_{t' \in \RR} \PP(|\max_{j \in [p]} \tilde{X}_j^t - t'| \le \varepsilon ) \le 2  \EE \big[ {\max_{j \in [p] } |X_j - \mu_j | /\sigma_j} \big] {\varepsilon}.
		\end{equation}
		Combining \eqref{eq: single max anti con ineq} and \eqref{eq: single max eq var anti con}, we have 
		\begin{align*}
			\sup_{t \in \RR} \PP(|\max_{j \in [p]} X_j - t| \le \varepsilon ) \le    \sup_{t \in \RR} \PP(|\max_{j \in [p]} \tilde{X}_j^t | \le \varepsilon/\underline\sigma )  \le 2  \EE \big[ {\max_{j \in [p] } |X_j - \mu_j | /\sigma_j} \big] {\varepsilon}/\underline\sigma,
		\end{align*}
		which completes the proof of \eqref{eq: single max anti con}.
	\end{proof}
	\subsection{Proof of Lemma~\ref{lm: gauss comp rate}}\label{sec: proof lm gauss comp rate}
	\begin{proof}[\unskip\nopunct]
		The proof is a modification of
		the proof of Theorem~3.2 in \cite{cck2022improvedbootstrap}. Specifically, for any $v \in \RR^p$ and $s \in \RR$, we replace the smoothing function $m^y(x)$ in equation (65) of \cite{cck2022improvedbootstrap} with $\varphi^v_s(x)$ defined in \eqref{eq: smooth fun} and setting $\beta = \delta^{-1} \log p$. Next, we replace the function $h^y(x;t)$ in equation (72) of \cite{cck2022improvedbootstrap} as $h^v_{s,\delta}(x, t)$ in \eqref{eq: smooth indicate fun}. Then, we define the term
		$$
		\cI^v_s = \varphi^v_s(V) - \varphi^v_s(Z),
		$$
		where without loss of generality, we assume that $V$ and $Z$ are independent. Following the  arguments as in Step~2 of the proof for Theorem~3.2 in \cite{cck2022improvedbootstrap}, we let 
		$$
		\Psi(t) = \EE [\varphi_s^v( \sqrt{t} V + \sqrt{1-t}Z)], \quad \text{for all } t \in [0,1].
		$$
		By the multivariate Stein identity \cite{stein1981estimation}, we have that 
		\begin{align*}
			\Psi'(t) &= \frac{1}{2}\sum_{j = 1}^p \EE \left[\partial_j\varphi_s^v( \sqrt{t} V + \sqrt{1-t}Z) \left(\frac{1}{\sqrt{t}}V_j - \frac{1}{\sqrt{1-t}}Z_j\right)\right]\\
			& = \frac{1}{2}\sum_{j,k = 1}^p \EE \left[\partial_j\partial_k\varphi_s^v( \sqrt{t} V + \sqrt{1-t}Z) \left(\tau_{jk}(V) - \Sigma_{jk}\right)\right].
		\end{align*}
		%Recall we have established the following equation for any $x \in \RR^p$ in \eqref{eq: deriv eq}:   
		% Then it suffices for us to show that for any $x \in \RR^p$, 
		%    $$
		%    \partial_j\partial_k\varphi_s^v(x) =  \partial_j\partial_k\varphi_s^v(x)  h^v_{s,\delta}(x, 0).
		%   $$
		% Indeed when $h^v_{s,\delta}(x, 0) = 0$, by the property of the function $\zeta_{\beta}^v$ in \eqref{eq: property zeta}, we have that either $\zeta_{\beta}^v(x) \le s$ or $\zeta_{\beta}^v(x) > s + \delta$, which by the property of the function $\gamma_s$ indicates that $ \partial_j\partial_k\varphi_s^v(x) = 0$. Hence the claim follows. 
		Next,    following the same arguments as in the proof of Theorem~3.2 in \cite{cck2022improvedbootstrap} together with \eqref{eq: deriv eq} in Appendix~\ref{sec: proof claim recursive rate I}, we have that  
		\begin{equation}\label{eq: sec der eq}
				\partial_j\partial_k\varphi_s^v(\sqrt{t} V + \sqrt{1-t}Z) =  \partial_j\partial_k\varphi_s^v(\sqrt{t} V + \sqrt{1-t}Z)  h^{v'}_{s',\delta'}(Z, 0),
		\end{equation}
		where
		$$
		v' = \frac{1}{\sqrt{1-t}}v + \sqrt{\frac{t}{1-t}} V, \quad s' = \frac{1}{\sqrt{1-t}}s, \quad \delta' = \frac{1}{\sqrt{1-t}} \delta.
		$$
		Then, since $Z$ and $V$ are independent, and $Z$ satisfies Condition~\ref{cond: anti-con cov},  by Theorem~\ref{col: anti con equal var no cor 1} and Theorem~\ref{thm: anti con}, we have that 
		$$
		\EE [h^{v'}_{s',\delta'}(Z, 0)|V = x] \le \frac{K_0 \EE[\max_{\ell \in \cS_{\rm p}} |Z_{\ell}|/\sigma_{\ell}]}{C_{\cA,\cB}\sqrt{1-t}} \delta, \quad \forall x \in \RR^p,
		$$
		where $K_0 > 0$ is a universal constant.
		Then  we have that
		\begin{align*}
			|\Psi'(t)| & \le \frac{1}{2}\sum_{j,k = 1}^p \EE \left[h^{v'}_{s',\delta'}(Z, 0)\cdot\left|\partial_j\partial_k\varphi_s^v( \sqrt{t} V + \sqrt{1-t}Z) \right|  \cdot \left| \tau_{jk}(V) - \Sigma_{jk}\right|\right] \\
			& \le \frac{C_0 \delta^{-2} \log p}{2} \EE \left( \EE [h^{v'}_{s',\delta'}(Z, 0)|V ] \max_{j,k \in [p]}|\tau_{jk}(V) - \Sigma_{jk}|\right)\\
			& \le  \frac{K_0 C_0 \delta^{-1} \log p \EE[\max_{\ell \in \cS_{\rm p}} |Z_{\ell}|/\sigma_{\ell}]}{2 C_{\cA,\cB}\sqrt{1-t}} \EE \left(\max_{j,k \in [p]}|\tau_{jk}(V) - \Sigma_{jk}|\right),
		\end{align*}
	where the first inequality follows from \eqref{eq: sec der eq}, and the second inequality follows by applying the derivative bound in \eqref{eq: clt m der sum 1} of  Lemma~\ref{lm: clt third dev bd}.  In turn, we have
		\begin{align*}
			|\EE (\cI_s^v)| &= \left|\int_{0}^1 \Psi(t) dt \right| \le  \frac{K_0 C_0 \delta^{-1} \log p \EE[\max_{\ell \in \cS_{\rm p}} |Z_{\ell}|/\sigma_{\ell}] \Delta}{2 C_{\cA,\cB}} \int_{0}^1 \frac{1}{\sqrt{1-t}} dt\\
			& = K_0 C_0 C_{\cA,\cB}^{-1}\delta^{-1} \log p \EE[\textstyle\max_{\ell \in \cS_{\rm p}} |Z_{\ell}|/\sigma_{\ell}] \Delta.
		\end{align*}
		%Now applying the properties of the functions $\gamma_s$ and $\zeta_{\beta}^v$, {\color{red}EF: What properties? Be specific} 
		Then, for any $v \in \RR^p$ and $s \in \RR$ we have
		\begin{align*}
			&\PP\left(\max_{j \in \cA} (V_j + v_j) - \max_{j' \in \cB} (V_{j'} + v_{j'})  \le s\right)  \le \PP(\zeta_{\beta}^v(V) \le s + \delta) \le \EE (\varphi_{s + \delta}^v (V)) \\
			& \le \EE (\varphi_{s + \delta}^v (Z)) +   |\EE (\cI_s^v)| \le \PP(\zeta_{\beta}^v(Z) \le s + 2\delta) +   |\EE (\cI_s^v)|   \\
			& \le \PP(\max_{j \in \cA} (Z_j + v_j) - \max_{j' \in \cB} (Z_{j'} + v_{j'})  \le s + 3\delta) +   |\EE (\cI_s^v)|   \\
			& \le \PP(\max_{j \in \cA} (Z_j + v_j) - \max_{j' \in \cB} (Z_{j'} + v_{j'})  \le s ) +   |\EE (\cI_s^v)| + K_1 C_{\cA,\cB}^{-1}\EE[\max_{\ell \in \cS_{\rm p}} |Z_{\ell}|/\sigma_{\ell}] \delta  ,
		\end{align*}
		where the last but one inequality is a result of \eqref{eq: property zeta} and the choice of $\beta = \delta^{-1} \log p$, and the last inequality is a result of Theorem~\ref{col: anti con equal var no cor 1} and Theorem~\ref{thm: anti con}, and $K_1 > 0$ is a universal constant. Similarly,
		\begin{align*}
			&\PP\left(\max_{j \in \cA} (V_j + v_j) - \max_{j' \in \cB} (V_{j'} + v_{j'})  \le s\right)  \ge \PP(\zeta_{\beta}^v(V) \le s - \delta) \ge \EE (\varphi_{s - 2\delta}^v (V)) \\
			& \ge \EE (\varphi_{s - 2\delta}^v (Z)) -   |\EE (\cI_s^v)| \ge \PP(\zeta_{\beta}^v(Z) \le s - 2\delta) -   |\EE (\cI_s^v)|   \\
			& \ge \PP(\max_{j \in \cA} (Z_j + v_j) - \max_{j' \in \cB} (Z_{j'} + v_{j'})  \le s - 3\delta) -   |\EE (\cI_s^v)|   \\
			& \ge \PP(\max_{j \in \cA} (Z_j + v_j) - \max_{j' \in \cB} (Z_{j'} + v_{j'})  \le s ) -  |\EE (\cI_s^v)| - K_1 C_{\cA,\cB}^{-1}\EE[\max_{\ell \in \cS_{\rm p}} |Z_{\ell}|/\sigma_{\ell}] \delta  .
		\end{align*}
		Then \eqref{eq: gauss comp rate} holds by taking $\delta = \sqrt{\Delta \log p}$, which completes the proof.
		% and \eqref{eq: gauss comp rate argmax} is a direct result of \eqref{eq: gauss comp rate} due to the equivalance stated in \eqref{eq: equiv argmax}.
	\end{proof}
	\subsection{Proof of Lemma~\ref{prop: modified thm 3.1}}\label{sec: proof prop modified thm 3.1}
	\begin{proof}[\unskip\nopunct]
		The proof of Lemma~\ref{prop: modified thm 3.1} follows the proof of Theorem~3.1 in \cite{cck2022improvedbootstrap} by modifying  Lemma~3.1 in \cite{cck2022improvedbootstrap} using redefined smoothing functions. For ease of presentation, we denote the constants in Conditions~\ref{cond: fourth moment}-\ref{cond: anti-con rate} as \(C_v\), \(C_e\), \(C_b\), and \(C_a\) respectively, and we denote the constant in \eqref{eq: m2 cond} and \eqref{eq: m3 cond} as \(C_m\). Then, for any $\epsilon \in \{0,1\}^n$, we redefine the variable $\varrho_{\epsilon}$ in (21) of \cite{cck2022improvedbootstrap} as 
		$$
		\varrho_{\epsilon} = \sup_{v \in \RR^p, s \in \RR} \left|\PP\big(\cM_{\cA}^v (S_{n,\epsilon}^V)-  \cM_{\cB}^v (S_{n,\epsilon}^V) \le s\big) - \PP\big( \cM_{\cA}^v (S_{n}^Z)-  \cM_{\cB}^v (S_{n}^Z) \le s\big)\right|,
		$$
		where  $S_{n,\epsilon}^V = \frac{1}{\sqrt{n}} \sum_{i=1}^n \big( \epsilon_i V_i + (1-\epsilon_i) Z_i\big)$. We adopt the same sequence of random vectors $\epsilon^0, \ldots, \epsilon^D \in \{0,1\}^n$ as defined in Remark~3.1 of \cite{cck2022improvedbootstrap} and define $I_d = \{i \in [n]: \epsilon^d_i = 1\}$ for $d = 0, \ldots, D$, where $D = {\rm round}(4 \log n) + 1$, and ${\rm round} (\cdot)$ is the rounding operator that rounds the number inside to the nearest integer. For all $i \in [n]$ and $j,k,\ell \in [p]$, we let
		$$
		\begin{gathered}
			\mathcal{E}_{i, j k}^V=\EE\left[V_{i j} V_{i k}\right], \quad \mathcal{E}_{i, j k \ell}^V=\EE\left[V_{i j} V_{i k} V_{i \ell }\right], \\
			\mathcal{E}_{i, j k}^Z=\EE\left[Z_{i j} Z_{i k}\right], \quad \mathcal{E}_{i, j k \ell}^Z=\EE\left[Z_{i j} Z_{i k} Z_{i \ell}\right],
		\end{gathered}
		$$
		and for $n \ge 1$ and $d = 0, \ldots, D$, we consider the event $\cA_d$ defined in Remark~3.1 of \cite{cck2022improvedbootstrap} that
			{\small
		$$
	\begin{aligned}
			\mathcal{A}_d= & \Big\{\max_{j, k \in [p]}\!\Big|\frac{1}{\sqrt{n}}\! \sum_{i=1}^n \epsilon_i^d\!\!\left(\!\mathcal{E}_{i, j k}^V\!-\mathcal{E}_{i, j k}^Z\right)\!\!\Big| \!\leq \mathcal{B}_{n, 1, d}\!\Big\}\!  \bigcap\!\Big\{\max_{j, k, \ell \in [p]}\Big|\frac{1}{\sqrt{n}} \!\sum_{i=1}^n \!\epsilon_i^d\!\!\left(\mathcal{E}_{i, j k \ell}^V-\mathcal{E}_{i, j k \ell}^Z\right)\!\!\Big| \!\leq \mathcal{B}_{n, 2, d}\!\Big\} ,
		\end{aligned}
		$$
	}where $\mathcal{B}_{n, 1, d}, \mathcal{B}_{n, 2, d} > 0$ are some positive rates to be specified later.
		
		Then, we finish the proof in two steps.

		In Step~1, we first establish a recursive inequality of $\rho_{\epsilon^d}$, which is a counterpart of the rate in Lemma~3.1 of \cite{cck2022improvedbootstrap} in our case. Specifically, for any $d = 0, \ldots, D$ and any constant $\delta > 0$ such that 
		\begin{equation}\label{eq: delta scaling}
			C_e B_n\delta^{-1} \log ^2 (pn) \le \sqrt{n},
		\end{equation}
		we show that the following holds under  event $\cA_d$,
		\begin{equation}\label{eq: iter rate rho}
			\begin{aligned}
				\varrho_{\epsilon^d} \lesssim \,\, & r_n \delta + \varepsilon_n +  \frac{B_n^2 \delta^{-4} \log^5 (pn) }{n^2} \\
				&+  \big(r_n \delta + \varepsilon_n +   \EE (\varrho_{\epsilon^{d+1}}|\epsilon^d)\big)  \left( \frac{\cB_{n,1,d} \delta^{-2} \log p}{\sqrt{n}}  + \frac{\cB_{n,2,d} \delta^{-3} \log^2 p}{{n}} + \frac{B_n^2 \delta^{-4} \log^3 (pn)}{n} \right).
			\end{aligned}
		\end{equation}
		Then by the proof of Corollary~3.1 in \cite{cck2022improvedbootstrap}, there exists a constant $K_{\rm iter} > 0$ depending only on $C_v$, $C_e$ and $C_b$ such that by taking 
		\begin{equation}\label{eq: K_iter B_n}
			\cB_{n,1,d+1} \ge \cB_{n,1,d} + K_{\rm iter} B_n \log^{1/2}(pn) \text{ and }  \cB_{n,2,d+1} \ge \cB_{n,2,d} + K_{\rm iter} B_n^2 \log^{3/2}(pn),
		\end{equation}
		we  have a recursive inequality for $\EE (\varrho_{\epsilon^d} \II\{\cA_d\} )$ that  
		\begin{equation}\label{eq: iter rate event}
			\begin{aligned}
				\EE (\varrho_{\epsilon^d} \II\{\cA_d\} ) \lesssim  & \ r_n \delta + \varepsilon_n +  \frac{B_n^2 \delta^{-4} \log^5 (pn) }{n^2} +  \big(r_n \delta + \varepsilon_n +   \EE (\varrho_{\epsilon^{d+1}}  \II\{\cA_{d+1}\} )\big) \\
				& \times \left( \frac{\cB_{n,1,d} \delta^{-2} \log p}{\sqrt{n}}  + \frac{\cB_{n,2,d} \delta^{-3} \log^2 p}{{n}} + \frac{B_n^2 \delta^{-4} \log^3 (pn)}{n} \right).
			\end{aligned}
		\end{equation}  
		
		In Step~2, following \eqref{eq: iter rate event}, we employ an induction argument and follow the proof of Theorem~3.1 in \cite{cck2022improvedbootstrap} to show that \eqref{eq: CLT initial results} holds. 
		
		For the rest of the proof, same as in \cite{cck2022improvedbootstrap}, for ease of notation, we denote  $\PP(\cdot)$ and $\EE(\cdot)$ as $\PP(\cdot | \epsilon^d)$ and $\EE(\cdot |\epsilon^d )$ respectively, and we let $\lesssim$ be that the inequality holds up to a constant depending only on $C_v$, $C_e$, $C_b$, $C_a$ and $C_{\gamma}$.
		
		\subsubsection*{\bf Step~1}
		% To establish the recursive inequality of $\EE(\rho_{\epsilon^d} \II\{\cA_d\})$ for $d = 0, \ldots, D$, by Corollary~3.1 in \cite{cck2022improvedbootstrap}, we only need to show that the rate in Lemma~3.1 of \cite{cck2022improvedbootstrap} holds in our case, and then the recursive inequality follows by properly choosing the recursive rates $\cB_{n, 1,d+1}$ and $\cB_{n, 2,d+1}$ for $d = 0, \ldots, D$.  
		
		% Specifically, for any $d = 0, \ldots, D$ and any constant $\delta > 0$ such that 
		% \begin{equation}\label{eq: delta scaling}
			%     C_e B_n\delta^{-1} \log ^2 (pn) \le \sqrt{n},
			% \end{equation}
		% to establish the recursive inequality for $\EE(\rho_{\epsilon^d} \II\{\cA_d\})$, we only need to show that the following holds under the event $\cA_d$,
		% \begin{equation}\label{eq: iter rate rho}
			%     \begin{aligned}
				%         \varrho_{\epsilon^d} \lesssim \,\, & r_n \delta + \varepsilon_n +  \frac{B_n^2 \delta^{-4} \log^5 (pn) }{n^2} +  \big(r_n \delta + \varepsilon_n +   \EE (\varrho_{\epsilon^{d+1}}|\epsilon^d)\big) \\
				%         & \times \left( \frac{\cB_{n,1,d} \delta^{-2} \log p}{\sqrt{n}}  + \frac{\cB_{n,2,d} \delta^{-3} \log^2 p}{{n}} + \frac{B_n^2 \delta^{-4} \log^3 (pn)}{n} \right),
				%     \end{aligned}
			% \end{equation}
		% up to a constant depending only on $C_v$, $C_e$, $C_b$ and $C_a$.
		
		Similar to the proof of  Lemma~3.1 of \cite{cck2022improvedbootstrap}, to show \eqref{eq: iter rate rho}, we begin by constructing a smoothing approximation function of $\II( \cM_{\cA}^v(\cdot) - \cM_{\cB}^v(\cdot)  \le s)$ for any $v \in \RR^p$ and $s \in \RR$. In particular, given any $s \in \RR$ and $\delta > 0$, we find a five-times continuously differentiable function $\gamma^0: \RR \rightarrow [0,1]$ with derivatives bounded up to the fifth order such that $\gamma^0(t) = 1$ for any $t \le 0$ and $\gamma^0(t) = 0$ for any $t \ge 1$. Then for the given $\delta > 0$, we define $\gamma_s(t) = \gamma_{s, \delta}(t) = \gamma^0(\delta^{-1}(t-s))$. It is not difficult to show that for any $d = 1, \ldots, 5$, there exists an absolute constant $C_{\gamma} \ge 1$ such that $\|\gamma_s^{(d)}\|_{\infty} \le C_{\gamma}\delta^{-d}$, where $\gamma_s^{(d)}$ denotes the $d$-th derivative of $\gamma_s$, and $\II \{t \le s\}\le \gamma_s(t) \le \II\{t \le s + \delta\}$ for any $t \in \RR$. For any vector $v \in \RR^p$, define the function $\zeta_{\beta}^v: \RR^p \rightarrow \RR$ as
		$$
		\zeta_{\beta}^v (x) =  \zeta_{\beta, \cA}^v (x)  = \beta^{-1} \big( \log \big(\textstyle\sum_{j \in \cA} e^{\beta(x_j + v_j)} \big) - \log \big(\textstyle\sum_{j' \in \cB} e^{\beta(x_{j'} + v_{j'})} \big)  \big),
		$$
		% where $\bar\mu = n^{-1/2} \sum_{i=1}^n \mu_i$, 
		and it also holds that
		\begin{equation}\label{eq: property zeta}
			-\beta^{-1} \log p \le \zeta_{\beta}^v (x) - \{\max_{j \in \cA} (x_j + v_j) - \max_{j' \in \cB} (x_{j'} + v_{j'})\} \le \beta^{-1} \log p.
		\end{equation}
		Then, we define the smoothing function as
		\begin{equation}\label{eq: smooth fun}
			\varphi^v_s(x) = \varphi_{s,\beta}^v(x) = \gamma_s (\zeta_{\beta}^v (x)) .
		\end{equation}
		The following lemma bounds the partial derivatives of $\varphi^v_s$:
		\begin{lemma}\label{lm: clt third dev bd}
			Letting $\beta = \delta^{-1} \log p$, for any $j,k, \ell , r, h \in [p]$, there exists functions $U_{jk} (x): \RR^p \rightarrow \RR$, $U_{jk \ell} (x): \RR^p \rightarrow \RR$, $U_{jk \ell r} (x): \RR^p \rightarrow \RR$ and $U_{jk \ell r h} (x): \RR^p \rightarrow \RR$ and universal constants $C_0, C_1, c_1 > 0$ such that for any given $v \in \RR^p$ and $s \in \RR$, and for any $x \in \RR^p$, we have that 
			\begin{equation}\label{eq: clt m der bd 1}
				\begin{aligned}
					|\partial_j \partial_k \varphi^v_s(x)| \le U_{jk} (x), & \quad |\partial_j \partial_k \partial_{\ell} \varphi^v_s(x)| \le U_{jk \ell} (x), \\
					|\partial_j \partial_k \partial_{\ell} \partial_r \varphi^v_s(x)| \le U_{jk \ell r} (x), & \quad  |\partial_j \partial_k \partial_{\ell} \partial_r \partial_h \varphi^v_s(x)| \le U_{jk \ell r h} (x) ,  
				\end{aligned}
			\end{equation}
			% \begin{equation}\label{eq: clt m der bd 2}
				% \end{equation}
			\begin{equation}\label{eq: clt m der sum 1}
				\begin{aligned}
					\sum_{j,k =1}^p U_{jk } (x) \le C_0 \delta^{-2} \log p, &\quad \sum_{j,k, \ell=1}^p U_{jk\ell} (x) \le C_0 \delta^{-3} \log^{2} p, \\
					\sum_{j,k, \ell, r =1}^p U_{jk \ell r} (x) \le C_0 \delta^{-4} \log^3 p, &\quad \sum_{j,k, \ell, r, h=1}^p U_{jk\ell r h } (x) \le C_0 \delta^{-5} \log^{4} p.
				\end{aligned}
			\end{equation}
			%   \begin{equation}\label{eq: clt m der sum 2}    
				% \end{equation}
			Also, for any $x, y \in \RR^p$ such that $\max_{j \in [p]} |y_j| \le \beta^{-1}$, we have
			\begin{equation}\label{eq: clt m der order 2}
				\begin{aligned}
					c_1 U_{jk \ell r} (x) \le U_{jk \ell r} &(x + y) \le C_1 U_{jk \ell r} (x),  \\
					c_1 U_{jk \ell r h} (x) \le U_{jk \ell r h} &(x + y) \le C_1 U_{jk \ell r h} (x).
				\end{aligned}
			\end{equation}
			%  \begin{equation}\label{eq: clt m der order 2}
				% \end{equation}
		\end{lemma}
		\begin{proof}
			See Appendix~\ref{sec: proof lm clt third dev bd}.
		\end{proof}
		%The proof of Lemma~\ref{lm: clt third dev bd} is deferred to Section~\ref{sec: proof lm clt third dev bd}.
		{ Then, choosing $\beta = \delta^{-1}\log p$ and $\delta$ that satisfies \eqref{eq: delta scaling}, we replace the function $m^y(\cdot)$  in~\cite{cck2022improvedbootstrap} by $\varphi^v_s (\cdot)$.If \(\delta > 1\), then \(r_n \delta > 1\) and \eqref{eq: iter rate rho} holds trivially. Hence, we assume \(\delta \le 1\) without loss of generality for the rest of the proof. We redefine the function in equation~(72) in the proof of Lemma~3.1 in \cite{cck2022improvedbootstrap} as 
			\begin{equation}\label{eq: smooth indicate fun}
				h^v_{s, \delta}(x, t) \!:= \II\{-\delta - t < \max_{j \in \cA} (x_j \!+\! v_j)\!-\! \max_{j' \in \cB} (x_{j'} \!+\! v_{j'}) \!-\!s \le 2\delta + t \}, \, x \!\in\! \RR^p, \, t \ge 0,
			\end{equation}
			and we adopt the same notation that $W = S_{n, \epsilon^{d+1}}^V$ as defined in (73) of \cite{cck2022improvedbootstrap}. We redefine the error term in (71) of \cite{cck2022improvedbootstrap} as 
			\begin{equation}\label{eq: def I v s}
				\cI^v_s = \varphi_s^v (S^V_{n,\epsilon^d}) - \varphi_s^v (S^Z_{n}).
			\end{equation}
			Then, to show \eqref{eq: iter rate rho}, we first establish in the following claim the counterpart of (74) in \cite{cck2022improvedbootstrap}.
			% , where same as in \cite{cck2022improvedbootstrap}, we use notations $\PP(\cdot)$ and $\EE(\cdot)$ to represent $\PP(\cdot | \epsilon^d)$ and $\EE(\cdot |\epsilon^d )$ for ease of notation, and we denote by $\lesssim$ that the inequality holds up to a constant depending only on $C_v$, $C_e$, $C_b$, $C_a$ and $C_{\gamma}$.
			\begin{claim}\label{claim: recursive approx rate I}
				For $d = 0, \ldots, D$ and the term $\cI_s^v$ defined in \eqref{eq: def I v s}, when $\cA_d$ holds, we have that
				\begin{equation}\label{eq: iter rate approx}
					\begin{aligned}
						\sup_{v \in \RR^p, s \in \RR} |\EE (\cI_s^v)| \lesssim &\   \frac{B_n^2 \delta^{-4} \log^5 (pn) }{n^2} +  \big(r_n \delta + \varepsilon_n +   \EE (\varrho_{\epsilon^{d+1}})\big) \\
						& \times \left( \frac{\cB_{n,1,d} \delta^{-2} \log p}{\sqrt{n}}  + \frac{\cB_{n,2,d} \delta^{-3} \log^2 p}{{n}} + \frac{B_n^2 \delta^{-4} \log^3 (pn)}{n} \right).
					\end{aligned}
				\end{equation}
			\end{claim}
			\begin{proof}
				See Appendix~\ref{sec: proof claim recursive rate I}.
			\end{proof}
			Now applying the properties of the functions $\gamma_s$ and $\zeta_{\beta}^v$, for any $v \in \RR^p$ and $s \in \RR$, we have
			\begin{align*}
				&\PP\left(\cM_{\cA}^v (S_{n,\epsilon^d}^V) - \cM_{\cB}^v (S_{n,\epsilon^d}^V) \le s\right)  \le \PP(\zeta_{\beta}^v(S_{n,\epsilon^d}^V) \le s + \delta) \le \EE \big(\varphi_{s + \delta}^v (S_{n,\epsilon^d}^V)\big) \\
				&\quad \le \EE (\varphi_{s + \delta}^v (S_n^Z)) +   |\EE (\cI_s^v)| \le \PP(\zeta_{\beta}^v(S_n^Z) \le s + 2\delta) +   |\EE (\cI_s^v)|   \\
				&\quad \le \PP(\cM_{\cA}^v (S_n^Z) - \cM_{\cB}^v (S_n^Z)  \le s + 3\delta) +   |\EE (\cI_s^v)|   \\
				&\quad \le \PP(\cM_{\cA}^v (S_n^Z) - \cM_{\cB}^v (S_n^Z)  \le s ) +   |\EE (\cI_s^v)| + C_a ( 3 r_n \delta  + \varepsilon_n ),
			\end{align*}
			where the last inequality holds by Condition~\ref{cond: anti-con rate}. Similarly, we have
			\begin{align*}
				&\PP\left(\cM_{\cA}^v (S_{n,\epsilon^d}^V) - \cM_{\cB}^v (S_{n,\epsilon^d}^V) \le s\right)  \ge \PP(\zeta_{\beta}^v(S_{n,\epsilon^d}^V) \le s - \delta) \ge \EE \big(\varphi_{s - 2\delta}^v (S_{n,\epsilon^d}^V)\big) \\
				&\quad \ge \EE (\varphi_{s - 2\delta}^v (S_n^Z)) -   |\EE (\cI_s^v)| \ge \PP(\zeta_{\beta}^v(S_n^Z) \le s - 2\delta) -   |\EE (\cI_s^v)|   \\
				&\quad \ge \PP(\cM_{\cA}^v (S_n^Z) - \cM_{\cB}^v (S_n^Z)  \le s - 3\delta) -   |\EE (\cI_s^v)|   \\
				&\quad \ge \PP(\cM_{\cA}^v (S_n^Z) - \cM_{\cB}^v (S_n^Z)  \le s ) -   |\EE (\cI_s^v)| - C_a (3 r_n \delta  + \varepsilon_n).
			\end{align*}
			Taking the supremum over $v \in \RR^p$ and $s \in \RR$, we have that 
			$$
			\varrho_{\epsilon^d} \lesssim  \sup_{v \in \RR^p, s \in \RR}|\EE (\cI_s^v)|  +   r_n \delta  + \varepsilon_n,
			$$
			which along with \eqref{eq: iter rate approx} gives \eqref{eq: iter rate rho}, and \eqref{eq: iter rate event} follows consequently by Corollary~3.1 of \cite{cck2022improvedbootstrap}.
			% Then by the proof of Corollary~3.1 in \cite{cck2022improvedbootstrap}, there exists a constant $K_{\rm iter} > 0$ depending only on $C_v$, $C_e$ and $C_b$ such that by taking $\cB_{n,1,d+1} \ge \cB_{n,1,d} + K_{\rm iter} B_n \log^{1/2}(pn)$ and $\cB_{n,2,d+1} \ge \cB_{n,2,d} + K_{\cB} B_n^2 \log^{3/2}(pn)$, we further have that 
			%   \begin{equation}\label{eq: iter rate event}
				%        \begin{aligned}
					%         \EE (\varrho_{\epsilon^d} \II\{\cA_d\} ) \lesssim  & \ r_n \delta + \varepsilon_n +  \frac{B_n^2 \delta^{-4} \log^5 (pn) }{n^2} +  \big(r_n \delta + \varepsilon_n +   \EE (\varrho_{\epsilon^{d+1}}  \II\{\cA_{d+1}\} )\big) \\
					%         & \times \left( \frac{\cB_{n,1,d} \delta^{-2} \log p}{\sqrt{n}}  + \frac{\cB_{n,2,d} \delta^{-3} \log^2 p}{{n}} + \frac{B_n^2 \delta^{-4} \log^3 (pn)}{n} \right),
					%     \end{aligned}
				%   \end{equation}
			%   and the recursive inequality is established.
			\subsubsection*{\bf Step~2}
			Following the proof of Theorem~3.1 in \cite{cck2022improvedbootstrap}, we employ \eqref{eq: iter rate event} to show \eqref{eq: CLT initial results} through an inductive arguments. Let $\cB_{n,1,d}$ and $\cB_{n,2,d}$ be the same as (30) of \cite{cck2022improvedbootstrap} that
			$$
			\cB_{n,1,d} = C_1 (d+1) B_n \log^{1/2} (pn) \quad \text{and} \quad  \cB_{n,2,d} = C_1 (d+1) B_n^2 \log^{3/2} (pn),
			$$
			where $C_1 = C_m + K_{\rm iter}$, and recall that $K_{\rm iter}$ is the constant in \eqref{eq: K_iter B_n}.
			For each $d = 0, 1, \ldots, D$, we redefine 
			$$
			f_d = \inf \Big\{t \ge 1: \EE(\varrho_{\epsilon^d} \II(\cA_d)) \le t \Big(r_n \Big(\frac{B_n^2 \log^3(pn)}{n}\Big)^{1/4}+ \varepsilon_n\Big)\Big\},
			$$
			and  take 
			$$
			\delta = \delta_d = \frac{B_n^{1/2} \log^{3/4}(pn) ((d+1)f_{d+1})^{1/3}}{n^{1/4}},
			$$
			which satisfies \eqref{eq: delta scaling} under the condition $C_e B_n^2 \log^5(pn) \le n$. Then, we  have
			\begin{align*}
				&\frac{B_n^2 \delta^{-4} \log^5 (pn) }{n^2} \le \frac{\log^2(pn)}{n} \le r_n \left(\frac{B_n^2 \log^3(pn)}{n}\right)^{1/4} ,\\
				& \frac{\cB_{n,1,d} \delta^{-2} \log p}{\sqrt{n}}  \le \frac{C_1 (d+1)^{1/3}}{f_{d+1}^{2/3}}, \quad \\
				& \frac{\cB_{n,2,d} \delta^{-3} \log^2 p}{{n}} \bigvee \frac{B_n^2 \delta^{-4} \log^3 (pn)}{n} \le \frac{C_1 \vee 1}{f_{d+1}}.
			\end{align*}
			Combining the above rates with \eqref{eq: iter rate event} and following identical arguments as in the proof of Theorem~3.1 in \cite{cck2022improvedbootstrap}, we have that \eqref{eq: CLT initial results} holds as desired. }
	\end{proof}

	\section{Lemmas in Section~\ref{sec: proof col anti con eq var no cor 1}}\label{sec: proof tech lms}
	% We provide the proofs for the technical lemmas of the main theorems in this section. 
	% In this section we provide the supporting lemmas that facilitate the proofs in Section~\ref{sec: proof col anti con eq var no cor 1}. 
 Lemma~\ref{lm: anti con equal var no cor 1 dens} characterizes the joint density of \((M_{\cA}, M_{\cB})\) when there exist perfectly negatively correlated pairs between \(\cA\) and \(\cB\) such that the joint distribution is not absolutely continuous everywhere on $\RR^2$. Lemma~\ref{lm: anti con equal var no cor 1} provides the anti-concentration rate for \(M_{\cB} - M_{\cA}\) based on Lemma~\ref{lm: anti con equal var no cor 1 dens}. Lemma~\ref{lm: NA NB equiv} establishes the one-to-one relationship between the sets \(\cN_{\cA}\) and \(\cN_{\cB}\) defined in \eqref{eq: def Na Nb}, ensuring that Lemma~\ref{lm: anti con equal var no cor 1} is applicable. The proofs of Lemmas~\ref{lm: anti con equal var no cor 1 dens}-\ref{lm: NA NB equiv} are in Sections~\ref{sec: proof lm anti con equal var no cor 1 dens}, \ref{sec: proof lm anti con equal var no cor 1} and \ref{sec: proof lm NA NB equiv} respectively. 
	\begin{lemma}\label{lm: anti con equal var no cor 1 dens}
		Suppose that $(Z_1, Z_2, \ldots, Z_d)^{\top}\in\RR^d$ is a  Gaussian random vector with mean $(\mu_1,...,\mu_d)^\top$, homogeneous component-wise variance $\sigma^2 > 0$, and $ |\Corr (Z_i, Z_j)| < 1$ for all $i \ne j$. Define $M_{\cA} = \max_{i\in[d]} Z_i$ and $M_{\cB} = \max_{i\in[d]} \{- Z_i + c_i\}$, where $c_i$'s are constants, and let $\bar{c} = \max_{i\in[d]} c_i$ and $ \cI^* = \argmax_{i \in [d]} c_i$. Then, we have 
		\begin{enumerate}
			\item Let $\cL =  \{(x,y) \in \RR^2: y = - x + \bar{c}\}$. The joint distribution of $(M_{\cA}, M_{\cB})$ is absolutely continuous on $\RR^2 \backslash \cL$ with respect to the Lebesgue measure on $\RR^2$, and a version of its density is
			\begin{equation}\label{eq: dens joint max R2}
				\begin{aligned}
					\tilde{f}(x,  y) = \sum_{(i,j) \in [d]^2: i \ne j} & \PP\!\left(\!\!\begin{array}{c}
						\max_{i' \in [d]\backslash\{i\}} Z_{i'} \le x,    \\
						\max_{j' \in [d] \backslash \{j\}} -Z_{j'}+c_{j'} \le y 
					\end{array}  \! \bigg| Z_i = x, Z_j = - y + c_j  \!\right)\\
					& \times \phi_{i,j} (x, -y + c_j) ,
				\end{aligned}
			\end{equation}
			where $\phi_{i,j}(\cdot,\cdot)$ is the joint density for $(Z_i, Z_j)$. 
			
			\item  On the line  $\cL$ defined above, consider the distribution of $M_\cA$ when $(M_{\cA}, M_{\cB}) = (M_{\cA}, -M_{\cA}+ \bar{c} )$. We have that the distribution of $M_\cA$ on $\cL$ is absolutely continuous with respect to the Lebesgue measure on $\RR$, and a version of its density is  
			\begin{equation}\label{eq: dens joint max R1}
				\begin{aligned}
					& \tilde{f}(x) = \frac{d}{dx} \PP(M_{\cA} \le x, M_{\cB} = -M_{\cA}+ \bar{c} )\\
					&\ =\left\{
					\begin{array}{ll}
						\PP(x + c_i - \bar{c} \le Z_i \le x, \forall i \ne i^* | Z_{i^*} = x) \phi_{i^*} (x), & \,\text{if $|\cI^*|=1$ and $\cI^* = \{i^*\}$}, \\
						0, & \,\text{if $|\cI^*|>1$} .
					\end{array}   \right.
				\end{aligned}
			\end{equation}
			where $\phi_i (\cdot)$ is the marginal density for $Z_i$.
		\end{enumerate}
	\end{lemma}
	\begin{lemma}\label{lm: anti con equal var no cor 1}
		Let $(Z_1, Z_2, \ldots, Z_d)^{\top}\in\RR^d$ be a  Gaussian random vector with mean $(\mu_1,...,\mu_d)^{\top}$ and homogeneous component-wise variance $\sigma^2 > 0$. Also, assume that $ |\Corr (Z_i, Z_j)| < 1$ for any $i \ne j$. Let $M_{\cA} = \max_{i\in[d]} Z_i$ and $M_{\cB} = \max_{i\in[d]}\{ - Z_i + c_i\}$, where $c_i$'s are constants. 
		% Define $\bar\lambda_1 = \min_{\mu \in \RR} \max_{i=1}^d |\mu_i - \mu|$ and $\bar\lambda_2 = \min_{\mu \in \RR} \max_{i=1}^d |-\mu_i + c_i - \mu|$,
		Then for any $\varepsilon >0$, we have that
		% there exists a fixed constant $C > 0$ such that 
		\begin{equation}\label{eq: anti con max diff equal var}
			% \sup_{t \in \RR}\PP\Big(\big| M_{\cB} - M_{\cA} - t\big| \le \varepsilon \Big) 
			\cL(M_{\cB} - M_{\cA}, \varepsilon) \le 3 \EE \left(\max_{i \in [d]} |Z_i - \mu_i| /\sigma \right) \frac{\varepsilon}{(1 +  \underline\rho )\sigma},
		\end{equation} 
		where $ \underline\rho =  \min_{i \ne j}\Corr(Z_i, Z_j)$. 
		% Also at $t = 0$, we have a refined bound
		% \begin{equation}\label{eq: anti con max diff equal var t 0}
			%   \PP\Big(\big| M_{\cB} - M_{\cA} \big| \le \varepsilon \Big) \le C  \left\{\EE \left(\max_i |Z_i - \mu_i|  \right) +  \bar\lambda_1 \wedge \bar\lambda_2  \right\}\frac{\varepsilon}{(1 + \underline\rho )\sigma^2}.
			% \end{equation} 
	\end{lemma}
	\begin{lemma}\label{lm: NA NB equiv}
		Let $(X_1,\ldots,X_p)^{\top}\in\RR^d$ be a Gaussian random vector such that $\Corr(X_i, X_j) < 1$ for any $i \ne j$. Let $\{\cA, \cB\}$ be a nontrivial partition of $[p]$, and let $\cN_{\cA}$ and $\cN_{\cB}$ be defined  in \eqref{eq: def Na Nb}. Then, we have $|\cN_{\cA}| = |\cN_{\cB}| = d$ for some integer $d \le p$, and there exist constants $c_1, \ldots, c_d$ such that 
		\begin{equation}\label{eq: equal NA NB}
			\{X_j\}_{j \in \cN_{\cB}} = \{-X_i + c_i\}_{i \in \cN_{\cA}}.
		\end{equation}
		Moreover, we have $|\Corr(X_i, X_{i'})| < 1$ for any $i,i' \in \cN_{\cA}$ and $i \ne i'$, and  $|\Corr(X_j, X_{j'})| < 1$ for any $j,j' \in \cN_{\cB}$ and $j \ne j'$.
	\end{lemma}
	% We provide the proofs for Lemmas~\ref{lm: anti con equal var no cor 1 dens}-\ref{lm: NA NB equiv} in the following sections.
	\subsection{Proof of Lemma~\ref{lm: anti con equal var no cor 1 dens}}\label{sec: proof lm anti con equal var no cor 1 dens}
	\begin{proof}[\unskip\nopunct]
		For $k \in [d]$, define the line $\cL_k = \{(x,y) \in \RR^2: y = -x + c_k\}$. Then, we prove the results (1) and (2) of  Lemma~\ref{lm: anti con equal var no cor 1 dens} by characterizing the densities on $\RR^2 \backslash (\cup_k \cL_k)$ and $\cup_k \cL_k$ respectively.

		\subsubsection{Distribution on $\RR^2 \backslash (\cup_k \cL_k)$}
		
		\begin{proof}[\unskip\nopunct]\renewcommand{\qedsymbol}{}
			
			First, we show the absolute continuity of the joint distribution on $\RR^2 \backslash (\cup_k \cL_k)$. For
			% any $(x,y) \in \RR^2 \backslash (\cup_k \cL_k)$ and 
			any Borel measurable subset $\cI \subseteq \RR^2 \backslash (\cup_k \cL_k)$, 
			% such that $(x,y) \in \cI$, when $|\cI|$ is small enough, we have $\cI \cap  (\cup_k \cL_k) = \emptyset $. Then 
			we have 
			\begin{align*}
				\PP\left((M_{\cA}, M_{\cB})\in \cI\right) &\le \sum_{(i,j) \in [d]^2: i \ne j} \PP((Z_i, -Z_j + c_j) \in \cI) + \sum_{i=1}^d  \PP((Z_i, -Z_i + c_i) \in \cI) \\
				& = \sum_{(i,j) \in [d]^2: i \ne j} \PP((Z_i, -Z_j + c_j) \in \cI), 
			\end{align*}
			where the equality holds because $\PP((Z_i, -Z_i + c_i) \in \cI) = 0$ for any $i \in [d]$ and $\cI \subseteq \RR^2 \backslash (\cup_k \cL_k)$. Combined with the fact that $(Z_i, Z_j)$ is non-degenerate for any $i \ne j$, the absolute continuity follows.
			
			Next, we show that for a.e. $(x,y) \in \RR^2 \backslash (\cup_k \cL_k)$, we have
			\begin{equation}\label{eq: dens lim}
				\lim_{\varepsilon \downarrow 0}\varepsilon^{-2} \PP(x < M_{\cA} \le x + \varepsilon, y < M_{\cB} \le y+\varepsilon) = \tilde{f}(x, y).
			\end{equation}
			For any $(x,y) \in \RR^2 \backslash (\cup_k \cL_k)$,  we have 
			$$
			\{x < M_{\cA} \le x + \varepsilon, y < M_{\cB} \le y+\varepsilon \} = \cup_{p=1}^{d}\cup_{q=1}^{d} \cE_{\varepsilon, x, y}^{p,q},
			$$
			where the event $\cE_{\varepsilon, x, y}^{p,q}$ is defined as
			$$
			\cE_{\varepsilon, x, y}^{p,q} = \left\{\exists I_1 \subseteq [d], |I_1| =p  \text{ and } I_2 \subseteq [d], |I_2| = q: \begin{array}{ll}
				x < Z_i \le x + \varepsilon,  & \,\, \forall i \in I_1,  \\
				Z_{i'} \le x , &\,\,   \forall i' \notin I_1, \\
				y < - Z_j + c_j \le y + \varepsilon,  & \,\, \forall j \in I_2,  \\
				- Z_{j'} + c_{j'}\le y , &\,\,   \forall j' \notin I_2.
			\end{array}  \right\},
			$$
			and $\cE_{\varepsilon, x, y}^{p,q}$'s are disjoint. Then  \eqref{eq: dens lim} follows by showing
			\begin{equation}\label{eq: claim 1 lemma A2}
				\lim_{\varepsilon \downarrow 0} \varepsilon^{-2}   \PP(\cE_{\varepsilon, x, y}^{1,1} )  =  \tilde{f}(x, y), \quad \text{for a.e. $(x,y) \in \RR^2 \backslash (\cup_k \cL_k)$;} 
			\end{equation}
			\begin{equation}\label{eq: claim 2 lemma A2}
				\text{When either $p > 1$ or $q > 1$, $\PP (\cE_{\varepsilon, x, y}^{p,q} ) = o(\varepsilon^2)$, for a.e. $(x,y) \in \RR^2 \backslash (\cup_k \cL_k)$.}
			\end{equation}

			We show \eqref{eq: claim 1 lemma A2} first. For $p = q = 1$ and $(x,y) \in \RR^2 \backslash (\cup_k \cL_k)$, we have the disjoint decomposition
			$$
			\cE_{\varepsilon, x, y}^{1,1} \!= \!\!\!\displaystyle\bigcup_{i\in [d]} \displaystyle\bigcup_{j\in [d]} \!\!\!\left\{\! \max_{i' \in [d] \backslash \{i\}} \!Z_{i'}\! \le x, \max_{j' \in [d] \backslash \{j\} }\! \!-Z_{j'} \!+\! c_{j'}\! \le \!y, x < \!Z_i\! \le x\!+\! \varepsilon, y <\!- Z_j \!+\! c_j\!\le y + \varepsilon\! \right\}\!,
			$$
			and note that for $i = 1, \ldots, d$, for small enough $\varepsilon > 0$, by $(x,y) \notin \cup_k \cL_k$, we have
			\begin{align*}
				& \PP\left\{ \max_{i' \in [d] \backslash \{i\} } Z_{i'} \le x, \max_{i' \in [d] \backslash \{i\}} -Z_{i'} + c_{i'} \le y, x < Z_i \le x+ \varepsilon, y <- Z_i + c_i\le y + \varepsilon \right\} \\
				& \le \PP\left\{ x < Z_i \le x+ \varepsilon, y <- Z_i + c_i\le y + \varepsilon \right\}  = 0.
			\end{align*}
			Hence, we have
			\begin{align*}
				& \PP(\cE_{\varepsilon, x, y}^{1,1}) \\
				&= \!\!\!\! \!\!\!\! \sum_{(i,j) \in [d]^2 : i \ne j}\!\! \!\! \!\!  \PP\!\!\left\{ \max_{i' \in [d] \backslash \{i\} } \!Z_{i'} \!\le x, \max_{j' \in [d] \backslash \{j\} } \!\!\!-Z_{j'} \!+\! c_{j'} \!\le y, x < Z_i \le x+ \varepsilon, y \!<\!- Z_j + c_j\le y + \varepsilon \!\right\}\\
				& = \!\!\!\! \!\!\!\! \sum_{(i,j) \in [d]^2 : i \ne j} \int_{x}^{x+\varepsilon} \!\int_{y}^{y+\varepsilon} \PP\left\{ \max_{i' \in [d] \backslash \{i\}} Z_{i'} \le x, \max_{j' \in [d]\backslash \{j\} }\!\!\!-Z_{j'} + c_{j'} \le y| Z_i = u, Z_j \!=\! - v + c_j\! \right\} \\
				& \quad \quad \quad \quad \quad \quad \quad \quad \quad  \times \phi_{i,j}(u,-v + c_j) du dv.
			\end{align*}
			Then, we  show that for a.e. $(x,y) \in \RR^2 \backslash (\cup_k \cL_k)$, the following holds for any $i,j \in [d] $ and $ i \ne j$,
			\begin{equation}\label{eq: cor -1 dens lim non L}
				\begin{aligned}
					& \lim_{u \downarrow x, v \downarrow y} \PP\left\{ \max_{i' \in [d] \backslash \{i\}} Z_{i'} \le x, \max_{j' \in [d] \backslash \{j\}} -Z_{j'} + c_{j'} \le y| Z_i = u, Z_j = - v + c_j \right\} \\
					& = \PP\left\{ \max_{i' \in [d] \backslash \{i\} } Z_{i'} \le x, \max_{j' \in [d]\backslash \{j\} } -Z_{j'} + c_{j'} \le y| Z_i = x, Z_j = - y + c_j \right\}. 
				\end{aligned}
			\end{equation}
			For $i' \ne i$ and $j' \ne j$, we first define the residuals of orthogonal projections of $Z_{i'}$ and $-Z_{j'} + c_{j'}$ onto $(Z_i, -Z_j + c_j)$ as
			$$
			V_{i'} = Z_{i'} - \mu_{i'} - a_{i'i}(Z_i -  \mu_i) + a_{i'j}(Z_j - \mu_j),
			$$
			$$
			\text{and} \quad  U_{j'} = -Z_{j'} +  \mu_{j'} + a_{j'i}(Z_i -  \mu_i) - a_{j'j}(Z_j - \mu_j),
			$$ 
			where for any $\ell \in [d]$, we define $(a_{\ell i}, a_{\ell j})$ as,
			\begin{align*}
				& (a_{\ell i}, a_{\ell j})^{\top} \\
				& = \left(\EE[ (Z_i -\mu_i, -Z_j + \mu_j)^{\top}(Z_i -\mu_i, -Z_j + \mu_j)]\right)^{-1} \EE[(Z_i -\mu_i, -Z_j + \mu_j)^{\top} (Z_{\ell} - \mu_{\ell})]. 
			\end{align*}
			Correspondingly, we define the set $\cD_{i,j} = \cD_{i,j}^1 \cup \cD_{i,j}^2$, where 
			\begin{align*}
				\cD_{i,j}^1 &= \cup_{i': V_{i'}\equiv 0} \left\{(x,y) \in \RR^2: (1-a_{i'i})x - a_{i'j} y -\mu_{i'}+a_{i'i}\mu_{i} - a_{i'j}(\mu_j - c_j) = 0\right\} ,\\
				\cD_{i,j}^2 & = \cup_{j': U_{j'}\equiv 0} \left\{(x,y) \in \RR^2: a_{j'i}x +(1+a_{j'j} )y +\mu_{j'} - c_{j'} - a_{j'i}\mu_{i} + a_{j'j}(\mu_j - c_j) = 0\right\} .
			\end{align*}
			We  show that the set $\cD_{i,j}$ has a Lebesgue measure of 0 on $\RR^2$, and \eqref{eq: cor -1 dens lim non L} holds for any $(x,y) \in \RR^2\backslash \cD_{i,j}$.
			
			We first show that $(1-a_{i'i}, -a_{i'j}) \ne \mathbf{0}$ and $(a_{j'i}, 1 + a_{j'j}) \ne \mathbf{0}$ for any $i' \ne i$ and $j' \ne j$ such that $V_{i'}$ and $U_{j'}$ degenerate to 0, which implies that $\cD_{i,j}$ is a union of lines or empty sets that has a Lebesgue measure of 0 in $\RR^2$. For $i' \ne i $ such that $ V_{i'} \equiv 0$, we have that $a_{i' j} \ne 0$,  
			% when $a_{i'j} = 0$, we have that $|a_{i'i}| < 1$ 
			because otherwise $Z_i$ and $Z_{i'}$ are perfectly correlated, which contradicts our assumption. Similarly, for $j' \ne j$ such that $U_{j'} \equiv 0$, we have that $a_{j'i} \ne 0$
			% when $a_{j'i} = 0$,  $|a_{j'j}| < 1$ 
			as $Z_j$ and $Z_{j'}$ are not perfectly correlated. Hence the claim follows and  $\cD_{i,j}$ has a Lebesgue measure of 0 on $\RR^2$.
			
			Then we consider any $(x, y) \in \RR^2 \backslash \cD_{i,j}$ and show that \eqref{eq: cor -1 dens lim non L} holds. Define the sets $J_1 = \{i' \in [d]: a_{i'i} \le 0\}$, $J_1' = \{i' \in [d]: a_{i'i} \ge 0\}$, $J_2 = \{i' \in [d]: a_{i'j} \le 0\}$, and $J_2' = \{i' \in [d]: a_{i'j} \ge 0\}$. For any $(u,v)$ that is close enough to $(x,y)$, we have $(u,y) \notin \cD_{i,j}$, and hence
			\begin{align}
				& \lim_{v \downarrow y }\PP \Big(\max_{i' \in [d] \backslash \{i\}} Z_{i'} \le x, \max_{j' \in [d] \backslash \{j\}} -Z_{j'} + c_{j'} \le y \Big| Z_i = u, Z_j  = -v + c_j \Big)  \notag \\
				& = \lim_{v \downarrow y } \PP \left\{ \begin{array}{ll}
					V_{i'} \le x - \mu_{i'} - a_{i'i}(u -  \mu_i) + a_{i'j}(c_j - v - \mu_j), & i' \in [d] \backslash i\\
					U_{j'} \le y - c_{j'} +  \mu_{j'} + a_{j'i}(u -  \mu_i) - a_{j'j}(c_j - v - \mu_j) , &  j' \in [d] \backslash j
				\end{array}\right\} \notag\\
				& = \PP \left\{ \begin{array}{ll}
					V_{i'} \le x - \mu_{i'} - a_{i'i}(u -  \mu_i) + a_{i'j}(c_j - y - \mu_j), & i' \in [d] \backslash i, \ i' \in J_2\\
					U_{j'} \le y - c_{j'} +  \mu_{j'} + a_{j'i}(u -  \mu_i) - a_{j'j}(c_j - y - \mu_j) , & j' \in [d] \backslash j,\  j' \in J_2' \\
					V_{i'} < x - \mu_{i'} - a_{i'i}(u -  \mu_i) + a_{i'j}(c_j - y - \mu_j), & i' \in [d] \backslash i,\  i' \notin J_2\\
					U_{j'} < y - c_{j'} +  \mu_{j'} + a_{j'i}(u -  \mu_i) - a_{j'j}(c_j - y - \mu_j) , &  j' \in [d] \backslash j,\  j' \notin J_2'
				\end{array} \right\} \notag\\
				& = \PP \left\{ \begin{array}{ll}
					V_{i'} \le x - \mu_{i'} - a_{i'i}(u -  \mu_i) + a_{i'j}(c_j - y - \mu_j), & i' \in [d] \backslash i\\
					U_{j'} \le y - c_{j'} +  \mu_{j'} + a_{j'i}(u -  \mu_i) - a_{j'j}(c_j - y - \mu_j) , & j' \in [d] \backslash j
				\end{array} \right\},\label{eq: cor -1 joint dens y lim}
			\end{align}
			where the last equality holds since we exclude the discontinuous points for degenerate $V_{i'}$ and $U_{j'}$ at 0 by excluding the set $\cD_{i,j}$.
			
			Furthermore, by similar arguments, we have that
			\begin{align}
				& \lim_{u \downarrow x }\PP \left\{ \begin{array}{ll}
					V_{i'} \le x - \mu_{i'} - a_{i'i}(u -  \mu_i) + a_{i'j}(c_j - y - \mu_j), & i' \in [d] \backslash i\\
					U_{j'} \le y - c_{j'}+  \mu_{j'} + a_{j'i}(u -  \mu_i) - a_{j'j}(c_j - y - \mu_j) , &  j' \in [d] \backslash j 
				\end{array} \right\} \notag\\
				&= \PP \left\{ \begin{array}{ll}
					V_{i'} \le x - \mu_{i'} - a_{i'i}(x -  \mu_i) + a_{i'j}(c_j - y - \mu_j), & i' \in [d] \backslash i, \  i' \in J_1\\
					U_{j'} \le y - c_{j'} +  \mu_{j'} + a_{j'i}(x -  \mu_i) - a_{j'j}(c_j - y - \mu_j) , & j' \in [d] \backslash j,\  j' \in J_1' \\
					V_{i'} < x - \mu_{i'} - a_{i'i}(x -  \mu_i) + a_{i'j}(c_j - y - \mu_j), & i' \in [d] \backslash i,\  i' \notin J_1\\
					U_{j'} < y - c_{j'} +  \mu_{j'} + a_{j'i}(x -  \mu_i) - a_{j'j}(c_j - y - \mu_j) , &  j' \in [d] \backslash j,\  j' \notin J_1'
				\end{array} \right\} \notag\\
				& = \PP \left\{ \begin{array}{ll}
					V_{i'} \le x - \mu_{i'} - a_{i'i}(x -  \mu_i) + a_{i'j}(c_j - y - \mu_j), & i' \in [d] \backslash i\\
					U_{j'} \le y - c_{j'} +  \mu_{j'} + a_{j'i}(x -  \mu_i) - a_{j'j}(c_j - y - \mu_j) , & j' \in [d] \backslash j
				\end{array} \right\} \notag \\
				& = \PP\left\{ \max_{i' \in [d] \backslash \{i\} } Z_{i'} \le x, \max_{j' \in [d] \backslash \{ j\} } -Z_{j'} + c_{j'} \le y| Z_i = x, Z_j = - y + c_j \right\} \label{eq: cor -1 joint dens x lim}.
			\end{align}
			Combining \eqref{eq: cor -1 joint dens y lim} and \eqref{eq: cor -1 joint dens x lim}, we have that \eqref{eq: cor -1 dens lim non L} holds for a.e. $(x,y) \in \RR^2 \backslash (\cup_k \cL_k)$. Hence we have that, for a.e. $(x,y) \in \RR^2 \backslash (\cup_k \cL_k)$,
			\begin{align*}
				&\lim_{\varepsilon \downarrow 0} \varepsilon^{-2}   \PP(\cE_{\varepsilon, x, y}^{1,1} ) \\
				& = \!\!\!\! \!\!\!\!\sum_{(i,j) \in [d]^2: i \ne j } \!\!\!\PP\!\left\{ \!\max_{i' \in [d] \backslash \{i\}} \!Z_{i'} \!\le\! x, \max_{j' \in [d] \backslash \{j\}} \!\!-Z_{j'}\! +\! c_{j'} \!\le y| Z_i\! =\! x, Z_j \!=\! - y + c_j \!\!\right\}\!\phi_{i,j}(x,- y + c_j),
			\end{align*}
			and \eqref{eq: claim 1 lemma A2} follows. 
			
			Then, we show \eqref{eq: claim 2 lemma A2}. For $\cE_{\varepsilon, x, y}^{p,q}$ with either $p > 1$ or $q > 1$, $\PP (\cE_{\varepsilon, x, y}^{p,q} )$ will be upper bounded by summation of terms of the form 
			$$
			\PP\left(\begin{array}{ll}
				x < Z_i, Z_j \le x + \varepsilon\\
				y < -Z_{i'} + c_{i'}\le y + \varepsilon 
			\end{array}\right) \text{ or }   \PP\left(\begin{array}{ll}
				x < Z_{i'}\le x + \varepsilon\\
				y < - Z_i + c_i, -Z_j + c_j \le y + \varepsilon
			\end{array}\right), \quad i \ne j.
			$$
			We will discuss $ \PP\left(\begin{array}{ll}
				x < Z_i, Z_j \le x + \varepsilon\\
				y < -Z_{i'} + c_{i'}\le y + \varepsilon 
			\end{array}\right)$ for any $i,i',j \in [d]$ and $i \ne j$, and the results for $\PP\left(\begin{array}{ll}
				x < Z_{i'}\le x + \varepsilon\\
				y < - Z_i + c_i, -Z_j + c_j \le y + \varepsilon
			\end{array}\right)$ follow with similar arguments. First consider when $i' \notin \{i,j\}$, and we discuss the scenarios where $(Z_i, Z_j, -Z_{i'} + c_{i'})$ is degenerate and non-degenerate respectively. When $(Z_i, Z_j, -Z_{i'} + c_{i'})$ is non-degenerate, we have that 
   \begin{align*}
       \PP\left(\begin{array}{ll}
				x < Z_i, Z_j \le x + \varepsilon\\
				y < -Z_{i'} + c_{i'}\le y + \varepsilon 
			\end{array}\right) & = \int_{x}^{x+\varepsilon} \int_{x}^{x+\varepsilon} \int_{-y - \varepsilon+c_{i'}}^{-y +c_{i'}} \phi_{i,j,i'} (u,v,u') du dv du'\\
   & = O(\varepsilon^3) = o(\varepsilon^2), \quad \text{for all $(x,y) \in \RR^2$ as $\varepsilon \downarrow 0$},
   \end{align*}
 where $\phi_{i,j,i'} (\cdot)$ denotes the joint density of $(Z_i, Z_j, Z_{i'})$, and the second equality is due to the bound $\sup_{(u,v,u') \in \RR^3} \phi_{i,j,i'} (u,v,u') < + \infty$.
		When $(Z_i, Z_j, -Z_{i'} + c_{i'})$ is degenerate, there exists a non-zero vector $(a_1, a_2, a_3, a_4) \in \RR^4$ such that with probability 1, we have
			$$
			(Z_i, Z_j, -Z_{i'} + c_{i'}) \in \cP := \{(v_1, v_2, v_3) \in \RR^3: a_1\cdot  v_1 + a_2\cdot v_2 + a_3 \cdot v_3 = a_4 \},
			$$
			where $a_3 \ne 0$ by the assumption that $|\Corr(Z_i,Z_j)| < 1$ for any $i \ne j$. Then define the line $\cC = \{(x,y) \in \RR^2: (a_1 + a_2) \cdot x + a_3 \cdot y = a_4\}$, and for any $(x,y) \in \RR^2 \backslash \cC$, we have
			\begin{align*}
				\PP\left(\begin{array}{ll}
					x < Z_i, Z_j \le x + \varepsilon\\
					y < -Z_{i'} + c_{i'}\le y + \varepsilon 
				\end{array}\right) &= \PP\!\left( \!\left\{\begin{array}{ll}
					x < Z_i, Z_j \le x + \varepsilon\\
					y < -Z_{i'} + c_{i'}\le y + \varepsilon 
				\end{array}\right\} \cap \big\{\!(Z_i, Z_j, -Z_{i'} + c_{i'}) \!\in \cP\big\}\!\!\right) \\
				& = 0, \quad \text{for small enough $\varepsilon > 0$}.
			\end{align*}
			% which will be of order $O(\varepsilon^3) = o(\varepsilon^2)$, 
			
			Next we consider when $i' \in \{i,j\}$, and we assume without loss of generality that $i = i'$. Because $(x,y) \notin \cup_k \cL_k$,  for small enough $\varepsilon > 0$, we have that
			$$
			\left\{ (v_1,v_2) \in \RR^2:  x < v_1 \le x + \varepsilon, \quad y <v_2\le y + \varepsilon  \right\} \cap \cup_k \cL_k = \emptyset.
			$$
			Then, since $(Z_i, -Z_i + c_i) \in \cup_k \cL_k$ with probability 1, we have that
			$$
			\PP\left(\begin{array}{ll}
				x < Z_i, Z_j \le x + \varepsilon\\
				y < -Z_{i'} + c_{i'}\le y + \varepsilon 
			\end{array}\right) \le \PP \left( x < Z_i \le x + \varepsilon,  y < -Z_{i} + c_{i}\le y + \varepsilon \right) = 0,
			$$
			when $\varepsilon > 0$ is small enough. This completes the proof of \eqref{eq: claim 2 lemma A2}.
			
			Combining  \eqref{eq: claim 1 lemma A2} and  \eqref{eq: claim 2 lemma A2}, for a.e. $(x,y) \in \RR^2 \backslash (\cup_k \cL_k)$, we have that 
			\begin{align*}
				&\lim_{\varepsilon \downarrow 0}\varepsilon^{-2} \PP(x < M_{\cA} \le x + \varepsilon, y < M_{\cB} \le y+\varepsilon)\\
				& \quad =  \lim_{\varepsilon \downarrow 0}\varepsilon^{-2}  \sum_{p = 1}^d \sum_{q = 1}^d \PP (\cE_{\varepsilon, x, y}^{p,q} ) =   \lim_{\varepsilon \downarrow 0} \varepsilon^{-2}   \PP(\cE_{\varepsilon, x, y}^{1,1} )  =  \tilde{f}(x, y),
			\end{align*}
			and \eqref{eq: dens lim} follows.
		\end{proof}
		
		\subsubsection{Distribution on $\cup_k \cL_k$}
		
		\begin{proof}[\unskip\nopunct]\renewcommand{\qedsymbol}{}
			
			Now we consider the distribution of $(M_{\cA}, M_{\cB})$ on the line $\cL_k$ for each $k =1, \ldots , d$. For any Borel measurable subset $\cI \subseteq \RR$, we have that $\PP(M_{\cA} \in \cI, M_{\cB} = - M_{\cA} + c_k) \le \PP(M_{\cA} \in \cI) \le \sum_i \PP(Z_i \in \cI)$, and hence the absolute continuity holds. 
			
			Now for any given $k \in [d]$ and any $x \in \RR$, we consider the disjoint decomposition
			\begin{equation}\label{eq: lm cor -1 L decomp}
				\{x < M_{\cA} \le x + \varepsilon, M_{\cB} = -M_{\cA} + c_k \} = \cup_{p=1}^{d}\cup_{q=1}^{d} \cE_{\varepsilon, k, x}^{p,q},
			\end{equation}
			where the event $\cE_{\varepsilon,k, x}^{p,q}$ is
			\begin{align*}
				\cE_{\varepsilon, k,  x}^{p,q} &= \left\{
				\begin{array}{l}
					\exists I_1 \subseteq [d], |I_1| =p \\
					\exists I_2 \subseteq [d], |I_2| = q 
				\end{array}
				: \begin{array}{ll}
					x < Z_i \le x + \varepsilon,  & \,\, \forall i \in I_1  \\
					Z_{i'} \le x , &\,\,   \forall i' \notin I_1 \\
					-(x + \varepsilon ) + c_k \le - Z_j + c_j < -x + c_k,  & \,\, \forall j \in I_2  \\
					- Z_{j'} + c_{j'} < -(x + \varepsilon ) + c_k , &\,\,   \forall j' \notin I_2\\
					M_{\cB} = -M_{\cA} + c_k
				\end{array}  \right\}\\
				& = \left\{
				\begin{array}{l}
					\exists I_1 \subseteq [d], |I_1| =p \\
					\exists I_2 \subseteq [d], |I_2| = q 
				\end{array}
				: \begin{array}{ll}
					x < Z_i \le x + \varepsilon,  & \,\, \forall i \in I_1  \\
					Z_{i'} \le x , &\,\,   \forall i' \notin I_1 \\
					x + c_j - c_k < Z_j  \le x + \varepsilon + c_j - c_k ,  & \,\, \forall j \in I_2  \\
					Z_{j'}  > x + \varepsilon + c_{j'} - c_k , &\,\,   \forall j' \notin I_2\\
					M_{\cB} = -M_{\cA} + c_k
				\end{array}  \right\}.
			\end{align*}
			Recall the set $ \cI^* = \argmax_{i \in [d]} c_i$,  $\bar{c} = \max_{i\in[d]} c_i$ and the line  $\cL =  \{(x,y) \in \RR^2: y = - x + \bar{c}\} = \cup_{k \in \cI^*} \cL_k$. We  first show that for $k \notin \cI^*$,  $\PP( \cE_{\varepsilon, k,  x}^{p,q}) = 0$ for any $p,q$ when $\varepsilon > 0$ is small enough.
			Suppose that $k \notin \cI^*$. Then, there exists an $i \in [d]$ such that $c_i > c_k$. If $i \notin I_2$, we have $x + \varepsilon + c_i -c_k   > x + \varepsilon $, and if $i \in I_2$, we have that $x + c_i -c_k > x + \varepsilon$ for small enough $\varepsilon > 0$. Hence when $\varepsilon $ is small enough, we have that the conditions $Z_i > x + \varepsilon$ and  $Z_i \le x + \varepsilon$ cannot hold simultaneously. This implies that $\PP( \cE_{\varepsilon, k,  x}^{p,q}) = 0$ for any $p,q$. Therefore, combining with the decomposition \eqref{eq: lm cor -1 L decomp}, for any $k \notin \cI^*$, $\PP \{x < M_{\cA} \le x + \varepsilon, M_{\cB} = -M_{\cA} + c_k \} = 0$ when $\varepsilon$ is small enough. This also implies that $(M_{\cA}, M_{\cB})$ is absolutely continuous with respect to the Lebesgue measure of $\RR^2$ on $\cL_k$ for $k \notin \cI^*$, and we have that result (1) of Lemma~\ref{lm: anti con equal var no cor 1 dens} follows. 
			
			Finally, we consider the line $\cL = \cup_{k \in \cI^*} \cL_k$. For any given $k \in \cI^*$ (note that $\cE_{\varepsilon, k,  x}^{p,q}$ denotes the same set for any $k \in \cI^*$) and any $p,q \ge 1$, if there exists $i \ne j$ such that $i, j \in I_1 \cup I_2$, we have that $\PP(\cE_{\varepsilon, k,  x}^{p,q}) = O(\varepsilon^2)$ as $(Z_i, Z_j)$ is non-degenerate. Similarly, we have $\PP(\cE_{\varepsilon, k,  x}^{p,q}) = O(\varepsilon^2) = o (\varepsilon)$ for either $p > 1$ or $q > 1$.  Therefore, we can write
			$$
			\PP(\cup_{p=1}^{d}\cup_{q=1}^{d} \cE_{\varepsilon, k, x}^{p,q}) = \PP (\cE_{\varepsilon, k,  x}^{1,1} ) + o(\varepsilon)= \sum_{i \in [d]} \PP(\cE_{\varepsilon,  x}^{(i)} ) + o(\varepsilon),
			$$
			where 
			$$
			\cE_{\varepsilon,  x}^{(i)}  = \left\{
			\begin{array}{ll}
				c_i = \bar{c}, & \\
				x < Z_i \le x + \varepsilon,  &  \\
				x + \varepsilon + c_{i'} - \bar{c} < Z_{i'} \le x , &\,\,   \forall i' \ne i 
			\end{array}  \right\}.
			$$
			For each $i \in [d]$, it can be seen that $\PP(\cE_{\varepsilon, x}^{(i)}) = 0$ if $i \notin \cI^*$. 
			Furthermore, if $|\cI^*| > 1$, for each $i \in [d]$, there exists $i' \in \cI^*$ and $i' \ne i$ such that $x + \varepsilon + c_{i'} - \bar{c}  = x + \varepsilon \ge x$, and $\PP(\cE_{\varepsilon,  x}^{(i)}) = 0$. Hence, we have that the density is 0 for any $x \in \RR$ if $|\cI^*| > 1$. 
			
			We now consider $|\cI^*| = 1$ and $\cI^* = \{i^*\}$ for some $i^* \in [d]$. We have 
			\begin{align*}
				\PP(\cE_{\varepsilon,  x}^{(i^*)})  &= \PP\left\{
				\begin{array}{ll}
					x < Z_{i^*} \le x + \varepsilon,  &  \\
					x + \varepsilon + c_{i} - \bar{c} < Z_{i} \le x , &\,\,   \forall i \ne i^*
				\end{array}  \right\}\\
				& = \int_x^{x+\varepsilon} \PP \left(x + \varepsilon + c_{i} - \bar{c} < Z_{i} \le x ,\,   \forall i \ne i^* | Z_{i^*} = u \right) \phi_{i^*} (u)du\\
				& = \int_x^{x+\varepsilon} \PP \left(x  + c_{i} - \bar{c} \le Z_{i} \le x ,\,   \forall i \ne i^* | Z_{i^*} = u \right) \phi_{i^*} (u)du \\
				& \quad - \int_x^{x+\varepsilon} \PP \left(x + c_{i} - \bar{c} \le Z_i \le x+\varepsilon + c_{i} - \bar{c}, \,   \forall i \ne i^* | Z_{i^*} = u \right) \phi_{i^*} (u)du\\
				& = \int_x^{x+\varepsilon} \PP \left(x + c_{i} - \bar{c} \le Z_{i} \le x ,\,   \forall i \ne i^* | Z_{i^*} = u \right) \phi_{i^*} (u)du + O(\varepsilon^2),
			\end{align*}
			where the $O(\varepsilon^2)$ term in the last equality follows as $(Z_i, Z_{i^*})$ is non-degenerate for any $i \ne i^*$, which also gives
			$$
			\lim_{u \downarrow x} \PP \left(x + c_{i} - \bar{c} \le Z_{i} \le x ,\,   \forall i \ne i^* | Z_{i^*} = u \right) = \PP \left(x + c_{i} - \bar{c} \le Z_{i} \le x ,\,   \forall i \ne i^* | Z_{i^*} = x \right).
			$$
			Hence for $x \in \RR$, we have 
			\begin{align*}
				& \lim_{\varepsilon \downarrow 0}\frac{1}{\varepsilon} \PP  \{x < M_{\cA} \le x + \varepsilon, M_{\cB} = -M_{\cA} + c_k \} =  \lim_{\varepsilon \downarrow 0}\frac{1}{\varepsilon} \PP (\cE_{\varepsilon,  x}^{(i^*)}) \\
				& =  \PP \left(x + c_{i} - \bar{c} \le Z_{i} \le x ,\,   \forall i \ne i^* | Z_{i^*} = x \right) \phi_{i^*} (x),
			\end{align*}
			and result (2) of Lemma~\ref{lm: anti con equal var no cor 1 dens} holds. \end{proof}
	\end{proof}

	\subsection{Proof of Lemma~\ref{lm: anti con equal var no cor 1}}\label{sec: proof lm anti con equal var no cor 1}
	\begin{proof}[\unskip\nopunct]
		% Without loss of generality, we can assume that $\bar\lambda_2 \le \bar\lambda_1$, and that $Z_i$'s have means $(\mu_i - \mu^*_{\cA})$'s, and denote the new shift by $c_i' = c_i - \mu_{\cA}^* - \mu_{\cB}^*$ , where $\mu^*_{\cA} = \arg\min_{\mu \in \RR} \max_{i=1}^d |\mu_i - \mu|$ and $\mu_{\cB}^* = \arg\min_{\mu \in \RR} \max_{i=1}^d|-\mu_i + c_i - \mu|$, and we slightly abuse the notation and denote $\mu_i - \mu^*_{\cA}$ by $\mu_i$ and $c_i'$ by $c_i$ for notational convenience. 
		Recall that we define $\bar{c} = \max_i c_i$, the set $\cI^* = \argmax_i c_i$ and the line  $\cL =  \{(x,y) \in \RR^2: y = - x + \bar{c}\}$. By Lemma~\ref{lm: anti con equal var no cor 1 dens}, for any $t \in \RR$, we have that
		\begin{align}
			& \PP(|M_{\cB} - M_{\cA} - t | \le \varepsilon) = \int_{(x,y) \in \cL^c: |y-x-t| \le \varepsilon } \tilde{f}(x,y) dx dy+ \int_{|-2x+\bar{c}-t|\le \varepsilon} \tilde{f}(x) dx \notag\\
			& = \int_{|y-x-t| \le \varepsilon } \tilde{f}(x,y) dx dy+ \int_{-\frac{t-\bar{c}}{2} - \frac{\varepsilon}{2}}^{-\frac{t-\bar{c}}{2} + \frac{\varepsilon}{2}} \tilde{f}(x) dx \notag\\
			& \overset{u = y - x}{ =\joinrel= } \underbrace{\int_{|u-t| \le \varepsilon } \int_{\RR}\tilde{f}(x,x+u) dx du }_{= {\rm I_1}} + \underbrace{\int_{-\frac{t-\bar{c}}{2} - \frac{\varepsilon}{2}}^{-\frac{t-\bar{c}}{2} + \frac{\varepsilon}{2}} \tilde{f}(x) dx}_{= {\rm I_2}}, \label{eq: decomp I1 I2}
		\end{align}
		where the second equality follows from the fact that $\cL$ has measure 0 on $\RR^2$, and $\tilde{f}(x,y)$ is finite everywhere. We consider the terms ${\rm I_1}$ and ${\rm I_2}$ independently. For ${\rm I_1}$, we have that 
		\begin{align*}
			{\rm I}_1 \!&=\! \int_{t- \varepsilon }^{t + \varepsilon} \!\!\int_{\RR}  \sum_{(i,j) \in [d]^2: i \ne j} \bigg\{ \PP\!\left(\!\!\begin{array}{c}
				\max_{i' \in [d] \backslash \{i\}} Z_{i'} \le x   \\
				\max_{j' \in [d]\backslash\{ j\}} (-Z_{j'}+c_{j'} )\le x + u 
			\end{array}  \! \bigg|\! \begin{array}{c}
				Z_i = x \\
				Z_j = - x - u + c_j 
			\end{array} \!\!\right)\\
			& \quad \quad \quad \quad \times \phi_{i,j} (x, -x - u + c_j) \bigg\} d x d u\\
			& =  \int_{t- \varepsilon }^{t + \varepsilon} \sum_{j \in [d]}\sum_{i \in [d] \backslash \{j\}}\int_{\RR}  \PP\left( \begin{array}{c}
				\max_{i' \in [d]\backslash \{i\}} Z_{i'} \le x   \\
				\max_{j' \in [d]\backslash \{j\}} \tilde{Z}_{j'} \le x 
			\end{array}  \bigg| \begin{array}{c}
				Z_i = x \\
				\tilde{Z}_j = x
			\end{array} \right) \phi_{Z_i,\tilde{Z}_j}^u (x, x) d x d u\\
			& =  \int_{t- \varepsilon }^{t + \varepsilon} \sum_{j \in [d]}\int_{\RR}  \sum_{i \in [d] \backslash \{j\} } \PP\left( \begin{array}{c}
				\max_{i' \in [d] \backslash \{i\}} Z_{i'} \le x   \\
				\max_{j' \in [d]\backslash \{j\}} \tilde{Z}_{j'} \le x 
			\end{array}  \bigg| \begin{array}{c}
				Z_i = x \\
				\tilde{Z}_j = x
			\end{array} \right) \phi_{Z_i|\tilde{Z}_j}^u (x| x)\phi_{\tilde{Z}_j}^u (x) d x d u,
		\end{align*}
		where $\tilde{Z}_j = -Z_j + c_j - u$ for $j \in [d]$; $\phi_{Z_i,\tilde{Z}_j}^u (\cdot)$ denotes the joint density of $(Z_i, \tilde{Z}_j)$; $\phi_{\tilde{Z}_j}^u (\cdot)$ denotes the marginal density of $\tilde{Z}_j$; $\phi_{Z_i|\tilde{Z}_j}^u (\cdot)$ denotes the conditional density of $Z_i$ on $\tilde{Z}_j$, and the second equality follows from 
		$$
		\phi_{Z_i,\tilde{Z}_j}^u (x, x) = \phi_{i,j} (x, -x - u + c_j).
		$$
		% Following similar steps as in the proof of Theorem~\ref{thm: anti con},
		Then, we let
		\begin{align*}
			g_{u, j}(x) &=  \sum_{i \in [d] \backslash \{j\}} \PP\left( \begin{array}{c}
				\max_{i' \in [d] \backslash  \{i\} } Z_{i'} \le x   \\
				\max_{j' \in [d]\backslash \{ j\}} \tilde{Z}_{j'} \le x 
			\end{array}  \bigg| \begin{array}{c}
				Z_i = x \\
				\tilde{Z}_j = x
			\end{array} \right) \phi_{Z_i|\tilde{Z}_j}^u (x| x),\\
			H_{u,j}(x) &= \PP\left( \begin{array}{c}
				\max_{i \in [d]} Z_{i} \le x   \\
				\max_{j' \in [d] \backslash \{j\}} \tilde{Z}_{j'} \le x 
			\end{array}  \bigg|
			\tilde{Z}_j = x \right),
		\end{align*}
		and we have the following claim for $g_{u, j}(x)$ and $H_{u,j}(x)$.
		% \begin{align*}
			%     % \textbf{Claim 3: }& \text{for every $j \in \cB$ and a.e. $u \in \RR$, $H_{u,j}(x)$ is non-decreasing on $x \in \RR$;}\\
			%     \textbf{Claim 2: }& \text{for every $j \in \cB$ and a.e. $u \in \RR$}, \, g_{u, j}(x) \le (1+\underline\rho)^{-1} d  H_{u,j}(x)/ dx, \,\, \text{$\forall x \ne (c_j - u)/2$.}
			% \end{align*}
		\begin{claim}\label{claim: A31}
			{For every $j \in \cB$ and a.e. $u \in \RR$}, we have that
			\begin{equation}\label{eq: claim 2 g = 0}
				g_{u, j}(x) = 0, \quad \forall x < (c_j - u)/2,
			\end{equation}
			and 
			\begin{equation}\label{eq: claim 2 g < H}
				g_{u, j}(x) \le (1+\underline\rho)^{-1} d  H_{u,j}(x)/ dx, \quad \forall x > (c_j - u)/2,
			\end{equation}
			where $\underline{\rho} = \min_{i \ne j} \rho_{ij}$ with  $\rho_{ij} = \Corr(Z_i, Z_j) = \sigma^{-2} {\rm Cov} (Z_i, Z_j)$.
		\end{claim}
		\begin{proof}
			See Appendix~\ref{sec: proof claim A31}.
		\end{proof}
		
		% We will revisit and prove Claim~\ref{claim: A31}
		% % and Claim~4 
		% at the end of the proof. 
		Following Claim~\ref{claim: A31}, together with the fact that 
		$$
		\left|\int_{\RR}  \!\sum_{(i,j) \in [d]^2 : i \ne j}\!\!\!\! \!\!\! \PP\!\!\left(\!\begin{array}{c}
			\displaystyle \max_{i' \in [d] \backslash \{i\}} Z_{i'} \le x   \\
			\displaystyle \max_{j' \in [d] \backslash \{j\} } \!\!-Z_{j'}+c_{j'} \le x + u 
		\end{array}   \Bigg|\! \begin{array}{c}
			Z_i = x \\
			Z_j \!=\! - x \!-\! u \!+\! c_j\! 
		\end{array} \!\!\right)\!\!\phi_{i,j} (x, -x \!-\! u \!+\! c_j) d x\right|\!\! <\! \infty,
		$$
		we have 
		\begin{align}
			{\rm I}_1 &= \int_{t-\varepsilon}^{t+\varepsilon} \sum_{j \in [d]} \int_{\RR} g_{u,j} (x) \phi_{\tilde{Z}_j}^u (x) dx du \notag \\
			& \overset{\rm (a)}{=} \int_{t-\varepsilon}^{t+\varepsilon} \sum_{j \in [d]} \int_{(c_j - u)/2}^{+\infty} g_{u,j} (x) \phi_{\tilde{Z}_j}^u (x) dx du \notag\\
			& \overset{\rm (b)}{\le} \frac{1}{1+\underline\rho}\int_{t-\varepsilon}^{t+\varepsilon} \sum_{j \in [d]} \int_{(c_j - u)/2}^{+\infty} \phi_{\tilde{Z}_j}^u (x) d H_{u,j}(x) du \notag\\
			& = \frac{1}{1+\underline\rho}\!\int_{t-\varepsilon}^{t+\varepsilon} \sum_{j \in [d]} \Big\{\phi_{\tilde{Z}_j}^u (x) H_{u,j}(x) \Big|_{\frac{c_j - u}{2}}^{+\infty} \notag\\
			& \quad \quad \quad  \quad \quad \quad \quad \quad + \!\!\int_{(c_j - u)/2}^{+\infty}\!\!H_{u,j}(x)\phi_{\tilde{Z}_j}^u (x)  \frac{x+\mu_j-c_j+u}{\sigma^2}  dx  \!\Big\}\!d u \notag\\
			& \overset{\rm (c)}{\le} \frac{1}{(1+\underline\rho)}\! \!\int_{t-\varepsilon}^{t+\varepsilon} \!\!\!\int_{\RR}\!\! \Big\{\!\! \sum_{j \in [d]}\!\PP\Big( 
			% \begin{array}{c}
				%    \max_{i} Z_{i} \le x   \\
				\max_{j' \in [d]\backslash \{j\}} \!\tilde{Z}_{j'} \!\le\! x 
				% \end{array}  
			\Big|
			\tilde{Z}_j \!=\! x \!\Big)\!\frac{|x+\mu_j-c_j+u| }{\sigma^2} \phi_{\tilde{Z}_j}^u \!(x) \!\!
			%        \\
			%        &\quad + \sum_{i}\PP\left( \begin{array}{c}
				%    \max_{i' \ne i} Z_{i'} \le x   \\
				%      \max_{j} \tilde{Z}_{j} \le x 
				% \end{array}  \bigg|
			%        {Z}_i = x \right) \phi_{{Z}_i}^u (x) 
			\Big\}dx du \notag\\
			& = \frac{1}{(1+\underline\rho)\sigma^2} \int_{t-\varepsilon}^{t+\varepsilon} \EE \Big( \sum_{j \in [d]} \II\Big\{\max_{j' \in [d] \backslash \{j\}} \tilde{Z}_{j'} \le \tilde{Z}_j \Big\} |\tilde{Z}_j - \tilde\mu_j |\Big)  du \notag\\
			& \overset{\rm (d)}{\le} \frac{1}{(1+\underline\rho)\sigma^2} \int_{t-\varepsilon}^{t+\varepsilon} \EE \big( \max_{j \in [d]} |\tilde{Z}_j - \tilde\mu_j |\big)  du \notag\\
			& =2 \EE \big( \max_{i \in [d]} |Z_i - \mu_i|\big)  \frac{\varepsilon}{(1+\underline\rho)\sigma^2}, \label{eq: I1 bound}
			% (1+\underline\rho)^{-1}\left(\EE \big( \max_i |Z_i - \mu_i|\big) + \bar\lambda_2 + |t|+ \varepsilon \right)\varepsilon/\sigma^2,
		\end{align}
		where $\tilde{\mu}_j = \EE (\tilde{Z}_j) = -\mu_j + c_j - u$; equality (a) follows from \eqref{eq: claim 2 g = 0}; inequality (b) holds because of \eqref{eq: claim 2 g < H};
		% and the fact that
		% $$
		% \phi_{\tilde{Z}_j}^u (x) H_{u,j}(x) \Big|^{(c_j - u)/2}_{-\infty} \ge 0,
		% $$ 
		% $\phi_{\tilde{Z}_j}^u (x) H_{u,j}(x) \Big|_{(c_j - u)/2^+}^{(c_j - u)/2^-} \le 0$ and $\phi_{\tilde{Z}_j}^u (x)$ is continuous everywhere, 
		inequality (c) follows from $ H_{u,j}(x) \le \PP\big(\max_{j' \in [d]\backslash \{j\}} \tilde{Z}_{j'} \le x \big|\tilde{Z}_j = x \!\big)$ and the fact that
		$$
		\phi_{\tilde{Z}_j}^u (x) H_{u,j}(x) \Big|_{(c_j - u)/2}^{+\infty} \le 0, \quad \forall j \in [d],
		$$ 
		and inequality (d) follows by the assumption $|\Corr(Z_j, Z_{j'})| < 1$ for any $j' \ne j$ such that $\PP(\max_{j' \ne j} \tilde{Z}_{j'} = \tilde{Z}_j) = 0$ for any $j \in [d]$. 
		
		For term ${\rm I}_2$, by Lemma~\ref{lm: anti con equal var no cor 1 dens}, when $|\cI^*| > 1$, we have ${\rm I}_2 = 0$; When $\cI^* = \{i^*\}$ for some $i^* \in [d]$,  we have 
		\begin{align}
			{\rm I}_2 &= \int_{-\frac{t-\bar{c}}{2} - \frac{\varepsilon}{2}}^{-\frac{t-\bar{c}}{2} + \frac{\varepsilon}{2}}  \PP\big(x + c_i - \bar{c} \le Z_i \le x, \forall i \ne i^* | Z_{i^*} = x\big) \phi_{i^*} (x)  dx \notag\\
			& \le \int_{-\frac{t-\bar{c}}{2} - \frac{\varepsilon}{2}}^{-\frac{t-\bar{c}}{2} + \frac{\varepsilon}{2}} \phi_{i^*} (x) dx \le \frac{\varepsilon}{\sqrt{2 \pi}\sigma} \notag \\
			&\le (1+\underline\rho)^{-1}\EE \big( \max_i |Z_i - \mu_i|\big) \varepsilon/\sigma^2,\label{eq: I2 bound}
		\end{align}
		where the second inequality holds by the bound on the Gaussian density function, and the last inequality holds by the fact that 
		$$
		(1+\underline\rho)^{-1} \EE[\max_i |Z_i - \mu_i| / \sigma ]\ge  (1+\underline\rho)^{-1} \EE[  |Z_1 - \mu_1| / \sigma] = (1+\underline\rho)^{-1} \sqrt{2/\pi} \ge 1/\sqrt{2\pi}.
		$$
		Then combining \eqref{eq: decomp I1 I2}, \eqref{eq: I1 bound} and \eqref{eq: I2 bound}, for any $t \in \RR$, we have that 
		$$
		\PP(|M_{\cB} - M_{\cA} - t | \le \varepsilon) \le 3 (1+\underline\rho)^{-1}\EE \big( \max_i |Z_i - \mu_i|\big) \varepsilon/\sigma^2,
		$$
		and \eqref{eq: anti con max diff equal var} holds as desired.
		% following very similar arguments as in the proof of Theorem~\ref{thm: anti con}, by discussing values of $t$ and $\varepsilon$ and noting the fact that $(1+\underline\rho)^{-1} \ge 1/2$, we have the resulting bound \eqref{eq: anti con max diff equal var} and \eqref{eq: anti con max diff equal var t 0} hold.

	\end{proof}
	\subsection{Proof of Lemma~\ref{lm: NA NB equiv}}\label{sec: proof lm NA NB equiv}
	\begin{proof}[\unskip\nopunct]
		We first show that $|\cN_{\cA}| = |\cN_{\cB}|$. 
		For any $i \in \cN_{\cA}$, by the definitions of $\cN_{\cA}$ and $\cN_{\cB}$ in \eqref{eq: def Na Nb}, there exists some $j \in \cN_{\cB}$ such that $\Corr(X_i, X_j) = - 1$. Now assume that there also exists some other $j' \in \cN_{\cB}, j' \ne j$ such that $\Corr(X_i, X_{j'}) = -1$, then we have that $\Corr(X_{j}, X_{j'}) = 1$, which contradicts our assumption that $\Corr(X_{j}, X_{j'}) < 1$ for all $j,j' \in \cB$ and $j \ne j'$. Hence, for any $i \in \cN_{\cA}$, $X_i$ is perfectly negatively correlated with one and only one $X_j$ with $j \in \cN_{\cB}$. By the same argument, for any $j \in \cN_{\cB}$, $X_j$ is perfectly negatively correlated with one and only one $X_i$ with $i \in \cN_{\cA}$, and hence we have $|\cN_{\cA}| = |\cN_{\cB}|$. Besides, by the definition of $\cN_{\cA}$ and $\cN_{\cB}$, there exist constants $c_i$'s such that $\{X_j\}_{j \in \cN_{\cB}} = \{-X_i + c_i\}_{i \in \cN_{\cA}}$, and \eqref{eq: equal NA NB} holds.
		
		Meanwhile, for any $i, i' \in \cN_{\cA}$ and $i \ne i'$, suppose that $\Corr(X_i, X_{i'}) = -1$. Then by our assumption, there exists some $j \in \cN_{\cB}$ such that $\Corr(X_i, X_j) = -1$, and we  have that $\Corr(X_{i'}, X_j) = 1$, which contradicts our assumption that $\Corr(X_{i}, X_{j}) < 1$ for any $i \ne j$. Hence we have that $|\Corr(X_i, X_{i'})| < 1$ for any $i, i' \in \cN_{\cA}$ and $i \ne i'$. The same result holds for $\cN_{\cB}$ by the same argument, which completes our proof.
	\end{proof}
	\section{Proofs of Claims and Auxiliary Lemmas in Appendix}
	%\begin{proof}[\unskip\nopunct]\renewcommand{\qedsymbol}{}
	%We provide in this section the proofs of the supporting claims and auxiliary lemmas  used in previous sections.
	%\end{proof}
	\subsection{Proofs of claims for Lemma~\ref{lm: joint max dens}}\label{sec: proof support claims lm joint max dens}
	\subsubsection{Proof of Claim~\ref{claim: 241}}\label{sec: proof claim 241}
	\begin{proof}[\unskip\nopunct]
		We first  
		% $\cE_{\varepsilon, x, y}^{1,1}$, where we 
		decompose  $\cE_{\varepsilon, x, y}^{1,1}$ into disjoint events as  
		$$
		\cE_{\varepsilon, x, y}^{1,1} = \cup_{i \in \cA} \cup_{j \in \cB} \Big\{ \max_{i'\in \cA \backslash \{i\}} X_{i'} \le x, \max_{j' \in \cB \backslash \{j\}} X_{j'} \le y , x < X_i \le  x+ \varepsilon , y < X_j \le y + \varepsilon \Big\}, 
		$$
		and we have 
		\begin{align*}
			& \PP(\cE_{\varepsilon, x, y}^{1,1} )  =  \sum_{i \in \cA, j \in \cB} \PP \Big( \max_{i' \in \cA \backslash \{i\}} X_{i'} \le x, \max_{j' \in \cB \backslash \{j\}} X_{j'} \le y , x < X_i \le  x+ \varepsilon , y < X_j \le y + \varepsilon \Big)\\
			& \quad = \!\!\! \!\!\! \sum_{i \in \cA, j \in \cB}  \int_{x}^{x + \varepsilon} \!\! \!\int_{y}^{y + \varepsilon} \!\! \PP \Big(\max_{i' \in \cA \backslash \{i\}} X_{i'} \le x, \max_{j' \in \cB \backslash \{j\}} X_{j'} \le y \Big| X_i = u, X_j  = v \Big) \phi_{i,j}(u, v) du d v,
		\end{align*}
		where the last equality holds by the assumption that $(X_i, X_j)$ are non-degenerate for any $i \in \cA$ and $j \in \cB$, and $\phi_{i,j}(u, v)$ is the joint density at $(X_i, X_j) = (u,v)$.
		
		Then, we prove Claim~\ref{claim: 241} by showing that  
		% for the mapping $(u, v) \mapsto \PP \big(\max_{i' \in \cA \backslash \{i\}} X_{i'} \le x, \max_{j' \in \cB \backslash \{j\}} X_{j'} \le y \Big| X_i = u, X_j  = v \big)$, we have
		\begin{equation}\label{eq: right cont joint max dens}
			\begin{aligned}
				& \lim_{u \downarrow x, v \downarrow y} \PP \Big(\max_{i' \in \cA \backslash \{i\}} X_{i'} \le x, \max_{j' \in \cB \backslash \{j\}} X_{j'} \le y \Big| X_i  = u, X_j = v \Big) \\
				&\quad = \PP \Big(\max_{i' \in \cA \backslash \{i\}} X_{i'} \le x, \max_{j' \in \cB \backslash \{j\}} X_{j'} \le y \Big| X_{i}  = x, X_j = y \Big),
			\end{aligned}
		\end{equation}
		for a.e. $(x, y) \in \RR^2$ and every $i \in \cA, j \in \cB$. Specifically, for every $i \in \cA$ and $j \in \cB$, we  construct a set $\cD_{i,j} \subseteq \RR^2$ such that $\cD_{i,j} $ has a Lebesgue measure of 0 in $\RR^2$, and for any $(x,y) \in \RR^2 \backslash \cD_{i,j}$, \eqref{eq: right cont joint max dens} holds. 
		
		To construct the set $\cD_{i,j}$, for $i' \in [p] \backslash \{i, j\}$, we first define the residual from orthogonal projection of $X_{i'}$ onto $(X_i, X_j)$ as 
		$$
		U_{i'} = X_{i'} - \mu_{i'} - a_{i' i} (X_i - \mu_i) - a_{i' j} (X_j - \mu_j),
		$$
		where $ (a_{i' i}, a_{i' j}) = \Sigma_{i', i j} \Sigma_{ij}^{-1}$
		with $\Sigma_{ij} = \EE [(X_i - \mu_{i}, X_j - \mu_{j})^{\top} (X_i - \mu_{i}, X_j - \mu_{j})]$ and $\Sigma_{i', i j} = \EE [(X_{i'}- \mu_{i'})(X_i - \mu_{i}, X_j - \mu_{j})]$. We have that $U_{i'}$ is independent of $(X_i, X_j)$.
		Then, we consider the $U_{i'}$'s that are degenerate, i.e., $U_{i'} \equiv 0$, and we define $\cD_{i,j}$ as
		$$
		\cD_{i,j} = \displaystyle\bigcup_{i' \notin \{i,j\}: U_{i'} \equiv 0}\!\!\big\{\! (x,y) \!\in \RR^2: ( \II_{\cA}\{i'\} - a_{i'i})\cdot x + ( \II_{\cB}\{i'\} - a_{i' j} )\cdot y = \mu_{i'} - a_{i'i} \mu_i - a_{i'j} \mu_{j} \big\},
		$$
		where $\cD_{i,j} = \emptyset$ if  $U_{i'} \not\equiv 0$ for any $i' \notin \{i,j\}$. 
		We first show that $\cD_{i,j} $ is a union of lines or empty sets and has a Lebesgue measure of 0 in $\RR^2$. To do so, it suffices to show that $( \II_{\cA}\{i'\} - a_{i'i}, \II_{\cB}\{i'\} - a_{i'j} ) \ne \mathbf{0}$ for any $i' \in [p] \backslash \{i,j\}$ such that $ U_{i'} \equiv 0$. When $U_{i'} \equiv 0$, we have 
		$$
		X_{i'} - \mu_{i'} \equiv a_{i' i} (X_i - \mu_i) + a_{i' j} (X_j - \mu_j).
		$$
		For $(a_{i' i}, a_{i' j})$, since $\sigma_{i'}^2 > 0$, at least one of $a_{i' i}$ and  $a_{i' j}$ is non-zero. When $a_{i' i} a_{i' j} \ne 0$, we have $( \II_{\cA}\{i'\} - a_{i'i}, \II_{\cB}\{i'\} - a_{i'j} ) \ne \mathbf{0}$. When $a_{i' i} a_{i' j} = 0$, since ${\rm Var}(X_{i'} - X_i), {\rm Var}(X_{i'} - X_{j}) > 0$, we have that either $a_{i' i} \ne 1$ or $ a_{i' j} \ne 1$, and $( \II_{\cA}\{i'\} - a_{i'i}, \II_{\cB}\{i'\} - a_{i'j} ) \ne \mathbf{0}$. Hence we conclude that $\cD_{i,j}$ is a union of lines or empty sets and has a Lebesgue measure of 0 in $\RR^2$. 
		
		Next, we show that \eqref{eq: right cont joint max dens} holds for any $(x, y) \in \RR^2 \backslash \cD_{i,j}$. Define the sets $J_1 = \{i' \in [p] \backslash \{i,j\}: a_{i' j} \le 0\}$, $J_1^c = [p] \backslash \{i,j\} \backslash J_1$,  $J_2 = \{i' \in [p] \backslash \{i,j\}: a_{i' i} \le 0\}$, and $J_2^c =  [p] \backslash \{i,j\} \backslash J_2 $. Then we have that
		\begin{align*}
			& \lim_{v \downarrow y }\PP \Big(\max_{i' \in \cA \backslash \{i\}} X_{i'} \le x, \max_{j' \in \cB \backslash \{j\}} X_{j'} \le y \Big| X_i = u, X_j  = v \Big)  \\
			& = \lim_{v \downarrow y } \PP \left\{ \begin{array}{ll}
				U_{i'} \le x -  \mu_{i'} - a_{i' i} (u - \mu_i) - a_{i' j} (v - \mu_{j}), &\,\, \forall i' \in \cA\backslash \{i\}\\
				U_{j'} \le y -  \mu_{j'} - a_{j' i} (u - \mu_i) - a_{j' j} (v - \mu_{j}) , & \,\, \forall j' \in \cB \backslash \{j\}
			\end{array}\Bigg| X_i  = u, X_j  = v \right\}\\
			& = \lim_{v \downarrow y } \PP \left\{ \begin{array}{ll}
				U_{i'} \le x -  \mu_{i'} - a_{i' i} (u - \mu_i) - a_{i' j} (v - \mu_{j}), &\,\, \forall i' \in \cA\backslash \{i\}\\
				U_{j'} \le y -  \mu_{j'} - a_{j' i} (u - \mu_i) - a_{j' j} (v - \mu_{j}) , & \,\, \forall j' \in \cB \backslash \{j\}
			\end{array} \right\}\\
			& = \PP \left\{ \begin{array}{ll}
				U_{i'} \le x -  \mu_{i'} - a_{i' i} (u - \mu_i) - a_{i' j} (y - \mu_{j}), &\,\, \forall i' \in (\cA\backslash \{i\} ) \cap J_1\\
				U_{j'} \le y -  \mu_{j'} - a_{j' i} (u - \mu_i) - a_{j' j} (y - \mu_{j}) , & \,\, \forall j' \in (\cB \backslash \{j\})  \cap J_1 \\
				U_{i'} < x -  \mu_{i'} - a_{i' i} (u - \mu_i) - a_{i' j} (y - \mu_{j}), &\,\, \forall i' \in (\cA\backslash \{i\})  \cap J_1^c \\
				U_{j'} < y -  \mu_{j'} - a_{j' i} (u - \mu_i) - a_{j' j} (y - \mu_{j}) , & \,\, \forall j' \in (\cB \backslash \{j\})  \cap J_1^c
			\end{array} \right\}.
		\end{align*}
		We note that each $U_{i'}$ either degenerates to 0 or has a non-degenerate Gaussian distribution, and the mapping is discontinuous only for degenerate $U_{i'}$'s at $0$. When $\varepsilon$ is sufficiently small, since $(x, y ) \in \RR^2 \backslash \cD_{i,j}$, we have 
		$$
		\II_{\cA}\{i'\} \cdot x + \II_{\cB}\{i'\} \cdot y -  \mu_{i'} - a_{i' i} (u - \mu_i) - a_{i' j} (y - \mu_{j}) \ne 0, \quad \forall i' \in [p] \backslash\{i,j\}.
		$$
		Hence by continuity, we have 
		\begin{align}
			& \lim_{v \downarrow y }\PP \Big(\max_{i' \in \cA \backslash \{i\}} X_{i'} \le x, \max_{j' \in \cB \backslash \{j\}} X_{j'} \le y \Big| X_i  = u, X_j = v \Big) \notag \\ 
			&\quad =  \PP \left\{ \begin{array}{ll}
				U_{i'} \le x -  \mu_{i'} - a_{i' i} (u - \mu_i) - a_{i' j} (y - \mu_{j}), &\,\, \forall i' \in \cA\backslash \{i\}\\
				U_{j'} \le y -  \mu_{j'} - a_{j' i} (u - \mu_i) - a_{j' j} (y - \mu_{j}) , & \,\, \forall j' \in \cB \backslash \{j\}
			\end{array} \right\} \notag\\
			& \quad = \PP \Big(\max_{i' \in \cA \backslash \{i\}} X_{i'} \le x, \max_{j' \in \cB \backslash \{j\}} X_{j'} \le y \Big| X_i  = u, X_j = y \Big).  \label{eq: claim C1 y lim}
		\end{align}
		By similar arguments, we also have that
		\begin{align}
			& \lim_{u \downarrow x }\PP \Big(\max_{i' \in \cA \backslash \{i\}} X_{i'} \le x, \max_{j' \in \cB \backslash \{j\}} X_{j'} \le y \Big| X_i  = u, X_j = y \Big) \notag \\
			&\quad = \lim_{u \downarrow x } \PP \left\{ \begin{array}{ll}
				U_{i'} \le x -  \mu_{i'} - a_{i' i} (u - \mu_i) - a_{i' j} (y - \mu_{j}), &\,\, \forall i' \in \cA\backslash \{i\}\\
				U_{j'} \le y -  \mu_{j'} - a_{j' i} (u - \mu_i) - a_{j' j} (y - \mu_{j}) , & \,\, \forall j' \in \cB \backslash \{j\}
			\end{array} \right\} \notag\\
			&\quad = \PP \left\{ \begin{array}{ll}
				U_{i'} \le x -  \mu_{i'} - a_{i' i} (x- \mu_i) - a_{i' j} (y - \mu_{j}), &\,\, \forall i' \in (\cA\backslash \{i\} ) \cap J_2\\
				U_{j'} \le y -  \mu_{j'} - a_{j' i} (x - \mu_i) - a_{j' j} (y - \mu_{j}) , & \,\, \forall j' \in (\cB \backslash \{j\})  \cap J_2 \\
				U_{i'} < x -  \mu_{i'} - a_{i' i} (x - \mu_i) - a_{i' j} (y - \mu_{j}), &\,\, \forall i' \in (\cA\backslash \{i\})  \cap J_2^c \\
				U_{j'} < y -  \mu_{j'} - a_{j' i} (x - \mu_i) - a_{j' j} (y - \mu_{j}) , & \,\, \forall j' \in (\cB \backslash \{j\})  \cap J_2^c
			\end{array} \right\} \notag \\
			&\quad = \PP \left\{ \begin{array}{ll}
				U_{i'} \le x -  \mu_{i'} - a_{i' i} (x - \mu_i) - a_{i' j} (y - \mu_{j}), &\,\, \forall i' \in \cA\backslash \{i\}\\
				U_{j'} \le y -  \mu_{j'} - a_{j' i} (x - \mu_i) - a_{j' j} (y - \mu_{j}) , & \,\, \forall j' \in \cB \backslash \{j\}
			\end{array} \right\} \notag \\
			&\quad = \PP \Big(\max_{i' \in \cA \backslash \{i\}} X_{i'} \le x, \max_{j' \in \cB \backslash \{j\}} X_{j'} \le y \Big| X_i  = x, X_j = y \Big). \label{eq: claim C1 x lim}
		\end{align}
		Combining \eqref{eq: claim C1 y lim} and \eqref{eq: claim C1 x lim}, we have that \eqref{eq: right cont joint max dens} holds for any $(x, y) \in \RR^2 \backslash \cD_{i,j}$, which further gives that
		\begin{align*}
			\lim_{\varepsilon \downarrow 0} \varepsilon^{-2}   \PP(\cE_{\varepsilon, x, y}^{1,1} ) = \sum_{i,j}\PP \Big(\max_{i' \in \cA \backslash \{i\}} X_{i'} \le x, \max_{j' \in \cB \backslash \{j\}} X_{j'} \le y \Big| X_i  = x, X_j = y \Big) \phi_{i,j}(x, y),
		\end{align*}
		for a.e. $(x, y) \in \RR^2 $. This completes the proof.
	\end{proof}
	\subsubsection{Proof of Claim~\ref{claim: 242}}\label{sec: proof claim 242}
	\begin{proof}[\unskip\nopunct]
		% Now we move on to show that for $\cE_{\varepsilon, x, y}^{q,k} $ with either $q > 1$ or $k > 1$, $\PP(\cE_{\varepsilon, x, y}^{q,k}) = o(\varepsilon^2)$ when $\varepsilon \downarrow 0$. 
		First, observe that for $\cE_{\varepsilon, x, y}^{q,k} $ with either $q > 1$ or $k > 1$, $\PP(\cE_{\varepsilon, x, y}^{q,k})$ is bounded by a sum of terms of the form
		$$
		\PP\left(\begin{array}{l}
  x < X_{i}\le x + \varepsilon \\
			x < X_{i'} \le x + \varepsilon \\
			y < X_j\le y + \varepsilon 
		\end{array}\right) ,\quad  i, i' \in \cA,\quad i \ne i',  \quad j \in \cB , 
		$$
		$$
		\text{ or }   \PP\left(\begin{array}{l}
			x < X_i \le x + \varepsilon\\
   y < X_{j} \le y + \varepsilon \\
			y <  X_{j'}\le y + \varepsilon 
		\end{array}\right), \quad i \in \cA , \quad j, j' \in \cB, \quad j \ne j'.
		$$
		We consider $  \PP\left(\begin{array}{l}
  x < X_{i}, X_{i'}\le x + \varepsilon \\
			y < X_j\le y + \varepsilon 
		\end{array}\right)$ for $i, i' \in \cA$, $i \ne i'$ and $j \in \cB$ first, and we discuss the scenarios when $(X_i, X_{i'}, X_{j})^{\top}$ is degenerate and non-degenerate respectively.
		
		When $(X_i, X_{i'}, X_{j})^{\top}$ is non-degenerate, we have that  
		\begin{equation}\label{eq: dens res event bound 1}
			\PP\left(\begin{array}{l}
   x < X_{i} \le x + \varepsilon \\
				x < X_{i'} \le x + \varepsilon \\
				y < X_j\le y + \varepsilon 
			\end{array}\right) = O(\varepsilon^3) = o(\varepsilon^2), \quad  i, i' \in \cA,\quad i \ne i',  \quad j \in \cB, 
		\end{equation}
		as $\varepsilon \downarrow 0$ for any $(x,y) \in \RR^2$. 
		
		When $(X_i, X_{i'}, X_{j})^{\top}$  is degenerate, we have that there exists a non-zero vector $(c_1,c_2,c_3, c_4)$ such that $c_1 \cdot X_i + c_2 \cdot X_{i'} + c_3 \cdot X_j \equiv c_4$, and we define the two planes 
		$$
		\cP^1 = \{(v_1,v_2,v_3) \in \RR^3: v_1 = v_2\}, \, \cP^2 = \{(v_1,v_2,v_3) \in \RR^3: c_1 \cdot v_1 + c_2 \cdot v_2 + c_3 \cdot v_3 = c_4\},
		$$
		and their  intersection is 
		$$
		\cP^1 \cap \cP^2  = \{(v_1,v_2,v_3) \in \RR^3: (c_1 + c_2) v_1 + c_3 v_3= (c_1 + c_2) v_2 + c_3 v_3 = c_4\}.
		$$
		Then, we have that $(x,x,y) \in \cP^1$ and $(X_i, X_{i'}, X_{j})^{\top} \in \cP^2$ with probability 1, and for any $(x,y) \in \RR^2$  such that $ (c_1 + c_2) x + c_3 y \ne c_4$, the point $(x,x,y)$ does not fall onto the plane $\cP^2$.  Then, when $\varepsilon$ is sufficiently small, we have that 
		$$
		\{(v_1,v_2,v_3) \in \RR^3: x < v_1 \le x+\varepsilon,x <  v_2 \le x+\varepsilon, y < v_3 \le y+\varepsilon\} \cap  \cP^2= \emptyset,
		$$
		which implies 
		\begin{equation}\label{eq: dens res event bound 2}
			\PP\left(\begin{array}{l}
   x < X_i \le x + \varepsilon\\
				x <  X_{i'} \le x + \varepsilon \\
				y < X_j \le y + \varepsilon
			\end{array}\right) = 0, \quad  i, i' \in \cA,\quad i \ne i',  \quad j \in \cB.
		\end{equation}
		Hence it suffices to show that $(c_1+c_2, c_3) \ne \mathbf{0}$ so that $ (c_1 + c_2) x + c_3 y = c_4$ occurs with a Lebesgue measure of 0 on $\RR^2$. When $c_3 \ne 0$, we have $(c_1+c_2, c_3) \ne \mathbf{0}$.    
		%  When $c_3 \ne 0$, the intersection of the two planes $\cP^1$ and $\cP^2$ is the line: 
		%  $$
		%  \cP^1 \cap \cP^2  = \{(v_1,v_2,v_3) \in \RR^3: (c_1 + c_2) v_1 + c_3 v_3= (c_1 + c_2) v_2 + c_3 v_3 = c_4\},
		%  $$
		% which has measure 0 in $\RR^3$. 
		When $c_3 = 0$, since $(c_1,c_2,c_3,c_4)$ is non-zero, $(c_1,c_2)$ is non-zero, because otherwise, $c_4 = c_1 = c_2 = c_3 = 0$. If $c_1 + c_2 = 0$, then we have that ${\rm Var}(X_i - X_{i'}) = 0$, which contradicts our assumption that ${\rm Var} (X_i - X_j) > 0$ for any $i \ne j$. Hence we have that $c_1 + c_2 \ne 0$, and $(c_1+c_2, c_3) \ne \mathbf{0}$.
		
		Therefore, combining \eqref{eq: dens res event bound 1} and \eqref{eq: dens res event bound 2}, for a.e. $(x,y) \in \RR^2$,  we have that 
		$$
		\PP\left(\begin{array}{l}
			x < X_i \le x + \varepsilon \\
   x < X_{i'} \le x + \varepsilon \\
			y < X_j \le y + \varepsilon
		\end{array}\right) = o(\epsilon^2), \quad \text{for any } i, i' \in \cA, \, i \ne i' \text{ and } j \in \cB.
		$$
		%  and the intersection of $\cP^1$ and $\cP^2$ is the line: 
		%   $$
		%   \cP^1 \cap \cP^2  = \{(v_1,v_2,v_3) \in \RR^3: (c_1 + c_2) v_1 = (c_1 + c_2) v_2  = c_4\}.
		%   $$
		% Then when $ (c_1 + c_2) x + c_3 y \ne c_4$, we know that $(x,x,y)$ does not fall onto the plane $\cP^2$, and when $\varepsilon$ is sufficiently small, we have that $\{(v_1,v_2,v_3) \in \R^3: x < v_1,v_2 \le x+\varepsilon, y < v_3 \le y+\varepsilon\} \cap  \cP^2= \emptyset$ and in turn 
		% $$
		% \PP\left(\begin{array}{l}
			%       x < X_i, X_{i'} \le x + \varepsilon \\
			%        y < X_j \le y + \varepsilon
			%   \end{array}\right) = 0, \quad \text{for any } i, i' \in \cA, \, i \ne i' \text{ and } j \in \cB.
		%   $$
		
		Following similar arguments, we have that  $ \PP\left(\begin{array}{l}
			x < X_i \le x + \varepsilon\\
   y < X_{j} \le y + \varepsilon \\
			y < X_{j'}\le y + \varepsilon 
		\end{array}\right)  = o(\varepsilon^2)$ for any $i \in \cA$, $j,j' \in \cB$, $j \ne j'$. Hence 
		% except for at most a finite number of lines with measure 0,
		we have that $\PP(\cE_{\varepsilon, x, y}^{q,k}) = o(\varepsilon^2)$ when $\varepsilon \downarrow 0$ for a.e. $(x,y) \in \RR^2$ when $q > 1$ or $k > 1$, which completes the proof. 
	\end{proof}

	\subsubsection{Proof of Claim~\ref{claim: 243}}\label{sec: proof claim 243}
	\begin{proof}[\unskip\nopunct]
		We first decompose $\cE_{\varepsilon, x, y}^{1}$ into disjoint events,  
		$$
		\cE_{\varepsilon, x, y}^{1} = \bigcup_{i \in \cA} \Big\{ \max_{i' \in \cA \backslash \{i\}} X_{i'} \le x, \max_{j \in \cB } X_j \le y, x < X_i\le  x+ \varepsilon \Big\}, 
		$$
		and we have
		\begin{equation}\label{eq: E1 int}
			\PP(\cE_{\varepsilon, x, y}^{1} )  =  \sum_{i \in \cA}  \int_{x}^{x + \varepsilon}\PP \Big(\max_{i' \in \cA \backslash \{i\}} X_{i'} \le x, \max_{j \in \cB} X_j \le y \Big| X_i  = u \Big) \phi_{i}(u) du.
		\end{equation}
		We  show that the mapping $u \mapsto \PP \Big(\max_{i' \in \cA \backslash \{i\}} X_{i'} \le x, \max_{j \in \cB} X_j \le y \Big| X_i  = u \Big) $ is right continuous at $x$ for all $i \in \cA$ and a.e. $x \in \RR$. Specifically, given $i \in \cA$, for every $i' \in [p] \backslash \{i\}$, define the orthogonal residual  $V_{i'} = X_{i'} - \mu_{i'} - \sigma_{i}^{-2} \sigma_{i i'} (X_i  - \mu_i) $, where $\sigma_{i i'} = \EE [(X_i- \mu_i)(X_{i'} - \mu_{i'})]$. Then, define the set 
		$$
		\cD_i = \cup_{i' \in [p] \backslash \{i\}: V_{i'} \equiv 0} \big\{ x \in \RR: (\II_{\cA} (i') - \sigma_i^{-2} \sigma_{i i'}) \cdot  x -\mu_{i'} + \sigma_{i}^{-2} \sigma_{i i'} \cdot \mu_i + \II_{\cB}(i') \cdot y = 0\big\},
		$$
		where $\cD_i = \emptyset$ if there is no $i' \in [p] \backslash \{i\}$ such that $V_{i'} \equiv 0$. We  show that $\cD_i$ has a Lebesgue measure of 0 on $\RR$, and for every $x \in \RR^2 \backslash \cD_i$, we have
		\begin{equation}\label{eq: joint dens par right cont}
			\begin{aligned}
				& \lim_{u \downarrow x }\PP \Big(\max_{i' \in \cA 
					\backslash \{i\}} X_{i'} \le x, \max_{j \in \cB } X_j  \le y \Big| X_i  = u \Big)  \\
				& \quad = \PP \Big(\max_{i'\in \cA \backslash \{i\}} X_{i'} \le x, \max_{j\in \cB } X_j \le y \Big| X_i  = x \Big).
			\end{aligned}
		\end{equation}
		
		We first show that $\cD_i$ is a union of points or empty sets. Indeed, for $i' \in [p] \backslash \{i\}$, if $V_{i'} \equiv 0$, i.e., $(X_{i'}, X_i)$ is degenerate, we have $\sigma_{i}^{-2} \sigma_{i i'} \ne 0$. Meanwhile, for $i' \in \cA$ such that $V_{i'} \equiv 0$, if $\sigma_{i}^{-2} \sigma_{i i'} = 1$, we have $\mu_{i'} \ne \mu_i$, since otherwise, we have $X_{i'} \equiv X_i$, which contradicts our assumption that $X_{i'} \not\equiv X_i$ for any $i,i' \in \cA$ and $i \ne i'$. Hence, we have that $\cD_i$ is a properly defined union of points or empty sets with a Lebesgue measure of 0 on $\RR$.
		
		Next, we show that \eqref{eq: joint dens par right cont} holds for any $x \in \RR \backslash \cD_i$. Define the sets $J = \{i'  \in [p] \backslash \{i\}: \sigma_{i i' } \le 0\}$ and $J^c = [p] \backslash \{i\} \backslash J$. Then for any $x \in \RR \backslash \cD_i$, we have 
		\begin{align*}
			& \lim_{u \downarrow x }\PP \Big(\max_{i' \in \cA \backslash \{i\}} X_{i'} \le x, \max_{j \in \cB } X_j  \le y \Big| X_i  = u \Big)  \\
			& = \lim_{u \downarrow x } \PP \left\{ \begin{array}{ll}
				V_{i'} \le x -  \mu_{i'} - \sigma_{i}^{-2} \sigma_{i i'} (u - \mu_i) , &\,\, \forall i' \in \cA\backslash \{i\}\\
				V_{j} \le y - \mu_{j} - \sigma_{i}^{-2} \sigma_{i j} (u - \mu_i) , & \,\, \forall j \in \cB
			\end{array} \right\}\\
			& = \PP \left\{ \begin{array}{ll}
				V_{i'} \le x -  \mu_{i'} - \sigma_{i}^{-2} \sigma_{i i'} (x - \mu_i), &\,\, \forall i' \in (\cA\backslash \{i\} ) \cap J\\
				V_{j} \le y -  \mu_{j} - \sigma_{i}^{-2} \sigma_{i j} (x - \mu_i), & \,\, \forall j \in \cB   \cap J \\
				% V_{\tilde{S}''} \le z -  \mu_{\tilde{S}''} - \Sigma_{S}^{-1} \Sigma_{S \tilde{S}''} (x - \mu_i), & \,\, \forall \tilde{S}'' \in \cC  \cap J \\
				V_{i'} < x -  \mu_{i'} - \sigma_{i}^{-2} \sigma_{i i'} (x - \mu_i), &\,\, \forall i' \in (\cA\backslash \{i\})  \cap J^c \\
				V_{j} < y -  \mu_{j} - \sigma_{i}^{-2} \sigma_{i j} (x - \mu_i) , & \,\, \forall j \in \cB \cap J^c
				% V_{\tilde{S}''} < z -  \mu_{\tilde{S}''} - \Sigma_{S}^{-1} \Sigma_{S \tilde{S}''} (x - \mu_i) , & \,\, \forall \tilde{S}'' \in \cC \cap J^c
			\end{array} \right\}\\
			& = \PP \left\{ \begin{array}{ll}
				V_{i'} \le x -  \mu_{i'} - \sigma_{i}^{-2} \sigma_{i i'} (x - \mu_i), &\,\, \forall i' \in \cA\backslash \{i\}\\
				V_{j} \le y -  \mu_{j} - \sigma_{i}^{-2} \sigma_{ij} (x - \mu_i), & \,\, \forall j \in \cB
				% V_{\tilde{S}''} \le z -  \mu_{\tilde{S}''} - \Sigma_{S}^{-1} \Sigma_{S \tilde{S}''} (x - \mu_i), & \,\, \forall \tilde{S}'' \in \cC
			\end{array} \right\} \\
			& = \PP \Big(\max_{i'\in \cA \backslash \{i\}} X_{i'} \le x, \max_{j\in \cB } X_j \le y \Big| X_i  = x \Big), 
		\end{align*}
		and \eqref{eq: joint dens par right cont} holds.
		
		Combining \eqref{eq: E1 int} and \eqref{eq: joint dens par right cont}, for a.e. $x \in \RR$, we have
		\begin{align*}
			\lim_{\varepsilon \downarrow 0} \varepsilon^{-1}   \PP(\cE_{\varepsilon, x, y}^{1} ) = \sum_{i \in \cA} \PP \Big(\max_{i' \in \cA \backslash \{i\}} X_{i'} \le x, \max_{j \in \cB } X_j \le y \Big| X_i  = x \Big) \phi_{i}(x),
		\end{align*}
		which completes the proof.
	\end{proof}
	\subsubsection{Proof of Claim~\ref{claim: 244}}\label{sec: proof claim 244}
	\begin{proof}[\unskip\nopunct]
		We show that for $k \ge 2$, $\PP(\cE_{\varepsilon, x, y}^{k}) = o(\varepsilon)$ as $\varepsilon \downarrow 0$ for a.e. $x \in \RR$. When $|\cA| = 1$, the results hold trivially. When $|\cA| > 1$, we have that $\PP(\cE_{\varepsilon, x, y}^{k})$ is upper bounded by a sum of terms of the form $\PP(x < X_i \le x + \varepsilon, x < X_{i'} \le x + \varepsilon )$, where $i, i' \in \cA$ and $i \ne i'$.  When $(X_i, X_{i'})$ is non-degenerate, we have $\PP(x < X_i \le x + \varepsilon,x < X_{i'} \le x + \varepsilon ) = O(\varepsilon^2) = o(\varepsilon)$ for all $x \in \RR$, and Claim~\ref{claim: 244} holds. When $(X_i, X_{i'})$ is degenerate, there exists a non-zero vector $(c_1, c_2, c_3)$ such that $c_1 \cdot X_i + c_2 \cdot X_{i'} = c_3$. Note that $(c_1, c_2) \ne \mathbf{0}$, and $(c_1 + c_2)$ and $c_3$ cannot both be 0, since otherwise, we have  $X_i \equiv X_{i'}$, which contradicts our assumption. Hence the set $\cC := \{x \in \RR: (c_1 + c_2) \cdot x = c_3\}$ is either a single point or an empty set with a Lebesgue measure of 0 on $\RR$. When $x \in \RR \backslash \cC$, and $\varepsilon$ is sufficiently small, we have
		$$
		\{(v_1, v_2) \in \RR^2: x < v_1 \le x + \varepsilon, x < v_2 \le x + \varepsilon\} \cap \{(v_1, v_2) \in \RR^2: c_1 \cdot v_1 + c_2 \cdot v_2 = c_3\}= \emptyset,
		$$
		and $\PP(x < X_i \le x + \varepsilon, x <X_{i'} \le x + \varepsilon ) = 0$. Hence, Claim~\ref{claim: 244} follows.
	\end{proof}

	\subsection{Proof of Claim~\ref{claim: dup term thm 2.5} for Lemma~\ref{claim: 251}}\label{sec: proof claim dup term thm 2.5}
	\begin{proof}[\unskip\nopunct]
		We begin with $i, i'\in \cA$ and $i \ne i'$, and we  show that \eqref{eq: dup term A} holds for a.e. $u \in \RR$, i.e.,
		\begin{equation}\label{eq: id var check}
			\frac{  V_{i}  +\mu_{i}  + \sigma_{j}^{-2} \sigma_{ij} (u - \mu_{j}) }{1 - \sigma_{j}^{-2} \sigma_{ij}} \not\equiv \frac{V_{i'}  +\mu_{i'}  + \sigma_{j}^{-2} \sigma_{i' j} (u - \mu_{j}) }{1 - \sigma_{j}^{-2} \sigma_{i' j} }, \quad \text{ $\forall i, i'\in \cA$ and $i \ne i'$}.
		\end{equation}
		% We prove this by considering .
		Assume that the two random variables on the left-hand side (LHS) and RHS of \eqref{eq: id var check} are identical, i.e., $\tilde{V}_i \equiv \tilde{V}_{i'}$. Then, since $\EE V_{i} = \EE V_{i'} = 0$, we have that 
		\begin{align*}
			&(1 - \sigma_{j}^{-2} \sigma_{ij} )^{-1}V_{i}  =  (1 - \sigma_{j}^{-2} \sigma_{ij} )^{-1} \left(X_{i} - \mu_{i} - \sigma_{j}^{-1} \sigma_{ij} (X_{j}  - \mu_{j})\right)\\
			&  \equiv (1 - \sigma_{j}^{-2} \sigma_{i'j} )^{-1} \left(X_{i'} - \mu_{i'} - \sigma_{j}^{-1} \sigma_{i' j} (X_{j}  - \mu_{j})\right) =(1 - \sigma_{j}^{-2} \sigma_{i' j} )^{-1}V_{i'}.
		\end{align*}
		Then, if 
		$$
		(1 - \sigma_{j}^{-2} \sigma_{ij} )^{-1} \sigma_{j}^{-2} \sigma_{ij} = (1 - \sigma_{j}^{-2} \sigma_{i' j} )^{-1} \sigma_{j}^{-2} \sigma_{i' j} ,
		$$
		by some algebra, we have that $(1 - \sigma_{j}^{-2} \sigma_{ij} )^{-1} = (1 - \sigma_{j}^{-2} \sigma_{i'j} )^{-1} $, and hence $X_{i} - \mu_{i} \equiv X_{i'} - \mu_{i'} $, which contradicts the assumption that ${\rm Var} (X_{i} -X_{i'}  ) > 0$. Hence, 
		$$
		(1 - \sigma_{j}^{-2} \sigma_{ij} )^{-1} \sigma_{j}^{-2} \sigma_{ij} \ne (1 - \sigma_{j}^{-2} \sigma_{i' j} )^{-1} \sigma_{j}^{-2} \sigma_{i' j} ,
		$$
		which indicates that $\tilde{V}_i \equiv \tilde{V}_{i'}$ only holds at
		$$
		u = \left(\frac{\sigma_{j}^{-2} \sigma_{ij}}{1 - \sigma_{j}^{-2} \sigma_{ij}} - \frac{\sigma_{j}^{-2} \sigma_{i'j}}{1 - \sigma_{j}^{-2} \sigma_{i' j}}\right)^{-1} \left(\frac{\mu_{i'} } {1 - \sigma_{j}^{-2} \sigma_{i' j}} -\frac{\mu_{i} } {1 - \sigma_{j}^{-2} \sigma_{ij}} \right)  + \mu_{j}.
		$$
		Therefore, \eqref{eq: id var check} and equivalently \eqref{eq: dup term A} holds for a.e. $u \in \RR$.
		
		Secondly, for $i \in \cA$ and $i' \in \cS_j^+\backslash \cC_j \backslash \cA$, we show that \eqref{eq: dup term cross} holds for a.e. $u \in \RR$. In particular, if $\tilde{V}_i \equiv \tilde{V}_{i'}$ for $i \in \cA$ and $i' \in \cS_j^+\backslash \cC_j \backslash \cA$, we have
		$$
		\frac{  V_{i}  +\mu_{i}  + \sigma_{j}^{-2} \sigma_{ ij } (u - \mu_{j}) }{1 - \sigma_{j}^{-2} \sigma_{ij}} \equiv \frac{V_{i'}  + \mu_{i'}  - \sigma_{j}^{-2} \sigma_{i' j}  \mu_{j} }{1 - \sigma_{j}^{-2} \sigma_{i' j} } - u.
		$$
		This implies that 
		$$
		u = \frac{1 - \sigma_{j}^{-2} \sigma_{ij}}{1 - \sigma_{j}^{-2} \sigma_{i' j} } (\mu_{i'}  - \sigma_{j}^{-2} \sigma_{i' j}  \mu_{j} ) - (\mu_{i}  - \sigma_{j}^{-2} \sigma_{ij}  \mu_{j} ),
		$$
		which is a single point with a Lebesgue measure of 0 on $\RR$. Hence, \eqref{eq: dup term cross} holds for a.e. $u \in \RR$, which completes the proof.
	\end{proof}
	\subsection{ Proof of Lemma~\ref{lm: clt third dev bd}}\label{sec: proof lm clt third dev bd}
	\begin{proof}[\unskip\nopunct]
		For $j,k \in [p]$, we let
		\begin{align*}
			&\delta_{jk} = \II\{j = k\}, \quad \pi^{\cA}_j (x) = \frac{e^{\beta( x_j + v_j)}}{\sum_{j' \in \cA}e^{\beta( x_{j'} + v_{j'})}} , \quad \pi^{\cB}_j (x) = \frac{e^{\beta( x_j + v_j)}}{\sum_{j' \in \cB}e^{\beta( x_{j'} + v_{j'})}} .
			%   & \pi_{j}(x) = \left\{ \begin{array}{ll}
				%     \pi^{\cA}_j (x)  &,  \quad j \in \cA, \\
				%     -\pi^{\cB}_j (x)  &,  \quad j  \in \cB.
				% \end{array}\right.;\\
			% & w_{jk}(x) = \left\{ \begin{array}{ll}
				%      \delta_{jk}\pi^{\cA}_j (x) - \pi^{\cA}_j (x) \pi^{\cA}_k (x)  &,  \quad j,k \in \cA, \\
				%     - \delta_{jk}\pi^{\cB}_j (x) + \pi^{\cB}_j (x) \pi^{\cB}_k (x)  &,  \quad j,k \in \cB, \\
				%      0 &, \quad \text{otherwise}.
				% \end{array}\right..
			% & q_{jk\ell}^{\cA}(x) = \delta_{j k} \delta_{j \ell} \pi^{\cA}_j (x)- \delta_{jk} \pi^{\cA}_j (x) \pi^{\cA}_{\ell} (x) - \pi^{\cA}_j (x) \pi^{\cA}_k (x) (\delta_{j\ell} + \delta_{k \ell}) + 2  \pi^{\cA}_j (x)  \pi^{\cA}_k (x)  \pi^{\cA}_{\ell} (x);\\
			% & q_{jk\ell}^{\cB}(x) = \delta_{j k} \delta_{j \ell} \pi^{\cB}_j (x)- \delta_{jk} \pi^{\cB}_j (x) \pi^{\cB}_{\ell} (x) - \pi^{\cB}_j (x) \pi^{\cB}_k (x) (\delta_{j\ell} + \delta_{k \ell}) + 2  \pi^{\cB}_j (x)  \pi^{\cB}_k (x)  \pi^{\cB}_{\ell} (x);\\
			% & q_{jk\ell}(x) = \left\{ \begin{array}{ll}
				%     q_{jk\ell}^{\cA}(x)  &,  \quad j,k,\ell \in \cA, \\
				%   - q_{jk\ell}^{\cB}(x)  &,  \quad j,k,\ell \in \cB, \\
				%      0 &, \quad \text{otherwise}.
				% \end{array}\right..
		\end{align*}
		We  show by induction that for any $d \ge 1$,  $h = (h_1, \ldots, h_d)^{\top} \in [p]^d$, the partial derivative $\partial_{h_1} \ldots \partial_{h_d} \varphi^v_s(x)$ is the summation of $K_d$ terms of the  form 
		\begin{equation}\label{eq: varphi parital forms}
			\begin{aligned}
				\varphi^{(h)}_k(x)  = & (-1)^{q^k} \beta^{d_1^k} \gamma_s^{(d_2^k)}\big(\zeta_{\beta}(x)\big) \prod_{j \in \mathcal{O}^k_{\cA}}\big\{ \II_{\cA}(h_j )\pi_{h_j}^{\cA}(x) \big\} \\
				& \times \prod_{\ell \in \mathcal{O}^k_{\cB}} \big\{ \II_{\cB}(h_{\ell} )\pi_{h_{\ell}}^{\cB}(x)\big\} \prod_{r \in  \bar\cO^k} \delta_{h_{r} h_{\sigma_k(r)}} ,
			\end{aligned}
		\end{equation}
		where $K_d \le 2^d d!$, $k=1, \ldots, K_d$, $q^k = 1, 2$, $0 \le d_1^k \le d-1$ and $1 \le d_2^k  \le d$ are integers such that $d_1^k + d_2^k = d$, $\cO^k \subseteq [d]$ is a non-empty subset of $[d]$, $\bar{\cO}^k = [d] \backslash \cO^k$, $\cO^k = \cO^k_{\cA} \cup \cO^k_{\cB}$ is a partition of $\cO^k$, and $\sigma_k(\cdot)$ is a mapping from $\bar\cO^k$ to $\cO^k$. 
		Then by some algebra, when $d = 1$, for any $j \in [p]$, we have that 
		\begin{align*}
			\partial_j \varphi^v_s(x) & 
			% = \gamma_s' (\zeta_{\beta}(x)) \pi_j (x) 
			= \gamma_s' (\zeta_{\beta}(x)) \II_{\cA} (j)\pi_j^{\cA} (x) - \gamma_s' (\zeta_{\beta}(x)) \II_{\cB} (j) \pi_j^{\cB} (x),
			% \partial_j \partial_k \varphi(x) & = \gamma_s'' (\zeta_{\beta}(x)) \pi_j (x)\pi_k (x) + \beta  \gamma_s' (\zeta_{\beta}(x)) w_{jk}(x),\\
			%  \partial_j \partial_k \partial_{\ell} \varphi(x) & = \gamma_s''' (\!\zeta_{\beta}(x)\!) \pi_j (x)\pi_k (x) \pi_{\ell} (x)\! +\! \beta \gamma_s'' (\zeta_{\beta}(x)\!) \big(\pi_j (x) w_{k\ell}(x) \!+\! \pi_k (x) w_{j\ell} (x)\!+\! \pi_{\ell} (x) w_{jk}(x) \!\big)\\
			%  & \quad + \beta^2 \gamma_s' (\zeta_{\beta}(x))q_{jk\ell}(x).
		\end{align*}
		which is a special case of \eqref{eq: varphi parital forms} with $K_1 = 2$. Now assume that \eqref{eq: varphi parital forms} holds for $d$. Then, for $d + 1$ and any $h \in [p]^{d+1}$, denote by $h_{1:d} = (h_j)_{j=1}^d$. By some algebra, we have that, for any $j, j' \in [p]$, \
		\begin{align}
			\partial_{j'} \big(\II_{\cA} (j) \pi_j^{\cA} (x) \big) & = \beta \left\{  \delta_{jj'}\pi^{\cA}_j (x) - \pi^{\cA}_j (x) \pi^{\cA}_{j'} (x) \right\} \II_{\cA} (j)\II_{\cA} (j'),  \\
			\partial_{j'} \big(\II_{\cB} (j) \pi_j^{\cB} (x) \big) & = \beta \left\{  \delta_{jj'}\pi^{\cB}_j (x) - \pi^{\cB}_j (x) \pi^{\cB}_{j'} (x) \right\} \II_{\cB} (j)\II_{\cB} (j').
		\end{align}
		Then we have that
		\begin{align*}
			&\partial_{h_{d+1}} \varphi^{(h_{1:d})}_k(x) \\
			& =  (-1)^{q^k} \bigg\{ \beta^{d_1^k} \gamma_s^{(d_2^k + 1)}\big(\zeta_{\beta}(x)\big) \!\!\!\prod_{j \in \mathcal{O}^k_{\cA} \cup \{{d+1}\}}\!\!\big\{ \II_{\cA}(h_j )\pi_{h_j}^{\cA}(x) \big\}\!\! \! \prod_{\ell \in \mathcal{O}^k_{\cB}} \!\!\!\big\{ \II_{\cB}(h_{\ell} )\pi_{h_{\ell}}^{\cB}(x)\big\} \!\prod_{r \in  \bar\cO^k} \delta_{h_{r} h_{\sigma_k(r)}} \\
			& \quad -  \beta^{d_1^k} \gamma_s^{(d_2^k + 1)}\big(\zeta_{\beta}(x)\big)\!\! \prod_{j \in \mathcal{O}^k_{\cA} }\!\!\!\big\{ \II_{\cA}(h_j )\pi_{h_j}^{\cA}(x) \big\} \!\!\!\prod_{\ell \in \mathcal{O}^k_{\cB} \cup \{{d+1}\} } \!\!\!\big\{ \II_{\cB}(h_{\ell} )\pi_{h_{\ell}}^{\cB}(x)\big\} \!\!\!\prod_{r \in  \bar\cO^k} \delta_{h_{r} h_{\sigma_k(r)}} \\
			& \quad +\! \sum_{j \in \cO^k} \delta_{h_{d+1}, h_{j}} \cdot  \beta^{d_1^k + 1} \gamma_s^{(d_2^k)}\big(\zeta_{\beta}(x)\!\big) \!\!\!\!\prod_{j' \in \mathcal{O}^k_{\cA} }\!\!\!\!\big\{ \II_{\cA}(h_{j'} )\pi_{h_{j'}}^{\cA}(x) \big\}\!\!\! \!\prod_{\ell \in \mathcal{O}^k_{\cB}}\! \!\!\big\{ \II_{\cB}(h_{\ell} )\pi_{h_{\ell}}^{\cB}(x)\big\}\!\!\! \prod_{r \in  \bar\cO^k}\!\! \delta_{h_{r} h_{\sigma_k(r)}}\\
			& \quad - |\cO^k_{\cA}|\beta^{d_1^k + 1} \gamma_s^{(d_2^k)}\big(\zeta_{\beta}(x)\big) \!\!\!\prod_{j' \in \mathcal{O}^k_{\cA} \cup \{d+1\}}\!\!\!\big\{ \II_{\cA}(h_{j'} )\pi_{h_{j'}}^{\cA}(x) \big\}\!\!\! \prod_{\ell \in \mathcal{O}^k_{\cB}} \!\!\big\{ \II_{\cB}(h_{\ell} )\pi_{h_{\ell}}^{\cB}(x)\big\} \prod_{r \in  \bar\cO^k} \delta_{h_{r} h_{\sigma_k(r)}}\\
			& \quad - |\cO^k_{\cB}|\beta^{d_1^k + 1} \gamma_s^{(d_2^k)}\big(\zeta_{\beta}(x)\big) \!\!\!\prod_{j' \in \mathcal{O}^k_{\cA} }\!\!\!\big\{ \II_{\cA}(h_{j'} )\pi_{h_{j'}}^{\cA}(x) \big\}\!\!\! \!\prod_{\ell \in \mathcal{O}^k_{\cB} \cup \{d+1\} } \!\!\!\big\{ \II_{\cB}(h_{\ell} )\pi_{h_{\ell}}^{\cB}(x)\big\} \!\!\prod_{r \in  \bar\cO^k} \!\!\delta_{h_{r} h_{\sigma_k(r)}} \bigg\}.
		\end{align*}
		Hence, we have that \eqref{eq: varphi parital forms}  holds for $d+1$, and $K_{d+1} \le 2(d+1)K_d \le 2^{d+1} (d+1)!$.
		
		Furthermore, for any $d \in [5]$ and $h = (h_1, \ldots, h_d)^{\top} \in [p]^d$, let
		$$
		U_h (x) = \sum_{k=1}^{K_d}  \beta^{d_1^k} C_{\gamma}\delta^{-d_2^k}\!\!\prod_{j \in \mathcal{O}^k_{\cA}}\!\!\big\{ \II_{\cA}(h_j )\pi_{h_j}^{\cA}(x) \big\} \!\!\prod_{\ell \in \mathcal{O}^k_{\cB}} \!\!\big\{ \II_{\cB}(h_{\ell} )\pi_{h_{\ell}}^{\cB}(x)\big\} \!\!\prod_{r \in  \bar\cO^k} \!\! \delta_{h_{r} h_{\sigma_k(r)}}.
		$$
		We have that
		$$
		|\partial_{h_1} \ldots \partial_{h_d} \varphi^v_s(x)| = \left|\sum_{k=1}^{K_d}  \varphi^{(h)}_k(x)\right|  \le \sum_{k=1}^{K_d} | \varphi^{(h)}_k(x) | \le U_h (x) .
		$$
		Meanwhile, we note that for any $x,y \in \RR^p$ with $\max_{j \in [p]} |y_j| \le \beta^{-1}$, we have that 
		$$
		\sum_{j \in \cA} \pi^{\cA}_j(x) = 1, \quad \sum_{j \in \cB} \pi^{\cB}_j(x) = 1,
		$$
		$$
		e^{-2} \pi_j^{\cA}(x ) \le \pi_j^{\cA}(x + y) \le e^2 \pi_j^{\cA}(x ), \quad  e^{-2} \pi_j^{\cB}(x ) \le \pi_j^{\cB}(x + y) \le e^2 \pi_j^{\cB}(x ).
		$$
		Then, we  have
		$$
		\sum_{h_1 = 1}^p \ldots  \sum_{h_d = 1}^p  U_h (x) \le K_d C_{\gamma}  \left( \delta^{-d} \vee (\beta^1 /\delta^{d-1} ) \vee \ldots \vee (\beta^{d-1} /\delta)\right) \le K_d' \delta^{-d} \log^{d-1} p;
		$$
		$$
		e^{-2d}  U_h (x) \le  U_h (x + y) \le e^{2d}  U_h (x),
		$$
		where $K_d' = C_{\gamma} K_d$. Hence, we have that \eqref{eq: clt m der bd 1}-\eqref{eq: clt m der order 2} hold by taking  small enough $c_1$ and large enough  $C_0, C_1$, which completes the proof.
	\end{proof}
	\subsection{Proof of Claim~\ref{claim: recursive approx rate I} for Lemma~\ref{prop: modified thm 3.1}} \label{sec: proof claim recursive rate I}
	\begin{proof}
		We show \eqref{eq: iter rate approx} by modifying the rates in equations (76), (77), and (78) of \cite{cck2022improvedbootstrap} with our redefined smoothing functions and the anti-concentration condition. 
		Specifically, we redefine the function $f: [0,1] \rightarrow \RR$ in Step~1 of the proof of Lemma~3.1 in \cite{cck2022improvedbootstrap} by
		$$
		f(t) = \sum_{i=1}^{|I_d|} \EE \Big[\varphi_s^v\Big(W_i^{\sigma} + \frac{t V_{\sigma(i)}}{\sqrt{n}}\Big) - \varphi_s^v\Big(W_i^{\sigma} + \frac{t Z_{\sigma(i)}}{\sqrt{n}}\Big) \Big], \quad \forall t \in [0,1],
		$$
		where $\sigma$ is a random function with uniform distribution on the set of all one-to-one functions mapping $\{1,\ldots, |I_d|\}$ to $I_d$, and $\sigma$ is independent of $V_1,\ldots, V_n, Z_1,\ldots,Z_n$ and $\epsilon^{d+1}$; $W_i^{\sigma}$ is defined as 
		$$
		W_i^{\sigma} = \frac{1}{\sqrt{n}}\sum_{j=1}^{i-1} V_{\sigma(j)} + \frac{1}{\sqrt{n}} \sum_{j=i+1}^{|I_d|} Z_{\sigma(j)} + \frac{1}{\sqrt{n}} \sum_{j \notin I_d} Z_j, \quad \forall i = 1,\ldots , |I_d|.
		$$
		Then, following the proof of Lemma~3.1 in \cite{cck2022improvedbootstrap}, to show the validity of \eqref{eq: iter rate approx}, we only need to establish bounds on the derivatives of the function $f(\cdot)$ up to the fourth order. Specifically, we prove the following claims to establish the counterparts of the rates in equations (76), (77), and (78) of \cite{cck2022improvedbootstrap}.
		\begin{claim}\label{claim: f t 2nd}
			Under the event $\cA_d$, we have that
			\begin{equation}\label{eq: f t 2nd}
				\begin{aligned}
					|f^{(2)}(0) | \lesssim   &\  \frac{B_n^2 \delta^{-4} \log^5 (pn) }{n^2} \\
					& +  \big(r_n \delta + \varepsilon_n +   \EE (\varrho_{\epsilon^{d+1}})\big) \left(\frac{\cB_{n,1,d} \delta^{-2} \log p}{\sqrt{n}}  + \frac{B_n^2 \delta^{-4} \log^3 (pn)}{n} \right).
				\end{aligned}
			\end{equation}
		\end{claim}
		\begin{proof}
			See Appendix~\ref{sec: proof claim f t 2nd}.
		\end{proof}
		\begin{claim}\label{claim: f t 3rd}
			Under the event $\cA_d$, we have that
			\begin{equation}\label{eq: f t 3rd}
				\begin{aligned}
					|f^{(3)}(0) | \lesssim  &\ \frac{B_n^2 \delta^{-4} \log^5 (pn) }{n^2} \\
					& +  \big(r_n \delta + \varepsilon_n +   \EE (\varrho_{\epsilon^{d+1}})\big) \left( \frac{\cB_{n,2,d} \delta^{-3} \log^2 p}{{n}} + \frac{B_n^3 \delta^{-5} \log^5 (pn)}{n^{3/2}} \right).
				\end{aligned}
			\end{equation}
		\end{claim}
		\begin{proof}
			See Appendix~\ref{sec: proof claim f t 3rd}.
		\end{proof}
		\begin{claim}\label{claim: f t 4th}
			Under the event $\cA_d$, we have that
			\begin{equation}\label{eq: f t 4th}
				|f^{(4)}(\tilde{t}) | \lesssim   \frac{B_n^2 \delta^{-4} \log^3 p}{n^2} +  \big(r_n \delta + \varepsilon_n +   \EE (\varrho_{\epsilon^{d+1}})\big) \frac{B_n^2 \delta^{-4} \log^3 p}{n} , \quad \text{where $\tilde{t} \in (0,1)$}.
			\end{equation}
		\end{claim}
		\begin{proof}
			See Appendix~\ref{sec: proof claim f t 4th}.
		\end{proof}
		The rates \eqref{eq: f t 2nd}-\eqref{eq: f t 4th} are obtained by modifying Step~3, Step~4 and Step~5 of the proof of Lemma~3.1 in \cite{cck2022improvedbootstrap} correspondingly. Combining \eqref{eq: f t 2nd}-\eqref{eq: f t 4th}, following the same arguments as in Step~1 of the proof of Lemma~3.1 in \cite{cck2022improvedbootstrap}, we have that \eqref{eq: iter rate approx} holds. \end{proof}

	\subsubsection{Proof of Claim~\ref{claim: f t 2nd} via modification of Step~3 in \cite{cck2022improvedbootstrap}}\label{sec: proof claim f t 2nd}
	In Step~3 of the proof of Lemma~3.1 in \cite{cck2022improvedbootstrap}, to show the new rate with respect to equation (76) in that paper, we first need to show that equation (80) of \cite{cck2022improvedbootstrap} still holds for our redefined smoothing functions. Namely, we first show that 
	\begin{equation}\label{eq: deriv eq}
		\partial_j\partial_k\varphi_s^v(x) =  \partial_j\partial_k\varphi_s^v(x)  h^v_{s,\delta}(x, 0), \quad \forall x \in \RR^p.
	\end{equation}
	Indeed when $h^v_{s,\delta}(x, 0) = 0$, by the property of the function $\zeta_{\beta}^v$ in \eqref{eq: property zeta}, we have that either $\zeta_{\beta}^v(x) \le s$ or $\zeta_{\beta}^v(x) > s + \delta$, which by the property of the function $\gamma_s$ indicates that $ \partial_j\partial_k\varphi_s^v(x) = 0$. Hence the claim follows. Then for the error term $\cI_{2,1}$ defined below equation (79) in \cite{cck2022improvedbootstrap}, we have that 
	\begin{align*}
		&\sum_{j,k \in [p]} \EE[|\partial_j \partial_k \varphi^v_s(W)|]  =\!\! \sum_{j,k \in [p]} \!\EE[ h^v_{s,\delta}(W, 0) |\partial_j \partial_k \varphi^v_s(W)|] \le \!\!\sum_{j,k \in [p]}\! \EE[ h^v_{s,\delta}(W, 0) |U_{jk}(W)|] \\
		& \quad \le C_0 \delta^{-2} \log p \EE\big(h^v_{s,\delta}(W, 0)\big) \le C_0 \delta^{-2} \log p \left(C_a (3 r_n \delta + \varepsilon_n) +  2 \EE (\varrho_{\epsilon^{d+1}})\right),
	\end{align*}
 where the last inequality is by Condition~\ref{cond: anti-con rate} such that 
	\begin{align*}
		\EE\big(h^v_{s,\delta}(W, 0)\big)& \le \EE\big(h^v_{s,\delta}(S_n^Z, 0) \big) + 2 \EE (\varrho_{\epsilon^{d+1}} ) \le C_a (3 r_n \delta + \varepsilon_n) +  2 \EE (\varrho_{\epsilon^{d+1}}).
	\end{align*}  
	Hence,  replacing the rate in equation (81) of \cite{cck2022improvedbootstrap}, we have that
	$$
	|\cI_{2,1}| \lesssim \frac{\cB_{n,1,d} \delta^{-2} \log p}{\sqrt{n}} \left( r_n \delta + \varepsilon_n  +  \EE (\varrho_{\epsilon^{d+1}})\right).
	$$
	For the error term $\cI_{2,2}$ defined below equation (79) in \cite{cck2022improvedbootstrap}, we define 
	$$
	\chi^V_i = \II\{\|V_i\|_{\infty} \le C_e B_n \log(pn)\}, \quad \chi^Z_i = \II\{\|Z_i\|_{\infty} \le C_e B_n \log(pn)\},
	$$
	which correspond to the indicators $\tilde{V}_i$ and $\tilde{Z}_i$ defined in Step~3 of \cite{cck2022improvedbootstrap}. Then when $\chi^V_i = 1$, we have that $h_{s,\delta}^v (W_{\sigma^{-1}(i)}^{\sigma}, 2C_e B_n \log(pn)/\sqrt{n}) = 0$. This indicates $h_{s,\delta}^v (W_{\sigma^{-1}(i)}^{\sigma} + \hat{t}V_i/\sqrt{n}, 0) = 0$ for any $\hat{t} \in (0,1)$.
	% where $W_{\sigma^{-1}(i)}^{\sigma}$ and $\sigma(\cdot)$ are the same as defined in Step~1 of \cite{cck2022improvedbootstrap}. 
	Then following the same arguments as establishing \eqref{eq: deriv eq}, for all $i \in I_d$ and $ j, k, \ell, r \in [p]$ we have that
	\begin{equation}\label{eq: equiv 4th deriv phi}
		\begin{aligned}
			& \chi^V_i \Big| \partial_{j}\partial_k \partial_{\ell} \partial_r \varphi_s^v\Big(W_{\sigma^{-1}(i)}^{\sigma} + \frac{\hat{t}V_i}{\sqrt{n}}\Big) V_{i\ell} V_{i r}\Big|\\
			&  =  \chi^V_i h_{s,\delta}^v \big(W_{\sigma^{-1}(i)}^{\sigma}, 2C_e B_n \log(pn)/\sqrt{n}\big)  \Big| \partial_{j}\partial_k \partial_{\ell} \partial_r \varphi_s^v\Big(W_{\sigma^{-1}(i)}^{\sigma} + \frac{\hat{t}V_i}{\sqrt{n}}\Big) V_{i\ell} V_{i r}\Big| .
		\end{aligned}
	\end{equation}
	Then equation (82) in \cite{cck2022improvedbootstrap} also holds in our setting by replacing $h^y(W_{\sigma^{-1}(i)}^{\sigma};x)$ with $h_{s,\delta}^v \big(W_{\sigma^{-1}(i)}^{\sigma}, 2C_e B_n \log(pn)/\sqrt{n}\big)$. Following similar arguments, we have that
	\begin{align*}
		\EE & [h_{s,\delta}^v \big(W_{\sigma^{-1}(i)}^{\sigma}, 2C_e B_n \log(pn)/\sqrt{n}\big) U_{jk\ell r}(W_{\sigma^{-1}(i)}^{\sigma})] \EE [|V_{i \ell} V_{i r}|]\\
		&  \lesssim \EE  [ \chi^V_i  \chi^Z_i h_{s,\delta}^v \big(W_{\sigma^{-1}(i)}^{\sigma}, 2C_e B_n \log(pn)/\sqrt{n}\big) U_{jk\ell r}(W_{\sigma^{-1}(i)}^{\sigma})] \EE [|V_{i \ell} V_{i r}|]\\
		& \lesssim \EE  [ \chi^V_i  \chi^Z_i h_{s,\delta}^v \big(W, 4C_e B_n \log(pn)/\sqrt{n}\big) U_{jk\ell r}(W_{\sigma^{-1}(i)}^{\sigma})] \EE [|V_{i \ell} V_{i r}|]\\
		& \lesssim \EE  [ h_{s,\delta}^v \big(W, 4C_e B_n \log(pn)/\sqrt{n}\big) U_{jk\ell r}(W)] \EE [|V_{i \ell} V_{i r}|] \numberthis \label{eq: step 3 mid I22},
	\end{align*}
	which gives the equivalence of equation (83) of \cite{cck2022improvedbootstrap} in our setting. Also by    \eqref{eq: delta scaling}, we have that $ C_e B_n \log (pn) /\sqrt{n} \le \delta$, and again by Condition~\ref{cond: anti-con rate} we have 
	\begin{align*}
		\EE\big(h_{s,\delta}^v \big(W, 4C_e B_n \log(pn)/\sqrt{n}\big)\big)& \le \EE\big(h^v_{s,\delta}(S_n^Z,  4C_e B_n \log(pn)/\sqrt{n}) \big) + 2 \EE (\varrho_{\epsilon^{d+1}} )\\
		& \lesssim r_n \delta + \varepsilon_n +   \EE (\varrho_{\epsilon^{d+1}}),
	\end{align*}
	and then following the  arguments in Step~3 of \cite{cck2022improvedbootstrap}, we have 
	$$
	|\cI_{2,2}| \lesssim \frac{B_n^2 \delta^{-4} \log^3 p}{n} \big(r_n \delta + \varepsilon_n +   \EE (\varrho_{\epsilon^{d+1}})\big) + \frac{B_n^2 \delta^{-4} \log^5 (pn)}{n^2} .
	$$
	Hence,  \eqref{eq: f t 2nd} follows.

	\subsubsection{Proof of Claim~\ref{claim: f t 3rd} via modification of Step~4 in \cite{cck2022improvedbootstrap}}\label{sec: proof claim f t 3rd}
	The modification of Step~4 is very similar to that of Step~3. For the error terms $\cI_{3,1}$ and $\cI_{3,2}$ defined in Step~4 of \cite{cck2022improvedbootstrap},  applying our redefined anti-concentration Condition~\ref{cond: anti-con rate} and following the same steps as in Step~3, we have
	$$
	|\cI_{3,1}| \lesssim \frac{\cB_{n,2,d} \delta^{-3} \log^2 p}{{n}} \left( r_n \delta + \varepsilon_n  +  \EE (\varrho_{\epsilon^{d+1}})\right).
	$$
	For the error term $\cI_{3,2}$, following similar arguments as in showing \eqref{eq: equiv 4th deriv phi} and \eqref{eq: step 3 mid I22}, for any $j,k,\ell, r, h \in [p]$ and $i \in I_d$, we have that
	\begin{align*}
		&\EE\left(\chi_i^V \left|\partial_j \partial_k \partial_{\ell} \partial_r \partial_h \varphi_s^v (W_{\sigma^{-1}(i)}^{\sigma} + \frac{\hat{t}V_i}{\sqrt{n}}) V_{ir} V_{ih}\right|\right) \\
		& \quad \lesssim \EE \left(h_{s,\delta}^v \big(W, 4C_e B_n \log(pn)/\sqrt{n}\big) U_{jk\ell r h}(W) \right) \EE(|V_{ir} V_{ih}|),
	\end{align*}
	where $\hat{t} \in (0,1)$. By Condition~\ref{cond: anti-con rate} and following the arguments in Step~4 of \cite{cck2022improvedbootstrap}, we have
	$$
	|\cI_{3,2}| \lesssim \frac{B_n^3 \delta^{-5} \log^5 (pn)}{n^{3/2}} \big(r_n \delta + \varepsilon_n +   \EE (\varrho_{\epsilon^{d+1}})\big) + \frac{B_n^2 \delta^{-4} \log^5 (pn)}{n^2} ,
	$$
	which implies that  \eqref{eq: f t 3rd} holds.

	\subsubsection{Proof of Claim~\ref{claim: f t 4th} via modification of Step~5 in \cite{cck2022improvedbootstrap}}\label{sec: proof claim f t 4th}
	
	For the error terms $\cI_{4,1}$ and $\cI_{4,2}$ defined at the beginning of Step~5 in \cite{cck2022improvedbootstrap}, following the same arguments for showing \eqref{eq: equiv 4th deriv phi} and \eqref{eq: step 3 mid I22}, we have that 
	\begin{align*}
		& \frac{1}{n^2} \sum_{i \in I_d} \sum_{j,k,\ell,r \in [p]}  \EE \left(\chi^V_i \Big| \partial_{j}\partial_k \partial_{\ell} \partial_r \varphi_s^v\Big(W_{\sigma^{-1}(i)}^{\sigma} + \frac{\tilde{t}V_i}{\sqrt{n}}\Big) V_{i j} V_{i k} V_{i\ell} V_{i r}\Big|\right)\\
		& \quad \lesssim \frac{1}{n^2}\!\!\!\sum_{j,k,\ell,r \in [p]}\!\! \!\EE\! \left(h_{s,\delta}^v \big(W, 4C_e B_n \log(pn)/\sqrt{n}\big) U_{jk\ell r }(W) \!\right) \!\!\max_{j,k,\ell, r \in [p]} \sum_{i \in I_d} \EE \left(|V_{i j} V_{i k} V_{i\ell} V_{i r}|\right)\\
		& \quad \lesssim \frac{B_n^2 \delta^{-4} \log^3 p}{n} \big(r_n \delta + \varepsilon_n +   \EE (\varrho_{\epsilon^{d+1}})\big).
	\end{align*}
	Following the arguments in Step~5 of \cite{cck2022improvedbootstrap}, we have 
	$$
	|\cI_{4,1}| \lesssim \frac{B_n^2 \delta^{-4} \log^3 (p)}{n} \big(r_n \delta + \varepsilon_n +   \EE (\varrho_{\epsilon^{d+1}})\big) + \frac{B_n^2 \delta^{-4} \log^3 (p)}{n^2} .
	$$
	The bound for $\cI_{4,2}$ follows similar arguments, and \eqref{eq: f t 4th} holds. 
	
	% and the Slepian interpolant as $Z(t) = \sum_{i=1}^n Z_i(t)$, $t \in [0,1]$, where 
	% $$
	% Z_i(t) := \frac{1}{\sqrt{n}} \left\{\sqrt{t}[\sqrt{v} (\xi_i - \mu_i )+ \sqrt{1-v} (\eta_i - \mu_i)] + \sqrt{1-t}(W_i - \mu_i)\right\},
	% $$
	\subsection{Proof of Claim~\ref{claim: A31} for Lemma~\ref{lm: anti con equal var no cor 1}}\label{sec: proof claim A31}
	\begin{proof}[\unskip\nopunct]
		% Now we go back to show that Claim~\ref{claim: A31}
		% % and Claim~4 
		% holds for a.e. $u \in \RR$. 
		Recall that $\tilde{Z}_j = -Z_j + c_j - u$ for $j \in [d]$.    Given $j \in [d]$, for $i \in [d] \backslash \{j\}$, define the orthogonal residuals of projecting $Z_i$ and $\tilde{Z}_i$ onto $\tilde{Z}_j$ as $V_i = Z_i - \mu_i + \rho_{ij}(\tilde{Z}_j + \mu_j - c_j + u)$ and $\tilde{V}_i = - (Z_i - \mu_i) - \rho_{ij}(\tilde{Z}_j + \mu_j - c_j + u) = - V_i$. We have
		\begin{align*}
			H_{u,j}(x) &=  \PP\left( \begin{array}{ll}
				V_{i} \le x  - \mu_i + \rho_{ij}(\tilde{Z}_j + \mu_j - c_j + u) , & \, i \in [d] \backslash j \\
				\tilde{V}_i \le x + \mu_{i} - c_{i} + u - \rho_{ij}(\tilde{Z}_j + \mu_j - c_j + u), & \, i \in [d] \backslash j\\
				c_j - u - x \le x  & 
			\end{array}  \Bigg|
			\tilde{Z}_j = x \right)\\
			& =  \PP\left( \begin{array}{ll}
				V_{i} \le x  - \mu_i + \rho_{ij}(x + \mu_j - c_j + u) , & \, i \in [d] \backslash j\\
				- V_i \le x + \mu_{i} - c_{i} + u - \rho_{ij}(x + \mu_j - c_j + u), & \, i \in [d] \backslash j\\
				c_j - u - x \le x  & 
			\end{array} \right)\\
			& = \PP\left( \begin{array}{ll}
				U_i \le x , & \, i \in [d] \backslash j \\
				\tilde{U}_i \le x, & \, i \in [d] \backslash j\\
				(c_j - u)/2 \le x & 
			\end{array} \right),
		\end{align*}
		where 
		$$
		U_i \!= \!(1 + \rho_{ij})^{-1} (V_i + \mu_i - \rho_{ij}(\mu_j-c_j+u)\!),\, 
		\tilde{U}_i = \!(1 - \rho_{ij})^{-1} (\!-V_i - \mu_i + c_i  + \rho_{ij}(\mu_j-c_j)\!) - u.
		$$ 
		% Then it can be seen that $H_{u,j}(x)$ is non-decreasing on $x \in \RR$ and Claim~3 follows.
		When $x < (c_j - u)/2$, we have that $H_{u,j}(x) = g_{u, j}(x) = 0$, and \eqref{eq: claim 2 g = 0} holds.

		Next, we   show that  \eqref{eq: claim 2 g < H}  holds. When $x > (c_j - u)/2$, we have that 
		$$
		H_{u,j}(x) = \PP \big(\textstyle\max_{i \in [d]\backslash\{j\} } U_i \le x, \textstyle\max_{i \in [d]\backslash\{j\} } \tilde{U}_i \le x \big) .
		$$
		To apply \eqref{eq: max joint part x} of Lemma~\ref{lm: joint max dens} to bound  $d H_{u,j}(x) / d x$, we remove the duplicate terms in $\{U_i, \tilde{U}_i\}_{i \in [d]\backslash \{j\}}$. We first show that for any $i \ne i'$, $U_i - U_{i'} \not\equiv 0$ for a.e. $u \in \RR$. Indeed, if $U_i \equiv U_{i'}$, then $(1 + \rho_{ij})^{-1} V_i \equiv (1 + \rho_{i'j})^{-1} V_{i'} $, and by similar arguments for showing \eqref{eq: id var check} in  Claim~\ref{claim: dup term thm 2.5}, we have that $\rho_{ij} \ne \rho_{i'j}$, since otherwise we  have $Z_i - \mu_i \equiv Z_{i'} - \mu_{i'}$, which contradicts our assumption. Then we have that $U_i \equiv U_{i'}$ holds only at 
		$$
		u = \left(\frac{\rho_{ij}}{1+\rho_{ij}} -\frac{\rho_{i'j}}{1+\rho_{i'j}} \right)^{-1} \left(\frac{\mu_i - \rho_{ij}(\mu_j-c_j)}{1+\rho_{ij}}-\frac{\mu_{i'} - \rho_{i'j}(\mu_j-c_j)}{1+\rho_{i'j}}\right).
		$$
		Similarly, for any $i, i' \in [d]\backslash \{j\}$, if $U_i \equiv \tilde{U}_{i'}$, we have that 
		$$
		u = \frac{1+\rho_{ij}}{1-\rho_{i'j}}(  c_{i'} - \mu_{i'}  + \rho_{i'j}(\mu_j-c_j)) - (\mu_i - \rho_{ij}(\mu_j-c_j)).
		$$
		Hence, for a.e. $u \in \RR$, the duplicate terms may only be in $\{\tilde{U}_i\}_{i\ne j}$. Denote by $\tilde\cB_j \subseteq [d]$ the subset with duplicate terms in $\{\tilde{U}_i\}_{i\in [d] \backslash j}$ removed, i.e., let \(\tilde\cB_j\) be any subset of \([d] \backslash \{j\}\) such that for any \(i, i' \in \tilde\cB_j\) and \(i \ne i'\), we have \(\tilde{U}_i \not\equiv \tilde{U}_{i'}\). If \([d] \backslash \{j\} \backslash \tilde\cB_j \ne \emptyset\), then for any \(k \in [d] \backslash \{j\} \backslash \tilde\cB_j\), there exists a \(k' \in \tilde\cB_j\) such that \(\tilde{U}_k \equiv \tilde{U}_{k'}\). Then, by $x > (c_j - u)/2$, applying \eqref{eq: max joint part x} of Lemma~\ref{lm: joint max dens}, we have that
		\begin{align*}
			d H_{u,j}(x)/d x &= d \PP\left( \begin{array}{ll}
				U_i \le x , & \, i \ne j \\
				\tilde{U}_i \le x, & \, i \in \tilde\cB_j
			\end{array} \right) /d x= \underbrace{\sum_{i \in [d] \backslash \{j\}} \PP\left( \begin{array}{ll}
					U_{i'} \le x , & \, i' \ne i, j \\
					\tilde{U}_{i'} \le x, & \, i' \in \tilde\cB_j
				\end{array} \bigg| U_i = x \right) f_{U_{i}} (x) }_{\rm I_1 }\\
			& \quad \underbrace{+ \sum_{i \in \tilde\cB_j} \PP\left( \begin{array}{ll}
					U_{i'} \le x , & \, i' \ne j \\
					\tilde{U}_{i'} \le x, & \, i' \in \tilde\cB_j \backslash \{i\}
				\end{array} \bigg| \tilde{U}_i = x \right) f_{\tilde{U}_{i}} (x) }_{\rm I_2}\\
			& \overset{\rm (a)}{\ge} \sum_{i \in [d] \backslash \{j\}} \PP\left( \begin{array}{ll}
				U_{i'} \le x , & \, i' \ne i, j \\
				\tilde{U}_{i'} \le x, & \, i' \ne j\\
				(c_j - u)/2 \le x & 
			\end{array} \Bigg| U_i = x \right) f_{U_{i}} (x) \\
			& \overset{\rm (b)}{=} \sum_{i \in [d] \backslash \{j\}} \PP\left( \begin{array}{c}
				\max_{i' \ne i} Z_{i'} \le x   \\
				\max_{j' \ne j} \tilde{Z}_{j'} \le x 
			\end{array}  \bigg| \begin{array}{c}
				Z_i = x \\
				\tilde{Z}_j = x
			\end{array} \right) (1 + \rho_{ij})\phi_{Z_i|\tilde{Z}_j}^u (x| x) \\
			& \ge (1 + \underline\rho) g_{u, j}(x),
		\end{align*}
		where $f_{U_i}(\cdot)$ and $f_{\tilde{U}_i}(\cdot)$ denote the marginal densities for $U_i$ and $\tilde{U}_i$ respectively; inequality (a) holds as the terms in ${\rm I}_2$ are positive, and adding back constraints on the duplicate terms in $\{\tilde{U}_{i'}\}_{i' \in [d] \backslash\{j\}}$ does not change the value of ${\rm I}_1$; equality (b) holds by the equivalence that 
		$$
		\{(c_j - u)/2 \le x | \tilde{Z}_j = x\} =   \{Z_j \le x | \tilde{Z}_j = x\} .
		$$
		Therefore, \eqref{eq: claim 2 g < H} holds for a.e. $u \in \RR$, which completes the proof.
	\end{proof}
	\subsection{Log-Sobolev inequalities for Gaussian random vectors}\label{sec: log sob ineq}
	\begin{proof}[\unskip\nopunct]\renewcommand{\qedsymbol}{}
		We provide the Log-Sobolev inequality for Gaussian random vectors \citep{Berestycki2019ConcentrationOM} for self-completeness. This lemma is essential in establishing tail probability bounds for Gaussian maximal statistics.
	\end{proof}
	\begin{lemma}\label{lm: log-sob ineq Gaussian}
		Let \(Z = (Z_1, \ldots, Z_d)^{\top}\) be a standard Gaussian random vector on \(\RR^d\), and let \(F: \RR^d \rightarrow \RR\) be an \(L\)-Lipschitz function. Then, we have \(\EE |F(Z) |< \infty\) and 
		$$
		\PP\big(F(Z) \ge \EE (F(Z) ) + r\big) \le \exp\big(-r^2 / (2L^2)\big), \quad \forall r > 0.
		$$
	\end{lemma}
	\begin{proof}
		The inequality holds by Theorem~9 and Theorem~11 in \citep{Berestycki2019ConcentrationOM}.
	\end{proof}
\section{Additional Simulation Results}\label{sec: app add plot}
In this section, we present additional simulation results deferred from Section~\ref{sec: simu}. Specifically, we explore how the Lévy concentration function \(\cL(M_{\cB} - M_{\cA}, \varepsilon)\) scales with \(1 - \bar{\rho}\), where \(\bar{\rho} = \max_{i \in \cA, j \in \cB} \Corr(X_i, X_j)\). We consider both full rank and low rank settings. In the full rank setting, we generate \((X_{\cA}, X_{\cB})\) by setting \(\sigma_j^2 = 1\) for all \(j =1,\ldots, p\) and \(\text{Cov}(X_i, X_j) = \rho\) for all \(i, j \in [p]\) with \(i \neq j\). We vary \(\rho\) across 100 equidistant values between 0.9 and 0.99 to explore how \(\cL(M_{\cB} - M_{\cA}, \varepsilon)\) scales with \(1 - \bar{\rho} = 1 - \rho\). For the low rank setting,  we generate \(X_{\cA}\) and \(X_{\cB}\) by \eqref{eq: simu Xa Xb gen}. Here, the entries of \(\Gamma_{\cA}\) and \(\Gamma_{\cB}\) are drawn independently from a standard Gaussian distribution, and the row norms of \(\Gamma_{\cA}\) and \(\Gamma_{\cB}\) are normalized to unity. We repeat the generation of \(\Gamma_{\cA}\) and \(\Gamma_{\cB}\) 100 times under the low rank setting to create 100 different values of \(\bar{\rho}\) and examine the scaling of \(\cL(M_{\cB} - M_{\cA}, \varepsilon)\) with respect to \(1/\sqrt{1-\bar{\rho}}\) and \(1/(1-\bar{\rho})\) respectively. For all settings, the L\'evy concentration function $\cL(M_{\cB} - M_{\cA}, \varepsilon)$ is evaluated at $\varepsilon = 0.05$ over 500 independent repetitions of the generation scheme.

The results of the full rank setting are shown in Figure~\ref{fig: levy scale 1 - rho}(a) and (b). The L\'evy concentration function scales linearly with $1/\sqrt{1 - \bar\rho}$ across all settings of $p \in \{2, 2000, 3000, 4000\}$ in Figure~\ref{fig: levy scale 1 - rho}(a), while a sublinear relationship is observed between the L\'evy concentration function and $1 / (1-\bar\rho)$ in Figure~\ref{fig: levy scale 1 - rho}(b). The results suggest the potential to improve the extra factor of $\sqrt{1 - \bar\rho}$ in the anti-concentration rate of Theorem~\ref{col: anti con equal var no cor 1} and Theorem~\ref{thm: anti con}.

Figure~\ref{fig: levy scale 1 - rho}(c) and (d) display the results from the low rank setting. When \(p = 2\), where \(M_{\cB} - M_{\cA} = X_2 - X_1\) reduces to a Gaussian random variable, we observe that \(\cL(X_2 - X_1, \varepsilon)\) increases as \(1/\sqrt{1-\bar\rho}\) or \(1/(1-\bar\rho)\) increase. For \(p \ge 3\), this trend is less pronounced, indicating that there may be room to refine the anti-concentration bounds when the correlations are non-uniform.
\begin{figure}[htbp]
    \centering
    \begin{tabular}{cc}
         (a) $\cL(\max_{j \in [p]} X_j, \varepsilon)$ & (b) Density curves \\
         \!\!\!\!\includegraphics[width = 202 pt]{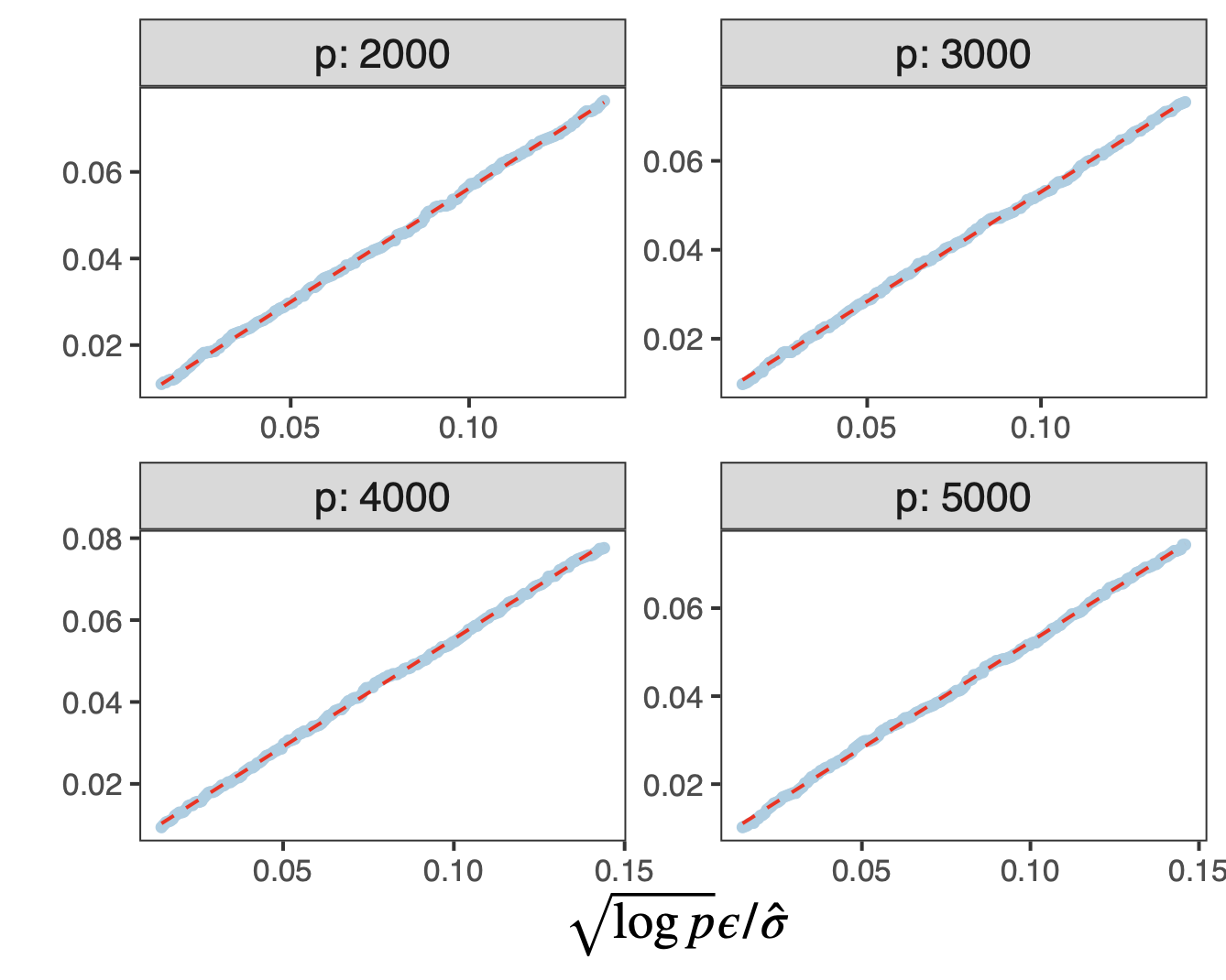}&  \!\!\!\!\!\includegraphics[width = 202 pt]{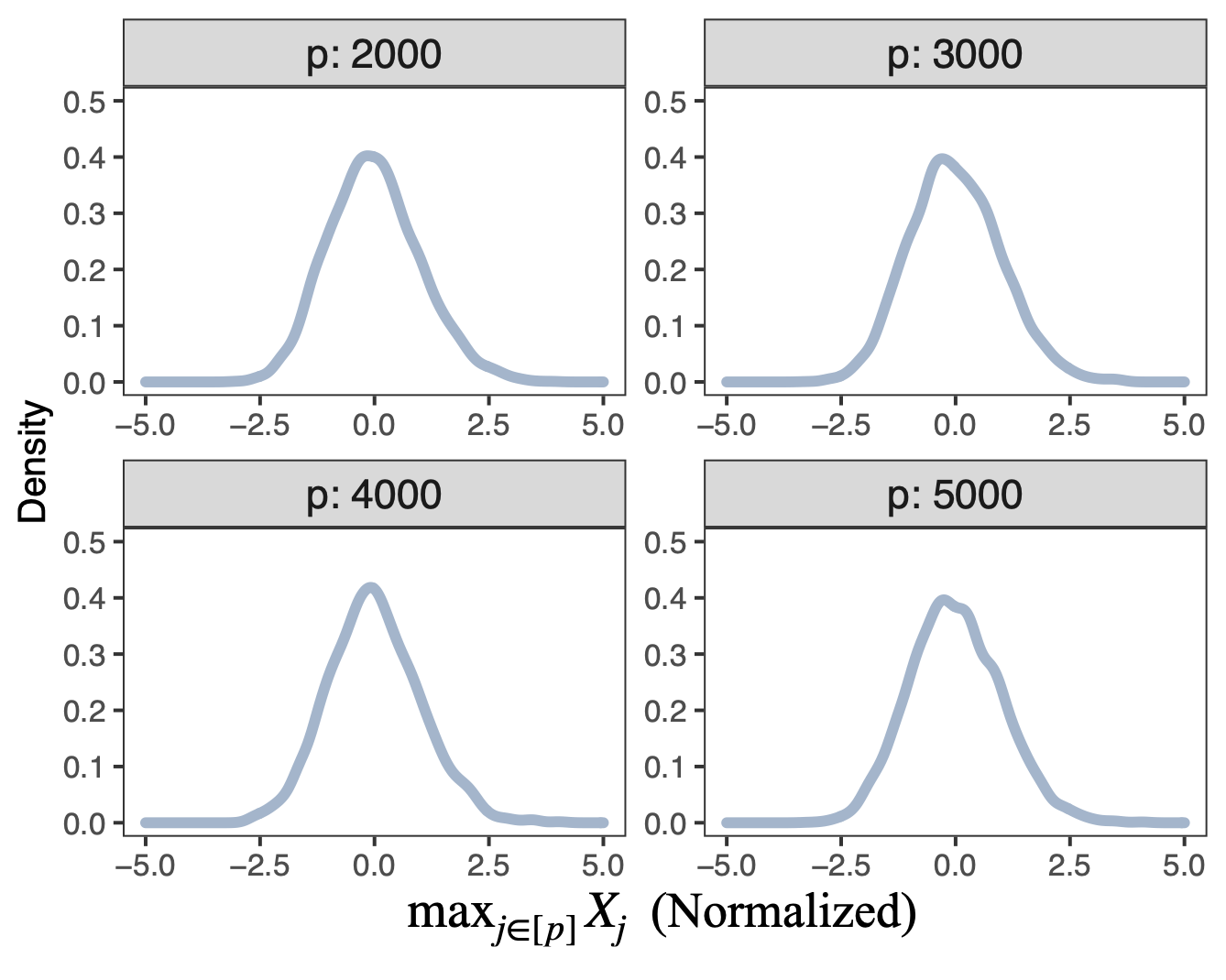}
    \end{tabular}
    \caption{Anti-concentration results of the single maximum $\max_{j \in [p]} X_j$ as a benchmark for the homogeneous variance case with $\mu = \mathbf{0}$. (a) Empirical evaluation of the L\'evy concentration function scaling linearly with the interval length $\varepsilon$; (b) For the density curves, $\max_{j \in [p]}X_j$ is normalized by the sample mean and standard deviation for comparability. }\label{fig: reference single max}
\end{figure}
\begin{figure}[htbp]
    \centering
    \begin{tabular}{cc}
         (a) $\cL(M_{\cB} - M_{\cA}, 0.05)$ & (b) $\cL(M_{\cB} - M_{\cA}, 0.05)$ \\
         \!\!\!\!\includegraphics[width = 202 pt]{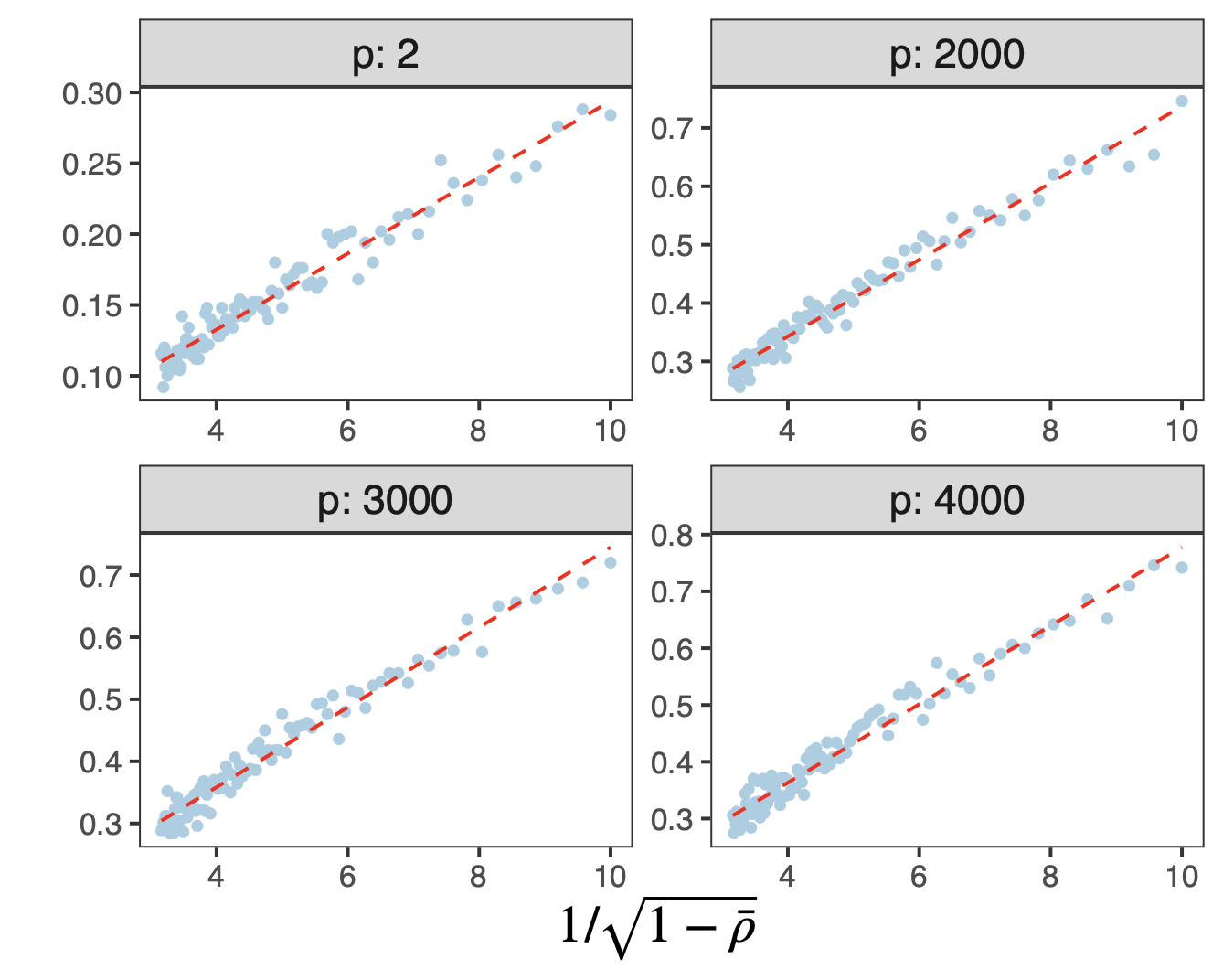}&  \!\!\!\!\!\includegraphics[width = 202 pt]{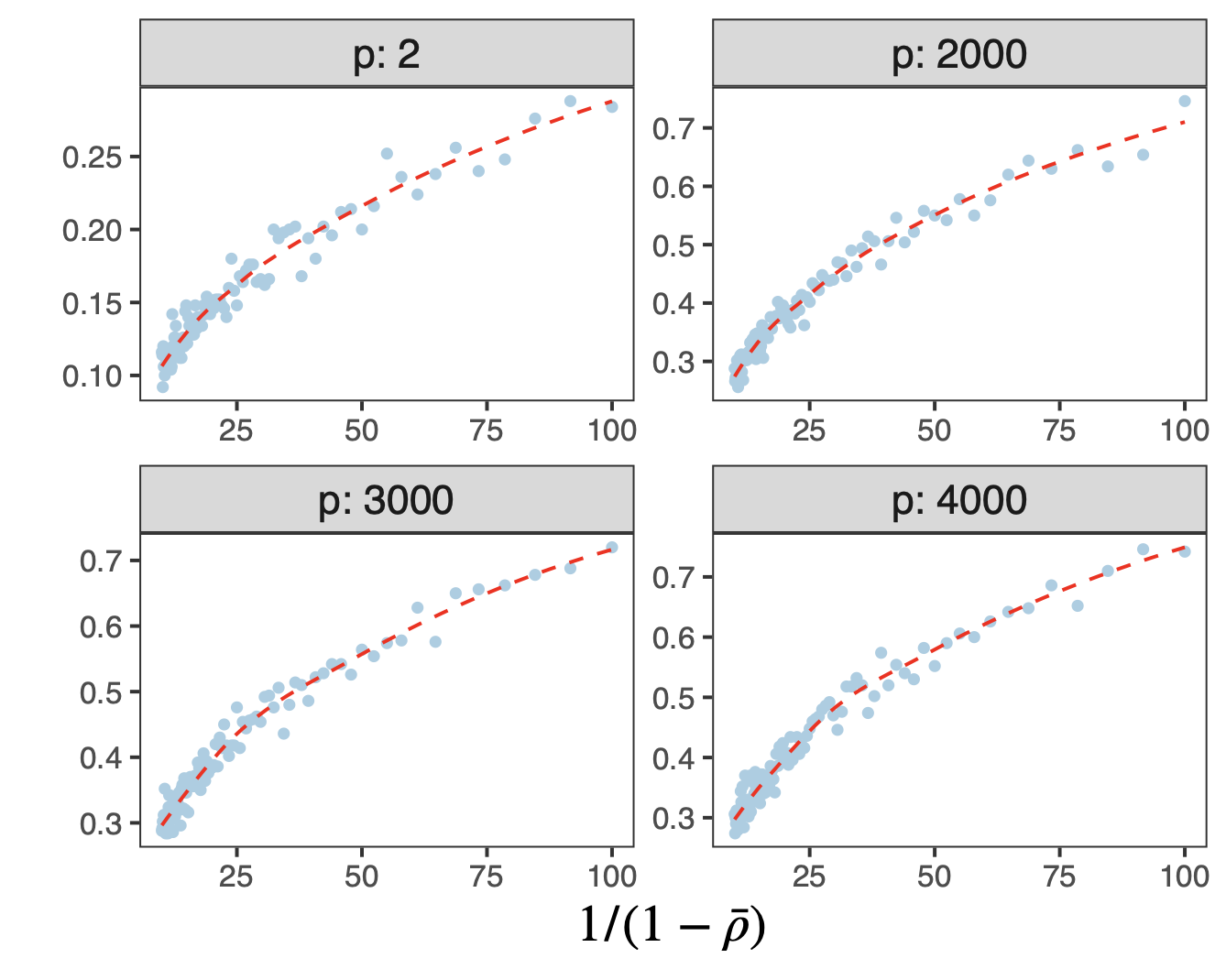}\\
         (c) $\cL(M_{\cB} - M_{\cA}, 0.05)$ & (d) $\cL(M_{\cB} - M_{\cA}, 0.05)$ \\
         \!\!\!\!\includegraphics[width = 202 pt]{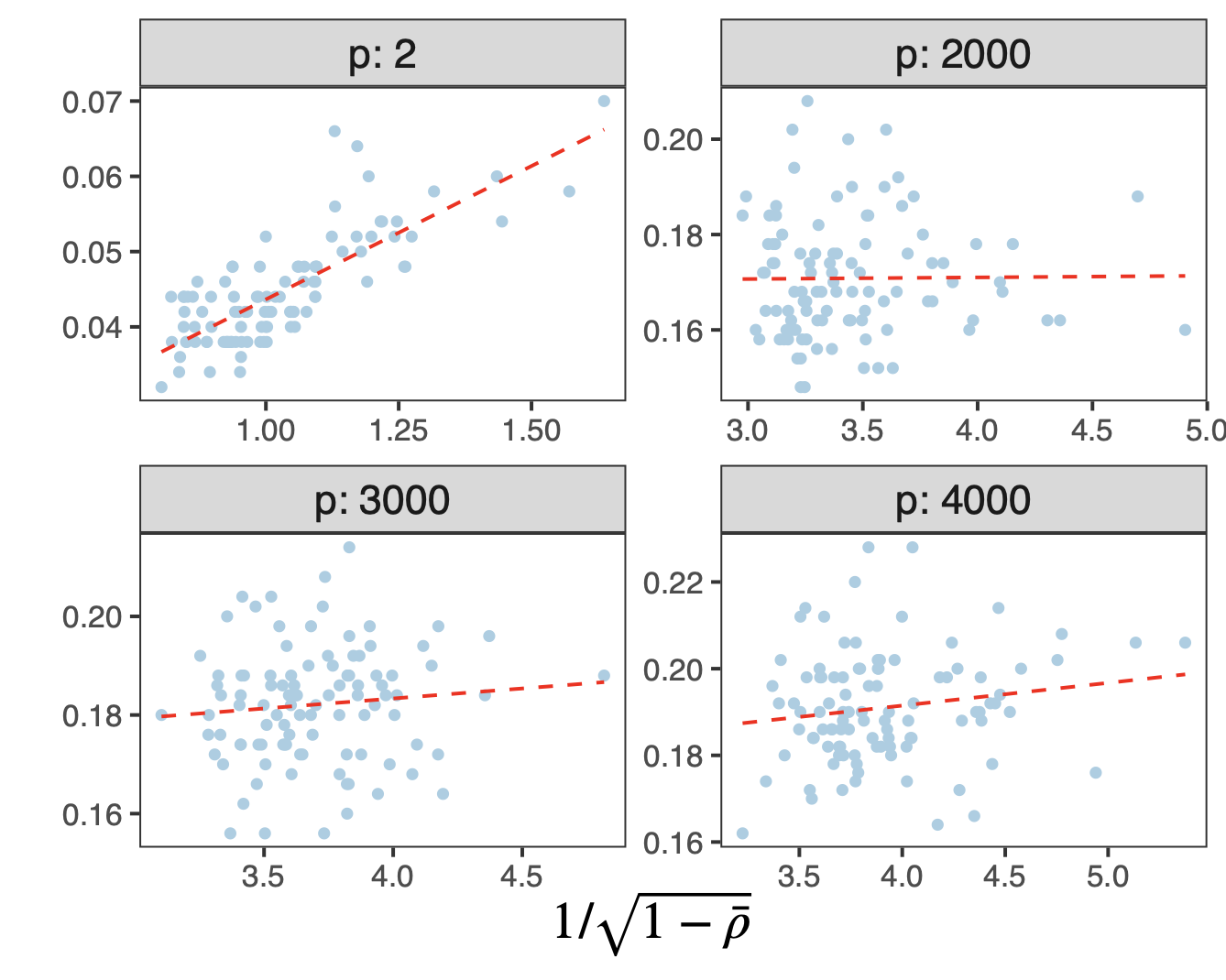}&  \!\!\!\!\!\includegraphics[width = 202 pt]{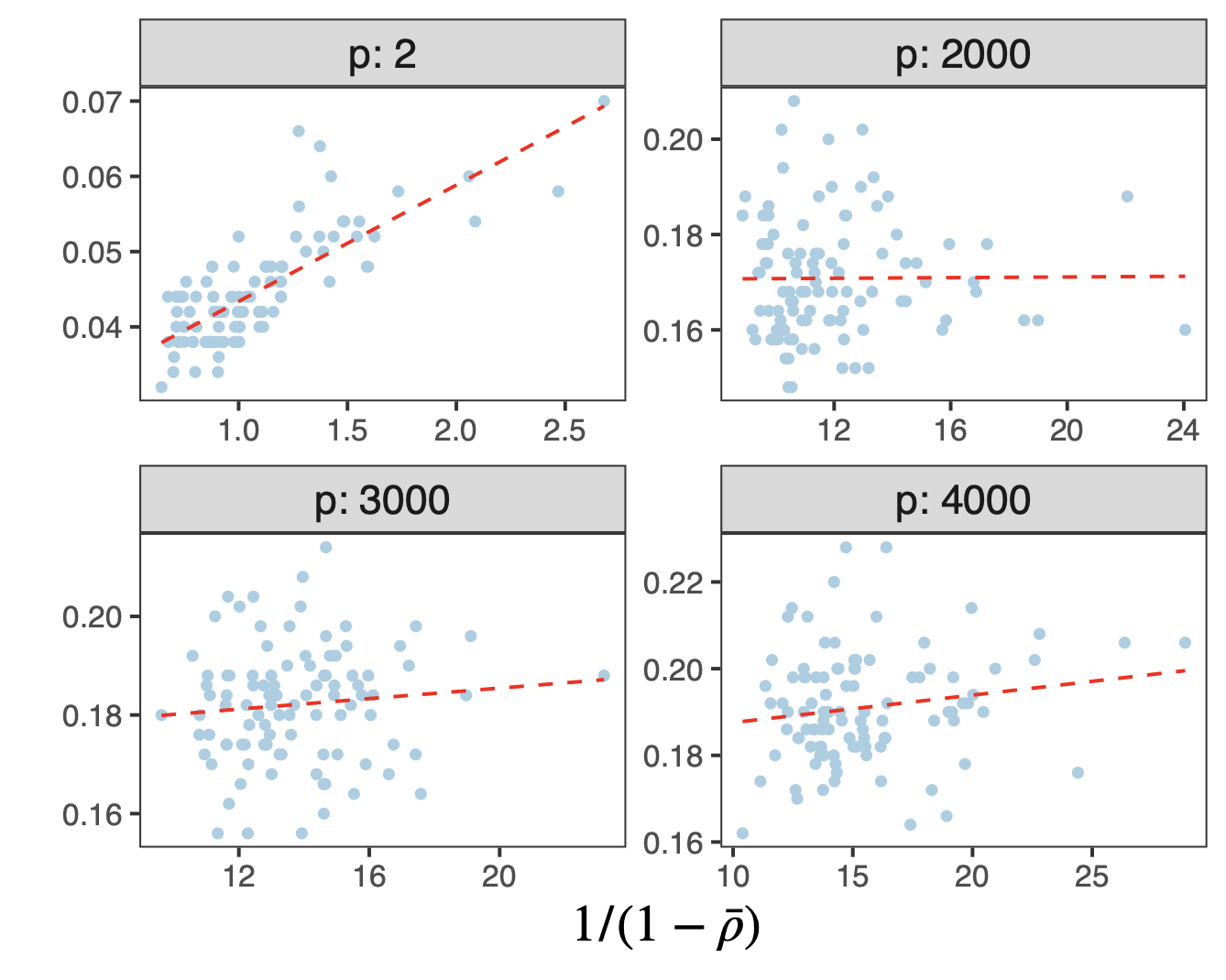}
    \end{tabular}
    \caption{Scaling of the L\'evy concentration function at $\varepsilon = 0.05$ with $1/\sqrt{1-\bar\rho}$ and $1/({1-\bar\rho})$ respectively. (a) and (b) present the results under the full rank setting, where $\sigma_j^2 = 1$ for all $j \in [p]$ and ${\rm Cov}(X_i, X_j) = \rho$ for all \(i, j \in [p]\) with \(i \neq j\). The L\'evy concentration grows linearly with $1/\sqrt{1-\bar\rho}$ and sublinearly with $1/({1-\bar\rho})$. (c) and (d) present the results under the low rank setting. The L\'evy concentration increases with $1/\sqrt{1-\bar\rho}$ and $1/({1-\bar\rho})$ at $p = 2$, while no significant relationship is observed when $p > 2$. }\label{fig: levy scale 1 - rho}
\end{figure}
\end{appendix}
\newpage
\bibliographystyle{imsart-nameyear} % Style BST file 
	\bibliography{sn-bibliography.bib}
\end{document}